\documentclass[10pt,reqno]{amsart}
\usepackage{amssymb,mathrsfs,color,mathtools, empheq}
\usepackage{pinlabel}

\mathtoolsset{showonlyrefs}

\usepackage{enumerate}

\usepackage{graphicx}
\usepackage{xcolor} 
\usepackage{tensor}
\usepackage{slashed}
\usepackage{cite}
\usepackage{cancel}
\usepackage{verbatim}
\usepackage{dsfont}
\usepackage{cite}
\usepackage[normalem]{ulem}

\usepackage{enumitem}

\definecolor{green}{rgb}{0,0.5,0} 

\newcommand{\nrm}[1]{\Vert#1\Vert}
\newcommand{\brk}[1]{\langle#1\rangle}
\newcommand{\Angles}[2]{\left\langle #1,#2\right\rangle}
\newcommand{\angles}[2]{\langle #1,#2\rangle}

\newcommand{\set}[1]{\{#1\}}

\renewcommand{\Re}{\mathrm{Re}}
\renewcommand{\Im}{\mathrm{Im}}

\newcommand{\aeq}{\simeq}
\newcommand{\aleq}{\lesssim}
\newcommand{\ageq}{\gtrsim}

\newcommand{\lap}{\Delta}

\newcommand{\ud}{\mathrm{d}}
\newcommand{\rd}{\partial}
\newcommand{\nb}{\nabla}

\newcommand{\imp}{\Rightarrow}

\newcommand{\alp}{\alpha}
\newcommand{\bt}{\beta}
\newcommand{\gmm}{\gamma}
\newcommand{\Gmm}{\Gamma}
\newcommand{\dlt}{\delta}
\newcommand{\Dlt}{\Delta}
\newcommand{\eps}{\epsilon}
\newcommand{\veps}{\varepsilon}
\newcommand{\kpp}{\kappa}

\newcommand{\sgm}{\sigma}

\newcommand{\tht}{\theta}

\newcommand{\omg}{\omega}

\newcommand{\bfd}{{\bf d}}

\newcommand{\bfh}{{\bf h}}

\newcommand{\bfD}{{\bf D}}

\newcommand{\bbC}{\mathbb C}

\newcommand{\bbH}{\mathbb H}

\newcommand{\bbN}{\mathbb N}

\newcommand{\bbR}{\mathbb R}
\newcommand{\bbS}{\mathbb S}

\newcommand{\bbZ}{\mathbb Z}

\newcommand{\calA}{\mathcal A}
\newcommand{\calB}{\mathcal B}

\newcommand{\calN}{\mathcal N}

\newcommand{\calP}{\mathcal P}

\newcommand{\calX}{\mathcal X}

\usepackage{wrapfig}
\usepackage{tikz}
\usetikzlibrary{arrows,calc,decorations.pathreplacing}
\definecolor{light-gray1}{gray}{0.90}
\definecolor{light-gray2}{gray}{0.80}

\definecolor{deepgreen}{cmyk}{1,0,1,0.5}

\newcommand{\Hp}{\mathbb{H}}
\newcommand{\N}{\mathbb{N}}
\newcommand{\R}{\mathbb{R}}
\newcommand{\Sp}{\mathbb{S}}
\newcommand{\Z}{\mathbb{Z}}

\newcommand{\h}{\mathbf{h}}
\newcommand{\m}{\mathbf{m}}

\newcommand{\al}{\alpha}
\newcommand{\be}{\beta}
\newcommand{\ga}{\gamma}

\newcommand{\fy}{\varphi}
\newcommand{\om}{\omega}
\newcommand{\la}{\lambda}

\newcommand{\s}{\sigma}

\newcommand{\ka}{\kappa}

\newcommand{\De}{\Delta}

\newcommand{\Ga}{\Gamma}

\newcommand{\p}{\partial}
\newcommand{\na}{\nabla}
\newcommand{\re}{\mathop{\mathrm{Re}}}

\newcommand{\supp}{\operatorname{supp}}

\newcommand{\loc}{\operatorname{loc}}
\newcommand{\ext}{\operatorname{ext}}

\makeatletter

\newcommand{\Rmnum}[1]{\expandafter\@slowromancap\romannumeral #1@}
\makeatother

\newcommand{\I}{\infty}

\newcommand{\ti}{\widetilde}

\newcommand{\ha}{\widehat}

\newcommand{\ang}[1]{\left\langle{#1}\right\rangle}
\newcommand{\abs}[1]{\left\lvert{#1}\right\rvert}
\newcommand{\ds}{\displaystyle}

\newcommand{\ant}[1]{\begin{align*}\begin{split} #1 \end{split}\end{align*}}
\newcommand{\EQ}[1]{\begin{equation}\begin{split} #1 \end{split}\end{equation}}

\newcommand{\Del}[1]{}

\numberwithin{equation}{section}

\newtheorem{thm}{Theorem}[section]
\newtheorem{cor}[thm]{Corollary}
\newtheorem{lem}[thm]{Lemma}
\newtheorem{prop}[thm]{Proposition}
\newtheorem{defn}[thm]{Definition}
\newtheorem{defnprop}[thm]{Definition/Proposition}

\newtheorem{remark}[thm]{Remark}
\newtheorem{rem}[thm]{Remark}

\newcommand{\mand}{{\ \ \text{and} \ \  }}

\newcommand{\mif}{{\ \ \text{if} \ \ }}

\newcommand{\mas}{{\ \ \text{as} \ \ }}

\renewcommand\Re{\mathrm{Re}\,}
\renewcommand\Im{\mathrm{Im}\,}

\newcommand{\dsp}{\frac{\ud s^\prime}{s^\prime}}
\renewcommand{\ds}{\frac{\ud s}{s}}
\newcommand{\dspp}{\frac{\ud s^{\prime\prime}}{s^{\prime\prime}}}

\newcommand{\snabla}{{\slashed{\nabla}}}

\newcommand{\tils}{{\tilde{s}}}

\newcommand{\tilw}{{\tilde{w}}}

\newcommand{\tilB}{{\tilde{B}}}

\newcommand{\tilF}{{\tilde{F}}}

\newcommand{\tilP}{{\mathbf{P}}}

\newcommand{\tilV}{{\tilde{V}}}

\newcommand{\ringH}{{\mathring H}}

\newcommand{\prin}{\mathrm{prin}}
\newcommand{\pert}{\mathrm{pert}}
\newcommand{\stat}{\mathrm{stat}}
\newcommand{\lot}{\mathrm{l.o.t.}}
\newcommand{\sym}{\mathrm{sym}}
\newcommand{\antisym}{\mathrm{a.s.}}

\newcommand{\ube}{\bsbeta}

\newcommand{\RR}{\mathcal{R}}

\newcommand{\bsb}{\boldsymbol{b}}
\newcommand{\bsa}{\boldsymbol{a}}

\newcommand{\low}{{\mathrm{low}}}

\newcommand{\jap}[1]{\langle #1\rangle}
\renewcommand{\dh}{\mathrm{dvol}_{\bfh}}
\newcommand{\tilphi}{{\widetilde{\phi}}}

\newcommand{\bsxi}{{\boldsymbol{\xi}}}
\newcommand{\bsz}{{\boldsymbol{z}}}
\newcommand{\Pea}{P}
\newcommand{\bsh}{{\boldsymbol{h}}}
\newcommand{\bsm}{{\boldsymbol{m}}}
\newcommand{\bsR}{{\boldsymbol{R}}}
\newcommand{\bsK}{{\boldsymbol{K}}}
\newcommand{\bsf}{{\boldsymbol{f}}}
\newcommand{\bsbeta}{{\boldsymbol{\beta}}}
\newcommand{\ringa}{{\mathring{\boldsymbol{a}}}}
\renewcommand{\RR}{\bsR}
\renewcommand{\h}{\bsh}
\renewcommand{\m}{\bsm}
\newcommand{\tilLE}{{\mathcal{LE}}}
\newcommand{\LE}{{\mathbb{LE}}}
\newcommand{\dsnot}{\frac{\ud s_0}{s_0}}
\newcommand{\bsg}{{\boldsymbol{g}}}
\newcommand{\bsT}{{\boldsymbol{T}}}
\newcommand{\bsX}{{\boldsymbol{X}}}
\newcommand{\bsw}{{\boldsymbol{w}}}
\newcommand{\nsigma}{\sigma}
\newcommand{\pfstep}[1]{\vskip.5em \noindent{\it #1.}}

\begin{document}

\title[Local Smoothing on hyperbolic space ]{Local smoothing estimates for Schr\"odinger equations on hyperbolic space}
\author{A. Lawrie}
\author{J. L\"uhrmann}
\author{S.-J. Oh}
\author{S. Shahshahani}

\begin{abstract}
We establish global-in-time frequency localized local smoothing estimates for Schr\"odinger equations on hyperbolic space $\bbH^d$. In the presence of symmetric first and zeroth order potentials, which are possibly time-dependent, possibly large, and have sufficiently fast polynomial decay, these estimates are proved up to a localized lower order error. Then in the time-independent case, we show that a spectral condition (namely, absence of threshold resonances) implies the full local smoothing estimates (without any error), after projecting to the continuous spectrum. In the process, as a means to localize in frequency, we develop a general Littlewood-Paley machinery on $\bbH^d$ based on the heat flow. Our results and techniques are motivated by applications to the problem of stability of solitary waves to nonlinear Schr\"odinger-type equations on $\bbH^{d}$. Specifically, some of the estimates established in this paper play a crucial role in the authors' proof of the nonlinear asymptotic stability of harmonic maps under the Schr\"odinger maps evolution on the hyperbolic plane; see~\cite{LLOS2}. 

As a testament of the robustness of approach, which is based on the positive commutator method and a heat flow based Littlewood-Paley theory, we also show that the main results are stable under small time-dependent perturbations, including polynomially decaying second order ones, and small lower order nonsymmetric perturbations. 
\end{abstract}


\thanks{
A. Lawrie was supported by NSF grant DMS-1700127. S.-J. Oh was supported in part by Samsung Science and Technology Foundation under Project Number SSTF-BA1702-02. The authors thank the Forschungsinstitut Oberwolfach for its hospitality where part of this work was conducted. They are grateful to A. Nahmod for discussions on the history of heat flow based Littlewood-Paley theory.
}

\maketitle

\section{Introduction} \label{s:intro}
The primary goal of this work is to establish local smoothing, or local energy decay, estimates for the linear Schr\"odinger equation
\begin{align*}
\begin{split}
(-i\partial_t +H)u=F,\qquad u(0)=u_0,
\end{split}
\end{align*}
on hyperbolic space $(\bbH^d,\bsh)$. Here $\bsh$ denotes the standard Riemannian metric on $\bbH^d$ and $H$ is the elliptic operator
\begin{align*}
\begin{split}
Hu := -\nabla_\mu\bsa^{\mu\nu}\nabla_\nu u+\frac{1}{i}(\bsb^\mu\nabla_\mu u+\nabla_\mu(\bsb^\mu u))+Vu.
\end{split}
\end{align*}
The precise assumptions on the coefficients are specified in Section~\ref{s:statement} below, but without being quantitative $\bsa$ is a symmetric real-valued two tensor with $\bsa-\bsh^{-1}$ small and decaying at spatial infinity, and $\bsb$ and $V$ decaying at spatial infinity with globally small imaginary parts. In particular, we allow $\bsa$, $\bsb$, and $V$ to depend on time. Roughly speaking, by a local smoothing estimate we mean an estimate of the form
\begin{align*}
\begin{split}
\|\jap{r}^{-\frac{1}{2}-\delta}|\nabla|^{\frac{1}{2}}u\|_{L^2(\bbR\times \bbH^d)}\lesssim \|u_0\|_{L^2(\bbH^d)}+\|\jap{r}^{\frac{1}{2}+\delta}|\nabla|^{-\frac{1}{2}}F\|_{L^2(\bbR\times \bbH^d)},\quad \delta>0,
\end{split}
\end{align*}
where $r$ denotes the geodesic distance to a fixed origin in $\bbH^d$, and $\jap{r}:=\sqrt{1+r^2}$. 

However, with an eye towards nonlinear applications, we will prove a sharper frequency localized version of the smoothing estimate above. That is, for a frequency localized function, the spatial weights and localizations in the local smoothing norms for our final estimate will depend on the frequency at which the function is localized; our estimate will be optimal in small spatial scales and for high frequencies. The complete local smoothing norm is then constructed by resolving $u$ into its frequency components (see Section~\ref{s:statement} below).  In fact, estimates established in this paper have a specific nonlinear application in the authors' recent proof of the nonlinear asymptotic stability of harmonic maps under the Schr\"odinger maps evolution on the hyperbolic plane in~\cite{LLOS2}; see Remark~\ref{r:combined}. 
The proof of our smoothing  estimate is based on a multiplier, or positive commutator, argument and as such is quite flexible and for instance allows us to treat time-dependent and complex-valued coefficients. 

The issue of frequency localization brings us to the second goal of the current work, which is to develop a robust Littlewood-Paley theory based on the heat flow on hyperbolic space. The bulk of this paper is devoted to this task, and the proof of the local smoothing estimate is short in comparison. Some of the specific issues that we need to treat in this regard are the failure of sharp Bernstein inequalities, and a quantitative understanding of the uncertainty principle, that is, the relation between spatial and frequency localizations. The heat flow based Littlewood-Paley theory developed in this work is abstract and the estimates are derived using integration by parts and the Duhamel principle for the heat equation, and therefore does not rely on the geometry of the hyperbolic space in an essential way, except for general spectral properties of the Laplacian. In particular, we systematically avoid heat kernel estimates and the Helgason Fourier transform, and this part of the analysis also extends to more general manifolds. We believe that this aspect of our work is of independent interest and hope that the ideas here can be used more broadly in the study of dispersive equations on manifolds.


\subsection{Statements of the main results} \label{s:statement}

Consider the linear Schr\"odinger equation on hyperbolic space $\Hp^d$ 
\EQ{ \label{eq:S} 
(-i \partial_t  + H )u =  F
}
with initial data $u(0) = u_0 \in L^2(\Hp^d)$. Here $H$ is a second order operator
\begin{equation} \label{eq:H_def}
 Hu := -\nabla_\mu\bsa^{\mu\nu}\nabla_\nu u+\frac{1}{i}(\bsb^\mu\nabla_\mu u+\nabla_\mu(\bsb^\mu u))+Vu,
\end{equation}
where $V(t, \cdot)$ is a smooth, complex-valued function, $\bsb(t, \cdot)$ is a smooth complex-valued vectorfield on $\bbH^{d}$, and $\bsa(t, \cdot)$ is a smooth real-valued symmetric two-tensor on $\bbH^{d}$.
We say that $H$ is \emph{stationary} if the coefficients are independent of time, and \emph{symmetric} if they are real-valued. In general $H_{\prin}$ will denote the principal part of $H$, that is,
\begin{align}\label{eq:Hprindef}
\begin{split}
H_{\prin}u:=-\nabla_\mu\bsa^{\mu\nu}\nabla_\nu u,
\end{split}
\end{align}
and $H_\lot$ the lower order part, that is,
\begin{align}\label{eq:Hlotdef}
\begin{split}
H_\lot u:=\frac{1}{i}(\bsb^\mu\nabla_\mu u+\nabla_\mu(\bsb^\mu u))+Vu.
\end{split}
\end{align}

\subsubsection{Main results for large symmetric lower order perturbations of $-\lap$}

Our first class of results concern the case when $H_{\prin} = -\Dlt := - \Dlt_{\bbH^{d}}$ and $H_{\lot}$ is symmetric, suitably decaying yet possibly large. Such an operator naturally arises as the linearization of a nonlinear Schr\"odinger-type equation on $\bbH^{d}$ around a solitary wave.

In order to capture the full strength of the local smoothing effect, which is crucial for recovering a derivative loss in Schr\"odinger-type equations, we aim to establish \emph{frequency localized local smoothing estimates} with spatial weights that are sharp for high frequencies and small spatial scales. 
One way to achieve frequency localization would be to use the Helgason Fourier transform on $\bbH^{d}$, and apply smooth cutoffs in the frequency space. As is well-known, however, such operations do not behave well in $L^{p}(\bbH^{d})$ (see \cite{ClSt}), and thus are not suitable in nonlinear applications. 

Instead, we will use the linear heat flow associated with the shifted Laplacian $\Delta+\rho^2$ on $\bbH^d$ as a robust means to localize frequencies. Here, $\rho^{2}$ is the spectral gap for $- \De$, whose spectrum is entirely absolutely continuous:
\EQ{ \label{eq:rho}
\s(-\De) = \s_{ac}(-\De) = [ \rho^2, \infty), \quad \rho:= \frac{d-1}{2}.
}
For any heat time $0 < s < \infty$ we define the frequency projections
\begin{equation}\label{eq:tilPdef}
 \tilP_{\geq s} u := e^{s(\Delta+\rho^2)} u,\qquad \tilP_s u := -s\partial_s \tilP_{\geq s}u=- s (\Delta +\rho^2)e^{s(\Delta+\rho^2)} u.
\end{equation}
For any $0 < s_0 < \infty$, we have the continuous resolution
\[
 u =  \int_0^{s_0} \tilP_s u \, \ds + \tilP_{\geq s_0} u.
\]

\begin{rem}[Heat flow based Littlewood-Paley theory] \label{rem:LP}
Use of the heat flow for frequency localization is classical. See for instance the books \cite{St70Book, CoifmanWeiss1}, or the papers \cite{KR06, CoifmanWeiss2}, for their application in frequency localization on manifolds. For further applications in harmonic analysis on non-Euclidean spaces including spaces of homogeneous type see \cite{CoifmanWeiss1,Nahmod1,Nahmod2} and the references therein.

The operators $\tilP_{\geq s}$ and $\tilP_{s}$ constitute \emph{heat flow based Littlewood-Paley projections} in the following sense. First, $\tilP_{\geq s} u = e^{s (\Dlt + \rho^{2})} u$ may be interpreted as a projection of $u$ to frequencies at most $s^{-\frac{1}{2}}$; this interpretation can be justified on $\bbR^{d}$ by observing that the Fourier transform of the heat flow $e^{s \lap} u$ is simply that of $u$ multiplied by a rescaled Gaussian adapted to the ball $\set{\abs{\xi} \aleq s^{-\frac{1}{2}}}$. On the other hand, $\tilP_{s} u = - s (\lap + \rho^{2}) e^{s (\lap + \rho^{2})} u$ is the analogue of a Littlewood-Paley projection to frequencies comparable to $s^{-\frac{1}{2}}$, the idea being that $s (\lap + \rho^{2})$ damps the frequencies much lower than $s^{-\frac{1}{2}}$. 

The heat flow based Littlewood-Paley theory, as described above, has successfully served as a viable substitute for the (naive but problematic) Fourier transform based Littlewood-Paley theory in the study of nonlinear dispersive equations on hyperbolic space; see, for example, \cite{IPS, LOS2, LOS5}.
\end{rem}

For each $\ell \in \bbZ$, the dyadic spatial annulus $A_\ell$ is defined by
\begin{align}\label{eq:rAldef}
 A_\ell:=\{2^{\ell}\leq r\leq 2^{\ell+1}\},\qquad r(x):=d(0,x), \quad \forall x\in\bbH^{d},
\end{align} 
where $0$ is a fixed origin in $\bbH^d$. Similarly the ball $A_{\leq\ell}$ and $A_{\geq\ell}$ are defined by
\begin{align*}
\begin{split}
A_{\leq\ell} := \{r\leq 2^{\ell}\},\qquad A_{\geq \ell}:=\{r\geq 2^{\ell}\}.
\end{split}
\end{align*}
For any heat time $s > 0$ we use the notation $k_s$ to denote the corresponding dyadic frequency
\begin{align*}
\begin{split}
 k_s := \lfloor \log_2 ( s^{-\frac{1}{2}} ) \rfloor .
\end{split}
\end{align*}
We can now define our frequency localized local smoothing spaces.
\begin{defn} \label{d:LEs}
For any $s>0$ the frequency localized local smoothing norm $\|\cdot\|_{LE_s}$ and its dual $\|\cdot\|_{LE_s^\ast}$ are defined as
\begin{align*}
 \| v \|_{LE_s} &:= s^{-\frac{1}{4}} \| v \|_{L^2(\R \times A_{\leq-k_s})} + \sup_{\ell \geq-k_s} \| r^{-\frac{1}{2}} v \|_{L^2(\R \times A_\ell)}, \\
 \| G \|_{LE_s^\ast} &:= s^{\frac{1}{4}} \| G \|_{L^2(\R \times A_{\leq-k_s})} + \sum_{\ell \ge-k_s} \| r^{\frac{1}{2}} G \|_{L^2( \R \times A_\ell)}.
\end{align*}
Similarly, the low-frequency analogues $\|\cdot\|_{LE_{\low}}$ and $\|\cdot\|_{LE_{\low}^\ast}$ are
\begin{align*}
 \| v \|_{LE_{\low}} &:= \|v\|_{L^2(\R \times A_{\leq0})} + \sup_{\ell \geq 0} \, \| r^{-\frac{3}{2}} v \|_{L^2(\R \times A_\ell)}, \\
 \| G \|_{LE_{\low}^\ast} &:= \| G \|_{L^2(\R \times A_{\leq0})} + \sum_{\ell \geq 0} \, \| r^{\frac{3}{2}} G \|_{L^2(\R \times A_\ell)}.
\end{align*}
\end{defn}

We build our local smoothing spaces $LE$ and $LE^\ast$ from frequency localized pieces. 

\begin{defn} \label{d:LE}
For any function $u \in C^\infty_0(\bbR\times \bbH^d)$ we let
\begin{equation} \label{eq:LE}
 \| u \|_{LE}^2 := \int_\frac{1}{8}^4 \| \tilP_{\geq s} u \|_{LE_{\low}}^2 \,\ds + \int_0^{\frac{1}{2}} s^{-\frac{1}{2}} \| \tilP_s u \|_{LE_s}^2 \, \ds.
\end{equation}
The space $LE$ is defined as the completion of smooth functions with compact support with respect to the norm $\|\cdot\|_{LE}$. Similarly for any $F\in C^\infty_0(\bbR\times \bbH^d)$
\begin{equation}\label{eq:LE*}
 \| F \|_{LE^\ast}^2 := \int_{\frac{1}{8}}^4 \| \tilP_{\geq s} F \|_{LE_{\low}^\ast}^2\,\ds + \int_0^{\frac{1}{2}} s^{\frac{1}{2}} \| \tilP_s F \|_{LE_s^\ast}^2 \, \ds.
\end{equation}
The space $LE^\ast$ is defined as the completion of smooth functions with compact support with respect to the norm $\|\cdot\|_{LE^\ast}$.
\end{defn}

\begin{rem}\label{rem:LEduality}
It follows from the definitions above that for any $s>0$
\begin{align*}
\begin{split}
| \angles{\tilP_su}{\tilP_sF}_{t,x}|\leq \|\tilP_su\|_{LE_s}\|\tilP_sF\|_{LE_s^\ast} 
\end{split}
\end{align*}
and
\begin{align*}
\begin{split}
| \angles{\tilP_{\geq s}u}{\tilP_{\geq s}F}_{t,x}|\leq \|\tilP_{\geq s}u\|_{LE_\low}\|\tilP_{\geq s}F\|_{LE_\low^\ast}. 
\end{split}
\end{align*}
But because of the integrations in the definitions of $LE$ and $LE^\ast$ and the fact that the heat flow based frequency projections are not sharply localized, $LE$ and $LE^\ast$ are not dual norms. In fact $\|\cdot\|_{LE^\ast}$ is a stronger norm than the dual of $\|\cdot\|_{LE}$, and we will show in Corollary~\ref{cor:LEduality} below that
\begin{align*}
\begin{split}
| \angles{u}{F}_{t,x}|\leq C\|u\|_{LE}\|F\|_{LE^\ast} 
\end{split}
\end{align*}
for an absolute constant $C \geq 1$. We have chosen to use $LE^\ast$ instead of the actual dual of $LE$ in Theorem~\ref{t:LE1-sym} below because the former is a more convenient norm to work with. In Theorem~\ref{t:LE-H} we will pass to a modified local smoothing norm $\tilLE$ where we will also work with the actual dual $\tilLE^\ast$. 
\end{rem}


For the first set of results, we assume that 
\begin{equation} \label{eq:prin-lap}
H_{\prin} = - \Dlt,
\end{equation}
and that $H_{\lot}$ obeys the following decay assumptions:
\begin{align} \label{eq:decay_assumptions}
\begin{split}
 &\sum_{k=0}^4 \|\nabla^{(k)}\bsb\|_{L^\infty(\bbR\times \bbH^d)} +\sum_{k=0}^1\, \sum_{\ell\geq0}\|r^{3}\nabla^{(k)}\bsb\|_{L^\infty(\bbR\times A_\ell)}  < \infty, \\
 &\sum_{k=0}^1 \, \Bigl( \|\nabla^{(k)}V\|_{L^\infty(\bbR\times \bbH^d)} + \sum_{\ell\geq0}\|r^{3}\nabla^{(k)}V\|_{L^\infty(\bbR\times A_\ell)} \Bigr) < \infty.
 \end{split}
\end{align}
For the definitions of the covariant derivatives $\nabla^{(k)}$ and norms of tensors we refer the reader to Section~\ref{s:prelim} below. 

The following is the first main result of this paper, which holds without any spectral assumptions on the operator $H$. In particular note that the coefficients of $H$ are allowed to depend on time.


\begin{thm} \label{t:LE1-sym} 
 Let $d \geq 2$ and assume that \eqref{eq:prin-lap} (i.e., $H_{\prin} = -\lap$) holds, and that $H_{\lot}$ is symmetric (i.e., $\bsb, V$ are real-valued) and \eqref{eq:decay_assumptions} holds.  Then, there exists a sufficiently large constant $R = R(d, \bsb, V) \gg 1$ so that for any $u_{0} \in L^{2}(\bbH^d)$ and $F \in \tilLE^{\ast}$, the solution $u(t)$ to the linear Schr\"odinger equation~\eqref{eq:S} with the initial value $u(0) = u_0$ satisfies 
  \begin{equation} \label{eq:LE1} 
  \| u \|_{LE} \lesssim \| u_0 \|_{L^2} + \| F \|_{LE^\ast} + \| u \|_{L^2(\bbR \times \{ r \leq R \})}.
 \end{equation}
\end{thm}


At the heart of the proof is the construction of a multiplier $Q$ of the form $\frac{1}{i} (\bt(r) \rd_{r} - \rd_{r}^{\ast} \bt(r))$ (in polar coordinates) such that the commutator $[i Q, \Delta]$ 
is positive (i.e., a \emph{positive commutator method}), which is moreover adapted to each heat flow based frequency localization $\tilP_{s} u$ or $\tilP_{\geq s} u$. A more detailed outline of the proof is given in Section~\ref{s:mult}, after the necessary preliminary material is set up.

\begin{rem} [Optimality of the space $LE$] \label{rem:LE-weight}
Our high-frequency component $\nrm{\cdot}_{LE_{s}}$ is optimal in terms of the $r$-weight, but the $r$-weight in our low-frequency component $\nrm{\cdot}_{LE_{\low}}$ is weaker in comparison to the free case \cite{Kaizuka1}. As discussed in Remark~\ref{rem:optimal-weight} below, this loss arises because we do not fully decompose the frequencies $\aleq 1$, and therefore our argument is nonoptimal for very low frequencies. Our current framework has the advantage that it only requires a local-in-time theory for the heat equation $\rd_{s} - \lap_{\bbH^{d}}$, which is very robust. As such, we believe that it can be readily extended to a broader class of manifolds; see Remark~\ref{rem:extension} below. On the other hand, it would be interesting to optimize the $r$-weight in the local smoothing estimate, by perhaps using the global-in-time theory for $\rd_{s} - \lap_{\bbH^{d}}$ to fully decompose the low frequencies.
\end{rem}

\begin{rem} [On the decay assumptions] \label{rem:decay}
The $r$-weight in \eqref{eq:decay_assumptions} (i.e., the decay assumption on $\bsb$ and $V$) is dictated by the difference between the weights in $LE$ and $LE^{\ast}$, or more precisely, by the weights $r^{-\frac{3}{2}}$ and $r^{\frac{3}{2}}$ in $\nrm{\cdot}_{LE_{\low}}$ and $\nrm{\cdot}_{LE^{\ast}_{\low}}$, respectively. In particular, if the $r$-weights in the $LE$ norm (and correspondingly in the $LE^{\ast}$ norm) improve, then these decay assumptions would improve as well.
\end{rem}

Under the additional assumption of stationarity (i.e., time independence), our next main theorem gives a spectral condition, under which the bounded region error $\| u \|_{L^2(\bbR \times \{ r \leq R \})}$ in the previous theorem may be removed. First we need a few more definitions. 

\begin{defn} \label{d:tilLE}
Fix $\nsigma>0$. For any $0<s\leq1$ define the modified local smoothing norms by
\begin{align*}
 \|v\|_{\tilLE_\low} &:= \|\jap{r}^{-\frac{3}{2}-\nsigma}v\|_{L^2(\bbR\times \bbH^d)},\\
 \|v\|_{\tilLE_s} &:= s^{-\frac{1}{4}}\|v\|_{L^2(\bbR\times A_{\leq-k_s})}+\sup_{-k_s\leq\ell < 0}\|r^{-\frac{1}{2}}v\|_{L^2(\bbR\times A_\ell)}+\|r^{-\frac{3}{2}-\nsigma}v\|_{L^2(\bbR\times A_{\geq0})},
\end{align*}
and the corresponding dual norms by
\begin{align*}
 \|G\|_{\tilLE_\low^\ast} &:= \|\jap{r}^{\frac{3}{2}+\nsigma}G\|_{L^2(\bbR\times \bbH^d)},\\
 \|G\|_{\tilLE_s^\ast} &:= s^{\frac{1}{4}}\|G\|_{L^2(\bbR\times A_{\leq-k_s})}+\sum_{-k_s\leq\ell < 0}\|r^{\frac{1}{2}}G\|_{L^2(\bbR\times A_\ell)}+\|r^{\frac{3}{2}+\nsigma}G\|_{L^2(\bbR\times A_{\geq0})}.
\end{align*}
The modified local smoothing norm and its dual are then given by
\begin{align*}
\begin{split}
 \|u\|_{\tilLE}^2 &:= \int_{\frac{1}{8}}^4\|\tilP_{\geq s}u\|_{\tilLE_\low}^2\,\ds+\int_0^{\frac{1}{2}}s^{-\frac{1}{2}}\|\tilP_s u\|_{\tilLE_s}^2\,\ds,\\
 \|F\|_{\tilLE^\ast}^2 &:= \int_{\frac{1}{8}}^4\|\tilP_{\geq s} F\|_{\tilLE_{\low}^\ast}^2\,\ds+\int_{0}^{\frac{1}{2}}s^{\frac{1}{2}}\|\tilP_s F\|_{\tilLE_s^\ast}^2\,\ds.
\end{split}
\end{align*}
\end{defn}

\begin{rem}
The norms $\tilLE$ and $\tilLE^\ast$ are dual norms as will be discussed in Remark~\ref{rem:completedualitytilLE} below. In the definitions above we could have replaced $\| r^{-\frac{3}{2}-\nsigma}v\|_{L^2(\bbR\times A_{\geq 0})}$ by 
\begin{align*}
\begin{split}
\sup_{\ell\geq0}\|r^{-\frac{3}{2}}v\|_{L^2(\bbR\times A_\ell)}.
\end{split}
\end{align*}
However, since the weight $r^{-\frac{3}{2}}$ is not optimal (see for instance Theorem~\ref{t:LE1-sym}) we have opted to work with the simpler definition above.
\end{rem}

To remove the bounded region error in Theorem~\ref{t:LE1-sym}, the key extra requirement is a spectral assumption on $H$, namely that $H$ has no \emph{threshold resonance}. 

\begin{defn} \label{def:th-res}
Let $H$ be stationary and symmetric. We say that $w$ is a \emph{threshold resonance} of $H$ if it is a nontrivial solution to the problem
\begin{equation*}
 H w = \rho^{2} w
\end{equation*}
satisfying\footnote{In Definition~\ref{def:LE0} below, we introduce the space $\tilLE_{0}$, with which \eqref{eq:w-in-LE0} simply reads $w \in \tilLE_{0}$.}
\begin{equation} \label{eq:w-in-LE0}
 1_{[0, 1]}(t) w(x) \in \tilLE,
\end{equation}
and the decay condition
\begin{align}
 \nrm{r^{-\frac{1}{2}} (\rd_{r} + \rho) w}_{L^{2} (A_{j})} \to 0. \label{eq:th-res}
\end{align}
\end{defn}

The precise statement of our next main result is as follows.
\begin{thm}  \label{t:LE-H} 
 Let $d \geq 2$ and assume that \eqref{eq:prin-lap} (i.e., $H_{\prin} = -\lap$) holds, and that $H_{\lot}$ is stationary and symmetric (i.e., $\bsb, V$ are time-independent and real-valued). Assume that \eqref{eq:decay_assumptions} holds and in addition
\begin{equation} \label{equ:LE-H_thm_additional_assumption}
 \sum_{k=0}^1\|\jap{r}^{3+2\nsigma}\nabla^{(k)}\bsb\|_{L^\infty}+\|\jap{r}^{3+2\nsigma}V\|_{L^\infty}<\infty.
\end{equation}
If $H$ has no resonance at the threshold as in Definition~\ref{def:th-res}, then for any $u_{0} \in L^{2}(\bbH^d)$ and $F \in \tilLE^{*}$, the solution $u(t)$ to the linear Schr\"odinger equation~\eqref{eq:S}
with the initial value $u(0) = u_0$ obeys the local smoothing estimate
\EQ{ \label{eq:le-H}
\| P_{c} u \|_{\tilLE} \lesssim   \| P_{c} u_0\|_{L^2} + \| P_{c} F \|_{\tilLE^\ast},
}
where $P_{c}$ is the spectral projection to $[\rho^{2}, \infty)$.
\end{thm} 

\begin{rem} \label{rem:spec}
From \eqref{eq:prin-lap} and \eqref{eq:decay_assumptions}, by standard spectral theory, $H$ is self-adjoint; thus the projection $P_{c}$ is well-defined. By Weyl's theorem, the essential spectrum of $H$ is $[\rho^{2}, \infty)$, which is the same as $H_{\prin} = -\lap$. The discrete spectrum $\sgm_{disc}(H)$ consists of discrete finite-multiplicity eigenvalues in $(-\infty, \rho^{2})$, whose only possible accumulation point is $\rho^{2}$. We note that under stronger decay assumptions on the potentials, it follows from Bouclet's work~\cite{Bou13} that $\rho^2$ cannot be an eigenvalue of $H$. 
\end{rem}

Our proof of Theorem~\ref{t:LE-H} exploits the well-known equivalence between a local smoothing estimate (e.g. \eqref{eq:le-H}) and the \emph{limiting absorption principle}, i.e.,  bounds for the resolvent $(\tau \pm i \eps - H)^{-1}$ in a suitable space which are uniform in $\eps$ and sharp in terms of $\tau$. One of the ingredients is the proof of absence of resonances embedded in $(\rho^{2}, \infty)$ for first and zeroth order perturbations, which extends the results of Donnelly \cite{Donnelly1} and Borthwick--Marzuola \cite{BorMar1}. We refer to Section~\ref{ss:LE-H-outline} for further discussion and remarks concerning the proof of Theorem~\ref{t:LE-H}.

While Definition~\ref{def:th-res} is concise, in practice it may not be easy to work with as the space $\tilLE$ is rather complicated. 
We observe, however, that the threshold resonance is a common obstruction to a wider class of estimates for $\rho^{2} - H$.
Exploiting this idea, we obtain the following more concrete characterization of the threshold (non)resonance:
\begin{prop} \label{p:no-th-res}
For $H$ obeying the assumptions of Theorem~\ref{t:LE-H}, the following statements are equivalent.
\begin{enumerate}
\item $H$ has no threshold resonance.
\item ($H^{1}_{thr}$-type bound for $\rho^{2} - H$) There exists $C_{0} > 0$ such that
\begin{equation} \label{eq:no-th-res-H1}
	\nrm{v}_{H^{1}_{thr}} \leq C_{0} \nrm{(\rho^{2} - H) v}_{H^{-1}_{thr}},
\end{equation}
where
\begin{equation} 
\nrm{v}_{H^{1}_{thr}} = \nrm{(\rd_{r} + \rho) v}_{L^{2}} + \nrm{\frac{1}{\sinh r} \slashed{\nabla} v}_{L^{2}} + \nrm{\frac{1}{\brk{r}} v}_{L^{2}},
\end{equation}
and $H^{-1}_{thr}$ is the dual of $H^{1}_{thr}$.
\item ($LE_{thr}^1$-type bound for $\rho^{2} - H$) There exists $C'_{0} > 0$ such that
\begin{equation} \label{eq:no-th-res-LE}
	\nrm{v}_{LE^{1}_{thr}} \leq C_{0}' \nrm{(\rho^{2} - H) v}_{LE^{\ast}_{thr}},
\end{equation}
where
\begin{align} 
\nrm{v}_{LE^{1}_{thr}} &= \nrm{(\rd_{r} + \rho) v}_{LE_{thr}} + \nrm{\frac{1}{\sinh r} \slashed{\nabla} v}_{LE_{thr}} + \nrm{\frac{1}{\brk{r}} v}_{LE_{thr}}, \\
\nrm{v}_{LE_{thr}} &= \nrm{v}_{L^{2}(A_{\leq 0})} + \sup_{\ell \geq 0} \nrm{r^{-\frac{1}{2}} v}_{L^{2}(A_{\ell})}, \\
\nrm{G}_{LE^{\ast}_{thr}} &= \nrm{G}_{L^{2}(A_{\leq 0})} + \sum_{\ell \geq 0} \nrm{r^{\frac{1}{2}} G}_{L^{2}(A_{\ell})}.
\end{align}
\end{enumerate}
\end{prop}

Estimate \eqref{eq:no-th-res-H1} is easy to verify for the free Laplacian by a simple integration by parts; see Section~\ref{ss:th-res}. In fact, the same proof applies to perturbations of the Laplacian with a magnetic potential (i.e., symmetric first-order potential) and a nonnegative electric potential (i.e., real-valued zeroth-order potential). This case furnishes an important first class of examples for which the local smoothing estimate \eqref{eq:le-H} holds.
\begin{cor} \label{c:no-th-res}
For Hamiltonians of the form
\begin{equation*}
	H = - \lap_{\bsb} + V = - (\nb_{\mu} + i \bsb_{\mu}) \bsh^{\mu \nu} (\nb_{\nu} + i \bsb_{\nu}) + V
\end{equation*}
where $\bsb$ is real-valued and $V \geq 0$, the bound \eqref{eq:no-th-res-H1} holds. Thus, if \eqref{eq:decay_assumptions} and \eqref{equ:LE-H_thm_additional_assumption} are satisfied, then \eqref{eq:le-H} holds as well.
\end{cor}

\subsubsection{Main results for small nonstationary, nonsymmetric perturbations}
Our next class of results concern a fairly general class of small perturbations of the cases considered in Theorems~\ref{t:LE1-sym} and \ref{t:LE-H}. The perturbations may be nonstationary (time-dependent), nonsymmetric (complex-valued) and second order. These results demonstrate the robustness of our approach, which is essentially a combination of a multiplier (or a positive commutator) argument and a heat flow based Littlewood-Paley theory.

We start with a generalization of Theorem~\ref{t:LE1-sym}. Let $\varepsilon_0>0$ be a small constant to be fixed. Instead of \eqref{eq:prin-lap}, we assume that $H_{\prin} = - \nb_{\mu} \bsa^{\mu \nu} \nb_{\nu}$, where $\bsa$ satisfies the smallness and decay condition
\begin{equation} \label{eq:decay_assumptions-prin}
\begin{split}
 &\sum_{k=0}^4 \, \Bigl( \|\nabla^{(k)}(\bsa-\bsh^{-1})\|_{L^\infty(\bbR\times \bbH^d)} + \sum_{\ell\geq0}\|r^{2}\nabla^{(k)}(\bsa-\bsh^{-1})\|_{L^\infty(\bbR\times A_\ell)} \Bigr) \\
 &\qquad +\sum_{k=0}^1\, \sum_{\ell\geq0}\|r^{3}\nabla^{(k)}(\bsa-\bsh^{-1})\|_{L^\infty(\bbR\times A_\ell)}\leq \varepsilon_0
 \end{split}
\end{equation}
as well as the following vanishing condition at $r = 0$:
\begin{equation} \label{eq:vanishing_assumption-prin}
	\bsa(\ud r, \ud \tht^{a}) \vert_{r = 0} = 0.
\end{equation}
Here $\bsh^{-1}$ denotes the inverse of the hyperbolic metric $\bsh$ on $\bbH^d$, $\tht^{a}$ is any angular coordinate in the polar coordinates on $\bbH^{d}$ (see Section~\ref{ss:geometry}). Moreover, we assume that $H_{\lot}$ is possibly nonsymmetric (i.e., $\bsb, V$ are complex-valued), satisfies \eqref{eq:decay_assumptions}, and the nonsymmetric part obeys the smallness and decay condition 
\begin{equation} \label{eq:decay_assumptions-lot}
\begin{split}
 &\sum_{k=0}^4 \|\nabla^{(k)}\Im \bsb\|_{L^\infty(\bbR\times \bbH^d)} +\sum_{k=0}^1\, \sum_{\ell\geq0}\|r^{3}\nabla^{(k)}\Im \bsb\|_{L^\infty(\bbR\times A_\ell)}  \leq \veps_{0}, \\
 &\sum_{k=0}^1 \, \Bigl( \|\nabla^{(k)} \Im V\|_{L^\infty(\bbR\times \bbH^d)} + \sum_{\ell\geq0}\|r^{3}\nabla^{(k)} \Im V\|_{L^\infty(\bbR\times A_\ell)} \Bigr) \leq \veps_{0}.
 \end{split}
\end{equation}
Under these assumptions, the conclusion of Theorem~\ref{t:LE1-sym} remains true:
\begin{thm} \label{t:LE1}
 Let $d \geq 2$ and assume that \eqref{eq:decay_assumptions}, \eqref{eq:decay_assumptions-prin}--\eqref{eq:decay_assumptions-lot} 
 hold with a suitably small $\veps_{0} > 0$\footnote{Here, $\veps_{0}$ depends on the underlying symmetric operator $H_{\mathrm{sym}} = - \lap + \frac{1}{i} (\Re \bsb^{\mu} \nb_{\mu} + \nb_{\mu} \Re \bsb^{\mu}) + \Re V$.}.  Then, there exists a sufficiently large constant $R = R(d, \bsb, V) \gg 1$ so that for any $u_{0} \in L^{2}(\bbH^d)$ and $F \in \tilLE^{\ast}$, the solution $u(t)$ to the linear Schr\"odinger equation~\eqref{eq:S} with initial value $u(0) = u_0$ satisfies \eqref{eq:LE1}, i.e.,
  \begin{equation*} 
  \| u \|_{LE} \lesssim \| u_0 \|_{L^2} + \| F \|_{LE^\ast} + \| u \|_{L^2(\bbR \times \{ r \leq R \})}.
 \end{equation*}
\end{thm}

Next, we generalize Theorem~\ref{t:LE-H}. Consider a second order operator of the form
\begin{equation*}
	H = H_{\stat} + H_{\pert}
\end{equation*}
where $H_{\stat}$ is a symmetric, stationary second order operator satisfying the hypothesis of Theorem~\ref{t:LE-H}, and $H_{\pert}$ is a perturbation of the form
\begin{equation*}
	H_{\pert} u = - \nb_{\mu} (\bsa_{\pert}^{\mu \nu} \nb_{\nu} u) + \frac{1}{i} (\bsb_{\pert}^{\mu} \nb_{\mu} u + \nb_{\mu} (\bsb_{\pert}^{\mu} u)) + V_{\pert} u.
\end{equation*}
Here, $\bsa_{pert}$ is a possibly time-dependent real-valued symmetric $2$-tensor obeying the conditions
\begin{equation} \label{eq:cLE2conditions-prin}
 \sum_{k=0}^4 \|\jap{r}^{3+2\nsigma}\nabla^{(k)}\bsa_{\pert}\|_{L^\infty_{t,x}}  \leq \kpp, 
\end{equation}
\begin{equation} \label{eq:cLE2conditions-prin-vanish}
	\bsa_{\pert}(\ud r, \ud \tht^{a}) \vert_{r = 0} = 0,
\end{equation}
and $\bsb_{\pert}^{\mu}$, $V_{\pert}$ are possibly complex-valued, time-dependent coefficients satisfying
\EQ{\label{eq:cLE2conditions}
 &\sum_{k=0}^4 \|\jap{r}^{3+2\nsigma}\nabla^{(k)}\bsb_{\pert} \|_{L^\infty_{t,x}} \leq \kpp, \\
 &\sum_{k=0}^4 \|\jap{r}^{3+2\nsigma}\nabla^{(k)}V_{\pert}\|_{L^\infty_{t,x}} \leq \kpp,
}
with $\kpp > 0$ small.

\begin{thm} \label{t:LE2}
Let $d \geq 2$. Suppose $H_{\stat}$ satisfies the assumptions imposed on $H$ in Theorem~\ref{t:LE-H}, and that $H_{\pert}$ obeys \eqref{eq:cLE2conditions-prin}--\eqref{eq:cLE2conditions}. Assume furthermore that $H_\stat$ has neither a threshold resonance nor any eigenvalues in $(-\infty,\rho^2]$. If $\kpp > 0$ in \eqref{eq:cLE2conditions} is sufficiently small depending on $H_{\stat}$, then for any $u_{0} \in L^{2}(\bbH^d)$ and $F \in \tilLE^{\ast}$, the solution $u(t)$ to the linear Schr\"odinger equation
\begin{equation*}
	(- i \rd_{t}  + H) u = F
\end{equation*}
with the initial value $u(0) = u_{0}$, obeys the local smoothing estimate 
\begin{equation*}
	\nrm{u}_{\tilLE} \aleq \nrm{u_{0}}_{L^{2}} + \nrm{F}_{\tilLE^{\ast}}.
\end{equation*}
\end{thm}

We refer to Section~\ref{s:mult} for an outline of the proofs of Theorems~\ref{t:LE1} and \ref{t:LE2}.

\begin{rem}[On second order perturbations] \label{rem:extension}
The main reason why we require the vanishing conditions \eqref{eq:vanishing_assumption-prin} and \eqref{eq:cLE2conditions-prin-vanish} on second order perturbations is technical simplicity: These conditions arise if we wish to treat the contribution of $\bsa - \bsh^{-1}$ perturbatively in the region $s^{\frac{1}{2}} \aleq r \ll 1$. It should be possible to remove them by supplementing our multiplier with a correction in the local region $\set{r \ll 1}$ that is better adapted to the local geometry of $\bsa$ (i.e., constructed based on the $\bsa$-distance function from $r = 0$). Alternatively, these conditions can be dropped if we use suboptimal weights in the region $\set{s^{\frac{1}{2}} \aleq r \ll 1}$ in our function spaces; see Remark~\ref{rem:suboptimal-w}.

In fact, in view of the robustness of our approach, the natural setting to which Theorems~\ref{t:LE1} and \ref{t:LE2} generalize is where $\bsa$ is an \emph{asymptotically hyperbolic manifold with no trapped geodesics}. Here, our multiplier must be combined with a pseudodifferential multiplier of Doi \cite{Doi} in a compact region $\set{r \aleq 1}$ that exploits the nontrapping condition. We refer to \cite{MMT} for the asymptotically flat case.
We plan to return to these issues in a future investigation.
\end{rem}

\subsubsection{Corollaries of the main theorems} We now state a number of corollaries of our main theorems, starting with a variant of our local smoothing estimate where the projections
\begin{align*}
\begin{split}
\Pea_{\geq s}u:=e^{s\Delta}u,\qquad \Pea_su:=-s\Delta e^{s\Delta}u
\end{split}
\end{align*}
are used instead of $\tilP_{\geq s}u$ and $\tilP_{s}u$. Associated to these projections we define the modified frequency localized local smoothing norms
\begin{align*}
 \|u\|_{\LE}^2 &:= \int_{\frac{1}{8}}^4\|\Pea_{\geq s}u\|_{\tilLE_\low}^2\,\ds+\int_0^{\frac{1}{2}}s^{-\frac{1}{2}}\|\Pea_s u\|_{\tilLE_s}^2\,\ds,\\
 \|F\|_{\LE^\ast}^2 &:= \int_{\frac{1}{8}}^4\|\Pea_{\geq s} F\|_{\tilLE_{\low}^\ast}^2\,\ds+\int_{0}^{\frac{1}{2}}s^{\frac{1}{2}}\|\Pea_s F\|_{\tilLE_s^\ast}^2\,\ds.
\end{align*}
The following is the analogue of Theorems~\ref{t:LE-H} and \ref{t:LE2} for the modified local smoothing spaces. 

\begin{cor}\label{c:LE1}
Let $d \geq 2$ and let $H$ be either as in Theorem~\ref{t:LE-H} or as in Theorem~\ref{t:LE2}. Then for any $u_{0} \in L^{2}(\bbH^d)$ and $F \in \tilLE^{\ast}$, the solution $u(t)$ to the linear Schr\"odinger equation~\eqref{eq:S} with the initial value $u(0) = u_0$ satisfies
\begin{align*}
\begin{split}
\|P_{c} u\|_{\LE}\lesssim \|P_{c} u_0\|_{L^2}+\|P_{c} F\|_{\LE^\ast},
\end{split}
\end{align*}
where $P_{c} = P_{c}(H)$ is the spectral projection to $[\rho^{2}, \infty)$ when $H$ is as in Theorem~\ref{t:LE-H}, and $P_{c} = I$ when $H$ is as in Theorem~\ref{t:LE2}. 
\end{cor}

Next, we note that our local smoothing estimate from Theorem~\ref{t:LE-H} can be used to prove non-endpoint Strichartz estimates for the operator $H$ when $H_\prin =-\Delta$. For this we also crucially rely on the existing Strichartz estimates for the unperturbed Schr\"odinger equation on $\bbH^d$, that is, when $H=-\Delta$. See for instance \cite{AP09, B07, IS}.

\begin{cor}\label{c:Strichartz}
Let $d \geq 2$ and let $H$ be either as in Theorem~\ref{t:LE-H} or as in Theorem~\ref{t:LE2} with $H_{\prin} = - \lap$. Let $u(t)$ be a solution to \eqref{eq:S}. Then for any two pairs of exponents $(p_1,q_1)$ and $(p_2,q_2)$ in
\begin{align} \label{eq:s-pairs} 
\begin{split}
\Big\{(\frac{1}{p},\frac{1}{q})\in(0,\frac{1}{2})\times(0,\frac{1}{2})\,\big\vert \,\frac{2}{p}\geq \frac{d}{2}-\frac{d}{q}\Big\}\cup\Big\{(\frac{1}{p},\frac{1}{q})=(0,\frac{1}{2})\Big\},
\end{split}
\end{align}
and any time interval $I\subseteq \bbR$, we have
\begin{align*}
\begin{split}
\|P_{c} u\|_{L^{p_1}(I;L^{q_1}(\bbH^d))}\lesssim \|P_{c} u_0\|_{L^2}+\|P_{c} F\|_{L^{p_2'}(I;L^{q_2'} (\bbH^d))},
\end{split}
\end{align*}
where $P_{c} = P_{c}(H)$ is the spectral projection to $[\rho^{2}, \infty)$ when $H$ is as in Theorem~\ref{t:LE-H}, and $P_{c} = I$ when $H$ is as in Theorem~\ref{t:LE2}. 
\end{cor}

Finally, we establish a combined Strichartz and local smoothing estimate that is useful in nonlinear applications. 
\begin{cor} \label{c:combined} 
Let $d \geq 2$ and let $H$ be either as in Theorem~\ref{t:LE-H} or as in Theorem~\ref{t:LE2} with $H_{\prin} = - \lap$. Let $u(t)$ be a solution to \eqref{eq:S}. Then for any two pairs of exponents $(p_1,q_1)$ and $(p_2,q_2)$ as in~\eqref{eq:s-pairs} and for any time interval $I\subseteq \bbR$, we have
\EQ{
\|P_{c} u\|_{(L^{p_1}_tL^{q_1}_x \cap \LE)(I \times \bbH^d))}\lesssim \|P_{c} u_0\|_{L^2}+\|P_{c} F\|_{(L^{p_2'}_tL^{q_2'}_x + \LE^*) (I \times \bbH^d))}.
}
\end{cor}

\begin{remark} \label{r:combined} 
Corollary~\ref{c:combined} is the main linear estimate used in~\cite{LLOS2}, where the operator $H$ arises by linearizing the Schr\"odinger maps evolution on $\Hp^2$ about a harmonic map in the caloric gauge; see~\cite[Eqn. (3.3) and Section 3.3]{LLOS2}. 
\end{remark} 

\subsection{History and related works} 

\subsubsection{Local smoothing estimates on (asymptotically) Euclidean space} 
Local smoothing estimates have a long and rich history, going back to Kato \cite{Kat1} in the context of the KdV equation, and (independently) to Constantin-Saut \cite{CoSa}, Sj\"olin \cite{Sjo} and Vega \cite{Veg} in the case of the free Sch\"odinger equation on $\bbR^{d}$. These estimates are basic tools for studying Schr\"odinger-type equations with derivative nonlinearities; see, for instance, the classical work \cite{KPV04} of Kenig-Ponce-Vega. Some recent works furthermore highlight the role of global-in-time local smoothing estimates as a ``core decay estimate'', which implies other useful decay estimates for nonlinear applications; see, for instance, the work of Rodnianski-Schlag \cite{RodS} or Tataru \cite{Tat08ajm}, which assert that global-in-time local smoothing estimates lead to global-in-time Strichartz estimates in various situations.

Our main results and approach were largely inspired by the work of Marzuola-Metcalfe-Tataru \cite{MMT} on global-in-time local smoothing estimates and Strichartz estimates for long-range perturbations of the Euclidean Laplacian. In particular, \cite{MMT} proves global-in-time local smoothing estimates with a lower order localized error (i.e., analogue of Theorem~\ref{t:LE1-sym}) in the symmetric and nontrapping case using a multiplier approach. Moreover, in the time-independent case with no threshold resonances or eigenvalues, this error is removed after projecting to the continuous spectrum (i.e., analogue of Theorem~\ref{t:LE-H}), by exploiting the equivalence between the local smoothing estimate and the limiting absorption principle. We emphasize that large metric perturbations of the Euclidean Laplacian are considered in~\cite{MMT}.

A related result is an earlier work by Rodnianski-Tao \cite{RoTa07} on global-in-time local smoothing estimates for time-independent, compactly supported, nontrapped yet possibly large metric perturbations of the Euclidean Laplacian, which relied on a quantitative version of Enss's method. 
There is also an extensive literature on the dispersive estimate (which is stronger than Strichartz estimates, and does not follow from a local smoothing estimate) for zeroth order perturbations of the Euclidean Laplacian; we refer to the article of Schlag~\cite{Schlag07_Survey} for an excellent survey. For a sample of some alternative approaches to global-in-time local smoothing estimates and Strichartz estimates for first and zeroth order perturbations of the Euclidean Laplacian, see \cite{EGS08, EGS09, DFVV}.

Finally, by the aforementioned equivalence, the subject of local smoothing estimates in the symmetric, time-independent case is intimately tied to the limiting absorption principle in spectral theory, i.e., resolvent estimates that are uniform near the real axis. We will refrain from attempting to cover the vast literature here, but instead refer to the recent work of Rodnianski-Tao \cite{RoTa15} for a nice review.

\subsubsection{Local smoothing and Strichartz estimates on (asymptotically) hyperbolic space} 
A global-in-time local smoothing estimate for the unperturbed hyperbolic Laplacian $H = -\lap$ is proved in Kaizuka \cite{Kaizuka1}, which heavily uses the Helgason Fourier transform. Dispersive estimates for $- i \rd_{t} + \lap$ were proved by Anker-Pierfelice \cite{AP09} and Ionescu-Staffilani \cite{IS}; by standard machinery, Strichartz estimates for $- i \rd_{t} + \lap$ then follow.

Concerning perturbations of the hyperbolic Laplacian, Borthwick-Marzuola \cite{BorMar1} recently proved the dispersive estimate for exponentially decaying, zeroth order perturbations of $-\lap$ (and also the matrix Hamiltonian, which arises when linearizing NLS around a standing wave). In the series of works \cite{ChenHassell1, ChenHassell3} Chen-Hassell established uniform $L^{p}$ bounds for resolvents of the Laplace-Beltrami operator on nontrapped, asymptotically hyperbolic (with exponential decay) manifolds, and in the work \cite{Chen18} Chen proved global-in-time Strichartz estimates in the same setting. 

Finally, we mention an interesting series of work by Li-Ma-Zhao~\cite{Li_Ma_Zhao} and Li \cite{Li1, Li2} in the closely related subject of perturbations of the wave equation on hyperbolic space. More precisely, in the context of their proof of stability of harmonic maps from $\bbH^{2}$ into $\bbH^{2}$ under the wave map evolution, they established a global-in-time local energy decay estimate (yet with exponentially decaying weights) and $L^{2}_{t} L^{p}_{x}$-Strichartz estimates for the first and zeroth order perturbations of the wave equation arising from linearization around such harmonic maps. The resulting potentials have the same form as those considered in Corollary~\ref{c:no-th-res}.


\subsection{Structure of the paper}
This paper is organized into two tiers. 
\begin{itemize}
\item The first tier consists of Section~\ref{s:prelim} and \ref{s:mult}, and is meant to serve as a guide to reading the paper. Section~\ref{s:prelim} contains preliminary material such as notation, definitions of the function spaces, the basic multiplier identity and the precise definition of the main multipliers used in the paper. Section~\ref{s:mult} contains an overview of the proofs of the main results.
\item The second tier consists of Sections~\ref{s:pr}--\ref{s:cors}. In Section~4, which takes up the bulk of the paper, we develop the heat flow based Littlewood-Paley theory. Sections~\ref{s:low}--\ref{s:cors} contain the detailed proofs of the main results. 
An outline of these sections can be found in Section~\ref{ss:outline}, following the overview of the proofs.
\end{itemize}


\section{Preliminaries, Function Spaces and Multipliers} \label{s:prelim} 

\subsection{Notation}\label{ss:notation} We will adhere to the following conventions for indices: We reserve greek lowercase indices $\ka, \la, \mu, \nu,$ \dots $ \in\{ 1, 2, \dots, d\}$ for coordinates on $(\Hp^d, \h)$. 
These indices are raised and lowered using the metric $\h$. We adopt the standard convention of summing over repeated upper and lower indices.  Tensors will be denoted by boldface letters.

We will use the notation $A_\ell$, $A_{\leq\ell}$, $k_s$, $\rho$, $P_s$, $P_{\geq s}$, $\tilP_{\geq s}$, and $\tilP_s$ introduced in Section~\ref{s:statement}. We  let
\begin{align*}
\begin{split}
\tilP_{s'\leq \cdot\leq s''}u:=\int_{s'}^{s''}\tilP_su\,\ds,\qquad \mathrm{and}\qquad \tilP_{\leq s'}=\int_0^{s'}\tilP_su\,\ds,
\end{split}
\end{align*}
and define $P_{s'\leq\cdot\leq s''}u$ and $P_{\leq s'}u$ similarly. Note that for any $0<s<s'$
\begin{align*}
\begin{split}
u=\tilP_{\leq s}u+\tilP_{s\leq\cdot\leq s'}u+\tilP_{\geq s'}u.
\end{split}
\end{align*}

For the operator $H$ we use the decomposition \eqref{eq:Hprindef}--\eqref{eq:Hlotdef} into principal and lower order parts.

Brackets $\angles{\cdot}{\cdot}$ denote the pairing in $L^2(\bbH^d)$ and $\angles{\cdot}{\cdot}_{t,x}$ the pairing in $L^2(\bbR\times \bbH^d)$. Similarly $L^p$ denotes $L^p(\bbH^d)$ and $L^p_{t,x}$ denotes $L^p(\bbR\times \bbH^d)$.

 \subsection{Geometry of the domain and polar coordinates} \label{ss:geometry} Let $\R^{d+1}$ denote the $(d+1)$-dimensional Minkowski space with rectilinear coordinates  $\{y^1,\dots, , y^d, y^0\}$ and metric $\m$ given in these coordinates by $\bsm = \textrm{diag}(1, \dots, 1,  -1)$.  
The $d$-dimensional hyperbolic space $\bbH^d$ is defined as 
\begin{align*}
\bbH^d := \{ y \in \R^{d+1} \colon (y^{0})^{2} - \sum_{a=1}^{d} (y^{a})^{2} = 1\, , \,y^0>0\}.
\end{align*}
The Riemannian metric $\bsh$ on $\bbH^d$ is obtained by pulling back the metric  $\m$ on $\R^{d+1}$ by the inclusion map $\iota: \bbH^d \hookrightarrow \R^{d+1}$, that is,
$
\bsh =  \iota^* \bsm. 
$
For any coordinate system $\{x^\mu\}$ on $\bbH^d$ denote by $\bsh_{\mu \nu}$ the components of $\bsh$ and by $\bsh^{\mu \nu}$ the components of the inverse matrix $\bsh^{-1}$, that is,  
\ant{
\bsh_{\mu \nu} :=  \bsh\left( \frac{\p}{\p x^\mu}, \frac{\p}{\p x^\nu} \right), \quad \bsh^{\mu \nu }:= \bsh^{-1}(dx^\mu, dx^\nu).
}
We denote the Christoffel symbols on $\bbH^d$ by 
\ant{
\Ga_{ \mu \nu}^{\ka}  = \frac{1}{2} \bsh^{ \ka \la}( \p_{\mu} \bsh_{\la \nu} + \p_\nu \bsh_{\mu \la} - \p_\la \bsh_{\mu \nu}).
} 
Given a vector field $X = X^\mu \frac{\p}{\p x^\mu}$ in $\Ga(T \bbH^d)$  or a $1$-form $\om = \om_\mu dx^\mu \in \Ga( T^*\bbH^d)$ the metric covariant derivative $\na$  is defined  in coordinates by 
\ant{
\na_{\p_\mu} X = (\p_{\mu} X^\nu + \Ga^{\nu}_{\mu \ka} X^\ka ) \p_\nu
, \quad \na_{\p_\mu} \om  = (\p_\mu \om_\nu  - \Gamma^\ka_{\mu \nu} \om_\ka ) dx^\nu.
}
From these definitions one can generalize the definition of   $\na$ to $(p,q)$-tensor fields of arbitrary rank by requiring that: 
$\nb_{\mu} f = \rd_{\mu} f$ for $(0, 0)$-tensors (i.e., functions), $\nb_{\mu}$ agrees with the preceding definition for $(1, 0)$- or $(0, 1)$-tensors (i.e., vector fields or 1-forms) and $\nb$ obeys the Leibniz rule with respect to tensor products.
We use the letter $\bsR$ for the Riemann curvature tensor on $\bbH^d$. Recall that for a vector field $\xi^\mu\frac{\p}{\p x^\mu}$ in $\Gamma(T\bbH^d)$ we have the formula
\EQ{ \label{eq:R} 
[ \na_\mu,  \na_\nu] \bsxi^\ka   = \bsR^\ka_{ \, \, \la \mu \nu} \bsxi^\la
}
as well as 
\EQ{ \label{eq:Ric} 
 [ \na_\mu,  \na_\nu] \bsxi^\mu   = \RR^\mu_{ \, \, \ka \mu \nu} \bsxi^\ka = \RR_{\ka \nu} \bsxi^\ka
}
where $\RR_{\ka \nu}$ denote the components of the Ricci curvature tensor. Since $\bbH^d$ has constant sectional curvature $ \bsK = -1$ we note that 
\EQ{\label{eq:Rich}
&\RR_{\mu \nu} = -(d-1) \bsh_{\mu \nu},\\
&\RR_{\mu\nu\kappa\lambda}=-(\bsh_{\mu\kappa}\bsh_{\nu\lambda}-\bsh_{\mu\lambda}\bsh_{\nu\kappa}).
}

We will often use geodesic polar coordinates $(r,\theta)$ on $\bbH^d$. In the identification of $\bbH^d$ with a Riemannian submanifold of $\bbR^{d+1}$ above, these are defined explicitly by the coordinate map
\begin{align*}
\begin{split}
(r,\theta)\in (0,\infty)\times \bbS^{d-1}\mapsto (\sinh r \,\theta,\cosh r)\in \bbR^{d+1}.
\end{split}
\end{align*}
Arbitrary local coordinates for $\theta\in \bbS^{d-1}$ will be denoted by $\theta^a:~a=1,\dots,d-1$. The hyperbolic metric in polar coordinates is given by
\begin{align*}
\begin{split}
dr\otimes dr+ \sinh^2r \slashed{\bsg}_{\theta_a\theta_b}d\theta^a\otimes d\theta^b,
\end{split}
\end{align*}
where $\slashed{\bsg}$ denotes the standard metric on $\bbS^{d-1}$. Note that the coordinate function $r$ coincides with the distance function $r$ defined above in \eqref{eq:rAldef} if we choose the origin to correspond to $\{r=0\}$. The Laplacian in polar coordinates takes the form
\EQ{
 \De = \De_{\Hp^d} = \partial_r^2 + (d-1)\coth r \partial_r + \frac{1}{\sinh^2 r} \De_{\Sp^{d-1}}
}
and we define $\snabla$ by 
\EQ{
 \abs{\na f}^2  := \abs{\na f}_{\h}^2 = \abs{\p_r f}^2 + \frac{1}{\sinh^2r}\abs{\snabla f}^2.
}
We will also use the notation
\begin{align*}
\begin{split}
\slashed{\bsh} = \sinh^2 r \, \slashed{\bsg}, \qquad \slashed{\bsh}^{-1}=\sinh^{-2} r \, \slashed{\bsg}^{-1},
\end{split}
\end{align*}
and for an arbitrary tensor field $\bsT$ write $|\bsT|_{\slashed{\bsh}}$ for the angular part of the norm computed using $\slashed{\bsh}$. For instance for a vector field $\bsX$ we have
\begin{align*}
\begin{split}
|\bsX|_{\slashed{\bsh}}^2:=\slashed{\bsh}_{\theta_{a}\theta_{b}}\bsX^{\theta_a}\bsX^{\theta_b}.
\end{split}
\end{align*}

The formal adjoint of $\partial_r$, denoted by $\partial_r^\ast$, has the explicit representation
\begin{align*}
\begin{split}
\partial_r^\ast:=-\partial_r-(d-1)\coth r\qquad \Rightarrow\qquad \angles{\partial_rf}{g}=\angles{f}{\partial_r^\ast g}.
\end{split}
\end{align*}
In terms of $\partial_r$ and $\partial_r^\ast$ the Laplacian satisfies
\begin{align*}
\begin{split}
-\Delta=\partial_r^\ast\partial_r - \frac{1}{\sinh^2r}\Delta_{\bbS^{d-1}}.
\end{split}
\end{align*}


\subsection{$L^p$ and Sobolev spaces on $\Hp^d$}\label{s:fs} 

With the Riemannian structure we can define the relevant function spaces on $\Hp^d$. Let $f: \Hp^d \to \R$ be a smooth function. The $L^p( \Hp^d)$ spaces are defined for $1 \le p < \infty$ by 
\begin{align*}
\|f \|_{L^p(\bbH^d)} :=  \left( \int_{\bbH^d} | f(x) |^p\, \dh \right)^{\frac{1}{p}}, \quad \| f \|_{L^{\infty}( \bbH^d)} :=  \sup_{x \in \Hp^d} \abs{f(x)}.
\end{align*} 
These definitions can be extended to tensors $\bsf$ in which case $|\bsf|$ is defined using the Riemannian metric $\bsh$ as
\begin{align*}
\begin{split}
|\bsf|^2=\bsh_{\mu_1\lambda_1}\dots\bsh_{\mu_p\lambda_p}\bsh^{\nu_1\kappa_1}\dots\bsh^{\nu_q\kappa_q}\bsf^{\mu_1\dots\mu_p}_{\nu_1\dots\nu_q}\bsf^{\lambda_1\dots\lambda_p}_{\kappa_1\dots\kappa_q}.
\end{split}
\end{align*}
The Sobolev norms $W^{k, p}( \Hp^d) $ are defined as follows, see for example~\cite{Heb}: 
\EQ{ \label{sobh4} 
\| f\|_{W^{k, p}(\Hp^d)}:=    \sum_{\ell=0}^k  \left( \int_{\Hp^d} \abs{ \na^{(\ell)} f}^p  \, \dh \right)^{\frac{1}{p}}
}
where $\na^{(\ell)} f$ is the $\ell$th covariant derivative of $f$ with the convention that $\na^{(0)} f = f$. One then defines the Sobolev space~$W^{k, p}(\Hp^d)$ to be the completion of all smooth compactly supported functions $f \in C^{\infty}_0( \Hp^d)$ under the norm~\eqref{sobh4}; see~\cite[Theorem $2.8$]{Heb}. We note that for $p=2$ we write $W^{k, 2}(\Hp^d) =: H^k(\Hp^d)$. The Sobolev spaces $W^{s,p}(\bbH^d)$ for $s \geq 0$ are then defined by interpolation.

Alternatively one can define Sobolev spaces on $\Hp^d$ using the spectral theory of the Laplacian $-\Delta$. Given $s \in {\mathbb R}$, the fractional Laplacian $(-\Delta)^{\frac{s}{2}}$ is a well-defined operator on, say $C^{\infty}_{0}(\Hp^{d})$ via the spectral theory of $-\Delta$.  Given a function $f \in C^{\I}_{0}(\Hp^d)$, we set
\begin{align*}
\|f \|_{\ti W^{s, p}(\Hp^d)} = \|(- \Delta)^{\frac{s}{2}} f\|_{L^p(\Hp^d)}
\end{align*}
and define $\ti W^{s, p}( \Hp^d)$ to be the completion of $C^{\infty}_0( \Hp^d)$ under the  norm above.  
The fractional Laplacian $(- \De)^{\frac{s}{2}}$ is bounded on $L^p(\Hp^d)$ for all $s \le 0$ and all $p \in (1, \infty)$. Using this fact one can show that $\ti W^{s, p} = W^{s, p}$ with equivalent norms, where the latter space is defined using the Riemannian structure as above; see for example~\cite{Tat01hyp} for a proof.
In particular
\EQ{
 \| \na (-\De)^{\ell} f\|_{L^2(\bbH^d)} \simeq  \|  \na^{(2 \ell +1)} f \|_{L^2(\bbH^d)},\qquad \forall\ell\in\bbN. 
}
We also have the following basic estimate which is a consequence of the well-known fact that $-\De$ on $\bbH^{d}$ has a spectral gap of $ \rho^2$ (see e.g.,~\cite{Bray}).

\begin{lem}[Poincar\'e inequality] \label{l:poincare} 
Let $v \in C^{\infty}_{0}(\bbH^{d})$. Then the following inequalities hold.
\begin{align}
	\|v\|_{L^{2}(\bbH^d)} &\leq \frac{1}{\rho} \|\nabla v\|_{L^{2}(\bbH^d)}, \\
	\|v \|_{L^{2}(\bbH^d)} &\leq \frac{1}{\rho^2} \|\De v\|_{L^{2}(\bbH^d)}.
\end{align}
\end{lem} 

\subsection{Basic comparisons of $LE_s$ and $LE_\low$} Here we record a few basic relations between the spaces $LE_s$ and $LE_\low$ introduced in Section~\ref{s:statement}. We start by stating, without proof, some simple comparisons between $LE_s$, $LE_{2s}$, and $LE_\low$ which follow directly from the definitions.

\begin{lem}\label{lem:LE_comparison}
There exists a constant $C>0$ independent of $s$ such that for any function $v$,
\begin{align*}
\begin{split}
 C^{-1}\|v\|_{LE_{2s}}\leq \|v\|_{LE_s} \leq C\|v\|_{LE_{2s}}.
\end{split}
\end{align*}
Similarly for any function $G$,
\begin{align*}
\begin{split}
  C^{-1}\|G\|_{LE_{2s}^\ast}\leq \|G\|_{LE_s^\ast} \leq C\|G\|_{LE_{2s}^\ast}.
\end{split}
\end{align*}
Similar estimates hold with $LE_s$ and $LE_s^\ast$ replaced by $\tilLE_s$ and $\tilLE_s^\ast$.
\end{lem}

\begin{lem}\label{lem:LEs_LElow_comparision}
Let $I\subset (0,\infty)$ be a compact interval not containing zero. Then for any function $v$ and any $s\in I$
\begin{align*}
\begin{split}
\|v\|_{LE_\low}\leq C(I)\|v\|_{LE_s},
\end{split}
\end{align*}
where the constant $C(I)$ depends only on the interval $I$.
\end{lem}


\subsection{Time-independent and auxiliary function spaces} 
As a matter of technical convenience we introduce a subtle modification of the local smoothing spaces defined in Section~\ref{s:statement}.  First we introduce the notion of a slowly varying $\ell^1$ sequence. 

\begin{defn}  \label{d:A} 
We denote by $\calA$ the family of positive, slowly varying sequences $\{ \alpha_\ell \}_{\ell \geq 0}$ with the property that
\[
 \sum_{\ell \geq 0} \alpha_\ell = 1, \quad \alpha_{0} \simeq 1.
\]
Here a positive slowly varying sequence is any sequence $\{ \al_\ell\}_{\ell \ge 0} \in \ell^1(\bbN_0)$, $\al_\ell > 0$, such that 
\EQ{
 \abs{ \log_2( \al_{\ell}/\al_{\ell+1})} \le \eta , \quad \forall \,  \ell  \ge 0, \quad \frac{\al_{\ell}}{\al_k} \le 2^{\eta \abs{\ell - k}}\quad \forall \ell, k
}
for some absolute constant $0 < \eta \ll 1$ to be fixed in the course of the proof. We remark that any positive $\ell^1$ sequence can be element-wise dominated by a slowly-varying sequence of comparable $\ell^1$-norm. 
\end{defn}

The auxiliary spaces defined below will play a critical technical role in our analysis, where summation against the family of slowly varying sequences is used as a more flexible substitute for the supremum in the definition of the local smoothing norms. 

\begin{defn}\label{d:Xspaces}
For any $\{ \alpha_\ell \}_{\ell \geq 0} \in \calA$ we define the low-frequency auxiliary norms
\EQ{ \label{eq:Xallow2}
 \| v \|_{X_{\low, \alpha}}^2 &= \| v \|_{L^2(\R \times A_{\leq0})}^2 + \sum_{\ell \geq 0} \alpha_\ell \| r^{-\frac{3}{2}} v \|_{L^2(\R \times A_\ell)}^2, \\
 \| G \|_{X_{\low, \alpha}^\ast}^2 &= \| G \|_{L^2(\R \times A_{\leq0} )}^2 + \sum_{\ell \geq 0} \alpha_\ell^{-1} \| r^{\frac{3}{2}} G \|_{L^2(\R \times A_\ell)}^2.
}
Similarly for  high frequencies, for any $\{ \alpha_\ell \}_{\ell \geq 0} \in \calA$ we define 
\EQ{ \label{eq:Xalhigh} 
 \| v \|_{X_{s, \alpha}}^2 &= s^{-\frac{1}{2}} \| v \|_{L^2(\R \times A_{\leq-k_s})}^2 + \sum_{\ell \geq -k_s} \alpha_{\ell+k_s} \| r^{-\frac{1}{2}} v \|_{L^2(\bbR\times A_\ell)}^2, \\
 \| G \|_{X_{s, \alpha}^\ast}^2 &= s^{\frac{1}{2}} \| G \|_{L^2(\R \times A_{\leq-k_s})}^2 + \sum_{\ell \geq -k_s} \alpha_{\ell+k_s}^{-1} \| r^{\frac{1}{2}} G \|_{L^2(\bbR\times A_\ell)}^2.
}
Finally, corresponding to $\tilLE$, we define
\begin{align*}
\begin{split}
\|v\|_{\calX_{\low,\alpha}}:=\|\jap{r}^{-\frac{3}{2}-\nsigma}v\|_{L^2_{t,x}},\quad \|G\|_{\calX_{\low,\alpha}^\ast}=\|\jap{r}^{\frac{3}{2}+\nsigma}G\|_{L^2_{t,x}}
\end{split}
\end{align*}
and
\begin{align*}
\begin{split}
\|v\|_{\calX_{s,\alpha}}&= s^{-\frac{1}{2}} \| v \|_{L^2(\R \times A_{\leq-k_s})}^2 + \sum_{-k_s\leq \ell < 0} \alpha_{\ell+k_s} \| r^{-\frac{1}{2}} v \|_{L^2(\bbR\times A_\ell)}^2+ \| r^{-\frac{3}{2}-\nsigma} v \|_{L^2(\bbR\times A_{\geq0})}^2,\\
\|G\|_{\calX_{s,\alpha}^\ast}&= s^{-\frac{1}{2}} \| G \|_{L^2(\R \times A_{\leq-k_s})}^2 + \sum_{-k_s\leq \ell < 0} \alpha_{\ell+k_s}^{-1} \| r^{\frac{1}{2}} G \|_{L^2(\bbR\times A_\ell)}^2+ \| r^{\frac{3}{2}+\nsigma} G \|_{L^2(\bbR\times A_{\geq0})}^2.
\end{split}
\end{align*}
Here $\nsigma>0$ is a fixed positive number as in Definition~\ref{d:tilLE}.
\end{defn}

These norms satisfy
\EQ{
 \langle v, G \rangle_{t,x} &\leq \| v \|_{X_{\low, \alpha}} \| G \|_{X_{\low, \alpha}^\ast}, \\
 \langle v, G \rangle_{t,x} &\leq \| v \|_{X_{s, \alpha}} \| G \|_{X_{s, \alpha}^\ast} ,
}
and similarly for $\calX_{s,\alpha}$ and $\calX_{\low,\alpha}$. We will crucially rely on the following relations between these spaces and the local smoothing norms introduced earlier, whose straightforward proofs we omit.

\begin{lem} \label{l:Xal} 
The following relations hold 
\begin{align*}
 \| v \|_{LE_{\low}} &\simeq \sup_{\{ \alpha_\ell \} \in \calA} \, \| v \|_{X_{\low,\alpha}},  \quad \| G \|_{LE_{\low}^*} \simeq \inf_{\{ \alpha_\ell \} \in \calA} \, \| G \|_{X_{\low, \alpha}^\ast}, \\
 \| v \|_{LE_{s}} &\simeq \sup_{\{ \alpha_\ell \} \in \calA} \, \| v \|_{X_{s,\alpha}}, \qquad \, \, \| G \|_{LE_{s}^*} \simeq \inf_{\{ \alpha_\ell \} \in \calA} \, \| G \|_{X_{s, \alpha}^\ast},
\end{align*}
where the implicit constants in the last line are independent of $s>0$. Similar conclusions hold with the $LE$ spaces replaced by $\tilLE$ and the $X$ spaces by the $\calX$ spaces.
\end{lem} 

In Section~\ref{s:error} we will need to work with time-independent as well as higher regularity versions of the $\tilLE$ and $\tilLE^\ast$ spaces defined above. Due to the presence of the supremum in the definition of $\tilLE$ it is not possible to directly write this space as $L^2_tY$ for another space $Y$. In fact the natural choice for $Y$ would be $\tilLE$ where all $L^2$ norms are taken only over $\bbH^d$ instead of $\bbR\times \bbH^d$. Below we will describe the procedure for passing from $\tilLE$ to these time-independent spaces using the auxiliary $\calX_{\low,\alpha}$ and $\calX_{s,\alpha}$ spaces defined above, but first we give the precise definition of our candidate space.

\begin{defn}\label{def:LE0}
Fix $\sigma > 0$. For any $0 < s \leq 1$ let
\begin{align*}
\begin{split}
 \|v\|_{\tilLE_{0,s}} &:=s^{-\frac{1}{4}}\|v\|_{L^2(A_{\leq-k_s})}+\sup_{-k_s\leq\ell < 0}\|r^{-\frac{1}{2}}v\|_{L^2(A_\ell)}+\|r^{-\frac{3}{2}-\nsigma}v\|_{L^2(A_{\geq0})},\\
 \|G\|_{\tilLE_{0,s}^\ast} &:=s^{\frac{1}{4}}\|G\|_{L^2(A\leq -k_s)}+\sum_{-k_s\leq\ell < 0}\|r^{\frac{1}{2}}G\|_{L^2(A_\ell)}+\|r^{\frac{3}{2}+\nsigma}G\|_{L^2(A_{\geq0})},
\end{split}
\end{align*}
and set
\begin{align*}
\begin{split}
 \|v\|_{\tilLE_{0,\low}} &:= \|\jap{r}^{-\frac{3}{2}-\nsigma}v\|_{L^2},\\
 \|G\|_{\tilLE_{0,\low}^\ast} &:= \|\jap{r}^{\frac{3}{2}+\nsigma}G\|_{L^2}.
\end{split}
\end{align*}
Then for any $\kappa\geq0$, we define
\begin{align*}
\begin{split}
 \|u\|_{\tilLE_0^\kappa}^2 &:=\int_{\frac{1}{8}}^4\|\tilP_{\geq s}u\|_{\tilLE_{0,\low}}^2\,\ds+\int_0^{\frac{1}{2}}s^{-\frac{1}{2}-\kappa}\|\tilP_{s}u\|_{\tilLE_{0,s}}^2\,\ds,\\
 \|F\|_{(\tilLE_0^\kappa)^\ast}^2 &:=\int_{\frac{1}{8}}^4\|\tilP_{\geq s}G\|_{\tilLE_{0,\low}^\ast}^2\,\ds+\int_{0}^{\frac{1}{2}}s^{\frac{1}{2}+\kappa}\|\tilP_sF\|_{\tilLE_{0,s}^\ast}^2\,\ds.
\end{split}
\end{align*}
When $\kappa=0$ we simply write $\tilLE_0$ and $\tilLE_0^\ast$ instead of $\tilLE_0^0$ and $(\tilLE_0^0)^\ast$. 
\end{defn}
With this definition, note that the condition \eqref{eq:w-in-LE0} simply becomes
\begin{equation*}\tag{\ref{eq:w-in-LE0}$'$}
w_{0} \in \tilLE_{0}.
\end{equation*}

Next we explain the relation between $\tilLE$ and $\tilLE_0$. First we need another definition.

\begin{defn}\label{def:X0spaces}
Given any function\footnote{To be precise, we have to impose mild conditions on $\alpha: (0,4]\to\calA$ such that the $s$ integrals in this definition are well-defined. For instance one can require that $\alpha$ is constant on dyadic $s$ intervals. For simplicity of exposition we have ignored this technical issue here and in the rest of the paper.} $\alpha \colon (0,4]\to\calA$ taking values in the family of slowly varying sequences $\calA$ defined in Definition~\ref{d:A} let
\begin{align*}
\begin{split}
\|u\|_{\calX_\alpha}^2:=\int_{\frac{1}{8}}^4\|\tilP_{\geq s}u\|_{\calX_{\low,\alpha(s)}}^2\,\ds+\int_0^{\frac{1}{2}}s^{-\frac{1}{2}}\|\tilP_s u\|_{\calX_{s,\alpha(s)}}^2\,\ds,\\
\|F\|_{\calX_\alpha^\ast}^2:=\int_{\frac{1}{8}}^4\|\tilP_{\geq s} F\|_{\calX_{\low,\alpha(s)}^\ast}^2\,\ds+\int_0^{\frac{1}{2}}s^{\frac{1}{2}}\|\tilP_s F\|_{\calX_{s,\alpha(s)}^\ast}^2\,\ds,
\end{split}
\end{align*}
where for fixed $s$ the spaces $\calX_{\low,\alpha(s)}$, $\calX_{s,\alpha(s)}$, $\calX_{\low,\alpha(s)}^\ast$, and $\calX_{s,\alpha(s)}^\ast$ are as in Definition~\ref{d:Xspaces}. Then define the spatial versions $\calX_\alpha^0$ and $(\calX_\alpha^0)^\ast$ by the requirement that
\begin{align*}
\begin{split}
\|u\|_{\calX_\alpha}=:\|u\|_{L_t^2 \calX_\alpha^0},\qquad \|F\|_{\calX_\alpha^\ast}=:\|F\|_{L_t^2(\calX_\alpha^0)^\ast}.
\end{split}
\end{align*}
\end{defn}

The relation between the various spaces defined above is summarized in the following lemma.

\begin{lem}  \label{l:albeLE} 
We have the equivalences
\begin{align*}
 \|u\|_{\tilLE}\simeq\sup_{\alpha : (0,4]\to\calA}\|u\|_{\calX_{\alpha{}}},\quad\| F \|_{\tilLE^*} \simeq \inf_{\al : (0,4]\to\calA} \|F\|_{  \calX_{\alpha{}}^\ast  },
\end{align*}
where the supremum and infimum are taken over all functions $\alpha \colon (0, 4] \to\calA$. Similar conclusions hold for $\tilLE_0$, $\tilLE_0^\ast$ and $\calX_\alpha$, $(\calX_\alpha^0)^\ast$.
\end{lem} 
\begin{proof}
We give the proof only for $\tilLE$. The other equivalences follow similarly using the time-independent analogue of Lemma~\ref{l:Xal}. By Lemma~\ref{l:Xal}, we have 
\begin{align*}
\begin{split}
\|u\|_{\tilLE}^2 &=\int_{\frac{1}{8}}^4 \|\tilP_{\geq s}u\|_{\tilLE_{low}}^2\,\ds+\int_0^{\frac{1}{2}} s^{-\frac{1}{2}}\|\tilP_su\|_{\tilLE_s}^2\,\ds\\
&\simeq\int_{\frac{1}{8}}^4\sup_{ \{ \alpha_\ell \} \in\calA}\|\tilP_{\geq s}u\|_{\calX_{low,\alpha}}^2\,\ds+\int_0^{\frac{1}{2}}\sup_{\{ \alpha_\ell \}\in\calA}s^{-\frac{1}{2}}\|\tilP_su\|_{\calX_{s,\alpha}}^2\,\ds .
\end{split}
\end{align*}
Now it is immediate that
\EQ{
 \sup_{\al:(0,4]\to \calA} \int_0^{\frac{1}{2}} s^{-\frac{1}{2}}\|\tilP_su\|_{\calX_{s,\alpha(s)}}^2\,\ds &\le \int_0^{{\frac{1}{2}}}\sup_{\{ \alpha_\ell \}\in\calA}s^{-\frac{1}{2}}\|\tilP_su\|_{\calX_{s,\alpha}}^2\,\ds,\\
 \sup_{\al:(0,4]\to \calA} \int_{\frac{1}{8}}^{4} s^{-\frac{1}{2}}\|\tilP_{\geq s}u\|_{\calX_{\low,\alpha(s)}}^2\,\ds &\le \int_{\frac{1}{8}}^{4}\sup_{\{ \alpha_\ell \}\in\calA}s^{-\frac{1}{2}}\|\tilP_{\geq s}u\|_{\calX_{\low,\alpha}}^2\,\ds,
}
and hence
\EQ{
\sup_{\alpha:(0,4]\to\calA}\|u\|_{\calX_{\alpha{}}} \le  C \|u\|_{\tilLE}.
}
For the opposite inequality, note that for each $\eps>0$  we can find a function $\al_\eps: (0,4] \to \calA$ so that 
\begin{align*}
 \|\tilP_su\|_{\calX_{s,\alpha_\eps(s)}}^2 &\ge \sup_{\{ \alpha_\ell \}\in\calA}\|\tilP_su\|_{\calX_{s,\alpha}}^2  - s^{\frac{3}{2}}\eps,\\
 \|\tilP_{\geq s}u\|_{\calX_{\low,\alpha_\eps(s)}}^2 &\ge \sup_{\{ \alpha_\ell \}\in\calA}\|\tilP_{\geq s}u\|_{\calX_{\low,\alpha}}^2  - \frac{\eps}{2\log(32)},
\end{align*}
and hence
\begin{align*}
 \int_0^{\frac{1}{2}} s^{-\frac{1}{2}}\|P_su\|_{\calX_{s,\alpha_\eps{(s)}}}^2\,\ds &\ge \int_0^{\frac{1}{2}}\sup_{\{ \alpha_\ell \}\in\calA}s^{-\frac{1}{2}}\|P_su\|_{\calX_{s,\alpha}}^2\,\ds - \frac{1}{2}\eps,\\
 \int_{\frac{1}{8}}^{4} \|\tilP_{\geq s}u\|_{\calX_{\low,\alpha_\eps{(s)}}}^2\,\ds &\ge \int_{\frac{1}{8}}^{4}\sup_{\{ \alpha_\ell \}\in\calA}s^{-\frac{1}{2}}\|\tilP_{\geq s}u\|_{\calX_{\low,\alpha}}^2\,\ds - \frac{1}{2}\eps.
\end{align*}
This means that we also have 
\EQ{
\|u\|_{\tilLE} \le C \sup_{\alpha:(0,4]\to\calA}\|u\|_{\calX_{\alpha{}}}.   
}
\end{proof} 

Finally we define a modified version of the local smoothing space $\tilLE_0$ which is relevant in the context of proving elliptic regularity estimates in the exterior of a ball $\{r\leq R\}\subseteq \bbH^d$ in Section~\ref{s:error}. The modified norm, which will be used for functions that are supported in this exterior region, agrees with $\tilLE_0$ outside a fixed ball but is less refined for small $r$. The precise definition is as follows.

\begin{defn}\label{def:LEext}
Let $\tilLE_{0,1}$ and $\tilLE_{0,\low}$ be as in Definition~\ref{def:LE0} with $s=1$. The $\tilLE_{0,\ext}^\kappa$ norm, with $\kappa\geq0$, is then defined as
\begin{align*}
\begin{split}
\|u\|_{\tilLE_{0,\ext}^\kappa}^2=\int_{\frac{1}{8}}^4\|\tilP_{\geq s} u\|_{\tilLE_{0,\low}}^2\,\ds+\int_{0}^{\frac{1}{2}}s^{-\frac{1}{2}-\kappa}\|\tilP_s u\|_{\tilLE_{0,1}}^2\,\ds.
\end{split}
\end{align*}
\end{defn}


The following is a simple corollary of the definitions.

\begin{lem}\label{lem:LE_LEext1}
For any $\kappa\geq0$
\begin{align*}
\begin{split}
\|u\|_{\tilLE_{0,\ext}^\kappa}\lesssim \|u\|_{\tilLE_0^\kappa}.
\end{split}
\end{align*}
\end{lem}
\begin{proof}
It suffices to note that for $s\leq \frac{1}{2}$
\begin{align*}
\begin{split}
\|\tilP_s u\|_{\tilLE_{0,1}} &\lesssim \|\tilP_s u\|_{L^2(A_{\leq-k_s})}+\sum_{m=-k_s}^{0}\|\tilP_s u\|_{L^2(A_m)}+\|r^{-\frac{3}{2}-\nsigma}\tilP_s u\|_{L^2(A_{\geq 0})}\\
&\lesssim \|\tilP_s u\|_{\tilLE_{0,s}}+\|\tilP_s u\|_{\tilLE_{0,s}}\sum_{m=-k_s}^0 2^{\frac{m}{2}}\lesssim \|\tilP_s u\|_{\tilLE_{0,s}}.
\end{split}
\end{align*}
\end{proof}

The usefulness of the modified norm $\tilLE_{0,\ext}$ comes from the fact, which we will prove in Section~\ref{s:pr}, that  $\|\chi_{\geq R}v\|_{\tilLE_{0,\ext}^\kappa}\simeq \|\chi_{\geq R}v\|_{\tilLE_{0}^\kappa}$ for large $R$, where $\chi_{\geq R}$ denotes a cutoff to the region $\{r\geq R\}\subseteq \bbH^d$.

\subsection{The Main Multiplier Identity} \label{ss:multiplier}

Let $Q$ be the self-adjoint operator
\begin{align}\label{eq:Qdef1}
\begin{split}
 Q := \frac{1}{i} \bigl( \ube^\mu\nabla_\mu+\nabla_\mu\ube^\mu \bigr),
\end{split}
\end{align}
where, in polar coordinates, the vector field $\ube$ is given by $\ube=\beta \partial_r$ for a radial function $\beta$. The operator $Q$ can then also be written in the form
\begin{align}\label{eq:Qdef2}
\begin{split}
Q = \frac{1}{i} \bigl( \beta\partial_r-\partial_r^\ast\beta \bigr),
\end{split}
\end{align}
where $\partial_r^\ast:=-\partial_r-(d-1)\coth r$ is the formal adjoint of $\partial_r$. Recall that in terms of $\partial_r$ and $\partial_r^\ast$
\begin{align*}
\begin{split}
-\Delta=\partial_r^\ast\partial_r-\frac{1}{\sinh^2r}\Delta_{\bbS^{d-1}}.
\end{split}
\end{align*}
We state our main multiplier identity in terms of $Q$ in this general form. In specific applications we will choose the radial function $\beta$ differently depending on the context, so that $ \Re\langle i\Delta u , Qu\rangle$ exhibits some positivity (possibly up to lower order errors). The precise form of the multipliers will be given in Section~\ref{ss:mults}, and our choices will be motivated in Section~\ref{s:mult}.

\begin{lem}[Main multiplier identity]\label{lem:Q}
Let $Q$ be as in \eqref{eq:Qdef2} where $\beta$ is a smooth bounded radial function with bounded first two derivatives. Suppose $u$ satisfies 
\begin{align}\label{eq:multdecayassumption}
 \begin{split}
  \lim_{j\to\infty}\|r^{-\frac{1}{2}}u\|_{L^2(A_j)}=\lim_{j\to\infty}\|r^{-\frac{1}{2}}|\nabla u|\|_{L^2(A_j)}=0.
 \end{split}
\end{align}
If $u\in H^2(\bbH^d)$ then 
\begin{align}\label{eq:Q1}
 \begin{split}
  \Re\langle i\Delta u , Qu\rangle &= 2\angles{(\partial_r\beta)\partial_ru}{\partial_ru}-2\rho^2\angles{(\partial_r\beta)u}{u}+2\angles{\beta \coth r\sinh^{-2}r\snabla u}{\snabla u}\\
   &\quad \quad -\frac{1}{2}\angles{(\Delta\partial_r\beta)u}{u}-\rho\angles{(\Delta(\beta\coth r)-2\rho\partial_r\beta)u}{u}.
 \end{split}
\end{align}
If $u$ is only in $H^2_{\mathrm{loc}}(\bbH^d)$ then
\begin{align}\label{eq:Q1alt}
\begin{split}
   &\lim_{j\to\infty}\Re\langle i\chi_j\Delta u , Qu\rangle \\
   &\qquad = 2\lim_{j\to\infty}\angles{(\partial_r\beta)\chi_j\partial_ru}{\partial_ru}-2\rho^2\lim_{j\to\infty}\angles{(\partial_r\beta)\chi_ju}{u}\\
   &\qquad \qquad +2\lim_{j\to\infty}\angles{\beta \chi_j\coth r\sinh^{-2}r\snabla u}{\snabla u}-\frac{1}{2}\lim_{j\to\infty}\angles{(\Delta\partial_r\beta)\chi_ju}{u}\\
   &\qquad \qquad -\lim_{j\to\infty}\rho\angles{(\Delta(\beta\coth r)-2\rho\partial_r\beta)\chi_ju}{u},
\end{split}
\end{align}
where $\chi_j(x)=\chi(2^{-j}|x|)$ and $\chi$ is a radial positive function equal to one on $\{r\leq1\}$ and equal to zero on $\{r\geq 2\}$.
\end{lem}
 
Before the proof, a few remarks concerning the key identity \eqref{eq:Q1} are in order. First, in view of the self-adjointness of $\Delta$ and $Q$, note that
\begin{align*}
\begin{split}
\Re\langle i\Delta u, Qu\rangle =\frac{1}{2}\langle [i Q,\Delta]u,u\rangle.
\end{split}
\end{align*}
This means that the above integral identity is equivalent to computation of the commutator $[iQ,\Delta]$. For this reason this type of estimate is also known as a \emph{positive commutator estimate}. 

Our second remark concerns the positivity of the RHS of \eqref{eq:Q1}. In all our applications, the function $\bt$ will be chosen so that 
\begin{equation} \label{eq:mult-positivity}
\rd_{r} \bt \geq 0, \qquad \bt \coth r - \rd_{r} \bt \geq 0, 
\end{equation}
which imply the coercivity of the first order terms on the RHS of \eqref{eq:Q1}:
\begin{equation} \label{eq:mult-positivity-top}
	2\angles{(\partial_r\beta)\partial_ru}{\partial_ru}+2\angles{\beta \coth r\sinh^{-2}r\snabla u}{\snabla u}
	\geq  2\brk{ (\partial_{r} \bt) \nb u, \nb u}.
\end{equation}
Heuristically, this information is sufficient for ensuring positivity for the high-frequency part of $u$ (as the remaining lower order terms are expected to be smaller). On the other hand, the exact form of the lower order terms becomes important for the low-frequency part. Expanding out the zeroth order terms on the RHS of \eqref{eq:Q1} and keeping only the terms with worst decay as $r \to \infty$, we arrive at the expression
\begin{equation} \label{eq:mult-positivity-low}
	- 2 \rho^{2} \Re \brk{(\rd_{r} \bt) u, u} - 2 \rho \Re \brk{(\rd_{r}^{2} \bt) \coth r u, u} + \cdots,
\end{equation}
where the remainder is better in the sense that it is bounded from above by $(\abs{\rd_{r}^{3} \bt} + (\abs{\bt} + \abs{\rd_{r} \bt}) e^{-2r}) \nrm{u}_{L^{2}}^{2}$ in the region $\set{r \ageq 1}$. Although the first term occurs with the opposite sign (compared to the first order terms), its sum with the other main terms can be shown to be positive, through a weighted Hardy-Poincar\'e inequality. We refer to Section~\ref{ss:LE1-outline} below for more discussion.

\begin{proof}
We first give a formal proof assuming all integration by parts can be justified, and then explain why the decay assumptions on $u$ are sufficient to justify the vanishing of boundary terms in \eqref{eq:Q1alt}. Then, if $u\in H^2(\bbH^d)$, \eqref{eq:Q1} is a direct consequence of \eqref{eq:Q1alt}. We will use the simple product rule $\partial_r^\ast(fg)=(\partial_r^\ast f)g-f(\partial_rg)$ during the proof. First note that
\begin{align*}
\begin{split}
 \Re\angles{i\Delta u}{Qu} = \Re\angles{\partial_r^\ast\partial_ru}{2\beta\partial_ru-(\partial_r^\ast\beta)u}-\Re\Angles{\frac{\Delta_{\bbS^{d-1}} u}{\sinh^2r}}{2\beta\partial_ru-(\partial_r^\ast\beta)u}. 
\end{split}
\end{align*}
Now
\begin{align*}
\begin{split}
 \Re\angles{\partial_r^\ast\partial_ru}{2\beta\partial_ru-(\partial_r^\ast\beta)u} &= 2\angles{\beta_r\partial_ru}{\partial_ru} +\angles{\beta}{\partial_r|\partial_ru|^2}-\angles{\partial_ru}{(\partial_r^\ast\beta)\partial_ru} \\
 &\quad \quad -\frac{1}{2}\angles{\partial_r\partial_r^\ast\beta}{\partial_r|u|^2}\\
 &= 2\angles{\beta_r\partial_ru}{\partial_ru}+\frac{1}{2}\angles{\Delta(\partial_r^\ast\beta)u}{u}\\
 &= 2\angles{\beta_r\partial_ru}{\partial_ru}-2\rho^2\angles{\beta_ru}{u}-\frac{1}{2}\angles{(\Delta\beta_r)u}{u}\\
 &\quad \quad -\rho\angles{(\Delta(\beta\coth r)-2\rho\beta_r)u}{u}.
\end{split}
\end{align*}
This proves the lemma in the radial case. For the angular derivatives we have
\begin{align*}
\begin{split}
  &-\Re\angles{\sinh^{-2}r\Delta_{\bbS^{d-1}} u}{2\beta\partial_ru-(\partial_r^\ast\beta)u} \\
  &\qquad = \angles{\sinh^{-2}r\beta}{\partial_r|\snabla u|^2}-\angles{\sinh^{-2}r(\partial_r^\ast\beta)\snabla u}{\snabla u}\\
  &\qquad = \angles{\partial_r^\ast(\sinh^{-2}r\beta)-\sinh^{-2}r\partial_r^\ast\beta}{|\snabla u|^2}\\
  &\qquad = 2\angles{\beta \coth r\sinh^{-2}r\snabla u}{\snabla u}.
\end{split}
\end{align*}
To justify the integration by parts we replace $\beta(r)$ by $\chi(2^{-j}r)\beta(r)$, where $\chi$ is as in the statement of the lemma. Letting $j\to\infty$ then justifies the formal computations above. For instance, since $\chi'(2^{-j}\cdot)$ is supported in $[2^j,2^{j+1}]$, the boundary terms
\begin{align*}
\begin{split}
2^{-j}\langle \chi'(2^{-j}r)\beta w,w\rangle,
\end{split}
\end{align*}
with $w$ equal to $u$ or $\nabla u$, go to zero as $j\to\infty$ by \eqref{eq:multdecayassumption}.
\end{proof}


\subsection{Slowly varying functions $\al$ and the multiplier} \label{ss:mults}
Here, we state the specific multipliers that are used in the remainder of the paper. The motivation for our choices will be explained in Section~\ref{s:mult} below. We give these definitions here so that we are able to refer to them already in Section~\ref{s:pr} to prove various estimates, which are needed later in the multiplier argument of Sections~\ref{s:low}--\ref{s:trans}. 

We first introduce the notion of a \emph{slowly varying function} $\al$, which enters in the definition of our multipliers.
In practice, we will recover each slowly varying sequence $\{\al_\ell \}_{\ell \geq 0} \in \calA$ from such a slowly varying function $\al$.

\begin{defnprop}[Slowly varying functions $\al$] \label{d:alpha} 
Given a sequence \\ $\{ \alpha_\ell \}_{\ell \geq 0} \in \calA$, we extend it to a sequence $\{ \alpha_\ell \}_{\ell \in \bbZ}$ by setting $\alpha_{\ell} \simeq 1$ for $\ell < 0$. Then we can find a smooth, positive, slowly varying function $\alpha \colon [0,\infty) \to (0, \infty)$ satisfying
\begin{align*}
 \alpha(y) \simeq \alpha_\ell, \quad y \simeq 2^\ell,\quad\qquad  \alpha'(y) = 0, \quad  \forall \, 0 \leq y \leq 1,
\end{align*}
the  bounds 
\begin{align*}
 \bigl| \alpha^{(j)}(y) \bigr| &\leq C \frac{\alpha(y)}{\langle y \rangle^j}, \quad 0 \leq j \leq 3,\\
 |\alpha'(y)| &\leq C \eta \frac{\alpha(y)}{y}, \quad \forall \, y \ge 1, \\
 |\alpha''(y)| &\leq C \eta \frac{\alpha(y)}{y^2}, \quad \forall \, y \ge 1,\\
 \alpha(y) &\geq \frac{c}{\langle y \rangle^\eta}, \qquad \forall \, y \geq 0,
\end{align*}
for some absolute constants $C, c \geq 1$, and such that the function
\EQ{
 [0,\infty) \ni y \mapsto \frac{\alpha(y)}{\langle y \rangle} 
 }
 is non-increasing. Here $\eta$ is as in Definition~\ref{d:A}. We call a function $\alpha$ satisfying these properties a slowly varying function associated to $\{ \alpha_\ell \}_{\ell \geq 0} \in \calA$. 
\end{defnprop} 

The construction of such an $\alp$ is straightforward, and is left to the reader. We also refer to \cite[p.~1523 or p.~1525]{MMT} and \cite[p.~581]{Tat08ajm} for a similar construction.
 
We are now ready to state the two multipliers which are used in our low- and high-frequency multiplier arguments in Sections~\ref{s:low}--\ref{s:trans}. 
We start with the low-frequency definitions.
\begin{defn} \label{d:beta_low} 
To each slowly varying function $\al$ as in Definition~\ref{d:alpha} we associate a weight function $\beta^{(\alpha)}$ by
\begin{equation} \label{eq:beta_low}
\beta^{(\alpha)}(r) := \int_0^r \frac{\alpha(y)}{\langle y \rangle} \, dy,
\end{equation}
and the vector field $\bsbeta^{(\alpha)}$ by
\begin{align*}
 \begin{split}
  \bsbeta^{(\alpha)} = \bsbeta^{(\alpha),\mu}\partial_\mu,
 \end{split}
\end{align*}
where in polar coordinates $\bsbeta^{(\alpha),r}=\beta^{(\alpha)}$ and $\bsbeta^{(\alpha),\mu}=0$ for $\mu\neq r$. When there is no risk of confusion, we sometimes drop the superscript $(\alpha)$ from the notation of $\beta^{(\alpha)}$ and $\bsbeta^{(\alpha)}$. The self-adjoint low-frequency multiplier $Q^{(\alpha)}$ is then defined as 
\EQ{ \label{eq:Qlowdef1} 
Q^{(\al)} v &= \frac{1}{i} \Big(\bsbeta^{(\alpha), \mu} \na_\mu v + \na_\mu (\bsbeta^{(\alpha), \mu} v) \Big) \\
&=  \frac{1}{i} \Big( 2  \beta^{(\alpha)}\partial_rv + (\p_r \be^{(\alpha)}) v + (d-1) \coth r \be^{(\alpha)} v \Big).
} 
When there is no risk of confusion we write $Q$ instead of $Q^{(\alpha)}$.
\end{defn} 
We will often use two properties of $\be^{(\alpha)}$ as defined above, namely, 
\EQ{ \label{eq:betabounds1} 
 \bigl| \be^{(\alpha)}(r) \bigr| + \bigl| \coth r \be^{(\alpha)}(r) \bigr| \lesssim 1, \mand \bigl| \p_r \be^{(\alpha)} (r) \bigr| \lesssim 1 .
}
The second property above is evident from the definition. Note that for large $r$ the first property follows from the fact that $\{\al_\ell \}_{\ell \geq 0} \in \calA$. Indeed, if $r\geq1$
\EQ{
\abs{\be^{(\alpha)}(r)} +  \abs{\coth r \be^{(\alpha)}(r)}  \lesssim \int_0^\infty \frac{\al(y)}{\ang{y}} \, \ud y  \lesssim 1+  \sum_{\ell \ge 0} \al_\ell \lesssim1.
}

The following is the analogous definition for high frequencies.
\begin{defn}\label{d:beta_high} 
To each slowly varying function $\al$ as in Definition~\ref{d:alpha}, $\delta>0$, and $s>0$, we  associate a weight function $\beta_s^{(\alpha)}$ by
\begin{equation} \label{eq:beta_high}
 \beta_{s}^{(\alpha)}(r) := s^{\frac{1}{2}}\int_0^{\delta s^{-\frac{1}{2}}r} \frac{\alpha(y)}{\langle y \rangle} \, dy,
\end{equation}
and the vector field $\bsbeta_{s}^{(\alpha)}$ by
\begin{align*}
\begin{split}
\bsbeta_{s}^{(\alpha)}=\bsbeta_{s}^{(\alpha),\mu}\partial_\mu,
\end{split}
\end{align*}
where in polar coordinates $\bsbeta_{s}^{(\alpha),r}=\beta_{s}^{(\alpha)}$ and $\bsbeta_{s}^{(\alpha),\mu}=0$ for $\mu\neq r$. When there is no risk of confusion, we sometimes suppress the subscript $s$ or the superscript $(\alpha)$ from the notation in $\beta_{s}^{(\alpha)}$ and $\bsbeta_{s}^{(\alpha)}$. The self-adjoint high-frequency multiplier $Q_s^{(\alpha)}$ is then defined as 
\EQ{\label{eq:Qsdef1}
Q_s^{(\al)} v &= \frac{1}{i} \Big( \bsbeta_{s}^{(\alpha), \mu} \na_\mu v + \na_\mu ( \bsbeta_{s}^{(\alpha), \mu} v) \Big) \\
& = \frac{1}{i} \Big( 2 \beta_s^{(\alpha)} \partial_r v + (\p_r \be_s^{(\alpha)}) v + (d-1) \coth r \be_{s}^{(\alpha)} v \Big).
} 
When there is no risk of confusion, we sometimes suppress the subscript $s$ or the superscript $(\alpha)$ from the notation in $Q_s^{(\alpha)}$.
\end{defn} 

We will often use three properties of $\be_s^{(\alpha)}$ as defined above which hold uniformly in $\{\alpha_\ell\}_{\ell\geq0}\in\calA$, $\delta\in(0,1]$, and $s\in(0,s_0]$ where $s_0>0$ is fixed:
\EQ{ \label{eq:betabounds2} 
\bigl| \be_{s}^{(\alpha)}(r) \bigr| \lesssim s^{\frac{1}{2}} , \quad \bigl| \coth r \be_{s}^{(\alpha)}(r) \bigr| \lesssim_{s_0} 1, \mand \bigl| \p_r \be_{s}^{(\alpha)} (r) \bigr| \lesssim 1 .
}

The relation between the choices of $\beta^{(\alpha)}$ and $\beta_s^{(\alpha)}$ above and our local smoothing spaces is given in the following lemma.

\begin{lem}\label{lem:betaLErelation}
Let $\beta^{(\alpha)}$ be as in Definition~\ref{d:beta_low}, then 
\begin{align*}
\begin{split}
 & \langle \jap{r}^{-2}(\partial_r \beta^{(\alpha)}) v, v \rangle_{t,x}\simeq \| v \|_{X_{\low, \alpha}}^2,\qquad\sup_{\{\alpha_\ell \} \in\calA} \angles{\jap{r}^{-2}(\partial_r\beta^{(\alpha)})v}{v}_{t,x}\simeq \|v\|_{LE_\low}^2.
\end{split}
\end{align*}
Similarly, if $\beta_{s}^{(\alpha)}$ is as in Definition~\ref{d:beta_high}, then
\begin{align*}
\begin{split}
 &s^{-\frac{1}{2}}\langle (\partial_r \beta_s^{(\alpha)}) v, v \rangle_{t,x}\simeq\| v \|_{X_{s, \alpha}}^2  ,\qquad\sup_{\{ \alpha_\ell \}\in\calA}s^{-\frac{1}{2}}\angles{(\partial_r\beta_{ s}^{(\alpha)})v}{v}_{t,x}\simeq \|v\|_{LE_s}^2,
\end{split}
\end{align*}
where the implicit constants are independent of $s$ and uniform in $\delta\in(0,1]$.
\end{lem}
We omit the straightforward proof.


\section{Overview of the scheme and outline of the paper} \label{s:mult}


In this section we describe the general scheme of the proofs of the main theorems. After this, we provide an outline of the remainder of the paper.


\subsection{Outline of the proof of Theorems~\ref{t:LE1-sym} and \ref{t:LE1}} \label{ss:LE1-outline}


In what follows, we focus on Theorem~\ref{t:LE1}, from which Theorem~\ref{t:LE1-sym} follows.

\subsubsection*{Outline of the proof, take 1: Caricature of the proof}
Let $Q$ be a time-independent self-adjoint operator. Then, 
\EQ{ \label{eq:dtQu} 
 \frac{1}{2} \frac{\ud}{\ud t} \langle u, Q u \rangle &= \re \, \langle (\partial_t + i H) u, Q u \rangle  - \re \, \langle i H u, Q u \rangle, 
}
and
\EQ{
-  \re \, \langle i H u, Q u \rangle =   \re \, \langle  i \De u, Q u \rangle  -    \re \, \langle i (H_{\prin} + \De) u, Q u \rangle -   \re \, \langle i H_{\lot} u, Q u \rangle.
}
Inserting the last two lines into the right-hand side of~\eqref{eq:dtQu}, we arrive at the identity
\EQ{
\re \, \langle  i \De u, Q u \rangle  &=   \frac{1}{2} \frac{\ud}{\ud t} \langle u, Q u \rangle  - \langle (\partial_t + i H) u, Q u \rangle  \\
& \quad +  \re \, \langle i (H_{\prin} + \De) u, Q u \rangle  +  \re \, \langle i H_{\lot} u, Q u \rangle.
}
On the other hand, we have the differential mass conservation law
\begin{equation} \label{equ:approx_mass_conservation}
 \frac{\ud}{\ud t} \, \| u(t) \|_{L^2}^2 = 2 \re \, \langle (\partial_t + i H) u, u \rangle , 
\end{equation}
which implies the approximate conservation of mass
\begin{equation} \label{equ:approx_mass_conservationalt}
\|u(t)\|_{L^2}^2\leq \|u(0)\|_{L^2}^2+ 2 \abs{\re \, \langle (\partial_t + i H) u, u \rangle_{t,x}}.
\end{equation}
Now if we can choose $Q$ obeying the key coercive bound
\begin{equation} \label{eq:Qcondstemp0}
 \re \, \langle  i \De u, Q u \rangle \gtrsim \|u\|_{LE}^2- C\|u\|_{L^2(\bbR\times\{r\leq R\})}^2,
\end{equation}
and upper bounds
\begin{align}\label{eq:Qcondstemp1}
\begin{cases}
 &\|Qu\|_{L^2}\lesssim \|u\|_{L^2}, \qquad \|Qu\|_{LE}\lesssim \|u\|_{LE}\\
 &|\Re\langle i H_{\lot} u, Q u \rangle_{t,x}|\lesssim \varepsilon \|u\|_{LE}^2 + C \|u\|_{L^2(\bbR\times\{r\leq R\})}^2\\
 &|\re \, \langle i (H_{\prin} + \De) u, Q u \rangle_{t,x}| \lesssim \varepsilon \|u\|_{LE}^2 + C \|u\|_{L^2(\bbR\times\{r\leq R\})}^2
\end{cases},
\end{align}
then from \eqref{eq:dtQu} and \eqref{equ:approx_mass_conservationalt}
\begin{align*}
\begin{split}
 \|u\|_{LE}^2 &\lesssim \sup_{t\in\bbR} \, \|u(t)\|_{L^2}^2 + \|F\|_{LE^\ast}\|u\|_{LE} + \varepsilon\|u\|_{LE}^2+\|u\|_{L^2(\bbR\times\{r\leq R\})}^2\\
 &\lesssim \|u(0)\|_{L^2}^2 + \|F\|_{LE^\ast}^2+\|u\|_{L^2(\bbR\times\{r\leq R\})}^2 + \varepsilon \|u\|_{LE}^2,
\end{split}
\end{align*}
which yields \eqref{eq:LE1} if $\varepsilon>0$ is chosen sufficiently small. 

The above procedure is roughly how we will go about proving Theorem~\ref{t:LE1}. However, due to the frequency localized nature of the local smoothing space $LE$, the operator $Q$ needs to be chosen differently depending on the frequency localization of $u$. Our goal will then be to reformulate the requirements \eqref{eq:Qcondstemp1} at a frequency localized level, in particular in terms of the spaces $LE_s$ and $LE_s^\ast$, and choose $Q_s$ for which the latter are satisfied when $u$ is replaced by $\tilP_s u$ or $\tilP_{\geq s} u$. Unfortunately, as our heat flow frequency localizations are not sharply localized, this process will yield errors which make the entire scheme more complicated. Moreover, we have to carefully analyze the interplay between spatial localizations and frequency projections in a manner that is consistent with the uncertainty principle. This is especially challenging because our frequency projections are based on the heat flow. 

\subsubsection*{Outline of the proof, take 2: Frequency localized estimates}
More precisely, consider the frequency localized equations
\begin{align}\label{eq:proj-temp1}
\begin{split}
 (\partial_t - i\Delta) \tilP_{\geq s} u&=\tilP_{\geq s}F- \tilP_{\geq s}H_\lot u-i\tilP_{\geq s}(H_\prin +\Delta)u,\\
 (\partial_t - i\Delta) \tilP_s u&=\tilP_sF -\tilP_sH_\lot u- i\tilP_s(H_{\prin}+\Delta) u,
\end{split}
\end{align}
where $0 < s \leq s_{0}$ for some $s_{0} > 0$ to be fixed later. We use a family of multipliers of the form $\set{Q^{(\alp)}}_{\alp}$ as in Definition~\ref{d:beta_low} for $\tilP_{\geq s} u$ with $s \aeq 1$, and $\set{Q^{(\alp)}_{s}}_{\alp}$ as in Definition~\ref{d:beta_high} for $\tilP_{s} u$ with $s \aleq 1$, where $\alp_{\ell}$ varies over $\calA$ (see Proposition/Definition~\ref{d:alpha}).

\subsubsection*{Motivation for the frequency-localized multipliers}
Let us give a heuristic motivation for such multipliers in the high frequency case, the low frequency case being similar. A basic requirement for the frequency-dependent multiplier $Q_{s}$ for $\tilP_{s} u$ is that $Q_{s} \tilP_{s}$ is bounded on $L^{2}$ (cf. the first line in \eqref{eq:Qcondstemp1}). Due to the frequency projection, a derivative is ``at most of size $s^{-\frac{1}{2}}$'' on $\tilP_{s} u$. Thus, for a multiplier of the form $Q_{s} = \frac{1}{i}(\bt_{s}(r) \rd_{r} - \rd_{r}^{\ast} \bt_{s}(r))$, the $L^{2}$-boundedness requirement leads to the condition $\abs{\bt_{s}} \leq C s^{-\frac{1}{2}}$, or equivalently
\begin{equation} \label{eq:proj-high-motivate}
	\abs{\bt_{s}(0) + \int_{0}^{r} \rd_{r} \bt_{s}(y) \, \ud y} \leq C s^{\frac{1}{2}}.
\end{equation}
On the other hand, recall from \eqref{eq:mult-positivity} that for coercivity of $i[Q_{s}, \lap]$, we wish to take $\bt_{s}(0) = 0$ and $\rd_{r} \bt_{s}$ to be positive, nonincreasing and as large as possible. These considerations motivate using as $\rd_{r} \bt_{s}$ a family of positive, nonincreasing and (barely) integrable functions of the form 
\begin{equation*}
	\frac{\alp(s^{-\frac{1}{2}} r)}{\brk{s^{-\frac{1}{2}} r}}
\end{equation*}
where $\alp$ varies in the family of slowly varying functions (see Definition~\ref{d:alpha}). Note that the scaling factor $s^{-\frac{1}{2}}$ matches with both \eqref{eq:proj-high-motivate} and the smallest possible spatial localization scale for $\tilP_{s} u$ (uncertainty principle). Our actual choice (Definition~\ref{d:beta_high}) is a tiny modification of the above, where an extra small scaling factor of $\dlt$ is added to make higher derivatives of the multiplier even smaller.

\subsubsection*{Coercivity for $\Re \brk{i \lap \tilP_{s} u, Q^{(\alp)}_{s} \tilP_{s}u}$ (Section~\ref{s:high})}
We now describe how the analogue of \eqref{eq:Qcondstemp0} is proved in our scheme.
We start with the high frequency multipliers. Recall from \eqref{eq:mult-positivity-top} in Section~\ref{ss:multiplier} that
\begin{equation*}
	\Re \brk{i \lap \tilP_{s} u, Q_{s}^{(\alp)} \tilP_{s} u} 
	\geq 2 \brk{ (\rd_{r} \bt_{s}^{(\alp)}) \nb \tilP_{s} u, \nb \tilP_{s} u} + \cdots,
\end{equation*}
where all lower order terms are omitted. Heuristically, since $\tilP_{s} u$ localizes $u$ to frequencies $\aeq s^{-\frac{1}{2}}$, and since $\rd_{r} \bt^{(\alp)}_{s}$ is slowly varying at spatial scales $\ageq s^{\frac{1}{2}}$, the top order term on the RHS should be comparable to $s^{-1} \brk{(\rd_{r} \bt_{s}^{(\alp)}) \tilP_{s} u, \tilP_{s} u}$. This term would dominate the lower order terms for small $s$, hence the desired coercivity would follow. Indeed, $\rd_{r} \bt_{s}^{(\alp)}$ is precisely the weight we use for the space $X_{s, \alp}$, and thus for $LE_{s}$. 

In making this discussion rigorous, however, we are faced with the issue that our heat flow based Littlewood-Paley projection $\tilP_{s}$ does not completely cutoff the frequencies below $s^{-\frac{1}{2}}$; this feature, in turn, arose from the desire to have a favorable $L^{p}$ theory $(p \neq 2)$, as discussed in Remark~\ref{rem:LP}. As a result, $\nb \tilP_{s} u$ cannot be directly related to $s^{-\frac{1}{2}} \tilP_{s} u$ with the same $s$, and we need to incur some low frequency (i.e., $\tilP_{\sgm} u$ with $\sgm > s$) error. See Remark~\ref{rem:no-bernstein} for more discussion.

With the above lessons in mind, the basic idea is to integrate the frequency-localized multiplier identities in $s$ in the range $0 < s \aleq 1$. Then $\nb \tilP_{s}$ can be replaced by $s^{-\frac{1}{2}} \tilP_{s}$ through an integration by parts argument (see Lemma~\ref{lem:integrated_bernstein_substitute}), from which we get the key coercivity property. The actual bound, which is a bit technical to state here, can be found in Proposition~\ref{prop:high_frequency_local_smoothing}. 

Let us close our heuristic discussion with a few remarks. First, as expected, our argument incurs a low frequency error term of the form
\begin{equation} \label{eq:low-freq-err} 
	\sup_{\{ \alpha_\ell \} \in \calA} \int_{I} \frac{1}{s} \brk{(\rd_{r} \bt_{s}^{(\alp)}) \tilP_{s} u, \tilP_{s} u} \, \ds,
\end{equation}
where $I = [s_{2}/8, s_{2}]$ with $s_{2} \aeq 1$. Treatment of this term is an important issue that will be discussed below. Second, in the proof of Proposition~\ref{prop:high_frequency_local_smoothing}, a technical subtlety arises when changing the order of integration in $s$ and supremum in $\alp$. It is dealt with by working with a suitable average of $\frac{1}{s} \brk{(\rd_{r} \bt_{s}^{(\alp)}) \tilP_{s} u, \tilP_{s} u}$ in $s$.

\subsubsection*{Coercivity for $\Re \brk{i \lap \tilP_{\geq s} u, Q^{(\alp)} \tilP_{\geq s} u}$ (Section~\ref{s:low})}
Leaving aside the issue of treating the low-frequency error \eqref{eq:low-freq-err}, we now turn to the issue of coercivity for $\tilP_{\geq s} u$. This discussion will explain the discrepancy between the spatial weights in $LE_{\low}$ and in $LE_{s}$.

Unlike in the high frequency case, the lower order terms cannot be simply treated as errors, and their precise forms matter. Recall from Section~\ref{ss:multiplier} that
\EQ{\label{eq:low-freq-key}
	& \Re \brk{i \lap \tilP_{\geq s} u, Q^{(\alp)} \tilP_{\geq s} u} \\ 
	& \geq 2 \brk{ (\rd_{r} \bt^{(\alp)}) \nb \tilP_{\geq s} u, \nb \tilP_{\geq s} u} - 2 \rho^{2} \brk{ (\rd_{r} \bt^{(\alp)}) \tilP_{\geq s} u, \tilP_{\geq s} u} \\
	& \phantom{\geq}- 2 \rho \brk{ (\rd_{r}^{2} \bt^{(\alp)}) \coth r \tilP_{\geq s} u, \tilP_{\geq s} u} + \cdots,
}
where the zeroth order terms with better decaying weights are omitted. If $\rd_{r} \bt^{(\alp)} \equiv 1$, then the RHS would be nonnegative by Poincar\'e's inequality. Strengthening this simple observation, we prove a \emph{weighted Hardy--Poincar\'e inequality} (Lemma~\ref{lem:HP}), which allows us bound the main terms on the RHS from below by 
\begin{equation} \label{eq:low-freq-main}
\brk{r^{-2} (\rd_{r} \bt^{(\alp)}) \tilP_{\geq s} u, \tilP_{\geq s} u}.
\end{equation}
This inequality explains the extra factor of $r^{-1}$ in the definition of $X_{\low, \alpha}$, and thus in $LE_{\low}$. 

\subsubsection*{Transitioning estimate (Section~\ref{s:trans})}
To conclude the discussion on coercivity, it remains to treat the low frequency error \eqref{eq:low-freq-err}. For the purpose of this heuristic discussion, we take $\rd_{r} \bt^{(\alp)}_{s} = \rd_{r} \bt^{(\alp)}$ in \eqref{eq:low-freq-err}, since $s \aeq 1$. 

In order to use \eqref{eq:low-freq-main}, which controls $\tilP_{\geq s} u$, we write $\tilP_{s} u = - s(\lap + \rho^{2}) \tilP_{\geq s} u = - s\rd_{s} \tilP_{\geq s} u$ and integrate by parts. A-priori, one may be worried that the spatial weight $\rd_{r} \bt^{(\alp)}$ in \eqref{eq:low-freq-err} is too big compared to that in \eqref{eq:low-freq-main}. However, we may arrange so that
\begin{equation} \label{eq:transition-key}
\eqref{eq:low-freq-err} \aleq \int_{\tilde{I}} \abs{\brk{(\rd_{r} \bt^{(\alp)}) \tilP_{\geq s} u, (\lap + \rho^{2}) \tilP_{\geq s} u}}\, \ds + \cdots,
\end{equation}
where $\tilde{I}$ is a slight enlargement of $I$, and we have omitted easily treated terms with sufficiently decaying weights. Now observe that, after an extra spatial integration by parts, the bad term coincides with the main term in \eqref{eq:low-freq-key} before application of the Hardy-Poincar\'e inequality! Therefore, the whole RHS of \eqref{eq:transition-key} can be estimated using the low-frequency multiplier identity. For details, see Lemma~\ref{lem:transitioning_estimate}. We remark that in order to see the key structure in \eqref{eq:transition-key}, it is crucial to use the shifted Laplacian $\lap + \rho^{2}$ to define the frequency projections $\tilP_{s}$, $\tilP_{\geq s}$.

\begin{rem} \label{rem:optimal-weight}
As noted in Remark~\ref{rem:LE-weight}, the $r$-weight in $LE_{\low}$ is not optimal in comparison to the free case \cite{Kaizuka1}. The loss occurs due to our use of a single type of multipliers $Q^{(\alp)} = \frac{1}{i} (\bt^{(\alp)} \rd_{r} - \rd_{r}^{\ast} \bt^{(\alp)})$ for all low frequencies, which decay (and thus are nonoptimal) at very low frequencies. We expect that this loss may be remedied by fully decomposing the low frequencies and using frequency-dependent multipliers.
\end{rem}

\subsubsection*{Boundedness of the multiplier $Q$ (Section~\ref{ss:Q-bdd})}
In the remainder of this subsection, we describe how the analogues of the upper bounds \eqref{eq:Qcondstemp1} are proved in our scheme. In what follows, we focus on the high-frequency multipliers $Q_{s}^{(\alp)}$, as the story for $Q^{(\alp)}$ is similar but only simpler. 

We begin with the boundedness properties of our frequency-localized multipliers, i.e., the analogue of the first line in \eqref{eq:Qcondstemp1}. The goal is to establish\footnote{Note that $\tilP_{s}$ is replaced by $\tilP_{\frac{s}{2}}$ on the RHS. This replacement occurs because $\nb \tilP_{s} u$ cannot be directly related with $\tilP_{s} u$, which in turn is because our heat flow based Littlewood-Paley projection $\tilP_{s}$ does not completely cutoff frequencies above $s^{-\frac{1}{2}}$. Fortunately, unlike the reverse direction, which necessitated the delicate transitioning estimate as discussed above, this issue turns out to be only a minor nuisance.}
\begin{align*} 
	\nrm{Q^{(\alp)}_{s} \tilP_{s} u}_{L^{2}} &\aleq  \nrm{\tilP_{\frac{s}{2}} u}_{L^{2}}, \\
	\nrm{Q^{(\alp)}_{s} \tilP_{s} u}_{X_{s, \gmm}} &\aleq  \nrm{\tilP_{\frac{s}{2}} u}_{X_{s, \gmm}},
\end{align*}
where $\alp, \gmm \in \calA$ and $s \aleq 1$. 

As discussed above, the first bound ($L^{2}$-boundedness) is a simple consequence of the bound \eqref{eq:proj-high-motivate}; see Proposition~\ref{p:QsL2}. The proof of the second bound ($X_{s, \gmm}$-boundedness) is far more tedious due to the presence of spatial weights; see Proposition~\ref{p:QsX}. The core principle, however, is quite simple: To reduce the matter to application of what we call the \emph{mismatch estimates} (Proposition~\ref{p:lpregcore}), which are estimates for the heat semi-group $e^{s \lap}$ with the input and the output localized in dyadic spatial annuli. We note that these estimates are only effective for spatial scales $\ageq s^{\frac{1}{2}}$ (uncertainty principle), which is consistent with our choice of $\bt_{s}^{(\alp)}$.

\subsubsection*{Contribution of $H_{\lot}$ (Sections~\ref{ss:lot-LE*} and \ref{ss:lot-comm})}
Next, we turn to the analogue of the second line in \eqref{eq:Qcondstemp1} for the frequency-localized multipliers. The goal is to prove
\begin{equation} \label{eq:lot-goal}
	\int_{0}^{2} \sup_{\set{\alp_{\ell}} \in \calA} \abs{\Re \brk{i \tilP_{s} H_{\lot} u, Q_{s}^{(\alp)} \tilP_{s} u}} \, \ds
	\leq \veps \nrm{u}_{LE}^{2} + C_{\veps} \nrm{u}_{L^{2}(\bbR \times \set{r \leq R}}^{2}
\end{equation}
for any $\veps > 0$ and $R = R(\veps, H_{\lot}) \geq 1$.

Our first key ingredient is the following set of estimates, which exploit the gain of one derivative, at the expense of growing spatial weights, when passing from $LE$ to $LE^{\ast}$:
\begin{align}
	\nrm{\bsb^{\mu} \nb_{\mu} u}_{LE^{\ast}} + \nrm{\nb_{\mu} (\bsb^{\mu} u)}_{LE^{\ast}} &\aleq C(\bsb) \nrm{u}_{LE}, \label{eq:LEstar-LE-b}\\
	\nrm{V u}_{LE^{\ast}} &\aleq C(V) \nrm{u}_{LE}. \label{eq:LEstar-LE-V}
\end{align}
Here, $C(\bsb)$ and $C(V)$ are positive constants that depend sublinearly on $\bsb$ and $V$, respectively. For the full statement, see Propositions~\ref{p:Hlotbound_for_V} and \ref{p:Hlotbound}. Their proofs are essentially applications of the mismatch estimates (Proposition~\ref{p:lpregcore}).

The preceding bounds are sufficient to handle the contribution of $H_{\lot}$ when $\bsb, V$ are small, e.g., in the exterior region $\set{r \geq R}$ for sufficiently large $R$ (where $\bsb, V$ are small due to our decay assumptions), as well as the globally small imaginary parts of $\bsb, V$. Thus, to prove \eqref{eq:lot-goal}, it remains to handle
\begin{equation} \label{eq:lot-large}
	\int_{0}^{2} \sup_{\set{\alp_{\ell}} \in \calA} \abs{\Re \brk{i \tilP_{s} H^{\sym}_{\lot}  (\chi_{r \leq R} u), Q_{s}^{(\alp)} \tilP_{s} (\chi_{r \leq R} u)}} \, \ds
\end{equation}
where $H^{\sym}_{\lot} u = \frac{1}{i}(\Re \bsb^{\mu} \nb_{\mu} u + \nb_{\mu} (\Re \bsb^{\mu} u)) + \Re V u$. 

To treat \eqref{eq:lot-large}, we first commute $\tilP_{s}$ and $H^{\sym}_{\lot}$, which gains a derivative; see Lemma~\ref{lem:commutator_B_Ps}. Then we may exploit a commutator structure arising from symmetry of $H^{\sym}_{\lot}$ and $Q_{s}^{(\alp)}$:
\begin{align*}
\begin{split}
\Re \brk{i H^{\sym}_{\lot} v, Q_{s}^{(\alp)} v} =\frac{1}{2} \brk{i [Q_{s}^{(\alp)}, H^{\sym}_{\lot}] v, v},
\end{split}
\end{align*}
where $v = \tilP_{s} (\chi_{r \leq R} u)$. The commutator $[Q_{s}^{(\alp)}, H^{\sym}_{\lot}]$ entails a smoothing effect\footnote{If our multiplier were a classical order zero operator, then this commutator structure would have gained a full derivative, and \eqref{eq:lot-large} would be bounded by $\nrm{u}_{L^{2}(\bbR \times \set{r \leq R}})$. However, when the derivatives in the commutator fall on the coefficients of our frequency localized multiplier $Q_{s}^{(\alp)}$, it is possible that an extra factor of $s^{-\frac{1}{2}}$ is produced (in other words, our multiplier does not belong to the classical symbol class). Nevertheless, using the precise structure of $Q_{s}^{(\alp)}$, the resulting error is again bounded by a small multiple of $\|u\|_{LE}$, which is acceptable.}, which allows us to bound \eqref{eq:lot-large} by the RHS of \eqref{eq:lot-goal}.

\subsubsection*{Contribution of $H_{\prin} + \Dlt$ (Sections~\ref{ss:commute_with_H_prin} and \ref{ss:DeltaHdifference})}
Finally, we discuss the contribution of second order perturbations $H_{\prin} + \lap$, i.e., the analogue of the third line in \eqref{eq:Qcondstemp1}. The overall idea is to exploit a commutator structure, as in the case of $H^{\sym}_{\lot}$. The first step is to commute $\tilP_{s}$ and $H_{\prin}$ and prove an estimate of the form
\begin{equation} \label{eq:prin-comm}
\int_{0}^{2} \sup_{\set{\alp_{\ell}} \in \calA} \abs{\Re \brk{i [\tilP_{s}, H_{\prin}] u, Q_{s}^{(\alp)} \tilP_{s} u}} \ds \aleq C(\ringa) \nrm{u}_{LE}^{2},
\end{equation}
where $C(\ringa)$ is a positive constant that depends sublinearly on $\ringa := \bsa - \bsh^{-1}$. 
The idea is that the commutator $[\tilP_{s}, H_{\prin}]$ essentially gains a derivative (as discussed in relation to Lemma~\ref{lem:commutator_B_Ps}), so that \eqref{eq:prin-comm} is roughly at the same difficulty as \eqref{eq:LEstar-LE-b}; see Proposition~\ref{p:commP_sH_prin} in Section~\ref{ss:commute_with_H_prin} for more details.

By \eqref{eq:prin-comm} and the fact that $\tilP_{s}$ commutes with $\lap$, we are left to estimate
\begin{equation*}
\Re \brk{i (H_{\prin} + \lap) \tilP_{s} u, Q_{s}^{(\alp)} \tilP_{s} u}.
\end{equation*}
Using the symmetry of $H_{\prin} + \lap$ (where $\tilP_{s}$ clearly commutes with $\lap$), we may rewrite the preceding expression as
\begin{equation*}
\frac{1}{2}  \brk{i [Q_{s}^{(\alp)}, H_{\prin} + \lap] \tilP_{s} u, \tilP_{s} u}.
\end{equation*}
Here the point is that $H_\prin+\Delta$ is again a self-adjoint second order operator with small decaying coefficients. This allows us to calculate the error term $\langle i[Q,(H_\prin+\Delta)]u,u\rangle$ in a similar way to the main commutator $\langle i[Q,\Delta] u,u\rangle$, and absorb the corresponding error terms in this way. The detailed argument can be found in Proposition~\ref{p:DeltaHpQhigh} in Section~\ref{ss:DeltaHdifference}.

\begin{rem} \label{rem:suboptimal-w}
There is a technical point concerning our proof of Proposition~\ref{p:DeltaHpQhigh}, which ultimately results in the vanishing condition \eqref{eq:vanishing_assumption-prin} for $\bsa^{r \tht_{a}}$. It turns out that in the region $\set{r \ll 1}$, $\nrm{u}_{LE}^{2}$ is too weak to bound the contribution of $\ringa^{\tht_{a} \tht_{b}}$ and $\bsa^{r \tht_{a}}$. For errors of the first type, this issue is fixed by exploiting the extra positive angular derivative term in \eqref{eq:Q1}; see Proposition~\ref{p:DeltaHpQhigh} and the corresponding introduction of $\tilde{F}_{s}^{(\alp)}$ in Section~\ref{s:high}. For errors of the second type, however, an extra assumption such as \eqref{eq:vanishing_assumption-prin} is necessary.

We note that this problem is due to the aforementioned issue that our multiplier is not a classical symbol, or more precisely, that a derivative of the coefficient of $Q_{s}^{(\alp)}$ incurs a factor of size $s^{-\frac{1}{2}}$ in the small $r$-region. One way to avoid it, and hence to drop the extraneous assumption \eqref{eq:vanishing_assumption-prin}, is to use 
\begin{equation*}
	\bt_{s}^{(\alp)} = s^{\frac{1}{2}} \int_{0}^{\dlt r} \frac{\alp(y)}{\brk{y}} \, \ud y,
\end{equation*}
instead of \eqref{eq:beta_high} for the high-frequency multipliers. However, this choice would lead to a weaker local smoothing space with a nonoptimal, frequency-independent spatial weight in the region $\set{r \ll 1}$.
\end{rem}


\subsection{Outline of the proof of Theorems~\ref{t:LE-H} and \ref{t:LE2}} \label{ss:LE-H-outline}

\subsubsection*{Limiting absorption principle and resonances (Sections~\ref{ss:elliptic_reg} and \ref{ss:LE-res})}
A well-known philosophy is that when $H$ is a symmetric and stationary perturbation of the Laplacian, Theorem~\ref{t:LE-H} is equivalent to a \emph{limiting absorption principle}. More precisely, essentially via Fourier transform in time, one can show that the local smoothing estimate in Theorem~\ref{t:LE-H} is equivalent to
\begin{align*}
\begin{split}
 \| P_c v \|_{ \tilLE_0} \lesssim  \| (\tau \pm i\eps - H) P_c v \|_{ \tilLE_0^*},
\end{split}
\end{align*}
uniformly in $\epsilon > 0$ and $\tau \in \bbR$, where $\tilLE_0$ and $\tilLE_0^\ast$ denote the spatial local smoothing norms introduced in Section~\ref{s:prelim}. 
Similarly Theorem~\ref{t:LE1-sym} (or rather the weaker version resulting from replacing $LE$ and $LE^\ast$ by $\tilLE$ and $\tilLE^\ast$ respectively) is equivalent to
\begin{equation} \label{eq:LE-H-outline-LE1}
\begin{split}
 \| v \|_{ \tilLE_0} \lesssim  \|(\tau \pm i\eps - H) v \|_{ \tilLE_0^*} + \| v \|_{L^2(\{r \le R\})},
\end{split}
\end{equation}
uniformly in $\epsilon > 0$ and $\tau \in \bbR$. A detailed proof of these (standard) equivalences can be found in Section~\ref{ss:LE-res} below.

Assuming that Theorem~\ref{t:LE-H} fails, we use a compactness argument to extract a solution $v_\infty\in \tilLE_0$ to 
\begin{align*}
\begin{split}
Hv_\infty =\tau v_\infty,
\end{split}
\end{align*}
which we may arrange so that $P_{c} v_{\infty} = v_{\infty}$ and $\|v_{\infty}\|_{L^2(\{r\leq R\})}=1$ using \eqref{eq:LE-H-outline-LE1}. To derive a contradiction, the first key step is to show that $v_\infty$ is a \emph{resonance}, which roughly is equivalent to satisfying one of the decay conditions
\begin{align*}
\begin{split}
\lim_{j\to\infty}\|r^{-\frac{1}{2}}(\partial_r+\rho\mp i\sqrt{\tau^2-\rho^2})v_\infty\|_{L^2(A_j)}= 0,
\end{split}
\end{align*}
which are known as the \emph{outgoing} or \emph{incoming radiation condition}. We refer to Proposition~\ref{p:LE-res} for a precise formulation of this part of the argument. 

Proof of the above decay condition again relies on a positive commutator estimate of the form described above with a different choice of multiplier $Q$. Here, unlike in the proof of Theorem~\ref{t:LE1} we have opted to choose the multiplier $Q$ independently of the frequency localization, so we work with $u$ rather than $\tilP_su$ or $\tilP_{\geq s}u$ throughout. Such a choice leads to new error terms involving $\nb u$, which are treated by elliptic regularity estimates in the space $\tilLE_{0}$ established in Section~\ref{ss:elliptic_reg}.


We remark that while using a frequency-independent multiplier is appealing in that it makes the arguments less tedious, it has the disadvantage that it does not allow us to distinguish between the difference in spatial weights in $LE_\low$ and $LE_s$. Indeed this is the reason we use the more uniform spaces $\tilLE$ and $\tilLE^\ast$ for Theorem~\ref{t:LE-H}.


\subsubsection*{Absence of embedded resonances (Section~\ref{ss:no-emb-res})}
To complete the proof of Theorem~\ref{t:LE2}, it remains to show that there are no resonances such as $v_{\infty}$. The cases $\tau \in (-\infty, \rho^{2})$ and $\tau = \rho^{2}$ are ruled out by the condition $P_{c} v_{\infty} = v_{\infty}$ and the hypothesis (i.e., threshold nonresonance), respectively. Therefore, we are left with the task of establishing the absence of any resonances embedded in the continuous spectrum $(\rho^{2}, \infty)$. 

Our proof proceeds in three steps. First, using yet another multiplier argument, we show that $v_{\infty}$ must decay faster than any polynomial in $r$ as $r \to \infty$ (see Lemma~\ref{l:poly-decay-key}). The radiation condition is used to ensure that the boundary terms arising in the multiplier argument vanish. Second, we prove that any solution to $(H - \tau) v_{\infty} = 0$ for $\tau > \rho^{2}$ with such a rapid decay must, in fact, be compactly supported. On the one hand, this part resembles the previous proof of the absence of embedded resonances in the case $H_{\lot} = 0$ by Donnelly \cite{Donnelly1} (see also Borthwick--Marzuola \cite{BorMar1} for the case  $H_{\lot} = V$), and on the other hand, it may be thought of as an application of a Carleman-type inequality from infinity. In the third and final step, we apply a standard unique continuation result for elliptic PDEs to conclude that $v_{\infty}$ is identically zero.

\begin{rem} 
Amusingly, the proof of the absence of embedded resonances for self-adjoint, first and zeroth order perturbations of the Laplacian seems simplified by the hyperbolic geometry.  Our simple choices of the multipliers and weights rely crucially on the exponential volume growth of the spheres $\set{r = \hbox{const}}$, and fail in the case of Euclidean space. Indeed, to the best of our knowledge, the argument in the Euclidean case becomes much more involved in the presence of a first order perturbation; compare \cite[Thm.~14.7.2]{Hor2} (for only zeroth order perturbations) with \cite[Thm.~17.2.8]{Hor3} (which also includes first order perturbations).
\end{rem}   

\subsubsection*{Outline of the proof of Theorem~\ref{t:LE2} (Sections~\ref{ss:pert-reg} and \ref{ss:LE2-pf})}
Theorem~\ref{t:LE2} is a simple consequence of the local smoothing estimates and the Littlewood-Paley machinery developed in this paper. We begin by splitting $u$ into the high and low frequency parts, i.e., $u = (1-\tilP_{\geq s_{0}}) u + \tilP_{\geq s_{0}} u$ for an appropriately small $s_{0}$, and claim that
\begin{equation} \label{eq:LE2-pf-key}
	\nrm{u}_{\tilLE} \aleq \nrm{u_{0}}_{L^{2}} + \nrm{F}_{\tilLE^{\ast}} + (s_{0}^{\dlt_{0}} + s_{0}^{-1} \kpp) \nrm{u}_{\tilLE}
\end{equation}
for some positive $\dlt_{0} > 0$. Then the desired local smoothing estimate would follow by taking $s_{0}, \kpp$ sufficiently small, and then absorbing the last term in the LHS. 

To prove \eqref{eq:LE2-pf-key}, by Theorem~\ref{t:LE1}, it suffices to estimate $\nrm{u}_{L^{2}(\bbR \times \set{r \leq R})}$ by the RHS of \eqref{eq:LE2-pf-key}. The contribution of the high frequency part $(1-\tilP_{\geq s_{0}}) u$ is easily bounded by $s_{0}^{\frac{1}{4}} \nrm{u}_{\tilLE}$, which is acceptable. For the low frequency bulk $\tilP_{\geq s_{0}} u$, we apply Theorem~\ref{t:LE-H} to the equation
\begin{equation*}
	(-i \rd_{t} + H_{\stat}) \tilP_{\geq s_{0}} u = \tilP_{\geq s_{0}} F - [\tilP_{\geq s_{0}}, H_{\stat}] u - \tilP_{\geq s_{0}} H_{\pert} u.
\end{equation*}
The estimates needed for treating the $\tilLE^{\ast}$ norm of the RHS are proved in Section~\ref{ss:pert-reg}. For $\nrm{[\tilP_{\geq s_{0}}, H_{\stat}] u}_{\tilLE^{\ast}}$, the idea is to gain a positive power of $s_{0}$ from the commutator structure. On the other hand, $\nrm{\tilP_{\geq s_{0}} H_{\pert} u}_{\tilLE^{\ast}}$ is bounded by $s_{0}^{-1} \kpp \nrm{u}_{\tilLE}$ by letting the derivatives fall on the low frequency projection and using the bounds \eqref{eq:cLE2conditions-prin}--\eqref{eq:cLE2conditions} on the coefficients. 


\subsection{Outline of the remainder of the article} \label{ss:outline}
In Section~\ref{s:pr}, which is the longest section of the paper, we develop the heat flow based Littlewood-Paley theory and prove various general estimates in the local smoothing spaces, as discussed above. A recurring theme is the uncertainty principle, or more concretely, the analysis of the relation between frequency and spatial localizations. Our key tools are the \emph{localized parabolic regularity estimates} proved in Section~\ref{ss:lpreg}. 

Sections~\ref{s:low}--\ref{s:trans} contain the proof of Theorem~\ref{t:LE1}, and therefore also of Theorem~\ref{t:LE1-sym}. The low and high-frequency estimates are derived in Sections~\ref{s:low} and~\ref{s:high} respectively, but as mentioned above additional work is required to combine these into a uniform estimate. This is achieved in Section~\ref{s:trans}. 

In Section~\ref{s:error}, we prove Theorem~\ref{t:LE-H}, Proposition~\ref{p:no-th-res} and Theorem~\ref{t:LE2}. Finally, we establish Corollaries~\ref{c:LE1} and \ref{c:Strichartz} in Section~\ref{s:cors}.


\section{Regularity theory of the heat flow: Littlewood-Paley theory} \label{s:pr} 

In this section we develop the heat flow based Littlewood-Paley theory and establish various general estimates in the local smoothing spaces that will be needed in Sections~\ref{s:low}--\ref{s:cors}. A key ingredient for most of the proofs is a quantitative understanding of the uncertainty principle, that is, the relation between spatial and frequency localization, which is captured in the localized parabolic regularity estimates derived in Section~\ref{ss:lpreg}.

We begin by recalling the definitions of the heat flow frequency projections associated with the shifted Laplacian as well as the usual Laplacian on hyperbolic space. For any heat time $0 < s < \infty$ we use the notations 
\[
 \tilP_{\geq s} u := e^{s(\Delta+\rho^2)} u,\qquad \tilP_s u := -s\partial_s \tilP_{\geq s}u=- s (\Delta +\rho^2)e^{s(\Delta+\rho^2)} u
\]
and 
\[
 P_{\geq s} u := e^{s \Delta} u, \qquad P_s u := -s\partial_s P_{\geq s} u = -s \Delta e^{s \Delta} u.
\]
Then we have for any $0 < s_0 < \infty$ the continuous resolutions 
\begin{align*}
 u &= \int_0^{s_0} \tilP_s u \, \frac{ds}{s} + e^{s_0 (\Delta+\rho^2)} u = \int_0^{s_0} \tilP_s u \, \frac{ds}{s} + \tilP_{\geq s_0} u, \\
 u &= \int_0^{s_0} \Pea_s u \, \frac{ds}{s} + e^{s_0 \Delta} u = \int_0^{s_0} \Pea_s u \, \frac{ds}{s} + \Pea_{\geq s_0} u. 
\end{align*}
Moreover, the following truncated resolutions of the $L^2$ norm will be useful in several occasions. 

\begin{lem}\label{l:L2res} 
For any $s_0 > 0$ we have 
\EQ{
\| u \|_{L^2}^2 = 4 \int_0^{s_0} \| \Pea_s u \|_{L^2}^2 \, \frac{\ud s}{s}  + 2\|s_0^{\frac{1}{2}}\nabla e^{s_0\Delta}u\|_{L^2}^2+\| \Pea_{\ge s_0} u \|_{L^2}^2 ,
}
and
\EQ{
 \|u\|_{L^2}^2 &= 4 \int_0^{s_0} \| \tilP_s u \|_{L^2}^2 \, \frac{\ud s}{s} + 2 \bigl( \|s_0^{\frac{1}{2}}\nabla e^{s_0(\Delta+\rho^2)}u\|_{L^2}^2-\rho^2\|s_0^{\frac{1}{2}} e^{s_0(\Delta+\rho^2)}u\|_{L^2}^2 \bigr) + \| \tilP_{\ge s_0} u \|_{L^2}^2 \\
 \|u\|_{L^2}^2 &= 2 \int_0^{s_0} \bigl( \| s^{\frac{1}{2}} \nabla e^{s(\Delta+\rho^2)} u \|_{L^2}^2 - \rho^2 \| s^{\frac{1}{2}} e^{s(\Delta+\rho^2)} u \|_{L^2}^2 \bigr) \, \ds + \| \tilP_{\geq s_0} u \|_{L^2}^2.
}
\end{lem}
\begin{proof}
For any $s_0>0$ we may write
\begin{align*}
\begin{split}
 \int_0^{s_0}s^2\Delta^2 e^{s\Delta}u\,\ds=&\int_0^{s_0}s\frac{\ud}{\ud s}\Delta e^{s\Delta}u \,\ud s=-\int_0^{s_0}\Delta e^{s\Delta}u\,\ud s+s_0\Delta e^{s_0\Delta}u\\
  =&u-e^{s_0\Delta}u+s_0\Delta e^{s_0\Delta}u.
\end{split}
\end{align*}
Therefore
\begin{align*}
\begin{split}
 \| u\|_{L^2}^2&=\int_0^{s_0}\angles{u}{s^2\Delta^2 e^{s\Delta}u}\,\ds+\angles{u}{e^{s_0\Delta}u}-\angles{u}{s_0\Delta e^{s_0\Delta}u}\\
 &=4\int_0^{\frac{s_0}{2}}\|s\Delta e^{s\Delta}u\|_{L^2}^2\,\ds+\angles{e^{\frac{s_0}{2}\Delta}u}{e^{\frac{s_0}{2}\Delta}u}+2\angles{(\frac{s_0}{2})^{\frac{1}{2}}\nabla e^{\frac{s_0}{2}\Delta}u}{(\frac{s_0}{2})^{\frac{1}{2}}\nabla e^{\frac{s_0}{2}\Delta}u}.
\end{split}
\end{align*}
The first identity follows by replacing $\frac{s_0}{2}$ by $s_0$. Similarly, for the second identity we write 
\begin{align*}
\begin{split}
 \int_{0}^{s_0}s^2(\Delta+\rho^2)^2e^{s(\Delta+\rho^2)}u\,\ds&=\int_0^{s_0}s\frac{\ud}{\ud s}(\Delta+\rho^2)e^{s(\Delta+\rho^2)}u\,\ud s \\
 &=-\int_0^{s_0}(\Delta+\rho^2)e^{s(\Delta+\rho^2)}u\,\ud s+s_0(\Delta+\rho^2)e^{s_0(\Delta+\rho^2)}u\\
 &=u-e^{s_0(\Delta+\rho^2)}u+s_0(\Delta+\rho^2)e^{s_0(\Delta+\rho^2)}u.
\end{split}
\end{align*}
It follows that
\begin{align*}
\begin{split}
 \|u\|_{L^2}^2&=4\int_0^{\frac{s_0}{2}}\|s(\Delta+\rho^2)e^{s(\Delta+\rho^2)}u\|_{L^2}^2\,\ds+ \angles{e^{\frac{s_0}{2}(\Delta+\rho^2)}u}{e^{\frac{s_0}{2}(\Delta+\rho^2)}u}\\
 &\quad + 2 \Bigl( \| {\textstyle (\frac{s_0}{2})^{\frac{1}{2}}} \nabla e^{\frac{s_0}{2}(\Delta+\rho^2)}u\|_{L^2}^2-\rho^2\| {\textstyle (\frac{s_0}{2})^{\frac{1}{2}} } e^{\frac{s_0}{2}(\Delta+\rho^2)}u\|_{L^2}^2 \Bigr).
\end{split}
\end{align*}
Again the desired estimate follows by replacing $\frac{s_0}{2}$ by $s_0$. The proof of the last identity is analogous.
\end{proof}


\subsection{Localized Parabolic Regularity} \label{ss:lpreg}

The starting point is the $L^{p}$ boundedness and regularity theory of the heat semi-groups $e^{s \De}$ and $e^{s(\Delta+\rho^2)}$. 
\begin{lem}[$L^p$ regularity of the heat flow] \label{l:preg}
Let $1 < p < \infty$. For any real-valued function $f \in C^{\infty}_{0}(\Hp^{d})$, integer $k>0$, and $s > 0$, the following bounds hold
\begin{align}
	\|e^{s \De} f\|_{L^{p}} &\aleq  \|f\|_{L^p}, \label{eq:pest} \\
	\| s^{\frac{k}{2}} \nb^{(k)} e^{s \De} f\|_{L^{p}} &\aleq \|f\|_{L^p}. \label{eq:napest} 
\end{align}
If $s\leq s_0$, then the same estimates hold with $e^{s\Delta}$ replaced by $e^{s(\Delta+\rho^2)}$ with constants which may depend on $s_0$.
\end{lem}
\begin{proof} 
 The statement for $e^{s(\Delta+\rho^2)}$ follows from the estimate for $e^{s\Delta}$ by observing that $e^{s\rho^2}$ is bounded for $s\leq s_0$. For $e^{s\Delta}$ the case $k\leq 1$ was proved, for instance, in~\cite[Lemma 2.11]{LOS5} using a general argument relying only on integration by parts.  There it was also shown that
\begin{align*}
\begin{split}
 \|s\Delta e^{s\Delta}f\|_{L^p}\lesssim \|f\|_{L^p}. 
\end{split}
\end{align*}
Since
\begin{align*}
\begin{split}
 &s^k\Delta^ke^{s\Delta}f=s\Delta e^{\frac{s}{2}\Delta}s^{k-1}\Delta^{k-1}e^{\frac{s}{2}\Delta}f, \\
 &s^{k+\frac{1}{2}}(-\Delta)^{k+\frac{1}{2}}e^{s\Delta}f=s(-\Delta) e^{\frac{s}{2}\Delta}s^{k-\frac{1}{2}}(-\Delta)^{k-\frac{1}{2}}e^{\frac{s}{2}\Delta}f,
\end{split}
\end{align*}
it follows by induction, and the equivalence of norms $\|\nabla g\|_{L^p}\simeq\|(-\Delta)^{-\frac{1}{2}}g\|_{L^p}$, that
\begin{align*}
\begin{split}
  \|s^k\Delta e^{s\Delta}f\|_{L^p}+\|s^{k+\frac{1}{2}}\nabla\Delta^ke^{s\Delta}f\|_{L^p}\lesssim \|f\|_{L^p},\qquad \forall k\in\bbN.
\end{split}
\end{align*}
The desired estimate now follows from the equivalence of norms (see Section~\ref{s:fs})
\begin{align*}
\begin{split}
\sum_{\ell=0}^k\|\nabla^{(\ell)}f\|_{L^p}\qquad \mathrm{and}\qquad \|(-\Delta)^{\frac{k}{2}}f\|_{L^p}.
\end{split}
\end{align*}
\end{proof} 

\begin{rem}
The previous lemma can also be proved for $p =1$ and $p=\infty$ using explicit $L^1$ bounds for the heat kernel $p_s$ and its derivatives $s \p_s p_s$ on $\Hp^d$ along with Young's inequality. This would  then allow generalization of many of the other $L^p$ estimates in this subsection to $p=1$ and $p=\infty$ as well. However, we have chosen not to include these estimates to keep all of our estimates independent of kernel bounds. 
\end{rem}

On occasion we will also need the following $L^2$-based fractional regularity theory of the heat semi-groups $e^{s\Delta}$ and $e^{s(\Delta+\rho^2)}$, see~\cite[Lemma 2.8]{LOS5} for a proof.
\begin{lem} \label{l:frac_preg}
 Let $f \in L^2$ and $\alpha \geq 0$. Then it holds that
 \begin{equation}
  \| s^\alpha (-\Delta)^\alpha e^{s\Delta} f \|_{L^2} \lesssim_\alpha \|f\|_{L^2}.
 \end{equation}
 If $s \leq s_0$, then the same estimate holds with $e^{s\Delta}$ replaced by $e^{s(\Delta+\rho^2)}$ with constants which may depend on $s_0$.
\end{lem}

We will also need a quick corollary of Lemma~\ref{l:preg} above. 
\begin{cor} \label{c:pregdiv} 
Let $1 < p < \infty$ and let $\bsxi$ be an arbitrary $(r,q)$ tensor field on $\bbH^d$. With $k:=r+q$ we define
\begin{align*}
\begin{split}
\nabla^{(k)}\cdot \bsxi:=\nabla_{\mu_{\sigma(1)}}\dots\nabla_{\mu_{\sigma(r)}}\nabla^{\mu_{\tau(r+1)}}\dots\nabla^{\mu_{\tau(r+q)}}\bsxi_{\mu_{r+1}\dots\mu_{r+q}}^{\mu_1\dots\mu_r},
\end{split}
\end{align*} 
where $\sigma$ and $\tau$ are arbitrary permutations of the indices $1,\dots,r$ and $r+1,\dots,r+q$, respectively. Then, 
\EQ{
\| s^{\frac{k}{2}}  e^{s \De} (\na^{(k)}\cdot\bsxi) \|_{L^{p}} \aleq \|\bsxi \|_{L^p}.
}
If $s\leq s_0$, then the same estimates hold with $e^{s\Delta}$ replaced by $e^{s(\Delta+\rho^2)}$ with constants which may depend on $s_0$.
\end{cor} 
\begin{proof} 
This is proved using~\eqref{eq:napest} and a duality argument. Indeed, 
\EQ{
\| s^{\frac{k}{2}}  e^{s \De} (\na^{(k)}\cdot\bsxi) \|_{L^{p}}  &= \sup_{\|f \|_{L^{p'}} = 1}\abs{ \ang{ s^{\frac{k}{2}}  e^{s \De} (\na^{(k)}\cdot\bsxi), \, f }} \\ 
 &\leq\sup_{ \|f \|_{L^{p'}} = 1} \abs{\ang{ |\bsxi| , \,|s^{\frac{k}{2}} \na^{(k)} e^{s \De} f|}} \\
 & \le \sup_{ \|f \|_{L^{p'}} = 1}  \| \bsxi \|_{L^p} \| s^{\frac{k}{2}} \na^{(k)} e^{s \De} f \|_{L^{p'}}  \\ 
 & \lesssim \sup_{ \|f \|_{L^{p'}} = 1}  \| \bsxi \|_{L^p} \| f \|_{L^{p'}}  \lesssim \| \bsxi \|_{L^p}, 
}
which completes the proof for $e^{s\Delta}$. The proof for $e^{s(\Delta+\rho^2)}$ follows from the boundedness of $e^{s\rho^2}$ when $s \leq s_0$. 
\end{proof} 

Let $\phi \in C^\infty_0$ be a smooth bump function such that $\phi(r) = 1$ if $\frac{1}{2} \le r \le 2$ and $\supp \phi (r) \subset [1/4, 4]$.   For each integer $\ell \in\Z$ define 
\EQ{ \label{eq:phidef} 
\phi_\ell(r)  := \phi( r/ 2^\ell).
}
The main result of this subsection are the localized parabolic regularity estimates established in the following proposition. 

\begin{prop}[Localized Parabolic Regularity] \label{p:lpregcore} 
Let $s_0 >0$ be fixed. For each~$s$ with $0 < s < s_0$, each $\ell, m \in \Z$ with $|\ell - m| \ge 10$ and $\max\{2^\ell, 2^{m}\} \ge s^{\frac{1}{2}}$, each $1 < p < \infty$, and each $N, k \in \N$ we have 
\begin{align} 
\|\phi_\ell e^{s\De} ( \phi_mv) \|_{L^p} &\lesssim_{s_0, N}  \left(s^{\frac{1}{2}} 2^{- \max\{\ell, m\}}\right)^N \|  \phi_m v \|_{L^p}, \label{eq:lpregcore0} \\
\| \phi_\ell  s^{k+\frac{1}{2}} \na (\Delta)^ke^{s \De} ( \phi_m v ) \|_{L^p} &\lesssim_{s_0, N, k} \left(s^{\frac{1}{2}} 2^{- \max\{\ell, m\}}\right)^N \|  \phi_m v \|_{L^p}, \label{eq:lpregcore1} \\
\| \phi_\ell s^k (\De)^k e^{s \De} ( \phi_m v ) \|_{L^p} &\lesssim_{s_0, N, k} \left(s^{\frac{1}{2}} 2^{- \max\{\ell, m\}}\right)^N \|  \phi_m v \|_{L^p} \label{eq:lpregcore2} .
\end{align}
The same estimates hold with $e^{s\Delta}$ replaced by $e^{s(\Delta+\rho^2)}$, $s\leq s_0$, with constants which may depend on $s_0$.
\end{prop} 

\begin{rem} \label{r:lpregcutoff}
Our proof can be applied to give a stronger result: Let $\psi_{R_1}(\cdot):=\psi(R_1^{-1}\cdot)$ and $\chi_{R_2}(\cdot):=\chi(R_2^{-1}\cdot)$ for two arbitrary smooth and bounded functions $\psi$ and $\chi$, such that the supports of $\psi_{R_1}$ and $\chi_{R_2}$ are disjoint. Then if $s^{\frac{1}{2}}\leq C_0 \max\{R_1,R_2\}$ for some absolute constant $C_0 > 0$, estimates \eqref{eq:lpregcore0}--\eqref{eq:lpregcore2} hold with $\phi_\ell$ replaced by $\psi_{R_1}$, $\phi_m$ replaced by $\chi_{R_2}$, $2^{-\max\{m,\ell\}}$ replaced by $\min\{R_1^{-1},R_2^{-1}\}$ and the implicit constant is allowed to depend on $C_0$. The same holds for Corollaries~\ref{c:lpregcov} and~\ref{c:lpregcorediv} and Lemma~\ref{l:lpreg1} below.
\end{rem} 

We will also need the following two corollaries of Proposition~\ref{p:lpregcore}.

\begin{cor}[Localized Parabolic Regularity for Covariant Derivatives]\label{c:lpregcov}
Under the assumptions of Proposition~\ref{p:lpregcore}, it holds that
\begin{equation}\label{eq:lpregcore3}
\| \phi_\ell s^{\frac{k}{2}} \nabla^{(k)} e^{s \De} ( \phi_m v ) \|_{L^2} \lesssim_{s_0, N, k} \left(s^{\frac{1}{2}} 2^{- \max\{\ell, m\}}\right)^N \|  \phi_m v \|_{L^2}.
\end{equation}
The same estimate holds with $e^{s\Delta}$ replaced by $e^{s(\Delta+\rho^2)}$, $s\leq s_0$, with constants which may depend on $s_0$.
\end{cor}

\begin{cor}[Localized Parabolic Regularity for Tensor Fields]\label{c:lpregcorediv}
Let $\bsxi$ be an arbitrary $(r,q)$ tensor field on $\bbH^d$ and with $k:=r+q$, let $\nabla^{(k)} \cdot \bsxi$ be as defined in Corollary~\ref{c:pregdiv}. Let $s_0 >0$ be fixed. For each $s$ with $0 < s \leq s_0$, each $\ell, m \in \Z$ with $|\ell - m| \ge 10$ and $ \max \{ 2^\ell, 2^{m} \} \geq s^{\frac{1}{2}}$, and each $N,k \in \N$ we have 
\EQ{ \label{eq:lpregcorediv} 
 \|   \phi_{\ell} s^{\frac{k}{2}} e^{s\De} \na^{(k)} \cdot (\phi_m \bsxi) \|_{L^2} &\lesssim_{s_0, N,  k} \left(s^{\frac{1}{2}} 2^{- \max\{\ell, m\}}\right)^N  \| \phi_m \bsxi \|_{L^2}.
 }
 The same statement holds with $e^{s\Delta}$ replaced by $e^{s(\Delta+\rho^2)}$, $s\leq s_0$, with constants which may depend on $s_0$. When $k=1$, $L^2$ can be replaced by $L^p$ for any $1<p<\infty$.
\end{cor} 
\begin{proof}
This follows from Corollary~\ref{c:lpregcov} by a duality pairing argument as in the proof of Corollary~\ref{c:pregdiv}. The statement for $k=1$ follows by the same duality argument from \eqref{eq:lpregcore1} with $k=0$.
\end{proof} 

\begin{proof}[Proof of Corollary~\ref{c:lpregcov}]
Since $e^{s\rho^2}$ is bounded when $s \leq s_0$, the estimate for $e^{s(\Delta+\rho^2)}$ follows from the estimate for $e^{s\Delta}$. For $k=0,1$ the latter follows from \eqref{eq:lpregcore0} and \eqref{eq:lpregcore1}, respectively. Now to prove \eqref{eq:lpregcore3} for $k=2$ we write
\begin{align*}
\begin{split}
\| \phi_\ell s \nabla^{(2)} e^{s \De} ( \phi_m v ) \|_{L^2}^2&=\angles{\phi_\ell s\nabla^{\mu_1}\nabla^{\mu_2}e^{s\Delta}(\phi_m v)}{\phi_\ell s\nabla_{\mu_1}\nabla_{\mu_2}e^{s\Delta}(\phi_m v)}\\
&=-2\angles{\phi_\ell s \nabla^{\mu_2}\nabla^{\mu_1}e^{s\Delta}(\phi_m v)}{(s^{\frac{1}{2}}\nabla_{\mu_1}\phi_\ell)s^{\frac{1}{2}}\nabla_{\mu_2}e^{s\Delta}(\phi_m v)}\\
&\quad -\angles{\phi_\ell s^{\frac{3}{2}}\nabla_{\mu_1}\nabla^{\mu_1}\nabla^{\mu_2}e^{s\Delta}(\phi_m v)}{\phi_\ell s^{\frac{1}{2}}\nabla_{\mu_2}e^{s\Delta}(\phi_m v)}\\
&=2\angles{(s^{\frac{1}{2}}\nabla^{\mu_2}\phi_\ell) s^{\frac{1}{2}} \nabla^{\mu_1}e^{s\Delta}(\phi_m v)}{(s^{\frac{1}{2}}\nabla_{\mu_1}\phi_\ell)s^{\frac{1}{2}}\nabla_{\mu_2}e^{s\Delta}(\phi_m v)}\\
&\quad+2\angles{\phi_\ell s^{\frac{1}{2}} \nabla^{\mu_1}e^{s\Delta}(\phi_m v)}{(s\nabla^{\mu_2}\nabla_{\mu_1}\phi_\ell)s^{\frac{1}{2}}\nabla_{\mu_2}e^{s\Delta}(\phi_m v)}\\
&\quad+2\angles{\phi_\ell s^{\frac{1}{2}} \nabla^{\mu_1}e^{s\Delta}(\phi_m v)}{(s^{\frac{1}{2}}\nabla_{\mu_1}\phi_\ell)s\Delta e^{s\Delta}(\phi_m v)}\\
&\quad-\angles{\phi_\ell s^{\frac{3}{2}}\nabla^{\mu_2}\Delta e^{s\Delta}(\phi_m v)}{\phi_\ell s^{\frac{1}{2}}\nabla_{\mu_2}e^{s\Delta}(\phi_m v)}\\
&\quad-\angles{\phi_\ell s^{\frac{3}{2}}[\nabla_{\mu_1}\nabla^{\mu_1},\nabla^{\mu_2}] e^{s\Delta}(\phi_m v)}{\phi_\ell s^{\frac{1}{2}}\nabla_{\mu_2}e^{s\Delta}(\phi_m v)}.
\end{split}
\end{align*}
In view of Remark~\ref{r:lpregcutoff}, the first four lines on the right-hand side above are now of the form that is controlled by \eqref{eq:lpregcore0}--\eqref{eq:lpregcore2}. That the last line is also of this form can be seen by the curvature formula \eqref{eq:R}--\eqref{eq:Rich}. We can now continue inductively in the same way to prove \eqref{eq:lpregcore3} for higher values of $k$.
\end{proof}

For technical reasons, we introduce a second type of bump function that we will denote by $\fy$ -- in practice this will be given by various derivatives of $\phi$. But, to  keep notation at a minimum, we let $\fy\in C^\infty_0$ be any smooth bump function with uniformly bounded derivatives such that $\supp \fy \in [1/4, 4]$. As usual we let 
\EQ{
\fy_\ell (r) := \fy( r/ 2^\ell)\quad \forall \ell \in \Z.
}
The main step in the proof of Proposition~\ref{p:lpregcore} is the following lemma. 
\begin{lem} \label{l:lpreg1}  
For each $s$ with $0<s < s_0$, each $\ell, m$ with  $2^\ell \ge 2^{m+10}$ and $2^{\ell} \ge s_0^{\frac{1}{2}}$, and  each $1 < p < \infty$ we have 
\begin{align} 
\| \fy_\ell e^{s \De} ( \phi_m v ) \|_{L^p} &\lesssim_{s_0} s^{\frac{1}{2}} 2^{-\ell} \| \phi_m v \|_{L^p}\label{eq:lpreg0}  \\
\| \fy_m s^{k+\frac{1}{2}} \na(\Delta)^k  e^{s \De} ( \phi_\ell v ) \|_{L^p} &\lesssim_{s_0} s^{\frac{1}{2}} 2^{-\ell} \| \phi_\ell v \|_{L^p}\label{eq:lpreg1} \\
\| \fy_\ell s^k (\De)^k e^{s \De} ( \phi_m v ) \|_{L^p} &\lesssim_{s_0} s^{\frac{1}{2}} 2^{-\ell} \| \phi_m v \|_{L^p}\label{eq:lpreg2} ,
\end{align} 
and we make a special note of the reversed roles of $m, \ell$ in~\eqref{eq:lpreg1}. 
\end{lem} 
\begin{proof} 
The basic ingredients in the proof are Lemma~\ref{l:preg} and Corollary~\ref{c:pregdiv}. 

First, we prove~\eqref{eq:lpreg0}. Define
\EQ{
 w_0 (s) := \fy_\ell e^{s \De} (\phi_m v ) .
}
Note that because the supports of $\fy_\ell $ and $\phi_m$ are disjoint, we have $ w_0(0) = 0$. Moreover, $w_0(s)$ solves the heat equation 
\EQ{
 (\p_s - \De) w_0(s) &= -(\De \fy_\ell)  e^{s \De}( \phi_m v) - 2 \na^\mu \fy_\ell  \na_\mu  e^{s \De} (\phi_m v)) =: \calN_0(s).
}
Using the Duhamel formula we have 
\EQ{\label{eq:swduh0} 
w_0(s) =  \int_0^{s}  s' e^{(s - s') \De} \calN_0(s') \, \frac{\ud s'}{s'}.  
}
Thus, using Lemma~\ref{l:preg}, 
\EQ{
\| w_0(s) \|_{L^p}  &\lesssim \Big\| \int_{0}^s  s' e^{(s - s') \De} \calN_0(s') \, \frac{\ud s'}{s'} \Big\|_{L^p} \lesssim \int_{0}^s s'  \|e^{(s-s')\De} \calN_0(s')\|_{L^p}  \frac{\ud s' }{s'}  \\
& \lesssim \int_{0}^s s'  \| (\De \fy_\ell)  e^{s' \De}( \phi_m v) \|_{L^p} \frac{\ud s'}{s'} + \int_0^s s' \| \na^\mu \fy_\ell  \na_\mu  e^{s' \De} (\phi_m v)) \|_{L^p} \, \frac{\ud s'}{s'} \\
& \lesssim \max(2^{-\ell}, 2^{-2\ell}) \int_{0}^s s'  \| e^{s' \De}( \phi_m v) \|_{L^p} \frac{\ud s'}{s'} \\
&\quad+ 2^{-\ell} \int_0^s (s')^{\frac{1}{2}} \| (s')^{\frac{1}{2}} \na  e^{s' \De} (\phi_m v)) \|_{L^p} \, \frac{\ud s'}{s'} \\
& \lesssim s\max(2^{-\ell}, 2^{-2\ell}) \| \phi_m v \|_{L^p} + s^{\frac{1}{2}} 2^{-\ell}  \| \phi_m v \|_{L^p}  \\
& \lesssim_{s_0} s^{\frac{1}{2}} 2^{-\ell}  \| \phi_m v \|_{L^p} ,
}
as desired. 

Next, we prove~\eqref{eq:lpreg2} and postpone the proof of~\eqref{eq:lpreg1} which requires a slightly different argument  until the end. We proceed as before, now defining,  
\EQ{
w_2 (s) := \fy_\ell  \De^k e^{s \De} (\phi_m v ) .
}
Then,  the goal is to estimate 
\EQ{
s^k w_2(s) = \fy_\ell s^k \De^k e^{s \De} ( \phi_m v ) ,
}
in $L^p$. 
Note again that because the supports of $\fy_\ell $ and $\phi_m$ are disjoint, we have $ w_2(0) = 0$. Moreover, $w_2(s)$ solves the following heat equation, 
\EQ{
(\p_s - \De) w_2(s) &= -(\De \fy_\ell) \De^k e^{s \De}( \phi_m v) - 2 \na^\mu \fy_\ell  \na_\mu (\De^k e^{s \De} (\phi_m v)) =: \calN_2(s),  \\
w_2(0) &= 0.
}
Using the Duhamel formula we have 
\EQ{\label{eq:swduh} 
s^k w_2(s) =  \int_0^{s/2}  s^k s' e^{(s - s') \De} \calN_2(s') \, \frac{\ud s'}{s'}  +  \int_{s/2}^s  s^k s' e^{(s - s') \De} \calN_2(s') \, \frac{\ud s'}{s'} .
}
Note that 
\EQ{\label{eq:dphiell} 
&\abs{\na^{(j)} \fy_\ell(r)}_{\h}  \lesssim \begin{cases} 2^{-\ell} \mif \ell \ge 0 \\ 2^{-j\ell} \mif \ell \le 0 \end{cases}  \quad \forall j \in \N.
}
We estimate the second integral above in $L^p$ as follows, 
\EQ{
 \Big\| \int_{s/2}^s  s^k s' e^{(s - s') \De} \calN_2(s') \, \frac{\ud s'}{s'} \Big\|_{L^p} &\lesssim \int_{\frac{s}{2}}^s s^k s'  \|e^{(s-s')\De}\calN_2(s')\|_{L^p}  \frac{\ud s' }{s'}  \\
 & \lesssim \int_{\frac{s}{2}}^s s^k s'  \| (\De \fy_\ell) \De^k e^{s' \De}( \phi_m v)\|_{L^p}  \frac{\ud s' }{s'} \\
 &\quad + \int_{\frac{s}{2}}^s s^k s'  \| \na^\mu \fy_\ell  \na_\mu \De^k e^{s' \De} (\phi_m v)\|_{L^p}  \frac{\ud s' }{s'} \\
 & \lesssim    \max(2^{-\ell}, 2^{-2\ell} ) \int_{\frac{s}{2}}^s s   \| (s')^k \De^k e^{s' \De}( \phi_m v)\|_{L^p}  \frac{\ud s' }{s'} \\
 &\quad + 2^{-\ell} \int_{\frac{s}{2}}^s s^{\frac{1}{2}} \|  (s')^{k+\frac{1}{2}}  \na \De^k e^{s' \De} (\phi_m v)\|_{L^p}  \frac{\ud s' }{s'} \\
 &\lesssim 2^{-\ell} s^{\frac{1}{2}} \| (\phi_m v)\|_{L^p}.
 }
 To handle the first integral in~\eqref{eq:swduh} we need to rewrite $\calN_2(s)$. Recall that  
 \EQ{
 \calN_2(s') &= -(\De \fy_\ell) \De^k e^{s' \De}( \phi_m v) - 2 \na^\mu \fy_\ell  \na_\mu \De^k e^{s' \De} (\phi_m v).
 }
Now by repeated applications of the product rule, we can write this as a linear combination of \emph{scalar} functions of the form
\begin{align*}
\begin{split}
  \nabla^{(2k+2-j)}((\nabla^{(j)}\fy_\ell)e^{s\Delta}(\phi_m v)),\qquad 1\leq j\leq 2k+2.
\end{split}
\end{align*}
 Then by Corollary~\ref{c:pregdiv}
\begin{align*}
\begin{split}
  \Big\| \int_{0}^{\frac{s}{2}}  s^k s' e^{(s - s') \De} \nabla^{(2k+2-j)} & \Big( (\nabla^{(j)} \fy_\ell)  e^{s' \De}( \phi_m v) \Big)  \, \frac{\ud s'}{s'} \Big\|_{L^p} \\
  &\lesssim\int_0^{\frac{s}{2}}\frac{s^ks'}{(s-s')^{k+1-\frac{j}{2}}}\|(\nabla^{(j)} \fy_\ell)  e^{s' \De}( \phi_m v)\|_{L^p}\,\dsp \\ 
  &\lesssim s^{\frac{j}{2}-1}\max(2^{-\ell},2^{-j\ell})\int_0^{\frac{s}{2}}s'\|\phi_mv\|_{L^p}\,\dsp\\
  &\lesssim 2^{-\ell} s^{\frac{1}{2}} \| (\phi_m v)\|_{L^p}.
\end{split}
\end{align*}
 
 Lastly, we prove~\eqref{eq:lpreg1}. Define 
 \EQ{
 W_1(s):= \fy_m\na \Delta^ke^{s\De} \phi_\ell v ,
 }
 and note that our goal is to show, 
 \EQ{
 \| s^{k+\frac{1}{2}} W_1(s) \|_{L^p} \lesssim s^{\frac{1}{2}} 2^{-\ell} \| \phi_\ell v \|_{L^p}.
 }
The difference from the previous cases is that $W_1(s)$ is a $(0, 1)$ tensor. But, we can compute its $L^p$ norm as 
 \EQ{
  \| s^{k+\frac{1}{2}} W_1(s) \|_{L^p} = \sup_{ \| \bsxi   \|_{L^{p'}} =1} \abs{\ang{ \fy_m s^{k+\frac{1}{2}} \na_\mu \Delta^ke^{s\De} \phi_\ell v , \, \bsxi^\mu}},
 }
 where the supremum is taken over all $(1, 0)$-tensors $\bsxi$. Note that the term inside the supremum above can be written as 
 \EQ{
 \ang{    \phi_\ell v , \,  \ti  \phi_{\ell} s^{k+\frac{1}{2}} \Delta^ke^{s\De} \na_\mu (\fy_m \bsxi^\mu)},
 }
 where $\ti \phi \in C^\infty_0$ is a fattened version of $\phi$, i.e, $\phi \ti \phi = \phi$ and $\supp \ti \phi \in [1/8, 8]$. Hence it suffices to show that 
 \EQ{ \label{eq:divw} 
 \|  \ti  \phi_{\ell} s^{k+\frac{1}{2}} \Delta^ke^{s\De} \na_\mu (\fy_m \bsxi^\mu) \|_{L^q} \lesssim s^{\frac{1}{2}} 2^{-\ell} \| \fy_m \bsxi \|_{L^q} ,\quad \forall 1 < q < \infty.
 }
 The advantage is that we can rephrase the estimate in terms of the scalar 
 \EQ{
 w_1(s):= \ti  \phi_{\ell} \Delta^ke^{s\De} \na_\mu (\fy_m \bsxi^\mu) .
 }
 Note that $w_1(0)= 0$ and $w_1(s)$  satisfies the following heat equation, 
 \EQ{
 (\p_s - \De) w_1 &=  - \De \ti \phi_\ell \Delta^ke^{s\De} \na_\mu (\fy_m \bsxi^\mu)  - 2 \na^\nu  \phi_\ell \na_\nu \Delta^ke^{s\De}  \na_\mu (\fy_m \bsxi^\mu) \\
 &=:\calN_1.
 }
 Using the Duhamel formula, 
\begin{align}\label{eq:sw1duhtemp1}
\begin{split}
 \|s^{k+\frac{1}{2}}w_1(s)\|_{L^p} \lesssim& \int_0^{\frac{s}{2}}s^{k+\frac{1}{2}}s'\|e^{(s-s')\Delta}\calN_1(s')\|_{L^p}\,\ds'\\
 &+\int_{\frac{s}{2}}^{s}s^{k+\frac{1}{2}}s'\|e^{(s-s')\Delta}\calN_1(s')\|_{L^p}\,\ds'.
\end{split}
\end{align}
Using Lemma~\ref{l:preg} and Corollary~\ref{c:pregdiv} we can estimate the second integral by 
 \EQ{
 &\int_\frac{s}{2}^s s^{k+\frac{1}{2}} s' \| e^{(s-s') \De} (\De \ti \phi_\ell) \Delta^ke^{s'\De} \na_\mu (\fy_m \bsxi^\mu) \|_{L^p} \, \frac{\ud s'}{s'}  \\
 & \quad + \int_{\frac{s}{2}}^s s^{k+\frac{1}{2}} s' \| e^{(s-s') \De}  (\na^\nu  \phi_\ell)  \nabla_\nu\Delta^ke^{s'\De}  \na_\mu (\fy_m \bsxi^\mu)  \|_{L^p} \, \frac{\ud s'}{s'} \\
 & \lesssim  \int_{\frac{s}{2}}^s s^{k+\frac{1}{2}} s' \| (\De \ti \phi_\ell) \Delta^ke^{s'\De} \na_\mu (\fy_m \bsxi^\mu) \|_{L^p} \, \frac{\ud s'}{s'} \\
 & \quad + \int_{\frac{s}{2}}^s s^{k+\frac{1}{2}} s' \|  (\na^\nu  \phi_\ell)\nabla_\nu \Delta^k e^{s'\De}  \na_\mu (\fy_m \bsxi^\mu) \Big) \|_{L^p} \, \frac{\ud s'}{s'} \\
 & \lesssim \max(2^{-\ell}, 2^{-2\ell}) \int_{\frac{s}{2}}^s s \| (s')^{k+\frac{1}{2}}e^{s'\De} \Delta^k\na_\mu (\fy_m \bsxi^\mu) \|_{L^p} \, \frac{\ud s'}{s'} \\
 &\quad + 2^{-\ell}  \int_{\frac{s}{2}}^s s^{\frac{1}{2}}  \| (s')^{\frac{1}{2}}\nabla e^{\frac{s'}{2}\De} (s')^{k+\frac{1}{2}}e^{\frac{s'}{2}\Delta}\Delta^k \na_\mu (\fy_m \bsxi^\mu)  \|_{L^p} \, \frac{\ud s'}{s'} \\
 &\lesssim_{s_0} s^{\frac{1}{2}} 2^{-\ell}\| \fy_m \bsxi \|_{L^p}.
 }
For the first integral in \eqref{eq:sw1duhtemp1} we have to argue differently again. By repeated applications of the product rule we can rewrite $\calN_1$ as a linear combination of terms of the form
\begin{align*}
\begin{split}
 \nabla^{(2k+2-j)}((\nabla^{(j)}\tilphi_\ell)e^{s\Delta}(\nabla_\mu(\fy_m\bsxi^\mu))),\qquad 1\leq j\leq 2k+2.
\end{split}
\end{align*}
Note that the difference with $\calN_2$ above is that we have kept one derivative on $\bsxi$ to preserve the scalar structure of the term to which $e^{s\Delta}$ is applied. Therefore to estimate the first integral in \eqref{eq:sw1duhtemp1} it suffices to note that
\begin{align*}
\begin{split}
  &\int_0^\frac{s}{2} s^{k+\frac{1}{2}} s' \| e^{(s-s') \De}  \nabla^{(2k+2-j)}((\nabla^{(j)}\tilphi_\ell)e^{s'\Delta}(\nabla_\mu(\fy_m\bsxi^\mu))) \|_{L^p} \, \frac{\ud s'}{s'}\\
 &\lesssim \int_0^\frac{s}{2} \frac{s^{k+\frac{1}{2}} s'}{(s-s')^{k+1-\frac{j}{2}}} \| (\nabla^{(j)}\tilphi_\ell)e^{s'\Delta}(\nabla_\mu(\fy_m\bsxi^\mu)) \|_{L^p} \, \frac{\ud s'}{s'}\\
 &\lesssim s^{\frac{j-1}{2}}\max(2^{-\ell}, 2^{-j\ell})\int_0^\frac{s}{2}(s')^{\frac{1}{2}}\|(s')^{\frac{1}{2}}e^{s'\Delta}\nabla_\mu(\fy_m\bsxi^\mu)\|_{L^p}\,\dsp\\
 &\lesssim_{s_0} s^{\frac{1}{2}} 2^{-\ell}\| \fy_m \bsxi \|_{L^p}.
\end{split}
\end{align*}
\end{proof} 

The proof of Proposition~\ref{p:lpregcore} will now follow by iterating Lemma~\ref{l:lpreg1} and a duality argument. We will also make use of Corollary~\ref{c:pregdiv}. 
\begin{proof}[Proof of Proposition~\ref{p:lpregcore}]
We will prove~\eqref{eq:lpregcore2} in detail and just remark that the proofs of~\eqref{eq:lpregcore0} and~\eqref{eq:lpregcore1} are similar. 

First we assume that $\ell \ge m + 10$. We also assume that $k=1$ and will later treat higher values of $k$ by induction. Note that setting $\fy_\ell =  \phi_\ell$, \eqref{eq:lpreg2} establishes the estimate~\eqref{eq:lpregcore2} with $N = 1$. We show how to iterate this estimate to obtain ~\eqref{eq:lpregcore2}  with  $N =2$. As in the proof of Lemma~\ref{l:lpreg1} define 
\EQ{
w(s) := \phi_\ell  \De e^{s \De} (\phi_m v ) .
}
Then, as before we have 
\EQ{ 
(\p_s - \De) w(s) &= -(\De \phi_\ell) \De e^{s \De}( \phi_m v) - 2 \na^\mu \phi_\ell  \na_\mu (\De e^{s \De} (\phi_m v))   \\
& = \De \phi_\ell  \De e^{s \De} (\phi_m v) - 2 \na_\mu \Big( \na^\mu \phi_\ell  \De e^{s \De} (\phi_m v) \Big),
}
where we have rewritten the right-hand side above to prepare for an application of Lemma~\ref{l:lpreg1}. Since $w(0) = 0$ we have 
\EQ{\label{eq:sws1} 
sw(s)  &= \int_0^s s s' e^{(s-s')\De} \De \phi_\ell  \De e^{s' \De} (\phi_m v) \frac{\ud s'}{s'}  \\
&\quad -2  \int_0^s s s' e^{(s-s')\De}\na_\mu \Big( \na^\mu \phi_\ell  \De e^{s' \De} (\phi_m v)\Big) \frac{\ud s'}{s'}.
}
We estimate the first term above as follows, 
\EQ{
\Big\| \int_0^s s s' e^{(s-s')\De} \De \phi_\ell  \De e^{s' \De} (\phi_m v) \frac{\ud s'}{s'} \Big\|_{L^p} & \lesssim  \int_0^s s s' \| e^{(s-s')\De} \De \phi_\ell  \De e^{s' \De} (\phi_m v) \|_{L^p} \frac{\ud s'}{s'}  \\
& \lesssim \int_0^s s  \|  (\De \phi_\ell)  s'\De e^{s' \De} (\phi_m v) \|_{L^p} \frac{\ud s'}{s'}.
}
Now, observe that 
\EQ{
\De( \phi_\ell(r))  = 2^{-2 \ell} \phi''(r/ 2^{\ell}) + 2^{-\ell} (d-1) \coth r \phi'(r/ 2^{\ell}),
}
so that, 
\begin{align}\label{eq:lapphi1}
\begin{split}
  \int_0^s s  \|  (\De \phi_\ell)  s'\De e^{s' \De} &(\phi_m v) \|_{L^p} \frac{\ud s'}{s'} \lesssim  2^{-2\ell} \int_0^s s  \|  \phi''(r/ 2^\ell) s' \De e^{s' \De} (\phi_m v) \|_{L^p} \frac{\ud s'}{s'} \\ 
 & \quad + 2^{-\ell}  \int_0^s s  \|   \coth r \phi'(r/ 2^{\ell}) s'\De e^{s' \De} (\phi_m v) \|_{L^p} \frac{\ud s'}{s'}.
\end{split}
\end{align}
To estimate the first term on the right above, we use Lemma~\ref{l:lpreg1} with $\fy_\ell(r)  := \phi''(r/ 2^{\ell})$, i.e., 
\EQ{ \label{eq:phi''1} 
\|  \phi''(r/ 2^\ell) s' \De e^{s' \De} (\phi_m v) \|_{L^p} \lesssim (s')^{\frac{1}{2}} 2^{-\ell} \| \phi_m v \|_{L^p} .
} 
For the second term we distiguish between the cases $\ell \ge 0$ and $\ell \le 0$. If $\ell \ge 0$ an 
application of  Lemma~\ref{l:lpreg1} with 
\EQ{
\fy_\ell(r) =  \coth (r/2^\ell) \phi'(r/ 2^{\ell}),
}
yields, 
\EQ{\label{eq:phi'1}
 \|   \coth r \phi'(r/ 2^{\ell}) s'\De e^{s' \De} (\phi_m v) \|_{L^p}& \leq \|   \coth (r/2^\ell) \phi'(r/ 2^{\ell}) s'\De e^{s' \De} (\phi_m v) \|_{L^p} \\\
& \lesssim (s')^{\frac{1}{2}} 2^{-\ell} \| \phi_m v \|_{L^p}.
}
If $\ell \le 0$, we first note that since
\begin{align*}
\begin{split}
 x\coth r\leq \coth (r/x),\qquad 0<x\leq1, 
\end{split}
\end{align*}
we have
\begin{align*}
\begin{split}
 \|   \coth r \phi'(r/ 2^{\ell}) s'\De e^{s' \De} (\phi_m v) \|_{L^p}\leq 2^{-\ell}\|   \coth (r/2^\ell) \phi'(r/ 2^{\ell}) s'\De e^{s' \De} (\phi_m v) \|_{L^p}. 
\end{split}
\end{align*}
Applying Lemma~\ref{l:lpreg1} with 
\EQ{
\fy_\ell(r) =   \coth (r/2^\ell) \phi'(r/ 2^{\ell}) 
}
as above then gives
\EQ{\label{eq:phi'2} 
\|   \coth r \phi'(r/ 2^{\ell}) s'\De e^{s' \De} (\phi_m v) \|_{L^p} \lesssim (s')^{\frac{1}{2}} 2^{-2\ell} \| \phi_m v \|_{L^p}.
}
Thus, for $\ell \ge 0$ we plug the estimates~\eqref{eq:phi''1} and~\eqref{eq:phi'1} into~\eqref{eq:lapphi1} to obtain, 
\EQ{\label{eq:lapphi2} 
\int_0^s s  \|  (\De \phi_\ell) s' \De e^{s' \De} (\phi_m v) \|_{L^p} \frac{\ud s'}{s'} &\lesssim   2^{-3\ell} s  \int_0^s (s')^{-\frac{1}{2}} \, \ud s'  \| \phi_m v \|_{L^p}  \\
& \quad +s  2^{-2\ell}   \int_0^s (s')^{-\frac{1}{2}} \, \ud s' \| \phi_m v \|_{L^p}  \\
& \lesssim s^{\frac{3}{2}}( 2^{-3\ell} + 2^{-2\ell})  \| \phi_m v \|_{L^p} \\
&\lesssim_{s_0} s 2^{-2\ell}  \| \phi_m v \|_{L^p} \mif \ell \ge 0.
}
Similarly, for $\ell \le 0$ we use~\eqref{eq:phi'2} to estimate the term with $\coth r$ in~\eqref{eq:lapphi1} yielding, 
\EQ{
\int_0^s s  \|  (\De \phi_\ell) s' \De e^{s' \De} (\phi_m v) \|_{L^p} \frac{\ud s'}{s'} &\lesssim 2^{-3\ell} s  \int_0^s (s')^{-\frac{1}{2}} \, \ud s'  \| \phi_m v \|_{L^p}  \\
& \lesssim s^{\frac{3}{2}} 2^{-3\ell} \| \phi_m v \|_{L^p}  \lesssim s 2^{-2\ell}  \| \phi_m v \|_{L^p}\mif \ell  \le 0.
}
Next we estimate the second term on the right-hand side of~\eqref{eq:sws1}. We begin by applying Corollary~\ref{c:pregdiv} to obtain 
\EQ{
\Big\|    \int_0^s s s' e^{(s-s')\De}\na_\mu &\Big( \na^\mu \phi_\ell  \De e^{s' \De} (\phi_m v)\Big) \frac{\ud s'}{s'} \Big\|_{L^p}  \\
&\lesssim \int_0^s \frac{s s' }{(s-s')^{\frac{1}{2}}} \| (s-s')^{\frac{1}{2}} e^{(s-s')\De}\na_\mu \Big( \na^\mu \phi_\ell  \De e^{s' \De} (\phi_m v)\Big) \|_{L^p} \frac{\ud s'}{s'} \\
&\lesssim \int_0^s \frac{s }{(s-s')^{\frac{1}{2}}} \|  \na \phi_\ell s' \De e^{s' \De} (\phi_m v) \|_{L^p} \frac{\ud s'}{s'}  \\
& = 2^{-\ell}  \int_0^s \frac{s  }{(s-s')^{\frac{1}{2}}} \|   \phi'(r/ 2^\ell)  s'\De e^{s' \De} (\phi_m v) \|_{L^p} \frac{\ud s'}{s'}  \\
&\lesssim s2^{-2\ell}  \| \phi_m v \|_{L^p} \int_0^s  \frac{ 1 }{(s')^{\frac{1}{2}}(s-s')^{\frac{1}{2}}}  \, \ud s'  \\
& \lesssim s 2^{-2 \ell} \| \phi_m v \|_{L^p} , 
}
where we applied Lemma~\ref{l:lpreg1} with $\fy_\ell (r) = \phi'(r/ 2^\ell)$ in the second to last line and then used that 
\EQ{
\int_0^s  \frac{  1}{(s')^{\frac{1}{2}}(s-s')^{\frac{1}{2}}}  \, \ud s'  &=  \int_0^{\frac{s}{2}} \frac{ 1 }{(s')^{\frac{1}{2}}(s-s')^{\frac{1}{2}}}  \, \ud s' + \int_{\frac{s}{2}}^s \frac{ 1}{(s')^{\frac{1}{2}} (s-s')^{\frac{1}{2}}}  \, \ud s' \lesssim 1.
}
We have thus shown that Lemma~\ref{l:lpreg1} implies 
\EQ{
 \| s w(s) \|_{L^p} \lesssim_{s_1} (s^{\frac{1}{2}} 2^{-\ell} )^2 \| \phi_m v \|_{L^p}.
}
Iterating the argument above gives, 
\EQ{\label{eq:ell>m} 
 \| \phi_\ell s \De e^{s\De} (\phi_m v) \|_{L^p} \lesssim_{s_1} (s^{\frac{1}{2}} 2^{-\ell} )^N \| \phi_m v \|_{L^p},
}
for any $N \in \N$, which proves \eqref{eq:lpregcore2} when $\ell > m+10$ and $k=1$. Now to get the same statement when $k\geq1$ we assume by induction that we have proved it for $k$ and prove it for $k+1$. Note that
\begin{align*}
\begin{split}
 \phi_\ell s^{k+1}\Delta^{k+1}e^{s\Delta}(\phi_mv)= \phi_\ell s\Delta e^{\frac{s}{2}\Delta}((\fy_{\leq\ell-5}+\fy_{\geq\ell-5})s^k\Delta e^{\frac{s}{2}\Delta}(\phi_mv)),
\end{split}
\end{align*}
where $\fy_{\leq \ell-5}$ is supported in $\{r\leq 2^{\ell-5}\}$ and $\fy_{\geq\ell-5}=1-\fy_{\leq\ell-5}$. Now since the supports of $\phi_\ell$ and $\fy_{\ell-5}$ are disjoint, using our induction hypothesis we get
\begin{align*}
\begin{split}
  \|\phi_\ell s\Delta e^{\frac{s}{2}\Delta}(\fy_{\leq\ell-5}s^k\Delta e^{\frac{s}{2}\Delta}(\phi_mv))\|_{L^p}&\lesssim (s^{\frac{1}{2}}2^{-\ell})^N\|\fy_{\leq\ell-5}s^k\Delta e^{\frac{s}{2}\Delta}(\phi_mv)\|_{L^p}\\
  &\lesssim (s^{\frac{1}{2}}2^{-\ell})^N\|\phi_mv\|_{L^p}.
\end{split}
\end{align*}
Similarly, since the support of $\fy_{\geq\ell-5}$ and $\phi_m$ are disjoint
\begin{align*}
\begin{split}
   \|\phi_\ell s\Delta e^{\frac{s}{2}\Delta}(\fy_{\geq\ell-5}s^k\Delta e^{\frac{s}{2}\Delta}(\phi_mv))\|_{L^p}&\lesssim\|\fy_{\geq\ell-5}s^k\Delta e^{\frac{s}{2}\Delta}(\phi_mv)\|_{L^p}\\
   &\lesssim (s^{\frac{1}{2}}2^{-\ell})^N\|\phi_mv\|_{L^p},
\end{split}
\end{align*}
completing the proof of \eqref{eq:lpregcore2} when $\ell > m+10$.

Next we use a duality argument to prove the estimate when $\ell + 10< m$. Our goal is to show that in this case, 
\EQ{ \label{eq:m>ell}
\| \phi_\ell s^k \De^k e^{s \De}( \phi_m v) \|_{L^p} \lesssim (s^{\frac{1}{2}} 2^{-m})^N \|  \phi_m v \|_{L^p} .
}
We let $\ti \phi \in C^\infty_0$ be a fattened version of $\phi$ so that $\phi (r)\ti \phi(r) = \phi(r)$ for all $r$, $\supp \ti \phi \subset[ 1/8, 8]$, and $\ti \phi_m(r):= \ti \phi(r/ 2^m)$. We note that the proof of~\eqref{eq:ell>m} can be easily modified to also establish the estimate, 
\EQ{ \label{eq:ell>m1}
 \|  \ti\phi_m s^k \De^k e^{s\De} ( \phi_\ell u) \|_{L^{q}} \lesssim (s^{\frac{1}{2}} 2^{-m})^N \| \phi_\ell u \|_{L^{q}},\, \, \forall\, m > \ell+10, \,  \,   \forall \, 1 < q < \infty. \quad
}
Using that $\phi_m(r) = \ti \phi_m(r) \phi_m(r)$ we have 
\EQ{
\| \phi_\ell s^k \De^k e^{s \De}( \phi_m v) \|_{L^p}  &= \sup_{ \|u \|_{L^{p'}} = 1} \abs{\ang{ \phi_\ell s \De e^{s \De}( \phi_m v), \, u }} \\
& =  \sup_{ \|u \|_{L^{p'}} = 1} \abs{\ang{   \phi_m v , \, \ti \phi_m s^k \De^k e^{s\De} ( \phi_\ell u)}} \\
& \le  \sup_{ \|u \|_{L^{p'}} = 1} \|  \phi_m v \|_{L^p}  \|  \ti\phi_m s^k \De^k e^{s\De} ( \phi_\ell u) \|_{L^{p'}}\\ 
& \lesssim \sup_{ \|u \|_{L^{p'}} = 1}  \|  \phi_m v \|_{L^p} (s^{\frac{1}{2}} 2^{-m})^N \| \phi_\ell u \|_{L^{p'}} \\
& \lesssim(s^{\frac{1}{2}} 2^{-m})^N \|  \phi_m v \|_{L^p},
}
where we used the estimate~\eqref{eq:ell>m1} in the second-to-last line above. 
This completes the proof of~\eqref{eq:m>ell}.

We remark that the only substantive difference in the proof of~\eqref{eq:lpregcore1} from that of~\eqref{eq:lpregcore2} is that the estimate~\eqref{eq:divw} is iterated instead of~\eqref{eq:lpreg1}. 
This finishes the proof of the proposition. 
\end{proof}

\subsection{Estimating lower order terms in $LE^*$}  \label{ss:lot-LE*}

Our main goal in this subsection is to prove Proposition~\ref{p:Hlotbound_for_V} and Proposition~\ref{p:Hlotbound} below which imply that 
\[
 \|Bu\|_{LE^\ast}\lesssim \|u\|_{LE} 
\]
for an operator $B$ of order one with sufficiently decaying coefficients. This is a manifestation of gain of regularity in our local smoothing estimate without frequency localizations. Along the way we will prove several general lemmas which are used in many other occasions in the rest of this work. To simplify notation, we will write $L^2$ for $L^2_{t,x}=L^2(\bbR\times \bbH^d)$ in this subsection. The detailed proofs below are worked out for the spaces $LE$ and $LE^\ast$ (or their frequency localized versions), but with minor modifications the same proofs yield the corresponding statements at the level of $\tilLE$, $\tilLE_0$, $\LE$ and their duals.

\begin{lem}\label{lem:PsLEs1}
Let $0 < s'\leq s$. Then we have for any scalar function $v$ that
\begin{align*}
\begin{split}
\|(s')^{\frac{k}{2}}\nabla^{(k)}e^{s'\Delta}v\|_{LE_s^\ast} &\lesssim \|v\|_{LE_s^\ast},\\
\|s^{\frac{k}{2}}\nabla^{(k)}e^{s\Delta}v\|_{LE_\low^\ast} &\lesssim \|v\|_{LE_\low^\ast},\\
\|(s')^{\frac{k}{2}}\nabla^{(k)}e^{s'\Delta}v\|_{LE_s} &\lesssim \|v\|_{LE_s},\\
\|s^{\frac{k}{2}}\nabla^{(k)}e^{s\Delta}v\|_{LE_\low} &\lesssim \|v\|_{LE_\low}.
\end{split}
\end{align*}
Moreover, for any $(r,q)$ tensor field $\boldsymbol{v}$, let $\nabla^{(k)} \cdot \boldsymbol{v}$ with $k := r + q$ be defined as in Corollary~\ref{c:pregdiv}.
Then it holds that
\begin{align*}
\begin{split}
\|e^{s'\Delta}(s')^{\frac{k}{2}}\nabla^{(k)}\cdot \boldsymbol{v}\|_{LE_s^\ast} &\lesssim \|\boldsymbol{v}\|_{LE_s^\ast},\\
\|e^{s\Delta}s^{\frac{k}{2}}\nabla^{(k)}\cdot \boldsymbol{v}\|_{LE_\low^\ast} &\lesssim \|\boldsymbol{v}\|_{LE_\low^\ast},\\
\|e^{s'\Delta}(s')^{\frac{k}{2}}\nabla^{(k)}\cdot \boldsymbol{v}\|_{LE_s} &\lesssim \|\boldsymbol{v}\|_{LE_s},\\
\|e^{s\Delta}s^{\frac{k}{2}}\nabla^{(k)}\cdot \boldsymbol{v}\|_{LE_\low} &\lesssim \|\boldsymbol{v}\|_{LE_\low},
\end{split}
\end{align*}
If $0 < s \leq s_0$, the same estimates hold with $e^{s\Delta}$ replaced by $e^{s(\Delta+\rho^2)}$ with constants depending on $s_0$. Finally, the analogous estimates hold at the level of $\tilLE$, $\tilLE_0$ and their duals.
\end{lem}
\begin{proof}
We present the proof in the scalar case and note that the tensor case follows by an almost identical argument. Also, as usual, the estimates for $e^{s(\Delta+\rho^2)}$ follow from the corresponding ones for $e^{s\Delta}$ since $e^{s\rho^2}$ is bounded when $s \leq s_0$. To simplify the notation we write $\calP_{s'}:=(s')^{\frac{k}{2}}\nabla^{(k)}e^{s'\Delta}$. Starting with the estimate in $LE_s^\ast$, first we show that
\begin{align*}
\begin{split}
\sum_{\ell\geq-k_s}2^{\frac{\ell}{2}}\|\phi_\ell\calP_{s'}v\|_{L^2}\lesssim \|v\|_{LE_s^\ast}.
\end{split}
\end{align*}
Here $\phi_\ell$ is a bump function as in \eqref{eq:phidef}. We also let $\phi_{\leq\ell-10}$ be a similar bump function supported in $\{r\leq2^{\ell-10}\}$, $\phi_{\geq\ell+10}=1-\phi_{\leq\ell+10}$, and $\tilphi_\ell=1-\phi_{\leq\ell-10}-\phi_{\geq\ell+10}$. Note that since $-k_{s'}\leq -k_s$ our localized parabolic regularity estimates can be applied to $\calP_{s'}$ when $\ell\geq -k_s$. It follows that the left-hand side above is bounded by
\begin{align*}
\begin{split}
&\sum_{\ell\geq-k_s}2^{\frac{\ell}{2}}\|\phi_\ell\calP_{s'}\phi_{\leq\ell-10}v \|_{L^2}+\sum_{\ell\geq-k_s}2^{\frac{\ell}{2}}\|\phi_\ell\calP_{s'}\tilphi_{\ell}v \|_{L^2}+\sum_{\ell\geq-k_s}2^{\frac{\ell}{2}}\|\phi_\ell\calP_{s'}\phi_{\geq\ell+10}v \|_{L^2}\\
&=:I+II+III.
\end{split}
\end{align*}
By localized parabolic regularity
\begin{align*}
\begin{split}
I &\lesssim \sum_{\ell\geq-k_{s}}2^{-N\ell+\frac{\ell}{2}}(s')^{\frac{N}{2}}\|v\|_{L^2(\bbR\times A_{\leq-k_s})}+\sum_{\ell\geq-k_s}\sum_{m=-k_s}^{\ell-10}2^{-N\ell+\frac{\ell}{2}}(s')^{\frac{N}{2}}\|\phi_mv\|_{L^2}\\
&\lesssim s^{\frac{1}{4}}\|v\|_{L^2(\bbR\times A_{\leq-k_s})}+\sum_{m\geq-k_s}\|\phi_mv\|_{L^2}\sum_{\ell\geq m+10}2^{-N\ell+\frac{\ell}{2}}(s')^{\frac{N}{2}}\\
&\lesssim s^{\frac{1}{4}}\|v\|_{L^2(\bbR\times A_{\leq-k_s})}+(s')^{\frac{N}{2}}s^{-\frac{N}{2}}\sum_{m\geq-k_s}2^{\frac{m}{2}}\|\phi_mv\|_{L^2} \\
&\lesssim \|v\|_{LE_s^\ast},
\end{split}
\end{align*}
where we have used the fact that $s's^{-1}\leq 1$. For $II$ we simply drop $\phi_\ell \calP_{s'}$, as this is a bounded operator on $L^2$, to get
\begin{align*}
\begin{split}
II\lesssim \sum_{\ell\geq -k_s}\|\tilphi_\ell v\|_{L^2}\lesssim \|v\|_{LE_s^\ast}.
\end{split}
\end{align*}
Finally, for $III$, using localized parabolic regularity (or simply by dropping $\phi_\ell\calP_{s'}$) we have
\begin{align*}
\begin{split}
III\lesssim& \sum_{\ell\geq-k_s}\sum_{m\geq \ell+10}2^{-mN}(s')^{\frac{N}{2}}2^{\frac{\ell}{2}}\|\phi_mv\|_{L^2}\\
=& \sum_{m\geq -k_s+10}\sum_{\ell=-k_s}^{m-10}2^{-mN}(s')^{\frac{N}{2}}2^{\frac{\ell}{2}}\|\phi_mv\|_{L^2}\\
\lesssim&\sum_{m\geq-k_s+10}2^{\frac{m}{2}}\|\phi_mv\|_{L^2}\lesssim \|v\|_{LE_s^\ast}.
\end{split}
\end{align*}
To complete the proof we still need to show that
\begin{align*}
\begin{split}
s^{\frac{1}{4}}\|\calP_{s'}v\|_{L^2(\bbR\times A_{\leq-k_s})}\lesssim \|v\|_{LE_s^\ast}.
\end{split}
\end{align*}
For this we estimate the left-hand side by
\begin{align*}
\begin{split}
s^{\frac{1}{4}}\|\calP_{s'}v\|_{L^2(\bbR\times A_{\leq-k_s})} \lesssim s^{\frac{1}{4}}\|\phi_{\leq-k_s}\calP_{s'}\phi_{\leq-k_s+10}v\|_{L^2}+s^{\frac{1}{4}}\|\phi_{\leq-k_s}\calP_{s'}\phi_{\geq-k_s+10}v\|_{L^2}.
\end{split}
\end{align*}
The first term on the right above is bounded by
\begin{align*}
\begin{split}
s^{\frac{1}{4}}\|v\|_{L^2(\bbR\times A_{\leq-k_s+10})}\lesssim \|v\|_{LE_s^\ast}.
\end{split}
\end{align*}
Using the boundedness of $\phi_{\leq-k_s}\calP_{s'}$ in $L^2$ the second term is bounded by
\begin{align*}
\begin{split}
s^{\frac{1}{4}}\sum_{m\geq-k_s+10}\|\phi_mv\|_{L^2}\lesssim \sum_{m\geq-k_s+10}2^{\frac{m}{2}}\|\phi_mv\|_{L^2}\lesssim \|v\|_{LE^\ast_s}.
\end{split}
\end{align*}
This completes the estimate in $LE_s^\ast$ and the estimate in $LE_s$ follows by duality:
\begin{align*}
\begin{split}
\|\calP_{s'}v\|_{LE_s}&\lesssim\sup_{\|\bsw\|_{LE_s^\ast}=1}|\angles{\calP_{s'}v}{\bsw}_{t,x}|=\sup_{\|\bsw\|_{LE_s^\ast}=1}|\angles{v}{e^{s'\Delta}(s')^{\frac{k}{2}}\nabla^{(k)}\cdot\bsw}_{t,x}|\\
&\leq \sup_{\|\bsw\|_{LE_s^\ast}=1}\|v\|_{LE_s}\|e^{s'\Delta}(s')^{\frac{k}{2}}\nabla^{(k)}\cdot\bsw\|_{LE_s^\ast}\lesssim \|v\|_{LE_s}.
\end{split}
\end{align*} 
For the estimate in  $LE_\low^\ast$ we first show that
\begin{align*}
\begin{split}
\sum_{\ell\geq 0}2^{\frac{3\ell}{2}}\|\phi_\ell\calP_sv\|_{L^2}\lesssim \|v\|_{LE_\low^\ast},
\end{split}
\end{align*}
by writing $v=\phi_{\leq\ell-10}v+\tilphi_\ell v+\phi_{\geq \ell+10}v$ as above. First
\begin{align*}
\begin{split}
\sum_{\ell\geq0}2^{\frac{3\ell}{2}}\|\phi_\ell\calP_s\phi_{\leq\ell-10}v\|_{L^2} &\leq \sum_{\ell\geq0}2^{\frac{3\ell}{2}-N\ell}s^{\frac{N}{2}}\|v\|_{L^2(\bbR\times A_{\leq0})}\\
&\quad + \sum_{\ell\geq10}\sum_{m=0}^{\ell-10}2^{\frac{3\ell}{2}-N\ell}s^{\frac{N}{2}}\|\phi_mv\|_{L^2}\\
&\lesssim \|v\|_{L^2(\bbR\times A_{\leq0})}+\sum_{m\geq0}2^{\frac{3m}{2}}\|\phi_mv\|_{L^2}\lesssim \|v\|_{LE_\low^\ast}.
\end{split}
\end{align*}
The contribution of $\phi_\ell \calP_s\tilphi_\ell v$ can be bounded as in the case of $LE_s^\ast$ by simply dropping $\phi_\ell\calP_s$ which is bounded in $L^2$. Next,
\begin{align*}
\begin{split}
\sum_{\ell\geq0}2^{\frac{3\ell}{2}}\|\phi_\ell\calP_s\phi_{\geq \ell+10}v\|_{L^2} &\lesssim \sum_{\ell\geq0}\sum_{m\geq\ell+10}2^{\frac{3\ell}{2}}2^{-mN}s^{\frac{N}{2}}\|\phi_mv\|_{L^2} \\
&\lesssim \sum_{m\geq10}2^{-Nm}s^{\frac{N}{2}}\|\phi_mv\|_{L^2}\sum_{\ell=0}^{m-10}2^{\frac{3\ell}{2}}\lesssim\|v\|_{LE_\low^\ast}.
\end{split}
\end{align*}
Finally to bound $\|\calP_sv\|_{L^2(\bbR\times A_{\leq0})}$ by $\|v\|_{LE_\low^\ast}$ we note that
\begin{align*}
\begin{split}
\|\calP_sv\|_{L^2(\bbR\times A_{\leq0})} &\lesssim \|\phi_{\leq0}\calP_s\phi_{\leq0}v\|_{L^2}+\|\phi_{\leq0}\calP_s\phi_{\geq0}v\|_{L^2}\\
&\lesssim \|\phi_{\leq0}v\|_{L^2}+\sum_{m\geq0}2^{-\frac{3m}{2}}2^{\frac{3m}{2}}\|\phi_mv\|_{L^2} \lesssim \|v\|_{LE_\low^\ast}.
\end{split}
\end{align*}
This completes the estimate for $LE_\low^\ast$ and the estimate for $LE_\low$ follows by duality as above. The proofs for the other local smoothing spaces are analogous.
\end{proof}

The next lemma is used in passing from $LE^\ast$ to $LE$ by exploiting the decay of the coefficients of the Schr\"odinger operator. 
\begin{lem}\label{lem:multLEs1}
Let $z$ be a scalar or tensor field, which is possibly time-dependent and complex-valued. For any $0 < s', s \lesssim 1$, we have 
\begin{align}
 \| z v\|_{LE_s^\ast} &\lesssim \Bigl( \|z\|_{L^\infty_{t,x}}+ \sum_{\ell\geq0}\|rz\|_{L^\infty(\bbR\times A_\ell)} \Bigr) \|v\|_{LE_{s'}}, \label{eq:z_LEstar_LE} \\
 \|zv\|_{LE_s^\ast} &\lesssim \Bigl( \|z\|_{L^\infty_{t,x}}+\sum_{\ell\geq0}\|r^2z\|_{L^\infty(\bbR\times A_\ell)} \Bigr) \|v\|_{LE_\low}. \label{eq:z_LEstar_LElow}
\end{align}
Moreover, for any $0 < s \lesssim 1$ it holds that
\begin{align}
\|z v\|_{LE_\low^\ast} &\lesssim \Bigl( \|z\|_{L^\infty_{t,x}}+ \sum_{\ell\geq0}\|r^2z\|_{L^\infty(\bbR\times A_\ell)} \Bigr) \|v\|_{LE_{s}}, \label{eq:z_LEstarlow_LE} \\
\|z v\|_{LE_\low^\ast} &\lesssim \Bigl( \|z\|_{L^\infty_{t,x}}+\sum_{\ell\geq0}\|r^3z\|_{L^\infty(\bbR\times A_\ell)} \Bigr) \|v\|_{LE_\low}. \label{eq:z_LEstarlow_LElow}
\end{align}
Analogously, for any $0 < s', s \lesssim 1$ we have 
\begin{align}
\| z v\|_{\tilLE_s^\ast} &\lesssim \|\jap{r}^{3+2\nsigma}z\|_{L^\infty_{t,x}}\|v\|_{\tilLE_{s'}},\label{eq:z_tilLEstar_LE}\\
\|zv\|_{\tilLE_s^\ast} &\lesssim  \|\jap{r}^{3+2\nsigma}z\|_{L^\infty_{t,x}}\|v\|_{\tilLE_\low},\label{eq:z_tilLEstar_LElow}\\
\| z v\|_{\tilLE_\low^\ast} &\lesssim \|\jap{r}^{3+2\nsigma}z\|_{L^\infty_{t,x}}\|v\|_{\tilLE_{s}},\label{eq:z_tilLEstarlow_LE}\\
\|zv\|_{\tilLE_\low^\ast} &\lesssim  \|\jap{r}^{3+2\nsigma}z\|_{L^\infty_{t,x}}\|v\|_{\tilLE_\low}.\label{eq:z_tilLEstarlow_LElow}
\end{align}
At the level of the spaces $\tilLE_{0, s}^\ast$ and $\tilLE_{0, \low}^\ast$ analogous estimates hold. 
\end{lem}
\begin{proof}
We begin with the proof of~\eqref{eq:z_LEstar_LE}. If $s'\leq s\lesssim 1$, we have 
\begin{align*}
\begin{split}
&s^{\frac{1}{4}}\|z v\|_{L^2(\bbR\times A_{\leq-k_s})}\\
&\lesssim (ss')^{\frac{1}{4}}s'^{-\frac{1}{4}}\|z\|_{L^\infty_{t,x}}\|v\|_{L^2(\bbR\times A_{\leq-k_{s'}})}+s^{\frac{1}{4}}\sum_{\ell=-k_{s'}}^{-k_s}2^{\frac{\ell}{2}}\|z\|_{L^\infty_{t,x}}\|r^{-\frac{1}{2}}v\|_{L^2(\bbR\times A_\ell)}\\
&\lesssim \|z\|_{L^\infty_{t,x}}(s')^{-\frac{1}{4}}\|v\|_{L^2(\bbR\times A_{\leq-k_{s'}})}+s^{\frac{1}{2}}\|z\|_{L^\infty_{t,x}}\sup_{\ell\geq-k_{s'}}\|r^{-\frac{1}{2}}v\|_{L^2(\bbR\times A_\ell)}\\
&\lesssim \|z\|_{L^\infty_{t,x}}\|v\|_{LE_{s'}}.
\end{split}
\end{align*}
Similarly, since $-k_{s}\geq -k_{s'}$
\begin{align*}
\begin{split}
\sum_{\ell \geq -k_s}2^{\frac{\ell}{2}}\|zv\|_{L^2(\bbR\times A_\ell)} &\lesssim \Bigl( \sum_{\ell\geq-{k_s}}\|rz\|_{L^\infty(\bbR\times A_\ell)} \Big) \sup_{\ell\geq-k_{s'}}\|r^{-\frac{1}{2}}v\|_{L^2(\bbR\times A_\ell)}\\
&\lesssim \Bigl( \|z\|_{L^\infty_{t,x}}+\sum_{\ell\geq0}\|rz\|_{L^\infty(\bbR\times A_\ell)} \Bigr) \|v\|_{LE_{s'}}.
\end{split}
\end{align*}
If $s\leq s'\lesssim1$ then 
\begin{align*}
\begin{split}
\|z v\|_{LE_s^\ast} &\lesssim s^{\frac{1}{4}}\|z v\|_{L^2(\bbR\times A_{\leq-k_s})}+\sum_{\ell=-k_s}^{-k_{s'}}2^{\frac{\ell}{2}}\|z v\|_{L^2(\bbR\times A_\ell)}+\sum_{\ell\geq -k_{s'}}2^{\frac{\ell}{2}}\|z v\|_{L^2(\bbR\times A_\ell)}\\
&\lesssim \|z\|_{L^\infty_{t,x}}\|v\|_{L^2(\bbR\times A_{\leq-k_{s'}})} \Big( s^{\frac{1}{4}}+\sum_{\ell=-k_s}^{-k_{s'}}2^{\frac{\ell}{2}} \Big)\\
&\quad + \Bigl( \sum_{\ell\geq-k_{s'}}\|rz\|_{L^\infty(\bbR\times A_\ell)} \Bigr) \sup_{\ell\geq-k_{s'}}\|r^{-\frac{1}{2}}v\|_{L^2(\bbR\times A_\ell)}\\
&\lesssim \Bigl( \|z\|_{L^\infty_{t,x}}+\sum_{\ell\geq0}\|rz\|_{L^\infty(\bbR\times A_\ell)} \Bigr)\|v\|_{LE_{s'}}.
\end{split}
\end{align*}
This completes the proof of \eqref{eq:z_LEstar_LE}. For \eqref{eq:z_LEstar_LElow} note that
\begin{align*}
\begin{split}
s^{\frac{1}{4}}\|z v\|_{L^2(\bbR\times A_{\leq-k_s})}+\sum_{\ell=-k_s}^{0}2^{\frac{\ell}{2}}\|z v\|_{L^2(\bbR\times A_\ell)} &\lesssim \|z\|_{L^\infty_{t,x}}\|v\|_{L^2(\bbR\times A_{\leq \max\{0,-k_s\}})} \\
&\leq \|z\|_{L^\infty_{t,x}} \|v\|_{LE_\low},
\end{split}
\end{align*}
and
\begin{align*}
\begin{split}
\sum_{\ell\geq0}2^{\frac{\ell}{2}}\|z v\|_{L^2(\bbR\times A_\ell)} &\lesssim \Bigl( \sum_{\ell\geq0}\|r^2z\|_{L^\infty(\bbR\times A_\ell)} \Big) \sup_{\ell\geq0} \|r^{-\frac{3}{2}}v\|_{L^2(\bbR\times A_\ell)}\\
&\lesssim \Big( \sum_{\ell\geq0}\|r^2z\|_{L^\infty(\bbR\times A_\ell)} \Big) \|v\|_{LE_\low}.
\end{split}
\end{align*}
For \eqref{eq:z_LEstarlow_LE} note that
\begin{align*}
\begin{split}
\|zv\|_{L^2(\bbR\times A_{\leq0})} &\lesssim s^{\frac{1}{4}}\|z\|_{L^\infty_{t,x}}s^{-\frac{1}{4}}\|v\|_{L^2(\bbR\times A_{\leq-k_s})} \\
&\quad +\|z\|_{L^\infty_{t,x}}\sup_{\ell\geq-k_s}\|r^{-\frac{1}{2}}v\|_{L^2(\bbR\times A_\ell)}\sum_{\ell=\min\{-k_s,0\}}^0 2^{\frac{\ell}{2}} \\
&\lesssim \|z\|_{L^\infty_{t,x}}\|v\|_{LE_s},
\end{split}
\end{align*}
and
\begin{align*}
\begin{split}
\sum_{\ell\geq0}2^{\frac{3\ell}{2}}\|z v\|_{L^2(\bbR\times A_\ell)}&\lesssim \sup_{\ell\geq 0}\|r^{-\frac{1}{2}}v\|_{L^2(\bbR\times A_\ell)}\sum_{\ell\geq0}\|r^2z\|_{L^\infty(\bbR\times A_\ell)}\\
&\lesssim \|v\|_{LE_s}\sum_{\ell\geq0}\|r^2z\|_{L^\infty(\bbR\times A_\ell)}.
\end{split}
\end{align*}
Finally for \eqref{eq:z_LEstarlow_LElow} observe that
\begin{align*}
\begin{split}
\|z v\|_{L^2(\bbR\times A_{\leq0})}\leq \|z\|_{L^\infty_{t,x}}\|v\|_{L^2(\bbR\times A_{\leq0})}\leq \|z\|_{L^\infty_{t,x}}\|v\|_{LE_\low},
\end{split}
\end{align*}
and
\begin{align*}
\begin{split}
\sum_{\ell\geq0}2^{\frac{3\ell}{2}}\|z v\|_{L^2(\bbR\times A_\ell)}&\lesssim \sup_{\ell\geq0}\|r^{-\frac{3}{2}}v\|_{L^2(\bbR\times A_\ell)}\sum_{\ell\geq0}\|r^3z\|_{L^\infty(\bbR\times A_\ell)}\\
&\lesssim \|v\|_{LE_\low}\sum_{\ell\geq0}\|r^3z\|_{L^\infty(\bbR\times A_\ell)}.
\end{split}
\end{align*}
The proofs are analogous in the case of the other smoothing spaces and are left to the reader.
\end{proof}

The following is a useful corollary of Lemmas~\ref{lem:PsLEs1} and~\ref{lem:multLEs1}.

\begin{lem} \label{lem:boundedness_LE_spatial_cutoff}
 Let $\chi_{\{ r > R\}}$ denote a smooth cut-off function to the spatial region $\{ r > R \}$. Then we have uniformly for all $R \geq 1$ that 
 \begin{equation}
  \| \chi_{\{ r > R \}} u \|_{LE} \lesssim \|u\|_{LE}.
 \end{equation}
 Analogous estimates hold at the level of $\tilLE$, $\tilLE_0$, $\LE$ and their dual spaces.
\end{lem}

\begin{proof}
We begin by recalling that 
\begin{align*}
 \| \chi_{\{r > R\}} u \|_{LE}^2 = \int_{\frac{1}{8}}^4 \bigl\| \tilP_{\geq s} \bigl(  \chi_{\{r > R\}} u \bigr) \bigr\|_{LE_\low}^2 \, \ds + \int_0^{\frac{1}{2}} s^{-\frac{1}{2}} \bigl\| \tilP_s \bigl( \chi_{\{r > R\}} u \bigr) \bigr\|_{LE_s}^2 \, \ds.
\end{align*}
Here we only describe how to estimate the second term on the right-hand side since the low-frequency integral in the first term is easier to treat. For given $0 < s \leq \frac{1}{2}$ and arbitrary $\frac{1}{2} \leq s_0 \leq 1$, we expand 
\[
 u = \int_0^s \tilP_{s'} u \, \frac{\ud s'}{s'} + \int_s^{s_0} \tilP_{s'} u \, \frac{\ud s'}{s'} + \tilP_{\geq s_0} u = \tilP_{\leq s} u + \tilP_{s \leq \cdot \leq s_0} u + \tilP_{\geq s_0} u
\]
and correspondingly have to bound 
\begin{equation} \label{equ:boundedness_cutoff_LE_high_freq_splitting}
 \begin{aligned}
  \int_0^{\frac{1}{2}} s^{-\frac{1}{2}} \bigl\| \tilP_s \bigl( \chi_{\{r > R\}} u \bigr) \bigr\|_{LE_s}^2 \, \ds &\lesssim \int_0^{\frac{1}{2}} s^{-\frac{1}{2}} \bigl\| \tilP_s \bigl( \chi_{\{r > R\}} \tilP_{\leq s} u \bigr) \bigr\|_{LE_s}^2 \, \ds \\
  &\quad + \int_0^{\frac{1}{2}} s^{-\frac{1}{2}} \bigl\| \tilP_s \bigl( \chi_{\{r > R\}} \tilP_{s \leq \cdot \leq s_0} u \bigr) \bigr\|_{LE_s}^2 \, \ds \\
  &\quad + \int_0^{\frac{1}{2}} s^{-\frac{1}{2}} \bigl\| \tilP_s \bigl( \chi_{\{r > R\}} \tilP_{\geq s_0} u \bigr) \bigr\|_{LE_s}^2 \, \ds.
 \end{aligned}
\end{equation}
In order to estimate the first term on the right-hand side of~\eqref{equ:boundedness_cutoff_LE_high_freq_splitting}, we use that $\tilP_s$ is bounded in $LE_s$ by Lemma~\ref{lem:PsLEs1} and that for any $0 < s, s' \lesssim 1$ we have (since $R \geq 1$)
\begin{equation} \label{equ:boundedness_cutoff_LE_observation}
 \| \chi_{\{r > R\}} v \|_{LE_s} = \| \chi_{\{ r > R\}} v \|_{LE_{s'}}.
\end{equation}
Then we find that
\begin{align*}
 \int_0^{\frac{1}{2}} s^{-\frac{1}{2}} \bigl\| \tilP_s \bigl( \chi_{\{r > R\}} \tilP_{\leq s} u \bigr) \bigr\|_{LE_s}^2 \, \ds &\lesssim \int_0^{\frac{1}{2}} s^{-\frac{1}{2}} \biggl( \int_0^s \| \chi_{\{ r > R\}} \tilP_{s'} u \|_{LE_{s}} \, \frac{\ud s'}{s'} \biggr)^2 \, \ds \\
 &\lesssim \int_0^{\frac{1}{2}} s^{-\frac{1}{2}} \biggl( \int_0^s \| \chi_{\{ r > R\}} \tilP_{s'} u \|_{LE_{s'}} \, \frac{\ud s'}{s'} \biggr)^2 \, \ds \\
 &\lesssim \int_0^{\frac{1}{2}} s^{-\frac{1}{2}} \biggl( \int_0^s \| \tilP_{s'} u \|_{LE_{s'}} \, \frac{\ud s'}{s'} \biggr)^2 \, \ds \\
 &\simeq \int_0^{\frac{1}{2}} \biggl( \int_0^s \Bigl( \frac{s'}{s} \Bigr)^{\frac{1}{4}} (s')^{-\frac{1}{4}} \| \tilP_{s'} u \|_{LE_{s'}} \, \frac{\ud s'}{s'} \biggr)^2 \, \ds \\
 &\lesssim \int_0^{\frac{1}{2}} (s')^{-\frac{1}{2}} \| \tilP_{s'} u \|_{LE_{s'}}^2 \, \frac{\ud s'}{s'},
\end{align*}
where in the last step we used Schur's test with the kernel $\chi_{\{ 0 \leq s' \leq s \leq \frac{1}{2}\}} \bigl( \frac{s'}{s} \bigr)^{\frac{1}{4}}$.

For the second term on the right-hand side of~\eqref{equ:boundedness_cutoff_LE_high_freq_splitting}, we insert the definition of the projection $\tilP_s = - s (\Delta + \rho^2) e^{s(\Delta + \rho^2)}$ to deal with the singularity of the $s^{-\frac{1}{2}}$ factor and obtain by Lemma~\ref{lem:PsLEs1} that
\begin{align*}
 &\int_0^{\frac{1}{2}} s^{-\frac{1}{2}} \bigl\| \tilP_s \bigl( \chi_{\{r > R\}} \tilP_{s \leq \cdot \leq s_0} u \bigr) \bigr\|_{LE_s}^2 \, \ds \\
 &\quad \lesssim \int_0^{\frac{1}{2}} s^{-\frac{1}{2}} \biggl( \int_s^{s_0} \bigl\| \tilP_s \bigl( \chi_{\{r > R\}} \tilP_{s'} u \bigr) \bigr\|_{LE_{s}} \, \frac{\ud s'}{s'} \biggr)^2 \, \ds \\
 &\quad \lesssim \int_0^{\frac{1}{2}} \biggl( \int_s^{s_0} s^{\frac{3}{4}} \bigl\| e^{s(\Delta+\rho^2)} (\Delta + \rho^2) \bigl( \chi_{\{r > R\}} \tilP_{s'} u \bigr) \bigr\|_{LE_{s}} \, \frac{\ud s'}{s'} \biggr)^2 \, \ds \\
 &\quad \lesssim \int_0^{\frac{1}{2}} \biggl( \int_s^{s_0} s^{\frac{3}{4}} \bigl\| (\Delta + \rho^2) \bigl( \chi_{\{r > R\}} \tilP_{s'} u \bigr) \bigr\|_{LE_{s}} \, \frac{\ud s'}{s'} \biggr)^2 \, \ds.
\end{align*}
We use the product rule to expand $(\Delta + \rho^2) \bigl( \chi_{\{r > R\}} \tilP_{s'} u \bigr)$. Then we can bound all resulting terms uniformly for $\frac{1}{2} \leq s_0 \leq 1$ in terms of
\[
 \int_0^1 (s')^{-\frac{1}{2}} \| \tilP_{s'} u \|_{LE_{s'}}^2 \, \frac{\ud s'}{s'} \lesssim \int_0^{1}(s')^{-\frac{1}{2}}\|\tilP_{\frac{s'}{2}}u\|_{LE_{\frac{s'}{2}}}^2\,\dsp\lesssim \|u\|_{LE}^2,
\]
by making use of the observation~\eqref{equ:boundedness_cutoff_LE_observation}, Lemmas~\ref{lem:LE_comparison} and~\ref{lem:PsLEs1}, Schur's test, and the fact that derivatives of $\chi_{\{r > R \}}$ are bounded in $L^\infty$ uniformly for all $R \geq 1$.

Finally, we turn to the third term on the right-hand side of~\eqref{equ:boundedness_cutoff_LE_high_freq_splitting}. Inserting the definition of the projection $\tilP_s = - s (\Delta + \rho^2) e^{s(\Delta + \rho^2)}$ and then using the product rule and Lemma~\ref{lem:PsLEs1}, we find that 
\begin{align*}
 \int_0^{\frac{1}{2}} s^{-\frac{1}{2}} \bigl\| \tilP_s \bigl( \chi_{\{r > R\}} \tilP_{\geq s_0} u \bigr) \bigr\|_{LE_s}^2 \, \ds &\lesssim \int_0^{\frac{1}{2}} s^{\frac{3}{2}} \bigl\| \bigl( \Delta \chi_{\{r > R\}} \bigr) \tilP_{\geq s_0} u \bigr\|_{LE_s}^2 \, \ds \\
 &\quad + \int_0^{\frac{1}{2}} s^{\frac{3}{2}} \bigl\| \bigl( \nabla^\mu \chi_{\{r > R\}} \bigr) \nabla_\mu \tilP_{\geq s_0} u \bigr\|_{LE_s}^2 \, \ds \\
 &\quad + \int_0^{\frac{1}{2}} s^{\frac{3}{2}} \bigl\| \chi_{\{r > R\}} (\Delta + \rho^2) \tilP_{\geq s_0} u \bigr\|_{LE_s}^2 \, \ds \\
 &=: I + II + III.
\end{align*}
To estimate term $I$ we note that by radial symmetry 
\begin{align*}
 \Delta \chi_{\{r > R\}} &= R^{-2} \psi''(r/R) + (d-1) \coth(r) R^{-1} \psi'(x/R)
\end{align*}
for some smooth cut-off function $\psi \in C^\infty(\bbR)$ with $\psi(r) = 1 $ for $r \gtrsim 1$. Since $\psi'(r/R)$ and $\psi''(r/R)$ localize to the region $\{ r \simeq R \}$ and since $R \geq 1$, we have in view of the definitions of $LE_s$ and $LE_\low$ uniformly for all $R \geq 1$ that 
\[
 \bigl\| \bigl( \Delta \chi_{\{r > R\}} \bigr) \tilP_{\geq s_0} u \bigr\|_{LE_s} \lesssim \bigl\| \tilP_{\geq s_0} u \bigr\|_{LE_\low}.
\]
Then we can easily carry out the integration in $s$ for the term $I$ to obtain that
\[
 I \lesssim \bigl\| \tilP_{\geq s_0} u \bigr\|_{LE_\low}^2,
\]
which upon averaging over $\frac{1}{2} \leq s_0 \leq 1$ yields the desired bound. The term $II$ can be estimated analogously, using in addition Lemma~\ref{lem:PsLEs1}. Finally, for term~$III$ we obtain 
\begin{align*}
 \int_0^{\frac{1}{2}} s^{\frac{3}{2}} \bigl\| \chi_{\{r > R\}} (\Delta + \rho^2) \tilP_{\geq s_0} u \bigr\|_{LE_s}^2 \, \ds &= \int_0^{\frac{1}{2}} s^{\frac{3}{2}} s_0^{-2} \bigl\| \chi_{\{r > R\}} \tilP_{s_0} u \bigr\|_{LE_s}^2 \, \ds \\
 &\lesssim \int_0^{\frac{1}{2}} s^{\frac{3}{2}} s_0^{-2} \| \tilP_{s_0} u\|_{LE_{s_0}}^2 \, \ds \\
 &\lesssim s_0^{-\frac{1}{2}} \| \tilP_{s_0} u \|_{LE_{s_0}}^2,
\end{align*}
where we could easily carry out the integration in $s$ and used that $s_0 \simeq 1$. Then averaging over $\frac{1}{2} \leq s_0 \leq 1$ leads to the desired bound. The proofs of the estimates for the other smoothing spaces are similar.
\end{proof}

The following two propositions are the main results of this subsection. The key ingredients for their proofs are Lemma~\ref{lem:PsLEs1} and Lemma~\ref{lem:multLEs1}.

\begin{prop} \label{p:Hlotbound_for_V} 
Let $V$ be an arbitrary potential function which is possibly time-dependent and complex-valued. Then it holds that
\begin{equation}
 \| Vu \|_{LE^*} \lesssim \Bigl( \|V\|_{L^\infty_{t,x}} + \sum_{\ell \geq 0} \|r^3 V\|_{L^\infty(\bbR\times A_\ell)} \Bigr) \|u\|_{LE} 
\end{equation}
and 
\begin{equation}
 \| Vu \|_{\tilLE^\ast} \lesssim \| \langle r \rangle^{3+2\nsigma} V \|_{L^\infty_{t,x}} \|u\|_{\tilLE}.
\end{equation}
Analogous estimates hold at the level of the spaces $\tilLE_0$, $\LE$ and their duals.
\end{prop} 

\begin{prop} \label{p:Hlotbound} 
Let $\bsb \in \Gamma(T\Hp^d)$ be a complex-valued and possibly time-dependent vector field. Then we have
\begin{equation} \label{equ:Hlotbound_LE_level}
 \| \bsb \cdot \na u  \|_{LE^*} \lesssim \sum_{k=0}^1 \Bigl( \|\nabla^{(k)}\bsb\|_{L^\infty_{t,x}}+\sum_{\ell\geq0}\|r^3\nabla^{(k)}\bsb\|_{L^\infty(\bbR\times A_\ell)} \Bigr) \|u\|_{LE}
\end{equation}
and 
\begin{equation}
 \| \bsb \cdot \na u  \|_{\tilLE^*} \lesssim \sum_{k=0}^1 \| \langle r \rangle^{3+2\nsigma} \nabla^{(k)} \bsb \|_{L^\infty_{t,x}} \|u\|_{\tilLE}.
\end{equation}
In particular, combined with Proposition~\ref{p:Hlotbound_for_V} it follows that
\begin{align*}
 \| \bsb \cdot \na u + \na \cdot (\bsb u) \|_{LE^\ast} &\lesssim \sum_{k=0}^1 \Bigl( \|\nabla^{(k)}\bsb\|_{L^\infty_{t,x}}+\sum_{\ell\geq0}\|r^3\nabla^{(k)}\bsb\|_{L^\infty(\bbR\times A_\ell)} \Bigr) \|u\|_{LE}, \\
 \| \bsb \cdot \na u + \na \cdot (\bsb u) \|_{\tilLE^\ast} &\lesssim \sum_{k=0}^1 \| \langle r \rangle^{3+2\nsigma} \nabla^{(k)} \bsb \|_{L^\infty_{t,x}}  \|u\|_{\tilLE}.
\end{align*}
Analogous estimates hold at the level of the spaces $\tilLE_0$, $\LE$ and their duals.
\end{prop} 
\begin{proof} 
In what follows we provide the details of the proof of the estimate~\eqref{equ:Hlotbound_LE_level} which is formulated at the level of the spaces $LE$ and $LE^\ast$. The proofs of the other claims proceed analogously and are left to the reader. For ease of notation we introduce the constants
\begin{align}\label{eq:Cjdef}
\begin{split}
&C_0(\bsz):=\|\bsz\|_{L^\infty_{t,x}}+\sum_{\ell\geq0}\|r\bsz\|_{L^\infty(\bbR\times A_\ell)},\\
&C_1(\bsz):=\|\bsz\|_{L^\infty_{t,x}}+\sum_{\ell\geq0}\|r^2\bsz\|_{L^\infty(\bbR\times A_\ell)},\\
&C_2(\bsz):=\|\bsz\|_{L^\infty_{t,x}}+\sum_{\ell\geq0}\|r^3\bsz\|_{L^\infty(\bbR\times A_\ell)},
\end{split}
\end{align} 
for an arbitrary tensor field $\bsz$ and note that $C_0(\bsz)\leq C_1(\bsz)\leq C_2(\bsz)$. We also let
\begin{align*}
\begin{split}
B:=\bsb\cdot\nabla.
\end{split}
\end{align*}
With this notation our task is to prove the following two estimates:
\begin{align}
&\int_{0}^{\frac{1}{2}}s^{\frac{1}{2}}\|\tilP_{s}Bu\|_{LE_s^\ast}^2\,\ds\lesssim \sum_{k=0}^1C_2^2(\nabla^{(k)}\bsb)\|u\|_{LE}^2,\label{eq:Hlowbound_high}\\
&\int_{\frac{1}{8}}^{4}\|\tilP_{\geq s}Bu\|_{LE_\low^\ast}^2\,\ds\lesssim \sum_{k=0}^1C_2^2(\nabla^{(k)}\bsb)\|u\|_{LE}^2.\label{eq:Hlowbound_low}
\end{align}
Starting with \eqref{eq:Hlowbound_high} we decompose $u$ as 
\begin{align*}
\begin{split}
u=\tilP_{\leq s} u+\tilP_{\geq s} u=\tilP_{\leq s}u+\tilP_{s\leq\cdot\leq s_0}u+\tilP_{\geq s_0}u,
\end{split}
\end{align*}
with $s_0\in[\frac{3}{4},1]$ arbitrary. Then
\begin{align*}
\begin{split}
\int_{0}^{\frac{1}{2}}s^{\frac{1}{2}}\|\tilP_{s}Bu\|_{LE_s^\ast}^2\,\ds &\lesssim \int_{0}^{\frac{1}{2}}s^{\frac{1}{2}}\|\tilP_{s}B\tilP_{\leq s}u\|_{LE_s^\ast}^2\,\ds+\int_{0}^{\frac{1}{2}}s^{\frac{1}{2}}\|\tilP_{s}BP_{\geq s}u\|_{LE_s^\ast}^2\,\ds\\
&\lesssim \int_{0}^{\frac{1}{2}}s^{\frac{1}{2}}\|\tilP_{s}B\tilP_{\leq s}u\|_{LE_s^\ast}^2\,\ds\\
&\quad +\int_{\frac{3}{4}}^1\int_{0}^{\frac{1}{2}}s^{\frac{1}{2}}\|\tilP_{s}B(\tilP_{s\leq\cdot\leq s_0}u+\tilP_{\geq s_0}u)\|_{LE_s^\ast}^2\,\ds\,\frac{\ud s_0}{s_0}\\
&\lesssim \int_{0}^{\frac{1}{2}}s^{\frac{1}{2}}\|\tilP_{s}B\tilP_{\leq s}u\|_{LE_s^\ast}^2\,\ds\\
&\quad + \int_{\frac{3}{4}}^1\int_{0}^{\frac{1}{2}}s^{\frac{1}{2}}\|\tilP_{s}B \tilP_{s\leq\cdot\leq s_0}u\|_{LE_s^\ast}^2\,\ds\,\frac{\ud s_0}{s_0}\\
&\quad + \int_{\frac{3}{4}}^1\int_{0}^{\frac{1}{2}}s^{\frac{1}{2}}\|\tilP_{s}B\tilP_{\geq s_0} u\|_{LE_s^\ast}^2\,\ds\,\frac{\ud s_0}{s_0}\\
&=:I+II+III.
\end{split}
\end{align*}
To estimate $I$ let us write
\begin{align*}
\begin{split}
 B \tilP_{\leq s} u = \nabla \cdot (\bsb \tilP_{\leq s}u)-(\nabla\cdot \bsb)\tilP_{\leq s}u.
\end{split}
\end{align*}
The contributions of these terms are bounded in a similar way and here we only consider the first term in detail. By Lemmas~\ref{lem:PsLEs1} and~\ref{lem:multLEs1}
\begin{align*}
\begin{split}
\int_0^{\frac{1}{2}}s^{\frac{1}{2}}\|\tilP_s\nabla\cdot(\bsb \tilP_{\leq s}u)\|_{LE_s^\ast}^2\,\ds &\lesssim \int_0^{\frac{1}{2}}s^{-\frac{1}{2}}\|\bsb \tilP_{\leq s}u\|_{LE_s^\ast}^2\,\ds \\
&\lesssim \int_0^{\frac{1}{2}}s^{-\frac{1}{2}}\Big(\int_0^s\|\bsb \tilP_{ s'}u\|_{LE_{s}^\ast}\,\dsp\Big)^2\,\ds \\
&\leq C_0^2(\bsb)\int_0^{\frac{1}{2}}\Big(\int_0^s(\frac{s'}{s})^{\frac{1}{4}}(s')^{-\frac{1}{4}}\|\tilP_{s'}u\|_{LE_{s'}}\,\dsp\Big)^2\,\ds\\
&\lesssim C_0^2(\bsb)\|u\|_{LE}^2,
\end{split}
\end{align*}
where in the last step we used the fact that $(\frac{s'}{s})^{\frac{1}{4}}\chi_{\{s'\leq s\}}$ is a Schur kernel. Applying the same argument to the second term in the decomposition of $BP_{\leq s}u$ above we conclude that
\begin{align*}
\begin{split}
I \lesssim \bigl( C_0^2(\bsb)+C_0^2(\nabla\cdot\bsb) \bigr) \|u\|_{LE}^2.
\end{split}
\end{align*}
For $II$ we keep the original form of $B\tilP_{s\leq \cdot\leq s_0}u$, that is,
\begin{align*}
\begin{split}
B\tilP_{s\leq \cdot\leq s_0}u=\bsb\cdot\nabla \tilP_{s\leq \cdot\leq s_0}u.
\end{split}
\end{align*}
Without the $s_0$ integration in the definition of $II$, arguing as above
\begin{align*}
\begin{split}
&\int_0^{\frac{1}{2}}s^{\frac{1}{2}}\|\tilP_s(\bsb\cdot\nabla \tilP_{s\leq\cdot\leq s_0}u)\|_{LE_s^\ast}^2\,\ds\\
&\lesssim\int_0^{\frac{1}{2}}s^{\frac{1}{2}}\Big(\int_s^{s_0}\|\bsb\cdot \nabla \tilP_{s'}u\|_{LE_{s}^\ast}\,\dsp\Big)^2\,\ds\\
&\lesssim C_0^2(\bsb)\int_0^{\frac{1}{2}}\Big(\int_s^{s_0}(\frac{s}{s'})^{\frac{1}{4}}(s')^{-\frac{1}{4}}\|s'^{\frac{1}{2}}\nabla \tilP_{s'}u\|_{LE_{s'}}\,\dsp\Big)^2\,\ds\\
&\lesssim C_0^2(\bsb)\int_0^{\frac{1}{2}}\Big(\int_s^{s_0}(\frac{s}{s'})^{\frac{1}{4}}(s')^{-\frac{1}{4}}\|\tilP_{\frac{s'}{2}}u\|_{LE_{\frac{s'}{2}}}\,\dsp\Big)^2\,\ds\\
&\lesssim C_0^2(\bsb)\|u\|_{LE}^2,
\end{split}
\end{align*}
where we used Schur's test as well as the equivalence of $LE_{s'}$ and $LE_{\frac{s'}{2}}$ from Lemma~\ref{lem:LE_comparison}. Since the implicit constants are uniform in $s_0\in[\frac{3}{4},1]$, integrating this estimate in $s_0$ shows that
\begin{align*}
\begin{split}
II\lesssim C_0^2(\bsb)\|u\|_{LE}^2.
\end{split}
\end{align*}
Similarly, for $III$ 
\begin{align*}
\begin{split}
\int_0^{\frac{1}{2}} s^{\frac{1}{2}}\|\tilP_s(\bsb\cdot\nabla \tilP_{\geq s_0}u)\|_{LE_s^\ast}^2\,\ds &\lesssim \int_0^{\frac{1}{2}}s^{\frac{1}{2}}\|\bsb\cdot\nabla \tilP_{\geq s_0}u\|_{LE_s^\ast}^2\,\ds \\
&\lesssim C_1^2(\bsb) \int_0^{\frac{1}{2}}s^{\frac{1}{2}}\|\nabla \tilP_{\geq s_0}u\|_{LE_\low}^2\,\ds \\
&\lesssim C_1^2(\bsb) \int_0^{\frac{1}{2}}s^{\frac{1}{2}}\|\tilP_{\geq \frac{s_0}{2}}u\|_{LE_\low}^2\,\ds \\
&\lesssim C_1^2(\bsb) \|\tilP_{\geq{\frac{s_0}{2}}}u\|_{LE_\low}^2.
\end{split}
\end{align*}
Integration of this in $s_0$ shows
\begin{align*}
\begin{split}
III \lesssim C_1^2(\bsb)\|u\|_{LE}^2,
\end{split}
\end{align*}
completing the proof of \eqref{eq:Hlowbound_high}.

To prove \eqref{eq:Hlowbound_low}, we decompose $u$ as
\begin{align*}
\begin{split}
u=\tilP_{\leq s}u+\tilP_{\geq s}u,
\end{split}
\end{align*}
and write
\begin{align}\label{eq:lowlowdecomptemp1}
\begin{split}
\tilP_{\geq s}Bu=\tilP_{\geq s}\nabla\cdot(\bsb (\tilP_{\leq s}u +\tilP_{\geq s}u))-\tilP_{\geq s}((\nabla\cdot \bsb)(\tilP_{\leq s}u+\tilP_{\geq s}u)).
\end{split}
\end{align}
We treat only the contribution of the first term on the right-hand side above in detail. Using Lemmas~\ref{lem:PsLEs1} and~\ref{lem:multLEs1}
\begin{align*}
\begin{split}
\int_{\frac{1}{8}}^4\|\tilP_{\geq s}\nabla\cdot(\bsb \tilP_{\leq s}u)\|_{LE_\low^\ast}^2\,\ds &\lesssim \int_{\frac{1}{8}}^4\|\bsb \tilP_{\geq s}u\|_{LE_{\low}^\ast}^2\,\ds\\
&\lesssim C_2^2(\bsb)\int_{\frac{1}{8}}^4\|\tilP_{\geq s}u\|_{LE_\low}^2\,\ds\lesssim C_2^2(\bsb)\|u\|_{LE}^2.
\end{split}
\end{align*}
Similarly, we find that 
\begin{align*}
\begin{split}
\int_{\frac{1}{8}}^4\|\tilP_{\geq s}\nabla\cdot(\bsb \tilP_{\leq s}u)\|_{LE_\low^\ast}^2\,\ds&\lesssim\int_{\frac{1}{8}}^4\Big(\int_0^s\|\bsb \tilP_{\geq \frac{7s'}{8}}\tilP_{\frac{s'}{8}}u\|_{LE_\low^\ast}\,\dsp\Big)^2\,\ds\\
&\lesssim C_1^2(\bsb)\int_{\frac{1}{8}}^4\Big(\int_0^s(s')^{\frac{1}{4}}(s')^{-\frac{1}{4}}\|\tilP_{\frac{s'}{8}}u\|_{LE_{\frac{s'}{8}}}\,\dsp\Big)^2\,\ds\\
&\lesssim C_1^2(\bsb)\int_0^{\frac{1}{2}}(s')^{-\frac{1}{2}}\|\tilP_{s'}u\|_{LE_{s'}}^2\,\dsp\lesssim C_1^2(\bsb)\|u\|_{LE}^2,
\end{split}
\end{align*}
where we used Cauchy-Schwarz and the change of variables $\frac{s'}{8}\mapsto s'$. Using a similar argument for the second term on the right-hand side of \eqref{eq:lowlowdecomptemp1} completes the proof of \eqref{eq:Hlowbound_low}, and hence of the proposition.
\end{proof}


\begin{proof}[Proof of Proposition~\ref{p:Hlotbound_for_V}]
 The proof is easier than the previous proof of Proposition~\ref{p:Hlotbound} because the operator $u \mapsto Vu$ does not involve derivatives. We therefore omit the details. 
\end{proof}

We end this subsection by noting that Lemma~\ref{lem:PsLEs1} allows us to prove the duality estimate in Remark~\ref{rem:LEduality}.

\begin{cor}\label{cor:LEduality}
There exists an absolute constant $C \geq 1$ such that we have for any $u$ and~$F$ that
\begin{align*}
\begin{split}
 |\angles{u}{F}_{t,x}| \leq C\|u\|_{LE}\|F\|_{LE^\ast}.
\end{split}
\end{align*}
Analogous estimates hold at the level of the spaces $\tilLE$, $\tilLE_0$, and $\LE$.
\end{cor}

\begin{proof}
We only prove the claim for the spaces $LE$ and $LE^\ast$ and observe that the other cases can be treated with analogous arguments. From the relation
\begin{align*}
\begin{split}
 s(\Delta+\rho^2)^2e^{2s(\Delta+\rho^2)}=\frac{1}{2}\frac{\ud}{\ud s} (s(\Delta+\rho^2) e^{2s(\Delta+\rho^2)})-\frac{1}{4}\frac{\ud}{\ud s}e^{2s(\Delta+\rho^2)}
\end{split}
\end{align*} 
we obtain that
\begin{align*}
\begin{split}
 u=4\int_0^{\infty}s^2(\Delta+\rho^2)e^{2s(\Delta+\rho^2)}u\,\ds 
\end{split}
\end{align*}
and that for any $s_0 > 0$,
\begin{align*}
\begin{split}
 e^{2s_0(\Delta+\rho^2)}u-2s_0(\Delta+\rho^2)e^{2s_0(\Delta+\rho^2)}u=4\int_{s_0}^\infty  s^2(\Delta+\rho^2)e^{2s(\Delta+\rho^2)}u\,\ds.
\end{split}
\end{align*}
It follows that for any $s_0>0$
\begin{align*}
\begin{split}
 \angles{u}{F}_{t,x} &= 4\int_0^{s_0}\angles{\tilP_su}{\tilP_sF}_{t,x}\,\ds \\
 &\quad + \angles{\tilP_{\geq s_0}u}{\tilP_{\geq s_0}F}_{t,x}-2\angles{s_0(\Delta+\rho^2)\tilP_{\geq s_0}u}{\tilP_{\geq s_0}F}_{t,x}.
\end{split}
\end{align*}
Integrating in $s_0$ we get
\begin{align*}
\begin{split}
  |\angles{u}{F}_{t,x}|&\lesssim \int_\frac{1}{4}^{\frac{1}{2}}\int_0^{s_0}|\angles{\tilP_su}{\tilP_sF}_{t,x}|\,\ds\,\frac{\ud s_0}{s_0}\\
  &\quad +\int_{\frac{1}{4}}^{\frac{1}{2}}|\angles{\tilP_{\geq s_0}u}{\tilP_{\geq s_0}F}_{t,x}|\,\frac{\ud s_0}{s_0}\\
  &\quad+\int_{\frac{1}{4}}^{\frac{1}{2}}|\angles{s_0(\Delta+\rho^2)\tilP_{\geq s_0}u}{\tilP_{\geq s_0}F}_{t,x}|\,\frac{\ud s_0}{s_0}\\
  &\lesssim \|u\|_{LE}\|F\|_{LE^\ast}\\
  &\quad+\Big(\int_{\frac{1}{4}}^{\frac{1}{2}}\|s_0(\Delta+\rho^2)\tilP_{\geq s_0}u\|_{LE_\low}^2\,\frac{\ud s_0}{s_0}\Big)^{\frac{1}{2}}\Big(\int_{\frac{1}{4}}^{\frac{1}{2}}\|\tilP_{\geq s_0}F\|_{LE_\low}^2\,\frac{\ud s_0}{s_0}\Big)^{\frac{1}{2}}\\
  &\lesssim \|u\|_{LE}\|F\|_{LE^\ast}
\end{split}
\end{align*}
where in the last step we used Lemma~\ref{lem:PsLEs1}.
\end{proof}


\begin{rem}\label{rem:completedualitytilLE}
Using Lemma~\ref{lem:PsLEs1} and \ref{lem:multLEs1} one can prove that $\tilLE$ and $\tilLE^\ast$ are in fact dual norms in the sense that
\begin{align*}
\begin{split}
\|u\|_{\tilLE}\simeq \sup_{\|F\|_{\tilLE^\ast}=1} |\angles{u}{F}_{t,x}|,\qquad \|F\|_{\tilLE^\ast}\simeq\sup_{\|u\|_{\tilLE}=1}|\angles{u}{F}_{t,x}|.
\end{split}
\end{align*}
The analogous statements hold for $\tilLE_0$, $\tilLE_0^\ast$ and $\LE$, $\LE^\ast$. The proofs of these statements use similar techniques as the ones already used in this subsection and are left to the reader.
\end{rem}



\subsection{Embeddings of the local smoothing spaces} \label{ss:le-embed}

In this subsection we collect several simple embedding properties of the local smoothing spaces defined in Section~\ref{s:prelim}.

\begin{lem} \label{l:wL2LE} 
The following estimates hold: 
\begin{equation} \label{eq:wL2LE}
 \begin{aligned}
  \|u\|_{L^2(\bbR\times A_{\leq0})}+\sup_{\ell\geq0}\| r^{-\frac{3}{2}} u \|_{L^2(\bbR\times A_\ell)} &\lesssim \| u \|_{LE}, \\
  \| \langle r \rangle^{-\frac{3}{2}-\nsigma} u \|_{L^2_{t,x}} &\lesssim \| u \|_{\tilLE}, \\
  \| \langle r \rangle^{-\frac{3}{2}-\nsigma} u \|_{L^2} &\lesssim \| u \|_{\tilLE_0}, \\
  \| \langle r \rangle^{-\frac{3}{2}-\nsigma} \nabla u \|_{L^2} &\lesssim \| u \|_{\tilLE_0^1}.
 \end{aligned}
\end{equation}
\end{lem} 


\begin{lem} \label{l:H12loc}
Fix an $R\ge 1$ and let $v_R \in C^{\infty}_0(\{r< R\})$. Then it holds that 
\begin{equation} \label{equ:H12loc}
 \|(-\De)^{\frac{1}{4}} v_R \|_{L^2} \lesssim_R \| v_R \|_{\tilLE_0},
\end{equation}
that is,
$ \tilLE_0  \subset H^{\frac{1}{2}}_{\loc}.$
\end{lem}

\begin{lem} \label{l:HminusLEstar}
We have for any $\nu \geq 1$ and any $v \in \tilLE^\ast_0$ that
\begin{align*}
\begin{split}
 \| (-\Delta)^{-\frac{\nu}{2}} v \|_{L^2} \lesssim \|v\|_{\tilLE_0^\ast},
\end{split}
\end{align*} 
that is, $\tilLE_0^\ast\subset H^{-\nu}$. 
\end{lem}

\begin{proof}[Proof of Lemma~\ref{l:wL2LE}]
The proofs of the asserted estimates are all similar. We therefore only provide the details for the last estimate involving the space $\tilLE_0^1$ and leave the other cases to the reader. 

To this end we write for any heat time $\frac{1}{2} \leq s_0 \leq 1$, 
\begin{equation}
 \nabla u = \int_0^{s_0} \nabla \tilP_s u \, \frac{\ud s}{s} + \nabla \tilP_{\geq s_0} u,
\end{equation}
from which we infer that 
\begin{equation} \label{equ:wL2LE_est1}
 \bigl\| \langle r \rangle^{-\frac{3}{2}-\nsigma} \nabla u \bigr\|_{L^2} \lesssim \int_0^{s_0} \bigl\| \langle r \rangle^{-\frac{3}{2}-\nsigma} \nabla \tilP_s u \bigr\|_{L^2} \, \ds + \bigl\| \langle r \rangle^{-\frac{3}{2}-\nsigma} \nabla \tilP_{\geq s_0} u \bigr\|_{L^2}.
\end{equation}
Using Lemma~\ref{lem:PsLEs1}, the fact that $\| \langle r \rangle^{-\frac{3}{2}-\nsigma} v \|_{L^2} \lesssim \| v \|_{\tilLE_{0,s}}$ for any $0 < s \lesssim 1$, and Cauchy-Schwarz, we obtain for the first term on the right-hand side of~\eqref{equ:wL2LE_est1} uniformly for all $\frac{1}{2} \leq s_0 \leq 1$ that 
\begin{align*}
 \int_0^{s_0} \bigl\| \langle r \rangle^{-\frac{3}{2}-\nsigma} \nabla \tilP_s u \bigr\|_{L^2} \, \ds &\lesssim \int_0^{s_0} s^{-\frac{1}{2}} \bigl\| s^{\frac{1}{2}} \nabla \tilP_s u \bigr\|_{\tilLE_{0,s}} \, \ds \\
 &\lesssim \int_0^{s_0} s^{-\frac{1}{2}} \| \tilP_{\frac{s}{2}} u \bigr\|_{\tilLE_{0,\frac{s}{2}}} \, \ds \\
 &\lesssim \biggl( \int_0^{1} s^{-\frac{3}{2}} \| \tilP_{\frac{s}{2}} u \bigr\|_{\tilLE_{0,\frac{s}{2}}}^2 \, \ds \biggr)^{\frac{1}{2}} \biggl( \int_0^1 s^{\frac{1}{2}} \, \ds \biggr)^{\frac{1}{2}} \\
 &\lesssim \|u\|_{\tilLE_0^1}.
\end{align*}
For the second term on the right-hand side of~\eqref{equ:wL2LE_est1} we find by Lemma~\ref{lem:PsLEs1} that 
\begin{align*}
 \bigl\| \langle r \rangle^{-\frac{3}{2}-\nsigma} \nabla \tilP_{\geq s_0} u \bigr\|_{L^2} &\lesssim s_0^{-\frac{1}{2}} \bigl\| s_0^{\frac{1}{2}} \nabla \tilP_{\geq s_0} u \bigr\|_{\tilLE_{0,\low}} \lesssim \| \tilP_{\geq \frac{s_0}{2}} u \|_{\tilLE_{0,\low}}.
\end{align*}
Hence, upon inserting the previous two bounds into the right-hand side of~\eqref{equ:wL2LE_est1} and averaging in $\frac{\ud s_0}{s_0}$ over $\frac{1}{2} \leq s_0 \leq 1$ we may conclude that $\| \langle r \rangle^{-\frac{3}{2}-\nsigma} \nabla u \|_{L^2} \lesssim \| u \|_{\tilLE_0^1}$, as desired. This finishes the proof of the lemma.
\end{proof} 

\begin{proof}[Proof of Lemma~\ref{l:H12loc}] 
By Lemma~\ref{l:L2res} we have for any fixed heat time $\frac{1}{2} \leq s_0 \leq 1$ that
\begin{equation} \label{equ:L2res_est1}
 \begin{aligned}
  \| (-\De)^{\frac{1}{4}} v_R \|_{L^2}^2  &\lesssim \int_0^{s_0} \| \tilP_s (-\De)^{\frac{1}{4}} v_R \|_{L^2}^2  \, \frac{\ud s}{s} + \|s_0^{\frac{1}{2}}\nabla \tilP_{\geq s_0}(-\Delta)^{\frac{1}{4}}v_R\|_{L^2}^2 \\
  &\quad \quad + \| \tilP_{\ge s_0} (-\De)^{\frac{1}{4}} v_R \|_{L^2}^2.
 \end{aligned}
\end{equation}
We begin by estimating the last term on the right-hand side above. Using Lemma~\ref{l:frac_preg} and Lemma~\ref{l:wL2LE} as well as keeping in mind the spatial support of $v_R$, we find that
\begin{align*}
 \| \tilP_{\ge s_0} (-\De)^{\frac{1}{4}} v_R \|_{L^2} &= s_0^{-\frac{1}{4}}  \| s_0^{\frac{1}{4}} (-\De)^{\frac{1}{4}} e^{s_0 (\De+\rho^2)} v_R \|_{L^2}  \\
 &\lesssim_{s_0} s_0^{-\frac{1}{4}} \| v_R \|_{L^2} \\
 &\lesssim_{s_0, R} \| \langle r \rangle^{-\frac{3}{2}-\nsigma} v_R \|_{L^2} \\
 &\lesssim_{s_0, R} \|v_R\|_{\tilLE_0}.
\end{align*}
The second term on the right-hand side of~\eqref{equ:L2res_est1} can be estimated analogously. Then for the first term on the right-hand side of~\eqref{equ:L2res_est1} we have by Lemma~\ref{l:frac_preg} that
\begin{align*}
 \int_0^{s_0} \| \tilP_s (-\De)^{\frac{1}{4}} v_R \|_{L^2}^2  \, \frac{\ud s}{s} &\lesssim \int_0^{s_0} s^{-\frac{1}{2}} \| s^{\frac{1}{4}} (-\De)^{\frac{1}{4}} e^{\frac{s}{2}(\Delta+\rho^2)} \tilP_{\frac{s}{2}} v_R \|_{L^2}^2 \, \frac{\ud s}{s} \\
 &\lesssim \int_0^{1} s^{-\frac{1}{2}} \| \tilP_{\frac{s}{2}} v_R \|_{L^2}^2\, \frac{\ud s}{s}.
\end{align*}
Below we will prove that 
\begin{equation} \label{equ:L2res_est2}
 \int_0^{1} s^{-\frac{1}{2}} \| \tilP_{\frac{s}{2}} v_R \|_{L^2}^2\, \frac{\ud s}{s} \lesssim \|v_R\|_{\tilLE_0}^2.
\end{equation}
Inserting the previous bounds into~\eqref{equ:L2res_est1} and then averaging in $\frac{\ud s_0}{s_0}$ over $\frac{1}{2} \leq s_0 \leq 1$ yields the desired estimate~\eqref{equ:H12loc}. It remains to verify~\eqref{equ:L2res_est2}. To this end we write 
\begin{equation} \label{equ:L2res_est3}
 \begin{aligned}
  \| \tilP_{\frac{s}{2}} v_R \|_{L^2} &\lesssim \| \tilP_{\frac{s}{2}} v_R \|_{L^2(A_{\leq -k_{\frac{s}{2}}})} + \sum_{-k_{\frac{s}{2}} \leq \ell \leq 0} \| \phi_\ell \tilP_{\frac{s}{2}} v_R \|_{L^2} + \| \tilP_{\frac{s}{2}} v_R \|_{L^2(A_{\geq 0})},
 \end{aligned}
\end{equation}
where $\phi_{\ell}$ is the bump function defined in~\eqref{eq:phidef}. We now control each of the terms on the right-hand side above. For the first two terms we obtain easily that
\begin{align*}
 &\| \tilP_{\frac{s}{2}} v_R \|_{L^2(A_{\leq -k_{\frac{s}{2}}})} + \sum_{-k_{\frac{s}{2}} \leq \ell \leq 0} \| \phi_\ell \tilP_{\frac{s}{2}} v_R \|_{L^2} \\
 &\lesssim s^{-\frac{1}{4}} \| \tilP_{\frac{s}{2}} v_R \|_{L^2(A_{\leq -k_{\frac{s}{2}}})} + \biggl( \sum_{-k_{\frac{s}{2}} \leq \ell \leq 0} 2^{\frac{1}{2} \ell} \biggr) \sup_{-k_{\frac{s}{2}} \leq \ell \leq 0} \, \| r^{-\frac{1}{2}} \tilP_{\frac{s}{2}} v_R \|_{L^2(A_\ell)} \\
 &\lesssim \| \tilP_{\frac{s}{2}} v_R \|_{\tilLE_{0, \frac{s}{2}}},
\end{align*}
which is of the desired form. To deal with the third term on the right-hand side of~\eqref{equ:L2res_est3} we exploit the spatial support of $v_R$ in $\{ r < R\}$ and the localized parabolic regularity estimates from Proposition~\ref{p:lpregcore} and Remark~\ref{r:lpregcutoff} (with $N=1$) as well as Lemma~\ref{l:wL2LE} to conclude that
\begin{align*}
 &\int_0^1 s^{-\frac{1}{2}} \| \tilP_{\frac{s}{2}} v_R \|_{L^2(A_{\geq 0})}^2 \, \ds \\
 &\lesssim \int_0^1 s^{-\frac{1}{2}} \| \chi_{\{1 \leq r \leq 2^{10} R\}} \tilP_{\frac{s}{2}} v_R \|_{L^2}^2 \, \ds + \int_0^1 s^{-\frac{1}{2}} \| \chi_{\{ r > 2^{10} R\}} \tilP_{\frac{s}{2}} v_R \|_{L^2}^2 \, \ds \\
 &\lesssim_{R} \int_0^1 s^{-\frac{1}{2}} \| r^{-\frac{3}{2}-\nsigma} \tilP_{\frac{s}{2}} v_R \|_{L^2(A_{\geq 0})}^2 \, \ds + \int_0^1 s^{-\frac{1}{2}} (s^{\frac{1}{2}} R^{-1})^2 \|v_R\|_{L^2}^2 \, \ds \\
 &\lesssim_{R} \int_0^1 s^{-\frac{1}{2}} \| \tilP_{\frac{s}{2}} v_R \|_{\tilLE_{0, \frac{s}{2}}}^2 \, \ds + \| \langle r \rangle^{-\frac{3}{2}-\nsigma} v_R \|_{L^2}^2 \\
 &\lesssim_{R} \|v_R\|_{\tilLE_0}^2.
\end{align*}
This finishes the proof of the lemma.
\end{proof} 

\begin{proof}[Proof of Lemma~\ref{l:HminusLEstar}]
By arguing as at the beginning of the proof of Lemma~\ref{l:L2res}, given any integer $k > 0$ and any fixed heat time $s_0 > 0$ we may write any function~$f$ as a linear combination
\begin{align*}
 f &= c_{2k} \int_0^{2s_0} s^{2k} (\Delta+\rho^2)^{2k}e^{s(\Delta+\rho^2)} f \,\ds \\
 &\quad \quad + c_{2k-1} s_0^{2k-1} (\Delta+\rho^2)^{2k-1} e^{2s_0(\Delta+\rho^2)}f + \ldots + c_0 e^{2s_0(\Delta+\rho^2)} f
\end{align*}
for appropriate constants $c_0, \ldots, c_{2k}$. Taking the inner product with $f$, integrating by parts and invoking Lemma~\ref{l:preg} gives
\begin{align*}
\|f\|_{L^2}^2&\lesssim \int_0^{s_0}\|s^k(\Delta+\rho^2)^ke^{s(\Delta+\rho^2)}f\|_{L^2}^2\,\ds+\|e^{\frac{s_0}{2}(\Delta+\rho^2)}f\|_{L^2}^2.
\end{align*}
Applying this estimate to $f = (-\Delta)^{-\frac{\nu}{2}}v$ with $k > \frac{\nu}{2} + 1$ and averaging in $\frac{\ud s_0}{s_0}$ over $\frac{1}{2} \leq s_0 \leq 1$ we get
\begin{equation} \label{equ:HminusLEstar_est1}
 \begin{aligned}
  \|(-\Delta)^{-\frac{\nu}{2}}v\|_{L^2}^2 &\lesssim \int_0^{1}\|(-\Delta)^{-\frac{\nu}{2}}s^{k}(\Delta+\rho^2)^{k}e^{s(\Delta+\rho^2)}v\|_{L^2}^2\,\ds \\
  &\quad \quad + \int_{\frac{1}{2}}^1 \|(-\Delta)^{-\frac{\nu}{2}}\tilP_{\geq \frac{s_0}{2}} v \|_{L^2}^2\,\dsnot.
 \end{aligned}
\end{equation}
Using the Poincar\'e inequality (that is, the fact that $\mathrm{Spec}{(-\Delta)} = [\rho^2,\infty)$, which implies that $(-\Delta)^{-\frac{\nu}{2}}$ is bounded in $L^2$), we have 
\begin{align*}
 \|(-\Delta)^{-\frac{\nu}{2}}\tilP_{\geq \frac{s_0}{2}}v\|_{L^2} \lesssim \|\tilP_{\geq \frac{s_0}{2}} v\|_{L^2} \lesssim \|\tilP_{\geq \frac{s_0}{2}} v \|_{\tilLE_{0,\low}^\ast}.
\end{align*}
Hence, the contribution of the second term on the right-hand side of~\eqref{equ:HminusLEstar_est1} is bounded by $\|v\|_{\tilLE_0^\ast}^2$, as desired.
For the first term on the right-hand side of~\eqref{equ:HminusLEstar_est1} we use the Poincar\'e inequality, Lemma~\ref{l:preg}, Lemma~\ref{l:frac_preg} as well as the fact that $\| v \|_{L^2} \lesssim s^{-\frac{1}{4}} \| v \|_{\tilLE_{0,s}^\ast}$ to conclude that
\begin{align*}
 &\int_0^{1}\|(-\Delta)^{-\frac{\nu}{2}}s^{k}(\Delta+\rho^2)^{k}e^{s(\Delta+\rho^2)}v\|_{L^2}^2\,\ds \\
 &\lesssim \int_{0}^{1} s^{\nu} \|s^{k-1-\frac{\nu}{2}}(-\Delta)^{-\frac{\nu}{2}}(\Delta+\rho^2)^{k-1}\tilP_sv\|_{L^2}^2\,\ds \\
 &\lesssim \int_{0}^{1} s^{\nu} \| \tilP_{\frac{s}{2}} v \|_{L^2}^2 \, \ds \\
 &\lesssim \int_0^{1} s^{\nu-\frac{1}{2}} \|\tilP_{\frac{s}{2}} v\|_{\tilLE_{0,\frac{s}{2}}^\ast}^2 \, \ds \\
 &\lesssim \|v\|_{\tilLE_0^\ast}^2.
\end{align*}
This finishes the proof of the lemma.
\end{proof}

\subsection{Estimates controlling frequency localized error in $L^2(\bbR\times \{r\leq R\})$ }  \label{ss:local-err}
In our main positive commutator estimate we will generate frequency localized $L^2$-errors in a bounded spatial region. To put these in the form stated in Theorem~\ref{t:LE1-sym} and Theorem~\ref{t:LE1}, we need to remove the frequency localizations, which is what the following lemma accomplishes. 

\begin{lem} \label{l:L2R}
Let $s_0>0$ and $N \in \N$ be fixed. Then for any $R \ge 1$  we have 
\begin{align}
 \| \tilP_{\ge s_0}u  \|_{L^2(\bbR \times \{r\leq R\})}^2 &\lesssim_{s_0, N} \| u \|_{L^2(\bbR \times \{r\leq 2^{10}R\})}^2 + R^{-N} \| \ang{r}^{-2} u \|_{L^2_{t, x}}^2, \quad \quad \label{eq:L2Rlow} \\
 \int_0^{s_0} \| \tilP_s u \|_{L^2(\bbR \times \{r\leq R\})}^2 \, \frac{\ud s}{s} &\lesssim_{s_0, N} \| u \|_{L^2(\bbR \times \{r\leq 2^{10}R\})}^2 + R^{-N} \| \ang{r}^{-2} u \|_{L^2_{t, x}}^2. \quad \quad \label{eq:L2Rhigh} 
\end{align}
\end{lem} 

\begin{proof}
We first prove~\eqref{eq:L2Rlow}. Note that with $\chi_{\leq R}$ a cutoff to the region $\{r\leq 2R\}$
\EQ{
 \| \tilP_{\ge s_0}u  \|_{L^2 (\bbR\times  \{r\leq R\})} \lesssim  \| \chi_{\le R} \tilP_{\ge s_0}( \chi_{\le 2^{10} R} u)  \|_{L^2} +  \| \chi_{\le R} \tilP_{\ge s_0} (\chi_{\ge 2^{10}R} u)  \|_{L^2 }.
}
For the first term on the right we have 
\EQ{
 \| \chi_{\le R} \tilP_{\ge s_0}( \chi_{\le 2^{10} R} u)  \|_{L^2} \lesssim \| e^{s_0 (\De+\rho^2)} ( \chi_{\le 2^{10} R} u)\|_{L^2} \lesssim \| ( \chi_{\le 2^{10} R} u)\|_{L^2}.
}
For the second term, by Proposition~\ref{p:lpregcore}, with $\phi_\ell$ as in \eqref{eq:phidef},
\EQ{
\| \chi_{\le R} \tilP_{\ge s_0} (\chi_{\ge 2^{10}R} u)  \|_{L^2 } &\lesssim  \sum_{\ell \ge \lfloor  \log_2 R \rfloor + 10} \| \chi_{\le R} e^{s_0 (\De+\rho^2)}  \phi_{\ell} u \|_{L^2} \\
& \lesssim_N \sum_{\ell \ge \lfloor  \log_2 R \rfloor + 10} (s_0^{\frac{1}{2}} 2^{-\ell})^N \|   \phi_{\ell} u \|_{L^2} \\
& \lesssim_N s_0^{\frac{N}{2}} \| \ang{r}^{-2} u \|_{L^2} \sum_{\ell \ge \lfloor  \log_2 R \rfloor + 10} (2^{-\ell})^{N-2} \\
&\lesssim_{s_0, N} R^{-N +2} \| \ang{r}^{-2} u \|_{L^2}, 
}
which finishes the proof of~\eqref{eq:L2Rlow}. We argue similarly to prove~\eqref{eq:L2Rhigh}. First, for each $s_0>0$ 
\EQ{
\int_0^{s_0} \| P_s u \|_{L^2_{t, x}(\bbR\times  \{r\leq R\})}^2 \, \frac{\ud s}{s}  & \le\int_0^{s_0} \| \chi_{\le R} P_{ s}( \chi_{\le 2^{10} R} u)  \|_{L^2}^2 \,  \frac{\ud s}{s}\\
&\quad+  \int_0^{s_0}\| \chi_{\le R} P_{s} (\chi_{\ge 2^{10}R} u)  \|_{L^2 }^2 \, \frac{\ud s}{s}.
}
For the first term, we note that,
\EQ{
 \int_0^{s_0}\| \chi_{\le R} \tilP_{ s}( \chi_{\le 2^{10} R} u)  \|_{L^2}^2\, \frac{\ud s}{s} &\lesssim\int_0^{s_0} \| \tilP_s ( \chi_{\le 2^{10} R} u)  \|_{L^2}^2\, \frac{\ud s}{s}  \\
 &\lesssim \| u  \|_{L^2(\bbR\times  \{r\leq 2^{10}R\})}^2.
}
For the second term we again use Proposition~\ref{p:lpregcore}, 
\EQ{
 &\int_0^{s_0}\| \chi_{\le R} \tilP_{s} (\chi_{\ge 2^{10}R} u)  \|_{L^2 }^2 \, \frac{\ud s}{s} \\
 &\quad \lesssim \int_0^{s_0} \Big( \sum_{\ell \ge \lfloor  \log_2 R \rfloor + 10} \| \chi_{\le R} s (\De+\rho^2) e^{s (\De+\rho^2)}  \phi_{\ell} u \|_{L^2} \Big)^2 \, \frac{\ud s}{s} \\
 &\quad \lesssim_N  \int_0^{s_0} \Big(\sum_{\ell \ge \lfloor  \log_2 R \rfloor + 10} (s^{\frac{1}{2}} 2^{-\ell})^N \|   \phi_{\ell} u \|_{L^2}\Big)^2 \, \frac{\ud s}{s} \\
 &\quad \lesssim_N  \int_0^{s_0} \Big(s^{\frac{N}{2}} \| \ang{r}^{-2} u \|_{L^2} \sum_{\ell \ge \lfloor  \log_2 R \rfloor + 10} (2^{-\ell})^{N-2}\Big)^2 \, \frac{\ud s}{s} \\
 &\quad \lesssim_ NR^{-2(N -2)} \| \ang{r}^{-2} u \|_{L^2}^2 \,  \int_0^{s_0}  s^{N}\, \frac{\ud s}{s} \\
 &\quad \lesssim_{s_0, N} R^{-2(N -2)} \| \ang{r}^{-2} u \|_{L^2}^2 , 
}
which completes the proof. 
\end{proof} 


\subsection{Boundedness of the multiplier $Q$}  \label{ss:Q-bdd}

In this subsection we establish that the low-frequency multiplier $Q^{(\al)}\tilP_{\ge s}$ used in Section~\ref{s:low} is bounded in $L^{2}_{t, x}$ and in $X_{\low, \gamma}$. Similarly, we prove that the high-frequency multiplier $Q_{s}^{(\al)}\tilP_s$ used in Section~\ref{s:high} is bounded in $L^2_{t, x}$ and in $X_{s, \gamma}$.  

\subsubsection{Bounding the low-frequency multiplier}

Here we fix $s_0 \simeq 1$. For any given slowly varying sequence $\{\al_\ell \} \in \calA$, let $Q^{(\alpha)}$ be as defined in Definition~\ref{d:beta_low}. Our goal is to prove the following two propositions. 
\begin{prop} \label{p:QlowL2} 
Fix $0 < s_0 \lesssim 1$. Then for any $s_0 \leq s \lesssim 1$ and any slowly varying sequence $\{\alpha_\ell\}\in\calA$,
\EQ{ \label{eq:QlowL2} 
 \| Q^{(\al)} \tilP_{\ge s} u \|_{L^2} &\lesssim_{s_0} \| \tilP_{\ge \frac{s}{2}} u\|_{L^2},\\
 \| Q^{(\al)} \tilP_{\ge s} u \|_{L^2_{t, x}} &\lesssim_{s_0} \| \tilP_{\ge \frac{s}{2}} u\|_{L^2_{t, x}},
}
where the implicit constant is independent of $\{\alpha_\ell\}\in\calA$.
\end{prop}

\begin{prop} \label{p:QlowX} 
Fix $0 < s_0 \lesssim 1$ and let $\{\gamma_k\} \in \calA$ be any slowly varying sequence. Then we have for any $s_0 \leq s \lesssim 1$ and any slowly varying sequence $\{\alpha_\ell\}\in\calA$
\EQ{
\| Q^{(\al)} \tilP_{\ge s} u \|_{X_{\low, \gamma}} \lesssim_{s_0} \| \tilP_{\ge \frac{s}{2}} u\|_{X_{\low, \gamma}},
}
where the implicit constant is independent of $\{\alpha_\ell\}\in\calA$.
\end{prop} 

\begin{cor}\label{cor:QLElow}
Fix $0 < s_0 \lesssim 1$. Then for any $s_0 \leq s \lesssim 1$ and $\{\alpha_\ell\}\in\calA$,
\begin{align*}
\begin{split}
 \|Q^{(\alpha)}\tilP_{\geq s}u \|_{LE_\low} \lesssim_{s_0} \|\tilP_{\geq \frac{s}{2}}u\|_{LE_\low},
\end{split}
\end{align*}
where the implicit constant is independent of $\{\alpha_\ell\}\in\calA$.
\end{cor}

First we prove Proposition~\ref{p:QlowL2}. 
\begin{proof}[Proof of Proposition~\ref{p:QlowL2}]
The second estimate is a consequence of the first. The proof of the first estimate is an easy application of Lemma~\ref{l:preg}. Note that, 
\EQ{
 \tilP_{\ge s}u = e^{s (\De+\rho^2)} u = e^{\frac{s}{2 }(\De+\rho^2)} e^{\frac{s}{2} (\De+\rho^2)} u. 
}
Therefore, since $s_0 \leq s \lesssim 1$, using~\eqref{eq:betabounds1},
\EQ{
\| Q^{(\al)} \tilP_{\ge s} u \|_{L^2} &\lesssim \| \bsbeta^{(\alpha), \mu} \nabla_\mu \tilP_{\ge s} u \|_{L^2} + \| \p_r \be \tilP_{\ge s} u \|_{L^2} + \| \coth(r) \be \tilP_{\ge s} u \|_{L^2} \\
& \lesssim_{s_0} \| s^{\frac{1}{2}} \na \tilP_{\ge s} u \|_{L^2}  + \| \tilP_{\ge s} u \|_{L^2} \\
& \simeq \| s^{\frac{1}{2}} \na e^{\frac{s}{2} (\De+\rho^2)} e^{\frac{s}{2} (\De+\rho^2)} u \|_{L^2}  + \| e^{\frac{s}{2} (\De+\rho^2)} e^{\frac{s}{2} (\De+\rho^2)} u \|_{L^2} \\
& \lesssim \| \tilP_{\ge \frac{s}{2}} u \|_{L^2},
}
where Lemma~\ref{l:preg} was applied to obtain the last inequality above. 
\end{proof} 
\begin{proof}[Proof of Proposition~\ref{p:QlowX}]
In this proof we use the notation $\chi_{\leq A}$ to denote a radial cutoff to the region $\{r\leq A\}$ for any positive number $A$.  We will also write $L^2$ instead of $L^2_{t,x}$ to simplify notation. As in the proof of Proposition~\ref{p:QlowL2}, we use~\eqref{eq:betabounds1} and the fact that $s \ge s_0$ to deduce that 
\EQ{ \label{eq:Qlow1} 
\| Q^{(\al)} \tilP_{\ge s} u \|_{X_{\low, \gamma}}^2 &:=  \| Q^{(\al)} \tilP_{\ge s} u \|_{L^2(\bbR \times A_{\leq0})}^2 + \sum_{\ell \ge 0} \gamma_\ell \| r^{-\frac{3}{2}} Q^{(\al)} \tilP_{\ge s} u \|_{L^2(\bbR \times A_\ell)}^2 \\
& \lesssim_{s_0}  \| s^{\frac{1}{2}} \na \tilP_{\ge s} u \|_{L^2(\bbR \times A_{\leq0})}^2 +  \| \tilP_{\ge s} u \|_{L^2(\bbR \times A_{\leq0})}^2  \\
& \quad  + \sum_{\ell \ge 0} \gamma_\ell 2^{-3 \ell} \|  s^{\frac{1}{2}} \na  \tilP_{\ge s} u \|_{L^2(\bbR \times A_\ell)}^2 +  \sum_{\ell \ge 0} \gamma_\ell 2^{-3 \ell} \|  \tilP_{\ge s} u \|_{L^2(\bbR\times A_\ell)}^2. 
}
The first two terms on the right-hand side above are estimated as follows. For the first term,  using Lemma~\ref{l:preg} and Proposition~\ref{p:lpregcore}, with $\phi_\ell$ as in Proposition~\ref{p:lpregcore}, we have 
\EQ{
\| s^{\frac{1}{2}} \na \tilP_{\ge s} u \|_{L^2(\bbR \times A_{\leq0})}  & \lesssim \| \chi_{\le 1} s^{\frac{1}{2}} \na e^{ \frac{s}{2} (\De+\rho^2)} ( \chi_{\le 2^{10}} e^{\frac{s}{2}(\De+\rho^2)} u) \|_{L^2}  \\
&\quad + \sum_{\ell \ge 10} \| \chi_{\le 1} s^{\frac{1}{2}} \na e^{ \frac{s}{2} (\De+\rho^2)} (  \phi_\ell e^{\frac{s}{2}(\De+\rho^2)} u) \|_{L^2}  \\
& \lesssim \|  \chi_{\le 2^{10}} e^{\frac{s}{2}(\De+\rho^2)} u \|_{L^2}   + \sum_{\ell \ge 10} (s^{\frac{1}{2}} 2^{-\ell})^{\frac{N}{2}} \|  \phi_\ell e^{\frac{s}{2}(\De+\rho^2)} u) \|_{L^2}   \\
& \lesssim_{s_0} \|  \tilP_{\geq\frac{s}{2}}  u \|_{L^2(\bbR \times A_{\leq10})}  \\
&\quad + \Big( \sum_{\ell \ge 10} \gamma_\ell 2^{-3\ell}\|  \tilP_{\geq\frac{s}{2}} u \|_{L^2(\bbR \times A_\ell)}^2\Big)^{\frac{1}{2}} \Big( \sum_{\ell \ge 10} \gamma_\ell^{-1} 2^{(-N+3) \ell} \Big)^{\frac{1}{2}} \\
&\lesssim \| \tilP_{\geq\frac{s}{2}} u \|_{X_{\low, \ga}}, 
}
where in the last line $N$ is fixed large enough so that the sum, 
\EQ{
\sum_{\ell \ge 10} \gamma_\ell^{-1} 2^{(-N+3) \ell} \le C
}
uniformly in $\{\gamma_k\} \in \calA$, which is possible by the definition of~$\calA$.  The second term on the right-hand side of~\eqref{eq:Qlow1} is estimated in an identical fashion. 

For the last two terms on the right-hand side of~\eqref{eq:Qlow1} we argue as follows. For the third term, we have 
\EQ{
 &\sum_{\ell \ge 0} \gamma_\ell 2^{-3 \ell} \|  s^{\frac{1}{2}} \na  \tilP_{\ge s} u \|_{L^2(\bbR\times A_\ell)}^2\\
  & \lesssim  \sum_{\ell \ge 0} \gamma_\ell 2^{-3 \ell} \| \phi_\ell   s^{\frac{1}{2}} \na  e^{\frac{s}{2} (\De+\rho^2)} ( \chi_{\le 1} e^{\frac{s}{2}(\De+\rho^2)} u) \|_{L^2}^2   \\
 & \quad + \sum_{\ell \ge 0}\gamma_\ell 2^{-3 \ell} \Big( \sum_{0 \le m \le \ell - 10}  \| \phi_\ell   s^{\frac{1}{2}} \na  e^{\frac{s}{2} (\De+\rho^2)} ( \phi_m e^{\frac{s}{2}(\De+\rho^2)} u) \|_{L^2}\Big)^2 \\
 & \quad +  \sum_{\ell \ge 0}\gamma_\ell 2^{-3 \ell} \Big( \sum_{\abs{\ell - m} \le  10}  \| \phi_\ell   s^{\frac{1}{2}} \na  e^{\frac{s}{2} (\De+\rho^2)} ( \phi_m e^{\frac{s}{2}(\De+\rho^2)} u) \|_{L^2}\Big)^2 \\
 & \quad + \sum_{\ell \ge 0} \gamma_\ell 2^{-3 \ell}  \Big( \sum_{ m \ge \ell + 10}\| \phi_\ell   s^{\frac{1}{2}} \na  e^{\frac{s}{2} (\De+\rho^2)} ( \phi_m e^{\frac{s}{2}(\De+\rho^2)} u) \|_{L^2}\Big)^2.
}
We bound the four terms on the right above as follows. For the first term, we use Lemma~\ref{l:preg} to deduce that 
\EQ{
 \sum_{\ell \ge 0} \gamma_\ell 2^{-3 \ell} \| \phi_\ell   s^{\frac{1}{2}} \na  e^{\frac{s}{2} (\De+\rho^2)} ( \chi_{\le 1} e^{\frac{s}{2}(\De+\rho^2)} u) \|_{L^2}^2 &\lesssim  \sum_{\ell \ge 0} \gamma_\ell 2^{-3 \ell} \|  \chi_{\le 1} e^{\frac{s}{2}(\De+\rho^2)} u \|_{L^2}^2 \\
 & \lesssim \| \tilP_{\geq\frac{s}{2}} u \|_{L^2(\bbR\times A_{\leq0})}^2 , 
}
where in the last line we used the fact that 
\EQ{
\sum_{\ell \ge 0} \gamma_\ell 2^{-3 \ell} \le C, 
}
uniformly in $\{ \gamma_{\ell} \} \in \calA$. For the second term, using Lemma~\ref{l:preg} 
\EQ{ 
\sum_{\ell \ge 0}\gamma_\ell 2^{-3 \ell} &\Big( \sum_{0 \le m \le \ell - 10}  \| \phi_\ell   s^{\frac{1}{2}} \na  e^{\frac{s}{2} (\De+\rho^2)} ( \phi_m e^{\frac{s}{2}(\De+\rho^2)} u) \|_{L^2}\Big)^2  \\
& \lesssim \sum_{\ell \ge 0} \Big( \sum_{0 \le m \le \ell - 10}\Big( \frac{\gamma_\ell}{\gamma_m}\Big)^{\frac{1}{2}} \left(\frac{2^{m}}{2^{\ell}}\right)^{\frac{3}{2}}  \gamma_m^{\frac{1}{2}}\|  r^{-\frac{3}{2}} e^{\frac{s}{2}(\De+\rho^2)} u \|_{L^2(\bbR\times A_m)}\Big)^2  \\
& \lesssim \sum_{m \ge 0} \gamma_m \|  r^{-\frac{3}{2}} \tilP_{\geq\frac{s}{2}} u \|_{L^2(\bbR\times A_m)}^2 ,
}
where the last line follows by Schur's test with the kernel 
\EQ{
k(m, \ell) := \chi_{m \le \ell -10} \left( \gamma_\ell/\gamma_m\right)^{\frac{1}{2}} \left(2^{m}/2^{\ell}\right)^{\frac{3}{2}} . 
}
Next, for the third term, we again use Lemma~\ref{l:preg} to deduce that 
\EQ{
 \sum_{\ell \ge 0}\gamma_\ell 2^{-3 \ell} &\Big( \sum_{\abs{\ell - m} \le  10}  \| \phi_\ell   s_0^{\frac{1}{2}} \na  e^{\frac{s}{2} (\De+\rho^2)} ( \phi_m e^{\frac{s}{2}(\De+\rho^2)} u) \|_{L^2}\Big)^2 \\
 & \lesssim \sum_{\ell \ge 0} \gamma_\ell 2^{-3\ell} \|  \ti\phi_\ell e^{\frac{s}{2}(\De+\rho^2)} u\|_{L^2}^2 \\
 & \lesssim \sum_{\ell \ge 0} \gamma_\ell \| r^{-\frac{3}{2}} \tilP_{\geq\frac{s}{2}} u  \|_{L^2(\bbR\times\ti A_\ell)}^2  \\
 & \lesssim \| \tilP_{\geq\frac{s}{2}} u \|_{X_{\low, \ga}}^2
}
where $\ti \phi_\ell$ and $\ti A_\ell$ are fattened versions of $\phi_\ell$ and $A_\ell$. Finally, for the last term we gain summability in $m$ from an application of Proposition~\ref{p:lpregcore}: 
\EQ{
\sum_{\ell \ge 0} \gamma_\ell 2^{-3 \ell} & \Big( \sum_{ m \ge \ell + 10}\| \phi_\ell   s^{\frac{1}{2}} \na  e^{\frac{s}{2} (\De+\rho^2)} ( \phi_m e^{\frac{s}{2}(\De+\rho^2)} u) \|_{L^2}\Big)^2 \\
& \lesssim \sum_{\ell \ge 0} \gamma_\ell 2^{-3 \ell}  \Big( \sum_{ m \ge \ell + 10} (s^{\frac{1}{2}} 2^{-m})^N\|  \phi_m e^{\frac{s}{2}(\De+\rho^2)} u \|_{L^2}\Big)^2  \\
& \lesssim \sum_{\ell \ge 0} \gamma_\ell 2^{-3 \ell} \sum_{ m \ge \ell + 10} \gamma_m \| r^{-\frac{3}{2}} e^{\frac{s}{2}(\De+\rho^2)} u \|_{L^2}^2 \sum_{m \ge \ell + 10} 2^{-m (2N-3)} \gamma_m^{-1}  \\
& \lesssim \sum_{ m \ge 10} \gamma_m \| r^{-\frac{3}{2}} \tilP_{\geq\frac{s}{2}} u \|_{L^2}^2\sum_{\ell = 0}^{m-10} \gamma_\ell 2^{-3 \ell}\lesssim \sum_{ m \ge 0} \gamma_m \| r^{-\frac{3}{2}} \tilP_{\geq\frac{s}{2}} u \|_{L^2}^2,
}
where again we used that 
\EQ{
\sum_{m \ge \ell + 10} 2^{-m (2N-3)} \gamma_m^{-1}   \le C\qquad \mathrm{and}\qquad\sum_{\ell \ge 0} \gamma_\ell 2^{-3 \ell}\leq C,
}
uniformly in $\{\gamma_\ell\} \in \calA$. The fourth and final term on the right-hand side of~\eqref{eq:Qlow1} is handled in an identical fashion. This completes the proof. 
\end{proof} 

\begin{proof}[Proof of Corollary~\ref{cor:QLElow}]
By Lemma~\ref{l:Xal} and Proposition~\ref{p:QlowX}
\begin{align*}
\begin{split}
 \|Q^{(\alpha)}\tilP_{\geq s}u\|_{LE_\low}&=\sup_{\{\gamma_\ell\}\in\calA}\|Q^{(\alpha)}\tilP_{\geq s}u\|_{X_{\low,\gamma}}\lesssim \sup_{\{\gamma_\ell\}\in\calA} \|\tilP_{\geq\frac{s}{2}}u\|_{X_{\low,\gamma}}\\
 &\lesssim \|\tilP_{\geq\frac{s}{2}}u\|_{LE_\low}.
\end{split}
\end{align*} 
\end{proof}

\subsubsection{Bounding the high-frequency multiplier}\label{s:Bounding_high_frequency_multiplier} 

Let $s_0>0$ be fixed. Given any slowly varying sequence $\{ \al_\ell \} \in \calA$ and any $0 < s \leq s_0$, we let $Q_s^{(\alpha)}$ be as defined in Definition~\ref{d:beta_high}. In the rest of this section the constants in our estimates are allowed to depend on $s_0$. 
\begin{prop}\label{p:QsL2} For any $0 < s \le s_0$ and $\{\alpha_\ell\}\in\calA$, we have 
\EQ{
\| Q_s^{(\al)} \tilP_{s} u  \|_{L^2} &\lesssim_{s_0} \| \tilP_{\frac{s}{2}} u\|_{L^2},\\
\| Q_s^{(\al)} \tilP_{s} u  \|_{L^2_{t, x}} &\lesssim_{s_0} \| \tilP_{\frac{s}{2}} u\|_{L^2_{t, x}},
} 
where the implicit constant is independent of $\{\alpha_\ell\} \in\calA$.
\end{prop} 

\begin{prop} \label{p:QsX} 
Let $\{\gamma_k\} \in \calA$ be any slowly varying sequence. Then for any $0 < s \le s_0$ and $\{\alpha_\ell\}\in\calA$, we have 
\EQ{
\| Q_s^{(\al)} \tilP_{s} u \|_{X_{s, \gamma}} \lesssim_{s_0} \| \tilP_{\frac{s}{2}} u\|_{X_{s, \gamma}},
}
where the implicit constant is independent of $\{\alpha_\ell\}, \{\gamma_k\} \in \calA$.
\end{prop} 

\begin{cor}\label{cor:QLEs}
For any $0<s\leq s_0$ and $\{\alpha_\ell\}\in\calA$, it holds that
\begin{align*}
\begin{split}
 \|Q^{(\alpha)}\tilP_{s}u \|_{LE_{s}} \lesssim_{s_0} \|\tilP_{ \frac{s}{2}}u\|_{LE_{s}},
\end{split}
\end{align*}
where the implicit constant is independent of  $\{\alpha_\ell\}\in\calA$.
\end{cor}
\begin{proof}[Proof of Proposition~\ref{p:QsL2}]
The second inequality is a consequence of the first. As in the previous section, the proof of the first inequality will be an easy application of Lemma~\ref{l:preg} along with the bounds~\eqref{eq:betabounds2}. Note that 
\EQ{
\tilP_s  u  = - s (\De+\rho^2) e^{s (\De+\rho^2)} u = - e^{\frac{s}{2} (\De+\rho^2)}s(\De+\rho^2)  e^{\frac{s}{2} (\De+\rho^2)} u .
}
Using the estimates~\eqref{eq:betabounds2} we have, 
\EQ{
\| Q_s^{(\al)} \tilP_{s} u  \|_{L^2} & \lesssim \| \bsbeta_{s}^{(\alpha), \mu} \na_\mu \tilP_{s} u \|_{L^2} + \| \p_r \be_{s}^{(\alpha)}  \tilP_{s} u \|_{L^2} + \| \coth(r) \be_{s}^{(\alpha)} \tilP_{s} u \|_{L^2} \\
& \lesssim_{s_0}  \| s^{\frac{1}{2}} \na \tilP_{s} u \|_{L^2} +  \| \tilP_{s} u \|_{L^2} \\
& \simeq \| s^{\frac{1}{2}} \na  e^{\frac{s}{2} (\De+\rho^2)} ( s (\De+\rho^2) e^{\frac{s}{2}(\De+\rho^2)} u)\|_{L^2} \\
&\quad+  \| e^{\frac{s}{2}(\De+\rho^2)} s (\De+\rho^2) e^{\frac{s}{2} (\De+\rho^2)} u \|_{L^2} \\
& \lesssim_{s_0}  \| s (\De+\rho^2) e^{\frac{s}{2}(\De+\rho^2)} u \|_{L^2} \lesssim \| \tilP_{\frac{s}{2}} u \|_{L^2} ,
}
which completes the proof.
\end{proof} 

\begin{proof}[Proof of Proposition~\ref{p:QsX}]
We proceed similarly to the proof of Proposition~\ref{p:QlowX} and use the same notation for the cutoffs $\chi_{\leq A}$ and $\phi_\ell$ and $L^2$ for $L^2_{t,x}$. Using the bounds~\eqref{eq:betabounds2} we find that
\EQ{ \label{eq:QsX_first_step}
 \| Q_s^{(\al)} \tilP_s u \|_{X_{s,\gamma}}^2 &:= s^{-\frac{1}{2}} \| Q_s^{(\al)} \tilP_s u \|_{L^2(\bbR \times A_{\leq-k_s})}^2 + \sum_{\ell \geq -k_s} \gamma_{\ell + k_s} \| r^{-\frac{1}{2}} \tilP_s u \|_{L^2(\bbR \times A_\ell)}^2 \\
 &\lesssim s^{-\frac{1}{2}} \| s^{\frac{1}{2}} \nabla \tilP_s u \|_{L^2(\bbR \times A_{\leq-k_s})}^2 + s^{-\frac{1}{2}} \| \tilP_s u \|_{L^2(\bbR \times A_{\leq-k_s})}^2 \\
 &\quad + \sum_{\ell \geq -k_s} \gamma_{\ell+k_s} \| r^{-\frac{1}{2}} s^{\frac{1}{2}} \nabla \tilP_s u \|_{L^2(\bbR \times A_\ell)}^2 \\
 &\quad + \sum_{\ell \geq -k_s} \gamma_{\ell + k_s} \| r^{-\frac{1}{2}} \tilP_s u \|_{L^2(\bbR \times A_\ell)}^2, 
}
where as usual the constants can depend on $s_0$. To estimate the first term on the right-hand side above we use Lemma~\ref{l:preg} and Proposition~\ref{p:lpregcore} to obtain that
\EQ{
 &s^{-\frac{1}{2}} \| s^{\frac{1}{2}} \nabla \tilP_s u \|_{L^2(\bbR \times A_{\leq-k_s})}^2 \\
 &\lesssim s^{-\frac{1}{2}} \| \chi_{\leq s^{\frac{1}{2}}} s^{\frac{1}{2}} \nabla e^{\frac{s}{2} (\Delta+\rho^2)} \bigl( \chi_{ \leq 2^{10} s^{\frac{1}{2}}} \tilP_{\frac{s}{2}} u \bigr) \|_{L^2}^2 \\
 &\quad + \Bigl( \sum_{\ell \geq -k_s + 10} s^{-\frac{1}{4}} \| \chi_{ \leq s^{\frac{1}{2}}} s^{\frac{1}{2}} \nabla e^{\frac{s}{2} (\Delta+\rho^2)} \bigl( \phi_\ell \tilP_{\frac{s}{2}} u \bigr) \|_{L^2} \Bigr)^2 \\
 &\lesssim s^{-\frac{1}{2}} \| \chi_{ \leq 2^{10} s^{\frac{1}{2}}} \tilP_{\frac{s}{2}} u \|_{L^2}^2 + \Bigl( \sum_{\ell \geq -k_s + 10} s^{-\frac{1}{4}} \bigl( s^{\frac{1}{2}} 2^{-\ell} \bigr)^{\frac{N}{2}} \| \phi_\ell \tilP_{\frac{s}{2}} u \|_{L^2} \Bigr)^2 \\
 &\lesssim s^{-\frac{1}{2}} \| \chi_{ \leq 2^{10} s^{\frac{1}{2}}} \tilP_{\frac{s}{2}} u \|_{L^2}^2 \\
 &\quad + \Bigl( \sum_{\ell \geq -k_s + 10} \gamma_{\ell + k_s} 2^{-\ell} \| \phi_\ell \tilP_{\frac{s}{2}} u \|_{L^2}^2 \Bigr) \Bigl( \sum_{\ell \geq -k_s + 10} \gamma_{\ell + k_s}^{-1} s^{\frac{1}{2}(N-1)} 2^{(-N+1) \ell} \Bigr) \\
 &\lesssim \| \tilP_{\frac{s}{2}} u \|_{X_{s,\gamma}}^2,
}
where in the last line we fixed $N$ sufficiently large so that
\EQ{
 \sum_{\ell \geq -k_s + 10} \gamma_{\ell + k_s}^{-1} s^{\frac{1}{2}(N-1)} 2^{(-N+1) \ell} \leq C
}
uniformly for all $0 < s \leq s_0$ and for all $\{ \gamma_k \} \in \calA$, which is possible by the definition of $\calA$. The second term on the right-hand side of~\eqref{eq:QsX_first_step} can be bounded in the same manner.

We now turn to estimating the third term on the right-hand side of~\eqref{eq:QsX_first_step}.
\EQ{ \label{eq:QsX_first_step_third_term_splitting}
 \sum_{\ell \geq -k_s} & \gamma_{\ell + k_s} \| r^{-\frac{1}{2}} s^{\frac{1}{2}} \nabla \tilP_s u \|_{L^2(\bbR \times A_\ell)}^2 \\
 &\lesssim \sum_{\ell \geq -k_s} \gamma_{\ell + k_s} 2^{-\ell} \| \phi_\ell s^{\frac{1}{2}} \nabla e^{\frac{s}{2}(\Delta+\rho^2)} \bigl( \chi_{\leq s^{\frac{1}{2}}} \tilP_{\frac{s}{2}} u \bigr) \|_{L^2}^2 \\
 &\quad + \sum_{\ell \geq -k_s} \gamma_{\ell + k_s} 2^{-\ell} \Bigl( \sum_{-k_s \leq m \leq \ell-10} \| \phi_\ell s^{\frac{1}{2}} \nabla e^{\frac{s}{2} (\Delta+\rho^2)} \bigl( \phi_m \tilP_{\frac{s}{2}} u \bigr) \|_{L^2} \Bigr)^2 \\
 &\quad + \sum_{\ell \geq -k_s} \gamma_{\ell + k_s} 2^{-\ell} \Bigl( \sum_{|m-\ell| < 10} \| \phi_\ell s^{\frac{1}{2}} \nabla e^{\frac{s}{2} (\Delta+\rho^2)} \bigl( \phi_m \tilP_{\frac{s}{2}} u \bigr) \|_{L^2} \Bigr)^2 \\ 
 &\quad + \sum_{\ell \geq -k_s} \gamma_{\ell + k_s} 2^{-\ell} \Bigl( \sum_{m \geq \ell + 10} \| \phi_\ell s^{\frac{1}{2}} \nabla e^{\frac{s}{2} (\Delta+\rho^2)} \bigl( \phi_m \tilP_{\frac{s}{2}} u \bigr) \|_{L^2} \Bigr)^2 .
}
For the first term on the right-hand side of~\eqref{eq:QsX_first_step_third_term_splitting} we use Lemma~\ref{l:preg} to infer that
\EQ{ 
 \sum_{\ell \geq -k_s} \gamma_{\ell + k_s} 2^{-\ell} \| \phi_\ell s^{\frac{1}{2}} \nabla e^{\frac{s}{2}(\Delta+\rho^2)} \bigl( \chi_{\leq s^{\frac{1}{2}}} \tilP_{\frac{s}{2}} u \bigr) \|_{L^2}^2 &\lesssim \sum_{\ell \geq -k_s} \gamma_{\ell+k_s} 2^{-\ell} \| \chi_{\leq s^{\frac{1}{2}}} \tilP_{\frac{s}{2}} u \|_{L^2}^2 \\
 &\lesssim s^{-\frac{1}{2}} \| \tilP_{\frac{s}{2}} u \|_{L^2(\bbR \times A_{\leq-k_s})}^2,
}
where we used that
\EQ{ 
 \sum_{\ell \geq -k_s} \gamma_{\ell + k_s} 2^{-\ell} \lesssim \sum_{\ell \geq -k_s} 2^{-\ell} \lesssim 2^{k_s} \simeq s^{-\frac{1}{2}}
}
uniformly for all $\{ \gamma_k \} \in \calA$. For the second term on the right-hand side of~\eqref{eq:QsX_first_step_third_term_splitting} we invoke Proposition~\ref{p:lpregcore} and Schur's test to conclude that 
\EQ{ 
 \sum_{\ell \geq -k_s} & \gamma_{\ell + k_s} 2^{-\ell} \Bigl( \sum_{-k_s \leq m \leq \ell-10} \| \phi_\ell s^{\frac{1}{2}} \nabla e^{\frac{s}{2} (\Delta+\rho^2)} \bigl( \phi_m \tilP_{\frac{s}{2}} u \bigr) \|_{L^2} \Bigr)^2 \\
 &\lesssim \sum_{\ell \geq -k_s} \gamma_{\ell + k_s} 2^{-\ell} \Bigl( \sum_{-k_s \leq m \leq \ell-10} \bigl( s^{\frac{1}{2}} 2^{-\ell} \bigr) \| \phi_m \tilP_{\frac{s}{2}} u \|_{L^2} \Bigr)^2 \\
 &\lesssim \sum_{\ell \geq -k_s} \biggl( \sum_{-k_s \leq m \leq \ell-10} \Bigl( \frac{\gamma_{\ell+k_s}}{\gamma_{m+k_s}} \Bigr)^{\frac{1}{2}} \Bigl( \frac{2^m}{2^\ell} \Bigr)^{\frac{1}{2}} \bigl( s^{\frac{1}{2}} 2^{-\ell} \bigr) \, \gamma_{m+k_s}^{\frac{1}{2}} \| r^{-\frac{1}{2}} \tilP_{\frac{s}{2}} u \|_{L^2(\bbR \times A_m)} \biggr)^2 \\
 &\lesssim \sum_{m \geq -k_s} \gamma_{m+k_s} \| r^{-\frac{1}{2}} \tilP_{\frac{s}{2}} u \|_{L^2(\bbR \times A_m)}^2.
}
Then for the third term on the right-hand side of~\eqref{eq:QsX_first_step_third_term_splitting} we use Lemma~\ref{l:preg} and crudely bound it by
\EQ{ 
 \sum_{\ell \geq -k_s} & \gamma_{\ell + k_s} 2^{-\ell} \Bigl( \sum_{|m-\ell| < 10} \| \phi_\ell s^{\frac{1}{2}} \nabla e^{\frac{s}{2} (\Delta+\rho^2)} \bigl( \phi_m \tilP_{\frac{s}{2}} u \bigr) \|_{L^2} \Bigr)^2 \\
 &\lesssim \sum_{\ell \geq -k_s} \gamma_{\ell + k_s} 2^{-\ell} \Bigl( \sum_{|m-\ell| < 10} \| \phi_m \tilP_{\frac{s}{2}} u \|_{L^2} \Bigr)^2 \\
 &\lesssim \sum_{m \geq -k_s} \gamma_{m+k_s} \| r^{-\frac{1}{2}} \tilP_{\frac{s}{2}} u \|_{L^2(\bbR \times A_m)}^2. 
}
Finally, for the fourth term on the right-hand side of~\eqref{eq:QsX_first_step_third_term_splitting} we use Proposition~\ref{p:lpregcore} to obtain that
\EQ{ 
 \sum_{\ell \geq -k_s} & \gamma_{\ell + k_s} 2^{-\ell} \Bigl( \sum_{m \geq \ell + 10} \| \phi_\ell s^{\frac{1}{2}} \nabla e^{\frac{s}{2} (\Delta+\rho^2)} \bigl( \phi_m \tilP_{\frac{s}{2}} u \bigr) \|_{L^2} \Bigr)^2 \\
 &\lesssim \sum_{\ell \geq -k_s} 2^{-\ell}\gamma_{\ell+k_s} \Bigl( \sum_{m \geq \ell + 10} \bigl( s^{\frac{1}{2}} 2^{-m} \bigr)^\frac{N}{2} \| \phi_m \tilP_{\frac{s}{2}} u \|_{L^2} \Bigr)^2 \\
 &\lesssim \sum_{\ell\geq-k_s}2^{-\ell}\gamma_{\ell+k_s}\sum_{m\geq\ell+10}\gamma_{m+k_s}\|r^{-\frac{1}{2}}\phi_m\tilP_\frac{s}{2}u\|_{L^2}^2\sum_{m\geq\ell+10}s^{\frac{N}{2}}2^{-m(N-1)}\gamma_{m+k_s}^{-1}\\
 &\lesssim\sum_{m\geq-k_s+10}\gamma_{m+k_s}\|r^{-\frac{1}{2}}\phi_m\tilP_\frac{s}{2}u\|_{L^2}^2\sum_{\ell\geq-k_s}\sum_{m\geq\ell+10}s^{\frac{N}{2}}2^{-\ell}2^{-m(N-1)}\big(\frac{\gamma_{\ell+k_s}}{\gamma_{m+k_s}}\big)\\
 &\lesssim\sum_{m\geq-k_s+10}\gamma_{m+k_s}\|r^{-\frac{1}{2}}\phi_m\tilP_\frac{s}{2}u\|_{L^2}^2.
}
Here to pass to the last line we have used the following estimate, where $\eta$ is as in Definition~\ref{d:A}:
\begin{align*}
\begin{split}
&\sum_{\ell\geq-k_s}\sum_{m\geq\ell+10}s^{\frac{N}{2}}2^{-\ell}2^{-m(N-1)}\big(\frac{\gamma_{\ell+k_s}}{\gamma_{m+k_s}}\big)\\
&\lesssim s^{\frac{N}{2}}\sum_{\ell\geq-k_s}\sum_{m\geq \ell+10}2^{-m(N-1)}2^{-\ell}2^{(m-\ell)\eta}\\
&\lesssim s^{\frac{N}{2}}\sum_{m\geq-k_s+10}2^{-m(N-1-\eta)}\sum_{-k_s\leq\ell\leq m-10}2^{-\ell(1+\eta)}\\
&\lesssim s^{\frac{N-1-\eta}{2}}\sum_{m\geq-k_s+10}2^{-m(N-1-\eta)}\lesssim 1.
\end{split}
\end{align*}
The last term on the right-hand side of~\eqref{eq:QsX_first_step} can be treated in an identical fashion, which completes the proof.
\end{proof} 

\begin{proof}[Proof of Corollary~\ref{cor:QLEs}]
This is similar to the proof of Corollary~\ref{cor:QLEs}. We omit the details.
\end{proof}

\subsection{Estimates needed to handle $F$} 
In this subsection we prove technical estimates concerning the pairings $\angles{\tilP_{\geq s }F}{Q^{(\alpha)}\tilP_{\geq s}u}_{t,x}$ and $\angles{\tilP_{s} F}{Q^{(\alpha)}_s \tilP_s u}_{t,x}$ that we need in Sections~\ref{s:low}--\ref{s:trans}. Here, $Q^{(\alpha)}$ and~$Q_s^{(\alpha)}$ are as in Definition~\ref{d:beta_low} and Definition~\ref{d:beta_high}, respectively. 

We start with the low frequencies.
\begin{prop} \label{p:FQlowu}
For any $\{\alpha_\ell \}\in\calA$ let $Q^{(\alpha)}$ be defined as in Definition~\ref{d:beta_low}. Then we have for any $\varepsilon>0$ that
\begin{align*}
 \begin{split}
  \int_{\frac{1}{8}}^4 \sup_{\{ \alpha_\ell \} \in \calA} \bigl| \angles{\tilP_{\geq s} F}{Q^{(\alpha)}\tilP_{\geq s}u}_{t,x} \bigr| \, \ds \leq \varepsilon \|u\|_{LE}^2 + C_\varepsilon \|F\|_{LE^\ast}^2.
 \end{split}
\end{align*}
Similarly
\begin{align*}
 \begin{split}
  \int_{\frac{1}{8}}^4 \bigl| \angles{\tilP_{\geq s}F}{\tilP_{\geq s}u}_{t,x} \bigr| \, \ds \leq \varepsilon \|u\|_{LE}^2 + C_\varepsilon \|F\|_{LE^\ast}^2.
 \end{split}
\end{align*}
\end{prop}
\begin{proof}
The second estimate follows from the bound $$|\angles{\tilP_{\geq}sF}{\tilP_{\geq s}u}|\leq \varepsilon\|\tilP_{\geq s}u\|_{LE_\low}^2 + C_\varepsilon\|\tilP_{\geq s}F\|_{LE^\ast_\low}^2$$ and the definitions of $LE$ and $LE^\ast$. Writing $Q$ for $Q^{(\alpha)}$, the first estimate follows by a similar argument if we can show that
\begin{align*}
\begin{split}
\int_{\frac{1}{8}}^{4}\sup_{\alpha\in\calA}\|Q\tilP_{\geq s}u\|_{LE_\low}^2\,\ds\lesssim \|u\|_{LE}^2.
\end{split}
\end{align*}
By Lemma~\ref{l:Xal} and Proposition~\ref{p:QlowX},  $\|Q\tilP_{\geq s}u\|_{LE_\low}\lesssim \|\tilP_{\geq\frac{s}{2}}u\|_{LE_\low}$, so
\begin{align*}
\begin{split}
\int_{\frac{1}{8}}^{4}\sup_{\alpha\in\calA}\|Q\tilP_{\geq s}u\|_{LE_\low}^2\,\ds\lesssim \int_{\frac{1}{8}}^2\|\tilP_{\geq s}u\|_{LE_\low}^2\,\ds+\int_{\frac{1}{16}}^\frac{1}{8}\|\tilP_{\geq s}u\|_{LE_\low}^2\,\ds.
\end{split}
\end{align*}
The first term above is bounded by $\|u\|_{LE}^2$. For the second term note that by the fundamental theorem of calculus
\begin{align*}
\begin{split}
\tilP_{\geq s}u=\tilP_{\geq 2s}u+\int_s^{2s}\tilP_{s'}u\,\dsp.
\end{split}
\end{align*}
Using the change of variables $2s\mapsto s$ we see that
\begin{align*}
\begin{split}
\int_{\frac{1}{16}}^\frac{1}{8}\|\tilP_{\geq 2s}u\|_{LE_\low}^2\,\ds\leq\|u\|_{LE}^2.
\end{split}
\end{align*}
On the other hand, by Lemma~\ref{lem:LEs_LElow_comparision}
\begin{align*}
\begin{split}
\|\tilP_{s'}u\|_{LE_\low}\lesssim \|\tilP_{s'}u\|_{LE_{s'}}
\end{split}
\end{align*}
uniformly in $s'\in[s,2s]\subseteq [\frac{1}{16},\frac{1}{4}]$, so using Cauchy-Schwarz
\begin{align*}
\begin{split}
\int_{\frac{1}{16}}^{\frac{1}{8}}\Big(\int_s^{2s}\|\tilP_{s'}u\|_{LE_\low}\,\dsp\Big)^2\,\ds\lesssim \int_{\frac{1}{16}}^{\frac{1}{4}}(s')^{-\frac{1}{2}}\|\tilP_{s'}u\|_{LE_{s'}}^2\,\dsp\lesssim \|u\|_{LE}^2,
\end{split}
\end{align*}
completing the proof.
\end{proof}

\begin{rem} \label{rem:FQlowu}
A similar argument shows that for any $\varepsilon>0$,
\begin{align*}
\begin{split}
\int_{\frac{1}{8}}^4 \sup_{\{\alpha_\ell\} \in\calA} \, \bigl| \angles{\tilP_{\geq s}F}{Q_{s_3}^{(\alpha)} \tilP_{\geq s}u}_{t,x} \bigr| \, \ds \leq \varepsilon \|u\|_{LE}^2 + C_{\varepsilon} \|F\|_{LE^\ast}^2,
\end{split}
\end{align*}
where for some fixed $s_3 \simeq 1$ the multiplier $Q_{s_3}^{(\alpha)}$ is as in Definition~\ref{d:beta_high}.
\end{rem}
Similarly, for high frequencies we will need the following result.

\begin{prop}\label{p:FQsu}
For any $\delta\in(0,1]$, $s>0$, and $\{\alpha_\ell\}\in\calA$, let $Q_s^{(\alpha)}$ be defined as in Definition~\ref{d:beta_high}. Then for any $\varepsilon>0$ it holds that
\begin{align*}
 \int_{0}^2 \sup_{\{ \alpha_\ell \} \in\calA} \, \bigl| \angles{\tilP_{s}F}{Q_s^{(\alpha)}\tilP_{s}u}_{t,x} \bigr| \, \ds \leq \varepsilon \|u\|_{LE}^2 + C_\varepsilon \|F\|_{LE^\ast}^2.
\end{align*}
Similarly
\begin{align*}
 \int_{0}^2 \bigl| \angles{\tilP_{s}F}{\tilP_{s}u}_{t,x} \bigr| \, \ds \leq \varepsilon \|u\|_{LE}^2 + C_\varepsilon \|F\|_{LE^\ast}^2.
\end{align*}
\end{prop}
\begin{proof}
By Lemmas~\ref{l:Xal},~\ref{lem:LE_comparison},~\ref{lem:PsLEs1}, and~\ref{lem:multLEs1} and Proposition~\ref{p:QsX}, we have 
\begin{align*}
 \sup_{\{ \alpha_\ell \} \in\calA} \, \|Q_s^{(\alpha)}\tilP_su\|_{LE_s} &\lesssim \|\tilP_{\frac{s}{4}}u\|_{LE_{\frac{s}{4}}},\\
 \|\tilP_su\|_{LE_s} &\lesssim \|\tilP_{\frac{s}{4}}u\|_{LE_{\frac{s}{4}}},\\
 \|\tilP_{s}F\|_{LE_{s}^\ast} &\lesssim \|\tilP_{\frac{s}{4}}F\|_{LE_{\frac{s}{4}}^\ast}.
\end{align*}
Then the desired estimates follow from bounding
\begin{align*}
 &\bigl| \angles{\tilP_s F}{Q_s^{(\alpha)}\tilP_su}_{t,x} \bigr| + \bigl| \angles{\tilP_s F}{\tilP_su}_{t,x} \bigr| \\
 &\qquad \qquad \lesssim \varepsilon \bigl( \|\tilP_su\|_{LE_s}^2+\|Q_s^{(\alpha)}\tilP_su\|_{LE_s}^2) + C_\varepsilon \|\tilP_sF\|_{LE_s^\ast}^2,
\end{align*}
and using the change of variables $\frac{s}{4}\mapsto s$.
\end{proof}

\subsection{Estimates needed to handle $H_{\lot}$} \label{ss:lot-comm}

One of the main advantages of the local smoothing space $LE$ and the inhomogeneous local smoothing space $LE^*$ is that they allow us to treat the lower order terms $H_\lot$ in $H$ as perturbations (up to a bounded region error) in the positive commutator argument presented in Sections~\ref{s:low}--\ref{s:trans}. In this section we establish the estimates required to do this. 

In general, we decompose $H_{\lot}$ into its symmetric and anti-symmetric parts as follows:
\begin{align*}
	H^{\antisym}_{\lot} u &= \Im \bsb^{\mu} \nb_{\mu} u + \nb_{\mu} (\Im \bsb^{\mu} u) + i \Im V u, \\
	H^{\sym}_{\lot} u &= \frac{1}{i} (\Re \bsb^{\mu} \nb_{\mu} u + \nb_{\mu} (\Re \bsb^{\mu} u)) + \Re V u.
\end{align*}
We start with the anti-symmetric part, which is small by hypothesis \eqref{eq:decay_assumptions-lot}. 
\begin{prop} \label{p:Hlot_antisym_Q}
For any $\dlt \in (0, 1]$, $s > 0$ and $\set{\alp_{\ell}} \in \calA$, let $Q^{(\alp)}$ and $Q^{(\alp)}_{s}$ be as in Definitions~\ref{d:beta_low} and \ref{d:beta_high}, respectively. Let $s_{3} \aeq 1$. Then the following bounds hold:
\begin{align*}
  \int_{\frac{1}{8}}^4 \sup_{\{\alpha_\ell\}\in\calA} \, \bigl| \Re\angles{i\tilP_{\geq s} H^{\antisym}_\lot u}{Q^{(\alpha)}\tilP_{\geq s}u}_{t,x} \bigr| \, \ds &\aleq  \varepsilon_{0} \|u\|_{LE}^2, \\
 \int_{\frac{1}{8}}^4 \sup_{\{\alpha_\ell\}\in\calA} \, \bigl| \Re\angles{i\tilP_{\geq s} H^{\antisym}_\lot u}{Q_{s_3}^{(\alpha)} \tilP_{\geq s}u}_{t,x} \bigr| \, \ds &\aleq \varepsilon_{0} \|u\|_{LE}^2, \\
 \int_{0}^2 \sup_{\{\alpha_\ell\} \in \calA} \, \bigl| \Re\angles{i\tilP_{s} H^{\antisym}_\lot u}{Q_s^{(\alpha)}\tilP_{s}u}_{t,x} \bigr| \, \ds &\aleq \varepsilon_{0} \|u\|_{LE}^2.
\end{align*}
\end{prop}
Recall that $\veps_{0}$ is the smallness parameter in the assumption \eqref{eq:decay_assumptions-lot} for $\Im \bsb$ and $\Im V$. Proposition~\ref{p:Hlot_antisym_Q} is thus a straightforward consequence of Propositions~\ref{p:Hlotbound_for_V}, \ref{p:Hlotbound}, \ref{p:FQlowu} and \ref{p:FQsu}, as well as Remark~\ref{rem:FQlowu}.

The following are the necessary estimates to handle the contribution of $H^{\sym}_\lot$ for low frequencies in Sections~\ref{s:low}--\ref{s:trans}.

\begin{prop} \label{p:commutator_Hlot_Q_low}
For any $\{ \alpha_\ell \} \in \calA$, let $Q^{(\alpha)}$ be as in Definition~\ref{d:beta_low}. Given any $\varepsilon > 0$, there exists a sufficiently large $R \equiv R(\varepsilon, \bsb, V) \geq 1$ such that
\begin{align*}
 \begin{split}
  \int_{\frac{1}{8}}^4 \sup_{\{\alpha_\ell\}\in\calA} \, \bigl| \Re\angles{i\tilP_{\geq s} H^{\sym}_\lot u}{Q^{(\alpha)}\tilP_{\geq s}u}_{t,x} \bigr| \, \ds \leq \varepsilon \|u\|_{LE}^2 + C_\varepsilon \|u\|^2_{L^2(\bbR\times\{r\leq R\})}.
 \end{split}
\end{align*}
\end{prop} 

\begin{rem} \label{rem:commutator_Hlot_Q_low}
The argument for the proof of Proposition~\ref{p:commutator_Hlot_Q_low} can also be used to establish the following closely related estimate: For any $\delta \in (0, 1]$, $s_3 \simeq 1$, and $\{ \alpha_\ell \}\in\calA$, let $Q_{s_3}^{(\alpha)}$ be as in Definition~\ref{d:beta_high}. Given any $\varepsilon > 0$, there exists a sufficiently large $R \equiv R(\varepsilon, \bsb, V) \geq 1$ such that
\begin{equation}
 \int_{\frac{1}{8}}^4 \sup_{\{\alpha_\ell\}\in\calA} \, \bigl| \Re\angles{i\tilP_{\geq s} H^{\sym}_\lot u}{Q_{s_3}^{(\alpha)} \tilP_{\geq s}u}_{t,x} \bigr| \, \ds \leq \varepsilon \|u\|_{LE}^2 + C_\varepsilon \|u\|^2_{L^2(\bbR\times\{r\leq R\})}.
\end{equation}
\end{rem}

In order to handle the contribution of $H^{\sym}_\lot$ for high frequencies in Sections~\ref{s:low}--\ref{s:trans}, we will need the following estimates.

\begin{prop} \label{p:commutator_Hlot_Q_high}  
 For any $\delta\in(0,1]$, $s>0$, and $\{ \alpha_\ell \}\in\calA$, let $Q_s^{(\alpha)}$ be as in Definition~\ref{d:beta_high}. Given any $\varepsilon > 0$, there exists a sufficiently large $R \equiv R(\varepsilon, \bsb, V) \geq 1$ such that 
 \begin{equation} \label{equ:commutator_Hlot_Q_high}
 \int_{0}^2 \sup_{\{\alpha_\ell\} \in \calA} \, \bigl| \Re\angles{i\tilP_{s} H^{\sym}_\lot u}{Q_s^{(\alpha)}\tilP_{s}u}_{t,x} \bigr| \, \ds \leq \varepsilon \|u\|_{LE}^2 + C_\varepsilon \|u\|^2_{L^2(\bbR\times\{r\leq R\})}.
 \end{equation}
\end{prop} 

\begin{rem}
 The argument for the proof of Proposition~\ref{p:commutator_Hlot_Q_high} can also be used to prove that given any $\varepsilon > 0$, there exists a sufficiently large $R \equiv R(\varepsilon, \bsb, V) \geq 1$ such that
 \begin{equation} \label{eq:extraPsHlotuPsuestimate}
 \int_{0}^2 \bigl| \Re\angles{i\tilP_{s} H^{\sym}_\lot u}{\tilP_{s}u}_{t,x} \bigr| \, \ds \leq \varepsilon \|u\|_{LE}^2 + C_\varepsilon \|u\|^2_{L^2(\bbR\times\{r\leq R\})}.
 \end{equation}
\end{rem}

In what follows, we drop the superscript $\sym$ for simplicity, and assume that $\bsb, V$ are real-valued. Moreover, we only present the proof of the more intricate high-frequency estimate of Proposition~\ref{p:commutator_Hlot_Q_high} and leave the easier, analogous proof of the low-frequency estimate of Proposition~\ref{p:commutator_Hlot_Q_low} to the reader. The proof of Proposition~\ref{p:commutator_Hlot_Q_high} requires the following auxiliary commutator estimate, which we establish first below.

\begin{lem} \label{lem:commutator_B_Ps}
For a smooth vector field $\bsb \in \Gamma(T \bbH^d)$ satisfying 
\[
 \sum_{k=0}^4 \| \nabla^{(k)} \bsb \|_{L^\infty_{t,x}} \lesssim 1,
\]
define the operator
\[
 B := \frac{1}{i} \bigl( \bsb \cdot \nabla + \nabla \cdot \bsb \bigr) = \frac{1}{i} \bigl( 2 \bsb \cdot \nabla + \nabla_\mu \bsb^\mu \bigr).
\]
Given $s_0 > 0$, there exists a constant $C \equiv C(s_0,\bsb)$ such that 
\begin{equation} \label{eq:BcomPsint}
 \int_0^{s_0} \bigl\| [B, \tilP_s]u \bigr\|_{L^2}^2 \, \ds \leq C \|u\|_{L^2}^2.
\end{equation}
\end{lem}
\begin{proof}
 We start by observing that the estimate~\eqref{eq:BcomPsint} for the zeroth order component $\nabla_\mu \bsb^\mu$ of the operator $B$ does not require us to exploit the commutator structure and follows easily from the boundedness of $\nabla_\mu \bsb^\mu$ and the resolution of the $L^2$ norm in terms of the frequency projections $\tilP_s$, see Lemma~\ref{l:L2res}. We therefore focus on the estimate for the first-order component $\bsb \cdot \nabla$ and note that
 \begin{equation} \label{equ:commutator_b_nabla_Ps}
  \begin{aligned}
   [ \bsb \cdot \nabla, \tilP_s ] u &= - s [ \bsb \cdot \nabla, \Delta ] e^{s(\Delta + \rho^2)} u - s \rho^2 [\bsb \cdot \nabla, e^{s(\Delta+\rho^2)}] u \\
   &\quad \, -s \Delta [\bsb \cdot \nabla, e^{s(\Delta+\rho^2)}] u.
  \end{aligned}
 \end{equation}
 Since $[ \bsb \cdot \nabla, \Delta ]$ is a second order differential operator with bounded coefficients, the estimate~\eqref{eq:BcomPsint} follows easily for the first term on the right-hand side of~\eqref{equ:commutator_b_nabla_Ps} using the parabolic regularity estimates from Lemma~\ref{l:preg}, elliptic regularity and the resolution of the $L^2$ norm in terms of the frequency projections $\tilP_s$. The second term on the right-hand side of~\eqref{equ:commutator_b_nabla_Ps} can also be handled in a straightforward manner. In order to estimate the contribution to~\eqref{eq:BcomPsint} of the last and main term on the right-hand side of~\eqref{equ:commutator_b_nabla_Ps}, we define
 \[
  w(s) := [ \bsb \cdot \nabla, e^{s(\Delta + \rho^2)} ] u.
 \]
 Then $w(s)$ satisfies the inhomogeneous heat equation 
 \begin{align*}
  \bigl( \partial_s - (\Delta + \rho^2) \bigr) w(s) &= - (\nabla_\mu \nabla^\mu \bsb^\nu) \nabla_\nu e^{s(\Delta+\rho^2)} u - 2 (\nabla_\mu \bsb^\nu) \nabla^\mu \nabla_\nu e^{s(\Delta+\rho^2)} u \\
  &\qquad-2\rho\, \bsb^\nu\nabla_\nu e^{s(\Delta+\rho^2)}u\\
  &= \nabla_\mu \nabla^\mu \bsb^\nu \nabla_\nu e^{s(\Delta+\rho^2)} u - 2 \nabla^\mu \bigl( (\nabla_\mu \bsb^\nu) \nabla_\nu e^{s(\Delta+\rho^2)} u \bigr) \\
  &\qquad-2\rho \,\bsb^\nu\nabla_\nu e^{s(\Delta+\rho^2)}u\\
  &=: W_1(s) + W_2(s)+W_3(s),
 \end{align*}
 or in Duhamel form
 \begin{equation} \label{equ:commutator_b_nabla_Ps_duhamel_w}
  w(s) = \int_0^s e^{(s-s')(\Delta+\rho^2)} s' (W_1(s') + W_2(s')+W_3(s')) \, \frac{\ud s'}{s'},
 \end{equation}
 and we need to show that 
 \[
  \int_0^{s_0} \| s \Delta w(s) \|_{L^2}^2 \, \ds \leq C \|u\|_{L^2}^2.
 \]
 We begin to estimate the contribution of $W_1(s')$. Here we split the heat time integration interval $[0,s]$ in~\eqref{equ:commutator_b_nabla_Ps_duhamel_w} into the consecutive intervals $[0, \frac{s}{2}]$ and $[\frac{s}{2}, s]$, which have to be dealt with in different manners,
 \begin{equation}
  \begin{aligned}
   &\int_0^{s_0} \biggl( \int_0^s \bigl\| s \Delta e^{(s-s')(\Delta+\rho^2)} s' W_1(s') \bigr\|_{L^2} \, \frac{\ud s'}{s'} \biggr)^2 \, \frac{\ud s}{s} \\
   &\quad \lesssim \int_0^{s_0} \biggl( \int_0^{\frac{s}{2}} \bigl\| s \Delta e^{(s-s')(\Delta+\rho^2)} s' W_1(s') \bigr\|_{L^2} \, \frac{\ud s'}{s'} \biggr)^2 \, \frac{\ud s}{s} \\
   &\quad \quad + \int_0^{s_0} \biggl( \int_{\frac{s}{2}}^s \bigl\| s \Delta e^{(s-s')(\Delta+\rho^2)} s' W_1(s') \bigr\|_{L^2} \, \frac{\ud s'}{s'} \biggr)^2 \, \frac{\ud s}{s} \\
   &\quad =: I + II.
  \end{aligned}
 \end{equation}
 For term $I$ we obtain by the parabolic regularity estimates from Lemma~\ref{l:preg} that $I$ is bounded by
 \begin{align*}
  &\int_0^{s_0} \biggl( \int_0^{\frac{s}{2}} \frac{s (s')^{\frac{1}{2}}}{s-s'} \bigl\| (s-s') \Delta e^{(s-s')(\Delta+\rho^2)} (\nabla^\mu \nabla_\mu \bsb^\nu) (s')^{\frac{1}{2}} \nabla_\nu e^{s'(\Delta+\rho^2)} u \bigr\|_{L^2} \, \frac{\ud s'}{s'} \biggr)^2 \ds \\
  &\lesssim \| \nabla^{(2)} \bsb \|_{L^\infty_{t,x}}^2 \|u\|_{L^2}^2 \int_0^{s_0} \biggl( \int_0^{\frac{s}{2}} \frac{s (s')^{\frac{1}{2}}}{s-s'} \, \frac{\ud s'}{s'} \biggr)^2 \ds \\
  &\lesssim_{\bsb, s_0} \|u\|_{L^2}^2.
 \end{align*}
 To bound the second term $II$, we distribute all derivatives of $\Delta$ onto $\bsb$ and $u$. Then we obtain by the parabolic regularity estimates from Lemma~\ref{l:preg} that 
 \begin{align*}
  II &\lesssim \int_0^{s_0} \biggl( \int_{\frac{s}{2}}^s s \bigl\| e^{(s-s') (\Delta+\rho^2)} s' \Delta ((\nabla^\mu \nabla_\mu \bsb^\nu) \nabla_\nu e^{s'(\Delta+\rho^2)} u \bigr) \bigr\|_{L^2} \, \frac{\ud s'}{s'} \biggr)^2 \, \ds \\
  &\lesssim \int_0^{s_0} \biggl( \sum_{k=0}^2 \| \nabla^{(2+k)} \bsb \|_{L^\infty_{t,x}} \|u\|_{L^2} \int_{\frac{s}{2}}^s s (s')^{\frac{k-1}{2}} \, \frac{\ud s'}{s'} \biggr)^2 \, \ds \\
  &\lesssim_{\bsb, s_0} \|u\|_{L^2}^2.
 \end{align*}
 The contribution of $W_2(s')$ has to be estimated more carefully using Schur's test. Again we split the heat time integration interval $[0,s]$ in~\eqref{equ:commutator_b_nabla_Ps_duhamel_w} into the consecutive intervals $[0, \frac{s}{2}]$ and $[\frac{s}{2}, s]$, and treat their contributions separately,
 \begin{equation}
  \begin{aligned}
   &\int_0^{s_0} \biggl( \int_0^s \bigl\| s \Delta e^{(s-s')(\Delta+\rho^2)} s' W_2(s') \bigr\|_{L^2} \, \frac{\ud s'}{s'} \biggr)^2 \, \frac{\ud s}{s} \\
   &\quad \lesssim \int_0^{s_0} \biggl( \int_0^{\frac{s}{2}} \bigl\| s \Delta e^{(s-s')(\Delta+\rho^2)} s' W_2(s') \bigr\|_{L^2} \, \frac{\ud s'}{s'} \biggr)^2 \, \frac{\ud s}{s} \\
   &\quad \quad + \int_0^{s_0} \biggl( \int_{\frac{s}{2}}^s \bigl\| s \Delta e^{(s-s')(\Delta+\rho^2)} s' W_2(s') \bigr\|_{L^2} \, \frac{\ud s'}{s'} \biggr)^2 \, \frac{\ud s}{s} \\
   &\quad =: III + IV.
  \end{aligned}
 \end{equation}
 For term $III$ we have by Corollary~\ref{c:pregdiv}, Schur's test, and Lemma~\ref{l:L2res} that this term is bounded by
 \begin{align*}
&\int_0^{s_0} \biggl( \int_0^{\frac{s}{2}} s (s')^{\frac{1}{2}} \bigl\|  e^{(s-s')(\Delta+\rho^2)} \Delta \nabla^\mu \bigl( (\nabla_\mu \bsb^\nu) (s')^{\frac{1}{2}} \nabla_\nu e^{s'(\Delta+\rho^2)} u \bigr) \bigr\|_{L^2} \, \frac{\ud s'}{s'} \biggr)^2 \, \ds \\
  &\lesssim \| \nabla \bsb \|_{L^\infty_{t,x}}^2 \int_0^{s_0} \biggl( \int_0^{\frac{s}{2}} \frac{s (s')^{\frac{1}{2}}}{(s-s')^{\frac{3}{2}}} \bigl\| (s')^{\frac{1}{2}} \nabla e^{s'(\Delta + \rho^2)} u \bigr\|_{L^2}  \, \frac{\ud s'}{s'} \biggr)^2 \, \ds \\
  &\lesssim_{\bsb, s_0} \int_0^{s_0} \bigl\| (s')^{\frac{1}{2}} \nabla e^{s'(\Delta + \rho^2)} u \bigr\|_{L^2}^2  \, \frac{\ud s'}{s'} \\
  &\lesssim_{\bsb, s_0} \|u\|_{L^2}^2,
 \end{align*}
 where we used that $\chi_{\{s' \leq \frac{s}{2}\}} \frac{s (s')^{\frac{1}{2}}}{(s-s')^{\frac{3}{2}}}$ is a Schur kernel. For term $IV$ we again first distribute all derivatives of $\Delta$ and $\nabla^\mu$ onto $\bsb$ and $u$, and then use parabolic regularity estimates and Schur's test to conclude that
 \begin{align*}
  IV &\lesssim \int_0^{s_0} \biggl( \int_{\frac{s}{2}}^s \bigl\| s e^{(s-s')(\Delta+\rho^2)} s' \Delta \nabla^\mu \bigl( (\nabla_\mu \bsb^\nu) \nabla_\nu e^{s'(\Delta+\rho^2)} u \bigr) \bigr\|_{L^2} \, \frac{\ud s'}{s'} \biggr)^2 \, \ds \\
  &\lesssim_{\bsb} \int_0^{s_0} \biggl( \int_{\frac{s}{2}}^s \sum_{k=1}^4 s (s')^{\frac{2-k}{2}} \bigl\| (s')^{\frac{k}{2}} \nabla^{(k)} e^{s'(\Delta+\rho^2)} u \bigr\|_{L^2} \, \frac{\ud s'}{s'} \biggr)^2 \, \ds \\
  &\lesssim_{\bsb, s_0} \|u\|_{L^2}^2.
 \end{align*}
The contribution of $W_3$ is similar to the previous terms but simpler. This finishes the proof of Lemma~\ref{lem:commutator_B_Ps}.
\end{proof}

We are now in a position to present the proof of Proposition~\ref{p:commutator_Hlot_Q_high}.

\begin{proof}[Proof of Proposition~\ref{p:commutator_Hlot_Q_high}]
 We recall that $H_\lot u = B u + V u$, where 
 \[
  B u = \frac{1}{i} \bigl( \bsb \cdot \nabla + \nabla \cdot \bsb \bigr) u.
 \]
 We only verify the estimate~\eqref{equ:commutator_Hlot_Q_high} for the first-order component $B$ of $H_\lot$ and leave the easier treatment of the potential part $V$ to the reader. For any $R \geq 1$ we denote by $\chi_{\{ r \leq R \}}$ a smooth cut-off function supported in the region $\{ r \geq R\}$ and equal to one on $\{ r \leq \frac{R}{2} \}$. Moreover, we set $\chi_{\{ r > R \}} := 1 - \chi_{\{r \leq R\}}$. Then we may expand
 \begin{align*}
  \Re \langle i \tilP_s B u, Q_s^{(\alpha)} \tilP_s u \rangle_{t,x} &= \Re \langle i \tilP_s B \bigl( \chi_{\{r > R\}} u\bigr), Q_s^{(\alpha)} \tilP_s \bigl( \chi_{\{r > R\}} u \bigr) \rangle_{t,x} \\
  &\quad + \Re \langle i \tilP_s B \bigl( \chi_{\{r > R\}} u\bigr), Q_s^{(\alpha)} \tilP_s \bigl( \chi_{\{r \leq R\}} u \bigr) \rangle_{t,x} \\
  &\quad + \Re \langle i \tilP_s B \bigl( \chi_{\{r \leq R\}} u\bigr), Q_s^{(\alpha)} \tilP_s \bigl( \chi_{\{r \leq R\}} u \bigr) \rangle_{t,x} \\
  &\quad + \Re \langle i \tilP_s B \bigl( \chi_{\{r \leq R\}} u\bigr), Q_s^{(\alpha)} \tilP_s \bigl( \chi_{\{r > R\}} u \bigr) \rangle_{t,x}.
 \end{align*}
 Our task therefore reduces to establishing for any $\varepsilon > 0$ that there exists a sufficiently large $R \equiv R(\varepsilon, \bsb, V) \geq 1$ such that the following four integrals 
 \begin{align}
  &\int_0^2 \sup_{\{\alpha_\ell\} \in \calA} \, \bigl| \Re \langle i \tilP_s B \bigl( \chi_{\{r > R\}} u\bigr), Q_s^{(\alpha)} \tilP_s \bigl( \chi_{\{r > R\}} u \bigr) \rangle_{t,x} \bigr| \, \ds,  \label{equ:commutator_BQ_far_far} \\
  &\int_0^2 \sup_{\{\alpha_\ell\} \in \calA} \, \bigl| \Re \langle i \tilP_s B \bigl( \chi_{\{r > R\}} u\bigr), Q_s^{(\alpha)} \tilP_s \bigl( \chi_{\{r \leq R\}} u \bigr) \rangle_{t,x} \bigr| \, \ds,  \label{equ:commutator_BQ_far_near} \\
  &\int_0^2 \sup_{\{\alpha_\ell\} \in \calA} \, \bigl| \Re \langle i \tilP_s B \bigl( \chi_{\{r \leq R\}} u\bigr), Q_s^{(\alpha)} \tilP_s \bigl( \chi_{\{r \leq R\}} u \bigr) \rangle_{t,x} \bigr| \, \ds,  \label{equ:commutator_BQ_near_near} \\
  &\int_0^2 \sup_{\{\alpha_\ell\} \in \calA} \, \bigl| \Re \langle i \tilP_s B \bigl( \chi_{\{r \leq R\}} u\bigr), Q_s^{(\alpha)} \tilP_s \bigl( \chi_{\{r > R\}} u \bigr) \rangle_{t,x} \bigr| \, \ds \label{equ:commutator_BQ_near_far} 
 \end{align}
 are each bounded by 
 \begin{equation}
  \varepsilon \|u\|_{LE}^2 + C_\varepsilon \|u\|_{L^2(\bbR \times \{r \leq R\})}^2.
 \end{equation}

 Let us begin with~\eqref{equ:commutator_BQ_far_far}. By Cauchy-Schwarz, Corollary~\ref{cor:QLEs}, Lemma~\ref{lem:PsLEs1}, Lemma~\ref{lem:LE_comparison}, and a simple change of variables we obtain that
 \begin{align*}
  &\int_0^2 \sup_{\{\alpha_\ell\} \in \calA} \, \bigl| \Re \langle i \tilP_s B \bigl( \chi_{\{r > R\}} u\bigr), Q_s^{(\alpha)} \tilP_s \bigl( \chi_{\{r > R\}} u \bigr) \rangle_{t,x} \bigr| \, \ds \\
  &\leq \int_0^2 \sup_{\{\alpha_\ell\} \in \calA} \, \bigl\| \tilP_s B \bigl( \chi_{\{r > R\}} u\bigr) \bigr\|_{LE_s^\ast} \bigl\| Q_s^{(\alpha)} \tilP_s \bigl( \chi_{\{r > R\}} u \bigr) \bigr\|_{LE_s} \, \ds \\
  &\lesssim \int_0^2 \bigl\| \tilP_{\frac{s}{4}} B \bigl( \chi_{\{r > R\}} u\bigr) \bigr\|_{LE_{\frac{s}{4}}^\ast} \bigl\| \tilP_{\frac{s}{4}} \bigl( \chi_{\{r > R\}} u \bigr) \bigr\|_{LE_{\frac{s}{4}}} \, \ds \\ 
  &\lesssim \biggl( \int_0^{\frac{1}{2}} s^{\frac{1}{2}} \bigl\| \tilP_s B  \bigl( \chi_{\{r > R\}} u\bigr) \bigr\|_{LE_s^\ast}^2 \, \ds \biggr)^{\frac{1}{2}} \biggl( \int_0^{\frac{1}{2}} s^{-\frac{1}{2}} \bigl\| \tilP_s  \bigl( \chi_{\{r > R\}} u \bigr) \bigr\|_{LE_s}^2 \, \ds \biggr)^{\frac{1}{2}} \\
  &\lesssim \bigl\| B \bigl( \chi_{\{r > R\}} u \bigr) \bigr\|_{LE^\ast} \| \chi_{\{r > R\}} u \|_{LE}.
 \end{align*}
 Now we note that since $B$ is a local operator and since $\chi_{\{ r > R \}} u$ is supported in $\{ r \geq \frac{R}{2} \}$, we may replace $B$ by 
 \[
  \frac{1}{i} \bigl( \bsb_{> \frac{R}{2}} \cdot \nabla + \nabla \cdot \bsb_{>\frac{R}{2}} \bigr), \quad \bsb_{> \frac{R}{2}} := \chi_{\{r > \frac{R}{2}\}} \bsb.
 \]
 Hence, by Proposition~\ref{p:Hlotbound} we have that 
 \[
  \bigl\| B \bigl( \chi_{\{r > R\}} u \bigr) \bigr\|_{LE^\ast} \lesssim C \bigl( \bsb_{> \frac{R}{2}} \bigr) \| \chi_{\{r > R\}} u \|_{LE},
 \]
 where the constant $C \bigl( \bsb_{> \frac{R}{2}} \bigr)$ depends on weighted $L^\infty_{t,x}$-norms of the vector field $\bsb$. In particular, due to our decay assumptions about $\bsb$, this constant can be made arbitrarily small by choosing $R \geq 1$ sufficiently large. Given any $\varepsilon > 0$ we may therefore choose $R \equiv R(\varepsilon, \bsb, V) \geq 1$ sufficiently large and invoke Lemma~\ref{lem:boundedness_LE_spatial_cutoff} to conclude that
 \begin{align}
  \eqref{equ:commutator_BQ_far_far} \lesssim \varepsilon \| \chi_{\{r > R\}} u \|_{LE}^2 \lesssim \varepsilon \|u\|_{LE}^2.
 \end{align}
 The estimate of~\eqref{equ:commutator_BQ_far_near} is analogous to the one of~\eqref{equ:commutator_BQ_far_far}.
 
 Now for~\eqref{equ:commutator_BQ_near_near} we exploit the self-adjointness of $B$, $Q_s^{(\alpha)}$, and $\tilP_s$ to obtain the following commutator structure 
 \begin{align*}
  &\Re \langle i \tilP_s B \bigl( \chi_{\{r \leq R\}} u\bigr), Q_s^{(\alpha)} \tilP_s \bigl( \chi_{\{r \leq R\}} u \bigr) \rangle_{t,x} \\
  &\qquad \qquad = -\frac{1}{2} \langle i [B, \tilP_s Q_s^{(\alpha)} \tilP_s] \bigl( \chi_{\{r \leq R\}} u\bigr), \bigl( \chi_{\{r \leq R\}} u \bigr) \rangle_{t,x}.
 \end{align*}
 We then expand into
 \begin{equation} \label{equ:commutator_BPQP}
  [B, \tilP_s Q_s^{(\alpha)} \tilP_s] = [B, \tilP_s] Q_s^{(\alpha)} \tilP_s + \tilP_s [B, Q_s^{(\alpha)}] \tilP_s + \tilP_s Q_s^{(\alpha)} [B, \tilP_s].
 \end{equation}
 The contributions of the first and third terms on the right-hand side of~\eqref{equ:commutator_BPQP} can be estimated in the same manner using Lemma~\ref{lem:commutator_B_Ps}. We only provide the details for the first term. By Lemma~\ref{lem:commutator_B_Ps} and Corollary~\ref{cor:QLEs}, we obtain
 \begin{align*}
  &\int_0^2 \sup_{\{ \alpha_\ell \} \in \calA} \, \bigl| \langle [B, \tilP_s] Q_s^{(\alpha)} \tilP_s \bigl( \chi_{\{r \leq R\}} u\bigr), \chi_{\{r \leq R\}} u \rangle_{t,x} \bigr| \, \ds \\
  &\lesssim \int_0^2 \sup_{\{ \alpha_\ell \} \in \calA} \, \bigl\| Q_s^{(\alpha)} \tilP_s \bigl( \chi_{\{r \leq R\}} u \bigr) \bigr\|_{L^2_{t,x}} \bigl\| [B, \tilP_s] \bigl( \chi_{\{r \leq R\}} u \bigr) \bigr\|_{L^2_{t,x}} \, \ds \\
  &\lesssim \biggl( \int_0^2 \bigl\| \tilP_{\frac{s}{2}} \bigl( \chi_{\{ r \leq R\}} u \bigr) \bigr\|_{L^2_{t,x}}^2 \, \ds \biggr)^{\frac{1}{2}} \biggl( \int_0^2 \bigl\| [B, \tilP_s] \bigl( \chi_{\{ r \leq R \}} u \bigr) \bigr\|_{L^2_{t,x}}^2 \, \ds \biggr)^{\frac{1}{2}} \\
  &\lesssim \| u \|_{L^2(\bbR \times \{ r \leq R\})}^2.
 \end{align*}
 We are thus left to estimate the contribution of the second and main term on the right-hand side of~\eqref{equ:commutator_BPQP}. Given any $\varepsilon > 0$ we obtain by Cauchy-Schwarz that 
 \begin{align*}
  &\int_0^2 \sup_{\{ \alpha_\ell \} \in \calA} \, \bigl| \langle \tilP_s [B, Q_s^{(\alpha)}] \tilP_s \bigl( \chi_{\{r \leq R\}} u\bigr), \chi_{\{r \leq R\}} u \rangle_{t,x} \bigr| \, \ds \\
  &\lesssim \biggl( \int_0^2 \bigl\| \tilP_s \bigl( \chi_{\{ r \leq R\}} u \bigr) \bigr\|_{L^2_{t,x}}^2 \, \ds \biggr)^{\frac{1}{2}} \biggl( \int_0^2 \sup_{\{\alpha_\ell\} \in \calA} \, \bigl\| [B, Q_s^{(\alpha)}] \tilP_s \bigl( \chi_{\{ r \leq R \}} u \bigr) \bigr\|_{L^2_{t,x}}^2 \, \ds \biggr)^{\frac{1}{2}} \\
  &\leq C_\varepsilon \|u\|_{L^2(\bbR \times \{r \leq R\})}^2 + \varepsilon \int_0^2 \sup_{\{\alpha_\ell\} \in \calA} \, \bigl\| [B, Q_s^{(\alpha)}] \tilP_s \bigl( \chi_{\{ r \leq R \}} u \bigr) \bigr\|_{L^2_{t,x}}^2 \, \ds.
 \end{align*}
 By direct computation we find that the commutator $[B, Q_s^{(\alpha)}]$ is given by
 \begin{equation}
  \begin{split}
   [B, Q_s^{(\alpha)}] &= - 4 \bsb^\mu (\nabla_\mu \bsbeta_s^{(\alpha),\nu} ) \nabla_\nu - 2 \bsb^\mu (\nabla_\mu \nabla_\nu \bsbeta_s^{(\alpha), \nu}) \\
   &\quad \, + 4 \bsbeta_s^{(\alpha), \nu} (\nabla_\nu \bsb^\mu) \nabla_\mu + 2 \bsbeta_s^{(\alpha), \nu} (\nabla_\nu \nabla_\mu \bsb^\mu).
  \end{split}
 \end{equation}
 In order to complete the estimate of~\eqref{equ:commutator_BQ_near_near} it thus suffices to prove that each of the following integrals
 \begin{align}
  &\int_0^2 \sup_{\{\alpha_\ell\} \in \calA} \, \bigl\| \bsb^\mu (\nabla_\mu \bsbeta_s^{(\alpha),\nu} ) \nabla_\nu \tilP_s \bigl( \chi_{\{r \leq R\}} u \bigr) \bigr\|_{L^2_{t,x}}^2 \, \ds \label{equ:commutator_BQ_term1_int} \\
  &\int_0^2 \sup_{\{\alpha_\ell\} \in \calA} \, \bigl\| \bsb^\mu (\nabla_\mu \nabla_\nu \bsbeta_s^{(\alpha), \nu}) \tilP_s \bigl( \chi_{\{r \leq R\}} u \bigr) \bigr\|_{L^2_{t,x}}^2 \, \ds \label{equ:commutator_BQ_term2_int} \\
  &\int_0^2 \sup_{\{\alpha_\ell\} \in \calA} \, \bigl\| \bsbeta_s^{(\alpha), \nu} (\nabla_\nu \bsb^\mu) \nabla_\mu \tilP_s \bigl( \chi_{\{r \leq R\}} u \bigr) \bigr\|_{L^2_{t,x}}^2 \, \ds \label{equ:commutator_BQ_term3_int} \\
  &\int_0^2 \sup_{\{\alpha_\ell\} \in \calA} \, \bigl\| \bsbeta_s^{(\alpha), \nu} (\nabla_\nu \nabla_\mu \bsb^\mu) \tilP_s \bigl( \chi_{\{r \leq R\}} u \bigr) \bigr\|_{L^2_{t,x}}^2 \, \ds \label{equ:commutator_BQ_term4_int}
 \end{align}
 is bounded by 
 \[
  \|u\|_{L^2(\bbR\times\{r\leq R\})}^2 + \|u\|_{LE}^2.
 \]
 In view of the fact that $|\beta_s^{(\alpha)}(r)| \lesssim s^{\frac{1}{2}}$ uniformly for all $\{ \alpha_\ell \} \in \calA$,  it is obvious by Lemma~\ref{l:preg} that \eqref{equ:commutator_BQ_term3_int} and \eqref{equ:commutator_BQ_term4_int} are bounded by $\|u\|_{L^2(\bbR\times\{r\leq R\})}^2$. Correspondingly, only the treatment of the first two integrals \eqref{equ:commutator_BQ_term1_int}--\eqref{equ:commutator_BQ_term2_int} requires more explanations. We begin with the estimate of the integral~\eqref{equ:commutator_BQ_term1_int}, which we split into
 \begin{equation} \label{equ:commutator_BQ_term1_int_splitting}
  \begin{aligned}
   &\int_0^2 \sup_{\{\alpha_\ell\} \in \calA} \, \bigl\| \bsb^\mu (\nabla_\mu \bsbeta_s^{(\alpha),\nu} ) \nabla_\nu \tilP_s \bigl( \chi_{\{r \leq R\}} u \bigr) \bigr\|_{L^2_{t,x}}^2 \, \ds \\
   &\lesssim \int_0^2 \sup_{\{\alpha_\ell\} \in \calA} \, \bigl\| \bsb^\mu (\nabla_\mu \bsbeta_s^{(\alpha),\nu} ) \nabla_\nu \tilP_s \bigl( \chi_{\{r \leq R\}} u \bigr) \bigr\|_{L^2(\bbR \times A_{\leq -k_s})}^2 \, \ds \\
   &\quad + \int_0^2 \sup_{\{\alpha_\ell\} \in \calA} \, \biggl( \sum_{\ell \geq -k_s} \bigl\| \phi_\ell \bsb^\mu (\nabla_\mu \bsbeta_s^{(\alpha),\nu} ) \nabla_\nu \tilP_s \bigl( \chi_{\{r \leq R\}} u \bigr) \bigr\|_{L^2_{t,x}} \biggr)^2 \, \ds.
  \end{aligned}
 \end{equation}
 Note that due to the radial symmetry of the vector field $\bsbeta_s^{(\alpha)}$, the only non-zero components of $\nabla_\mu \bsbeta_s^{(\alpha), \nu}$ are given by
 \begin{align*}
  \nabla_r \bsbeta_s^{(\alpha), r} = \partial_r \beta_s^{(\alpha)} \quad \text{and} \quad \nabla_{\theta_a} \bsbeta_s^{(\alpha), \theta_a} = \coth(r) \beta_s^{(\alpha)} \text{ for } a = 1, \ldots, d-1.
 \end{align*}
 We will use the following pointwise bounds for the non-zero components of $\nabla_\mu \bsbeta_s^{(\alpha), \nu}$, which hold uniformly for all $\{ \alpha_\ell \} \in \calA$,
 \begin{equation} \label{equ:commutator_BQ_term1_pointwise_bound1}
  \partial_r \beta_s^{(\alpha)}(r) = \delta \frac{\alpha(\delta s^{-\frac{1}{2}} r)}{\langle \delta s^{-\frac{1}{2}} r \rangle} \lesssim \frac{s^{\frac{1}{2}}}{(s+r^2)^{\frac{1}{2}}} \lesssim \begin{cases} 1 & \text{for } r \leq s^{\frac{1}{2}}, \\ s^{\frac{1}{2}} r^{-1} & \text{for } r \geq s^{\frac{1}{2}}, \end{cases} 
 \end{equation}
 and 
 \begin{equation} \label{equ:commutator_BQ_term1_pointwise_bound2}
  \coth(r) \beta_s^{(\alpha)}(r) \lesssim \begin{cases} 1 & \text{for } r \leq s^{\frac{1}{2}}, \\ s^{\frac{1}{2}} r^{-1} & \text{for } s^{\frac{1}{2}} \leq r \lesssim 1, \\ s^{\frac{1}{2}} & \text{for } r \gtrsim 1. \end{cases}
 \end{equation}
 Thus, using the localized parabolic regularity estimates from Proposition~\ref{p:lpregcore}, we obtain for the first term on the right-hand side of~\eqref{equ:commutator_BQ_term1_int_splitting} that 
 \begin{align*}
  &\int_0^2 \sup_{\{\alpha_\ell\} \in \calA} \, \bigl\| \bsb^\mu (\nabla_\mu \bsbeta_s^{(\alpha),\nu} ) \nabla_\nu \tilP_s \bigl( \chi_{\{r \leq R\}} u \bigr) \bigr\|_{L^2(\bbR \times A_{\leq -k_s})}^2 \, \ds \\
  &\lesssim_{\bsb} \int_0^2 \bigl\| \nabla \tilP_s \bigl( \chi_{\{ r \leq R\}} u \bigr) \bigr\|_{L^2(\bbR \times A_{<-k_s})}^2 \, \ds \\
  &\lesssim \int_0^2 s^{-\frac{1}{2}} \biggl( s^{-\frac{1}{4}} \bigl\| \chi_{\{ r \leq s^{\frac{1}{2}}\}} s^{\frac{1}{2}} \nabla e^{\frac{3}{4} s (\Delta + \rho^2)} \chi_{\{ r \leq 2^{10} s^{\frac{1}{2}}\}} \tilP_{\frac{s}{4}} \bigl( \chi_{\{ r \leq R\}} u \bigr) \bigr\|_{L^2_{t,x}} \biggr)^2 \, \ds \\
  &\quad + \int_0^2 s^{-\frac{1}{2}} \biggl( \sum_{m \geq -k_s+10} s^{-\frac{1}{4}} \bigl\| \chi_{\{ r \leq s^{\frac{1}{4}}\}} s^{\frac{1}{2}} \nabla e^{\frac{3}{4} s(\Delta + \rho^2)} \phi_m \tilP_{\frac{s}{4}} \bigl( \chi_{\{ r \leq R\}} u \bigr) \bigr\|_{L^2_{t,x}} \biggr)^2 \, \ds \\
  &\lesssim \int_0^2 s^{-\frac{1}{2}} \biggl( s^{-\frac{1}{4}} \bigl\| \chi_{\{ r \leq 2^{10}s^{\frac{1}{2}}\}} \tilP_{\frac{s}{4}} \bigl( \chi_{\{r \leq R\}} u \bigr) \bigr\|_{L^2_{t,x}} \biggr)^2 \, \ds \\
  &\quad + \int_0^2 s^{-\frac{1}{2}} \biggl( \sum_{m \geq -k_s+10} s^{\frac{1}{4}} 2^{-m} \bigl\| \phi_m \tilP_{\frac{s}{4}} \bigl( \chi_{\{r \leq R\}} u \bigr) \bigr\|_{L^2_{t,x}} \biggr)^2 \, \ds \\
  &\lesssim \int_0^2 s^{-\frac{1}{2}} \bigl\| \tilP_{\frac{s}{4}} \bigl( \chi_{\{ r \leq R\}} u \bigr) \bigr\|_{LE_{\frac{s}{4}}}^2 \, \ds \lesssim \|u\|_{LE}^2.
 \end{align*}
 For the second term on the right-hand side of~\eqref{equ:commutator_BQ_term1_int_splitting} we use the decay of $\bsb$ and the pointwise bound $s^{\frac{1}{2}} r^{-1}$ from \eqref{equ:commutator_BQ_term1_pointwise_bound1}--\eqref{equ:commutator_BQ_term1_pointwise_bound2} for the region $\{ r \geq s^{\frac{1}{2}} \}$ to obtain \begin{equation}
  \begin{aligned}
   &\int_0^2 \sup_{\{\alpha_\ell\} \in \calA} \, \biggl( \sum_{\ell \geq -k_s} \bigl\| \phi_\ell \bsb^\mu (\nabla_\mu \bsbeta_s^{(\alpha),\nu} ) \nabla_\nu \tilP_s \bigl( \chi_{\{r \leq R\}} u \bigr) \bigr\|_{L^2_{t,x}} \biggr)^2 \, \ds \\
   &\lesssim_{\bsb} \int_0^2 \biggl( \sum_{\ell \geq -k_s} 2^{-\ell} \bigl\| \phi_\ell s^{\frac{1}{2}} \nabla e^{\frac{3}{4} s(\Delta + \rho^2)} \phi_{\leq -k_s} \tilP_{\frac{s}{4}} \bigl( \chi_{\{ r \leq R\}} u \bigr) \bigr\|_{L^2_{t,x}} \biggr)^2 \, \ds \\
   &\quad + \int_0^2 \biggl( \sum_{\ell \geq -k_s} \sum_{-k_s \leq m \leq \ell+10} 2^{-\ell} \bigl\| \phi_\ell s^{\frac{1}{2}} \nabla e^{\frac{3}{4} s(\Delta + \rho^2)} \phi_m \tilP_{\frac{s}{4}} \bigl( \chi_{\{ r \leq R\}} u \bigr) \bigr\|_{L^2_{t,x}} \biggr)^2 \, \ds \\
   &\quad + \int_0^2 \biggl( \sum_{\ell \geq -k_s} \sum_{m > \ell+10} 2^{-\ell} \bigl\| \phi_\ell s^{\frac{1}{2}} \nabla e^{\frac{3}{4} s(\Delta + \rho^2)} \phi_m \tilP_{\frac{s}{4}} \bigl( \chi_{\{ r \leq R\}} u \bigr) \bigr\|_{L^2_{t,x}} \biggr)^2 \, \ds.
  \end{aligned}
 \end{equation}
 Then invoking the parabolic regularity estimates from Lemma~\ref{l:preg} for the first two terms on the right-hand side and the localized parabolic regularity estimates from Proposition~\ref{p:lpregcore} for the third term, we obtain easily that each of the three integrals on the right-hand side above is bounded by $\| \chi_{\{r \leq R\}} u \|_{LE}^2 \lesssim \|u\|_{LE}^2$, as desired. 
  
 Finally, we turn to the estimate of the integral~\eqref{equ:commutator_BQ_term2_int}. We first note that due to the radial symmetry of the vector field $\bsbeta_s^{(\alpha)}$, the only non-zero component of $\nabla_\mu \nabla_\nu \bsbeta_s^{(\alpha), \nu}$ is
 \begin{align*}
  \nabla_r \nabla_\nu \bsbeta_s^{(\alpha), \nu} = \partial_r^2 \beta_s^{(\alpha)} + 2 \rho \bigl( \coth(r) \partial_r \beta_s^{(\alpha)} - \sinh^{-2}(r) \beta_s^{(\alpha)} \bigr).
 \end{align*}
 Using the following pointwise bounds, which hold uniformly for all $\{ \alpha_\ell \} \in \calA$,
 \begin{equation}
  | \partial_r^2 \beta_s^{(\alpha)}(r) | \lesssim \delta^2 s^{-\frac{1}{2}} \frac{\alpha(\delta s^{-\frac{1}{2}} r)}{\langle \delta s^{-\frac{1}{2}} r \rangle^2} \lesssim \frac{s^{-\frac{1}{2}}}{1 + s^{-1} r^2} \lesssim \begin{cases} s^{-\frac{1}{2}} & \text{for } r \leq s^{\frac{1}{2}}, \\ s^{\frac{1}{2}} r^{-2} & \text{for } r \geq s^{\frac{1}{2}},  \end{cases} 
 \end{equation}
 and 
 \begin{equation}
  \bigl| \coth(r) (\partial_r \beta_s^{(\alpha)})(r) - \sinh^{-2}(r) \beta_s^{(\alpha)}(r) \bigr| \lesssim \begin{cases} s^{-\frac{1}{2}} & \text{for } r \leq s^{\frac{1}{2}}, \\ s^{\frac{1}{2}} r^{-2} & \text{for } s^{\frac{1}{2}} \leq r \lesssim 1, \\ 1 & \text{for } r \gtrsim 1, \end{cases}
 \end{equation}
 we may conclude in a similar manner as before that the integral~\eqref{equ:commutator_BQ_term2_int} is bounded by $\| \chi_{\{r \leq R\}} u \|_{L^2_{t,x}}^2 + \| u \|_{LE}^2$, as desired.
 
 In order to finish the proof of Proposition~\ref{p:commutator_Hlot_Q_high} it now only remains to estimate~\eqref{equ:commutator_BQ_near_far}. To this end we decompose~\eqref{equ:commutator_BQ_near_far} into
 \begin{align*}
  &\int_0^2 \sup_{\{\alpha_\ell\} \in \calA} \, \bigl| \Re \langle i \tilP_s B \bigl( \chi_{\{r \leq R\}} u\bigr), Q_s^{(\alpha)} \tilP_s \bigl( \chi_{\{r > R\}} u \bigr) \rangle_{t,x} \bigr| \, \ds \\
  &\lesssim \int_0^2 \sup_{\{\alpha_\ell\} \in \calA} \, \bigl| \Re \langle i \tilP_s \bigl( \chi_{\{ r > 2^{-20} R\}} B \chi_{\{r \leq R\}} u\bigr), Q_s^{(\alpha)} \tilP_s \bigl( \chi_{\{r > R\}} u \bigr) \rangle_{t,x} \bigr| \, \ds \\
  &\quad + \int_0^2 \sup_{\{\alpha_\ell\} \in \calA} \, \bigl| \Re \langle i \tilP_s \bigl( \chi_{\{ r \leq 2^{-20} R\}} B \chi_{\{r \leq R\}} u\bigr), Q_s^{(\alpha)} \tilP_s \bigl( \chi_{\{r > R\}} u \bigr) \rangle_{t,x} \bigr| \, \ds.
 \end{align*}
 The coefficients of the differential operator $B$ in the spatial region $\{ r \geq 2^{-20} R \}$ can be made arbitrarily small by choosing $R \gg 1$ sufficiently large, correspondingly the first term on the right-hand side can be estimated analogously to the term~\eqref{equ:commutator_BQ_far_far}. We further split the second term on the right-hand side into
 \begin{align*}
  &\int_0^2 \sup_{\{\alpha_\ell\} \in \calA} \, \bigl| \Re \langle i \tilP_s \bigl( \chi_{\{ r \leq 2^{-20} R\}} B \chi_{\{r \leq R\}} u\bigr), Q_s^{(\alpha)} \tilP_s \bigl( \chi_{\{r > R\}} u \bigr) \rangle_{t,x} \bigr| \, \ds \\
  &\lesssim \int_0^2 \sup_{\{\alpha_\ell\} \in \calA} \, \bigl| \Re \langle i \chi_{\{ r > 2^{-10} R\}} \tilP_s \bigl( \chi_{\{ r \leq 2^{-20} R\}} B \chi_{\{r \leq R\}} u\bigr), Q_s^{(\alpha)} \tilP_s \bigl( \chi_{\{r > R\}} u \bigr) \rangle_{t,x} \bigr| \, \ds \\
  &\quad + \int_0^2 \sup_{\{\alpha_\ell\} \in \calA} \, \bigl| \Re \langle i \tilP_s \bigl( \chi_{\{ r \leq 2^{-20} R\}} B \chi_{\{r \leq R\}} u\bigr), \chi_{\{ r \leq 2^{-10} R\}} Q_s^{(\alpha)} \tilP_s \bigl( \chi_{\{r > R\}} u \bigr) \rangle_{t,x} \bigr| \, \ds \\ 
  &=: I + II.
 \end{align*}
 Now by Cauchy-Schwarz, Corollary~\ref{cor:QLEs}, Lemma~\ref{lem:PsLEs1}, and Lemma~\ref{lem:boundedness_LE_spatial_cutoff}, we have for the term $I$ that
 \begin{align*}
  I &\lesssim \biggl( \int_0^2 s^{\frac{1}{2}} \bigl\| \chi_{\{ r > 2^{-10} R\}} \tilP_s \bigl( \chi_{\{ r \leq 2^{-20} R\}} B \chi_{\{r \leq R\}} u \bigr) \bigr\|_{LE_s^*}^2 \, \ds \biggr)^{\frac{1}{2}} \times \\
  &\qquad \qquad \qquad \qquad \qquad \times \biggl( \int_0^2 s^{-\frac{1}{2}} \sup_{\{ \alpha_\ell \in \calA \}} \, \bigl\| Q_s^{(\alpha)} \tilP_s \bigl( \chi_{\{ r > R \}} u \bigr) \bigr\|_{LE_s}^2 \, \ds \biggr)^{\frac{1}{2}} \\
  &\lesssim \biggl( \int_0^2 s^{\frac{1}{2}} \bigl\| \chi_{\{ r > 2^{-10} R\}} \tilP_s \bigl( \chi_{\{ r \leq 2^{-20} R\}} B \chi_{\{r \leq R\}} u \bigr) \bigr\|_{LE_s^*}^2 \, \ds \biggr)^{\frac{1}{2}} \|u\|_{LE}.
 \end{align*}
 Here we have to establish the following mismatch estimate: given any $\varepsilon > 0$, choosing $R \gg 1$ sufficiently large it holds that
 \begin{equation} \label{equ:commutator_BQ_near_far_mismatch1}
  \biggl( \int_0^2 s^{\frac{1}{2}} \bigl\| \chi_{\{ r > 2^{-10} R\}} \tilP_s \bigl( \chi_{\{ r \leq 2^{-20} R\}} F \bigr) \bigr\|_{LE_s^*}^2 \, \ds \biggr)^{\frac{1}{2}} \leq \varepsilon \|F\|_{LE^\ast}.
 \end{equation}
 Then with $F := B \chi_{\{ r \leq R \}} u$ and using that $\bigl\| B \chi_{\{ r \leq R \}} u \bigr\|_{LE^\ast} \lesssim_{\bsb} \|u\|_{LE}$ by Proposition~\ref{p:Hlotbound}, given any $\varepsilon > 0$ we may choose $R \gg 1$ sufficiently large to conclude that
 \[
  I \lesssim \varepsilon \|u\|_{LE}^2.
 \]
 For the term II we obtain by Cauchy-Schwarz, Lemma~\ref{lem:PsLEs1}, and Lemma~\ref{lem:boundedness_LE_spatial_cutoff} that
 \begin{align*}
  II &\lesssim \biggl( \int_0^2 s^{\frac{1}{2}} \bigl\| \tilP_s \bigl( \chi_{\{ r \leq 2^{-20} R\}} B \chi_{\{r \leq R\}} u \bigr) \bigr\|_{LE_s^\ast}^2 \, \ds \biggr)^{\frac{1}{2}} \times \\
  &\qquad \qquad \qquad  \times \biggl( \int_0^2 s^{-\frac{1}{2}} \sup_{\{\alpha_\ell\} \in \calA} \, \bigl\| \chi_{\{ r \leq 2^{-10} R\}} Q_s^{(\alpha)} \tilP_s \bigl( \chi_{\{ r > R\}} u \bigr) \bigr\|_{LE_s}^2 \, ds \biggr)^{\frac{1}{2}} \\
  &\lesssim_{\bsb} \|u\|_{LE} \biggl( \int_0^2 s^{-\frac{1}{2}} \sup_{\{\alpha_\ell\} \in \calA} \, \bigl\| \chi_{\{ r \leq 2^{-10} R\}} Q_s^{(\alpha)} \tilP_s \bigl( \chi_{\{ r > R\}} u \bigr) \bigr\|_{LE_s}^2 \, ds \biggr)^{\frac{1}{2}}.
 \end{align*}
 Then upon establishing the mismatch estimate 
 \begin{equation} \label{equ:commutator_BQ_near_far_mismatch2}
  \biggl( \int_0^2 s^{-\frac{1}{2}} \sup_{\{\alpha_\ell\} \in \calA} \, \bigl\| \chi_{\{ r \leq 2^{-10} R\}} Q_s^{(\alpha)} \tilP_s \bigl( \chi_{\{ r > R\}} u \bigr) \bigr\|_{LE_s}^2 \, ds \biggr)^{\frac{1}{2}} \leq \varepsilon \|u\|_{LE}
 \end{equation}
 for any given $\varepsilon > 0$ with $R \equiv R(\varepsilon) \gg 1$ chosen sufficiently large, we may also conclude the desired bound 
 \[
  II \lesssim \varepsilon \|u\|_{LE}^2.
 \]
 We are hence left to prove the two mismatch estimates \eqref{equ:commutator_BQ_near_far_mismatch1}--\eqref{equ:commutator_BQ_near_far_mismatch2}. Their proofs are similar and we begin with~\eqref{equ:commutator_BQ_near_far_mismatch1}. Here we have to show that given any $\varepsilon > 0$, we may choose $R \gg 1$ sufficiently large such that
 \begin{equation} \label{equ:commutator_BQ_near_far_mismatch1_todo}
  \int_0^2 s^{\frac{1}{2}} \biggl( \sum_{\ell \geq \log_2(R)-10} 2^{\frac{1}{2} \ell} \bigl\| \phi_\ell \tilP_s \bigl( \chi_{\{r \leq 2^{-20} R\}} F \bigr) \bigr\|_{L^2_{t,x}} \biggr)^2 \, \ds \leq \varepsilon^2 \|F\|_{LE^\ast}^2.
 \end{equation}
 In order to deal with a certain singularity arising at zero heat time in this estimate, we use the projections $\calP_s := s^2 (\Delta + \rho^2)^2 e^{s(\Delta + \rho^2)}$ to resolve $F$ for small heat times. More precisely, for a given heat time $0 < s \leq 2$, we decompose $F$ for any $2 \leq s_0 \leq 4$ into 
 \begin{equation} \label{equ:commutator_BQ_near_far_mismatch1_freq_decomp}
  F = \int_0^s \calP_{s'} F \, \frac{\ud s'}{s'} + \int_s^{s_0} \calP_{s'} F \, \frac{\ud s'}{s'} + \tilP_{s_0} F + \tilP_{\geq s_0} F.
 \end{equation}
 Then the contribution to~\eqref{equ:commutator_BQ_near_far_mismatch1_todo} of the first high-frequency term on the right-hand side is given by
 \begin{equation} \label{equ:commutator_BQ_near_far_mismatch1_high_freq1}
  \int_0^2 \biggl( \int_0^s \Bigl( \frac{s}{s'} \Bigr)^{\frac{1}{4}} (s')^{\frac{1}{4}} \sum_{\ell \geq \log_2(R)-10} 2^{\frac{1}{2} \ell} \bigl\| \phi_\ell \tilP_s \bigl( \chi_{\{r \leq 2^{-20} R\}} \calP_{s'} F \bigr) \bigr\|_{L^2_{t,x}} \, \frac{\ud s'}{s'} \biggr)^2 \, \ds. 
 \end{equation}
 Here we observe that the kernel $( \frac{s}{s'} )^{\frac{1}{4}}  \chi_{\{ s' \leq s\}}$ is not a Schur kernel due to a singularity at $s'=0$, and it is because of this singularity that we use the modified resolution~\eqref{equ:commutator_BQ_near_far_mismatch1_freq_decomp} of $F$. It allows us to now write 
 \begin{align*}
  \tilP_s \bigl( \chi_{\{r \leq 2^{-20} R\}} \calP_{s'} F \bigr) &= - \tilP_s \bigl( \chi_{\{ r \leq 2^{-20} R\}} s' (\Delta + \rho^2) \tilP_{s'} F \bigr) \\
  &= - s' (\Delta + \rho^2) \tilP_s \bigl( \chi_{\{ r \leq 2^{-20} R\}} \tilP_{s'} F \bigr) \\
  &\quad \, - s' \tilP_s \bigl( \Delta \chi_{\{ r \leq 2^{-20} R\}} \tilP_{s'} F \bigr) \\
  &\quad \, + 2 s' \tilP_s \nabla_\mu \bigl( \nabla^\mu \chi_{\{ r \leq 2^{-20} R\}} \tilP_{s'} F \bigr). 
 \end{align*}
 These three terms can be handled in a similar manner, noting that the supports of $\Delta \chi_{\{ r \leq 2^{-20} R\}}$ and $\nabla \chi_{\{ r \leq 2^{-20} R\}}$ are contained in $\{ r \leq 2^{-20} R\}$ and they can be thought of as cut-offs to the same region. We therefore concentrate on the contribution of the first term, that is, 
 \begin{equation}
  \int_0^2 \biggl( \int_0^s \Bigl( \frac{s'}{s} \Bigr)^{\frac{3}{4}} (s')^{\frac{1}{4}} \sum_{\ell \geq \log_2(R)-10} 2^{\frac{1}{2} \ell} \bigl\| \phi_\ell s (\Delta + \rho^2) \tilP_s \bigl( \chi_{\{ r \leq 2^{-20} R\}} \tilP_{s'} F \bigr) \bigr\|_{L^2_{t,x}} \, \frac{\ud s'}{s'} \biggr)^2 \, \ds. 
 \end{equation}
 Using the localized parabolic regularity estimates from Proposition~\ref{p:lpregcore}, the sum inside the integrals is bounded by 
 \begin{align*}
  &\sum_{\ell \geq \log_s(R)-10} 2^{-(N-\frac{1}{2})\ell} s^{\frac{N}{2}} \biggl( \bigl\| \phi_{\leq -k_{s'}} \tilP_{s'} F \bigr\|_{L^2_{t,x}} + \sum_{-k_{s'} \leq m \leq \log_2(R)-10} \bigl\| \phi_m \tilP_{s'} F \bigr\|_{L^2_{t,x}} \biggr) \\
  &\lesssim_N R^{-(N-\frac{1}{2})} s^{\frac{N}{2}-\frac{1}{4}} \Bigl( \frac{s}{s'} \Bigr)^{\frac{1}{4}} \| \tilP_{s'} F \|_{LE_{s'}^\ast}
 \end{align*}
 for any integer $N \geq 1$. Plugging back into the estimate above, we see that the contribution is bounded by 
 \begin{align*}
  R^{-2(N-\frac{1}{2})} \int_0^2 \biggl( \int_0^s \Bigl( \frac{s'}{s} \Bigr)^{\frac{1}{2}} (s')^{\frac{1}{4}} \| \tilP_{s'} F \|_{LE_{s'}^\ast} \, \frac{\ud s'}{s'} \biggr)^2 \, \ds \leq \varepsilon^2 \|F\|_{LE^\ast}^2,
 \end{align*}
 where in the last step we used the fact that $\chi_{\{ s' \leq s\}} (\frac{s'}{s})^{\frac{1}{2}}$ is a Schur kernel, and that $R^{-2(N-\frac{1}{2})}$ can be made arbitrarily small by choosing $R \gg 1$ sufficiently large. 
 
 Next, we consider the contribution to~\eqref{equ:commutator_BQ_near_far_mismatch1_todo} of the second high-frequency term $\int_s^{s_0} \calP_{s'} F \, \frac{\ud s'}{s'}$ in the decomposition~\eqref{equ:commutator_BQ_near_far_mismatch1_freq_decomp} of $F$. Using the localized parabolic regularity estimates from Proposition~\ref{p:lpregcore}, Schur's test and Lemma~\ref{lem:PsLEs1}, we obtain uniformly for all heat times $2 \leq s_0 \leq 4$ that
 \begin{align*}
  &\int_0^2 \biggl( \int_s^{s_0} \Bigl( \frac{s}{s'} \Bigr)^{\frac{1}{4}} (s')^{\frac{1}{4}} \sum_{\ell \geq \log_2(R)-10} 2^{\frac{1}{2} \ell} \bigl\| \phi_\ell \tilP_s \bigl( \chi_{\{r \leq 2^{-20} R\}} \calP_{s'} F \bigr) \bigr\|_{L^2_{t,x}} \, \frac{\ud s'}{s'} \biggr)^2 \, \ds \\
  &\lesssim R^{-(N-\frac{1}{2})} \int_0^4 (s')^{\frac{1}{2}} \| \calP_{s'} F \|_{LE_{s'}^\ast}^2 \, \frac{\ud s'}{s'} \\
  &\lesssim R^{-(N-\frac{1}{2})} \int_0^4 (s')^{\frac{1}{2}} \| \tilP_{\frac{s'}{8}} F \|_{LE_{\frac{s'}{8}}^\ast}^2 \, \frac{\ud s'}{s'} \\
  &\leq \varepsilon^2 \|F\|_{LE^\ast}^2,
 \end{align*}
 where in the last step we chose $R \gg 1 $ sufficiently large. Finally, using the localized parabolic regularity estimates from Proposition~\ref{p:lpregcore}, it is straightforward to bound the contributions to~\eqref{equ:commutator_BQ_near_far_mismatch1_todo} of the two low-frequency terms $\tilP_{s_0} F$ and $\tilP_{\geq s_0} F$ uniformly for all $2 \leq s_0 \leq 4$ by
 \begin{align*}
  \varepsilon^2 s_0^{\frac{1}{2}} \| \tilP_{\frac{s_0}{8}} F \|_{LE_{\frac{s_0}{8}}^\ast}^2 + \varepsilon^2 \| \tilP_{\geq s_0} F \|_{LE_\low^\ast}^2
 \end{align*}
 for sufficiently large $R \gg 1$. Then integrating in $\frac{\ud s_0}{s_0}$ over the heat time interval $2 \leq s_0 \leq 4$ yields that their contributions are bounded by $\varepsilon^2 \|F\|_{LE^\ast}^2$, as desired. This finishes the proof of the first mismatch estimate~\eqref{equ:commutator_BQ_near_far_mismatch1}. The proof of the second mismatch estimate~\eqref{equ:commutator_BQ_near_far_mismatch2} proceeds analogously, gaining smallness via the localized parabolic regularity estimates from Proposition~\ref{p:lpregcore}. Here it suffices to resolve $u$ with the usual frequency projections into 
 \[
  u = \int_0^s \tilP_{s'} u \, \frac{\ud s'}{s'} + \int_s^{s_0} \tilP_{s'} u \, \frac{\ud s'}{s'} + \tilP_{\geq s_0} u
 \]
 for any given heat time $0 < s \leq 2$ and any $2 \leq s_0 \leq 4$. Moreover, one uses the pointwise estimates for the coefficients of the multiplier $Q_s^{(\alpha)}$ given by
 \[
  |\beta_s^{(\alpha)}| \lesssim s^{\frac{1}{2}} \quad \text{and} \quad |\partial_r \beta_s^{(\alpha)}| + |\coth(r) \beta_s^{(\alpha)}| \lesssim 1
 \]
 uniformly for all $0 < s \leq 2$ and all $\{ \alpha_\ell \} \in \calA$. The details are left to the reader. This concludes the proof of Proposition~\ref{p:commutator_Hlot_Q_high}.
\end{proof}

\subsection{Estimates to commute $\tilP_{s}$, $\tilP_{\ge s_0}$ with $H_{\prin}$} \label{ss:commute_with_H_prin}

In this subsection we collect estimates that will be needed in Sections~\ref{s:low}--\ref{s:trans} to control the commutator of our frequency projections with the principal part $H_\prin$ of the operator $H$. We use the notation
\begin{align*}
\begin{split}
 \ringa:=\boldsymbol{a}-\boldsymbol{h}^{-1}, 
\end{split}
\end{align*}
so that the tensor $\ringa$ is globally small and decays at spatial infinity. Correspondingly, we may then write
\begin{align*}
\begin{split}
H_{\prin}=-\Delta-\nabla_{\mu}\ringa^{\mu\nu}\nabla_\nu.
\end{split}
\end{align*} 

For the high-frequency regime we need the following estimate.

\begin{prop} \label{p:commP_sH_prin}
For any $\delta \in (0, 1]$, $s > 0$, and $\{\alpha_\ell\}\in\calA$, let $Q^{(\alpha)}_s$ be as in Definition~\ref{d:beta_high}. Then we have 
\begin{align}\label{eq:princompmaintemp1}
\begin{split}
 \int_{0}^{2} \sup_{\{\alpha_\ell \}\in\calA} \, \bigl| \angles{[\tilP_s,H_\prin]u}{Q_s^{(\alpha)}\tilP_su}_{t,x} \bigr| \, \ds \lesssim C(\ringa) \|u\|_{LE}^2,
\end{split}
\end{align}
where
\begin{align}\label{eq:Cringadef1}
 C(\ringa) := \sum_{j=1}^4 \, \Bigl( \|\nabla^{(j)}\ringa\|_{L^\infty_{t,x}} + \sum_{\ell\geq0} \|r^2 \nabla^{(j)} \ringa \|_{L^\infty(\bbR \times A_\ell)} \Big).
\end{align}
\end{prop}

Similarly, in the low-frequency regime we will need the following estimate.

\begin{prop} \label{p:commP_leqsH_prin}
Given $\{\alpha_\ell\}\in\calA$, let $Q^{(\alpha)}$ be as in Definition~\ref{d:beta_low}. Then it holds that
\begin{align}\label{eq:princomplowmaintemp1}
\begin{split}
 \int_{\frac{1}{8}}^{4} \sup_{\{\alpha_\ell\}\in\calA} \, \bigl| \angles{[\tilP_{\geq s},H_\prin]u}{Q^{(\alpha)}\tilP_{\geq s}u}_{t,x} \bigr| \,\ds \lesssim C(\ringa) \|u\|_{LE}^2,
\end{split}
\end{align}
where
\begin{align}\label{eq:Cringadef2}
 \begin{split}
  C(\ringa) := \sum_{j=0}^1 \, \Bigl( \|\nabla^{(j)}\ringa\|_{L^\infty_{t,x}} + \sum_{\ell \geq 0} \|r^3 \nabla^{(j)} \ringa\|_{L^\infty(\bbR \times A_\ell)} \Bigr).
 \end{split}
\end{align}
\end{prop}

\begin{rem} \label{rem:commP_leqsH_prin}
We will also need the following analogues of estimates \eqref{eq:princompmaintemp1} and \eqref{eq:princomplowmaintemp1}:
\begin{align}\label{eq:princompmaintemp1noQ}
\begin{split}
 \int_{0}^{2} \bigl| \angles{[\tilP_s,H_\prin]u}{\tilP_su}_{t,x} \bigr| \, \ds \lesssim  C(\ringa) \|u\|_{LE}^2,
\end{split}
\end{align}
where $C(\ringa)$ is as in \eqref{eq:Cringadef1}, and
\begin{align}\label{eq:princomplowmaintemp1noQ}
\begin{split}
 \int_{\frac{1}{8}}^{4}  \bigl| \angles{[\tilP_{\geq s},H_\prin]u}{\tilP_{\geq s}u}_{t,x} \bigr| \, \ds &\lesssim C(\ringa) \|u\|_{LE}^2,\\
 \int_{\frac{1}{8}}^{4} \sup_{\{\alpha_\ell\}\in\calA} \, \bigl| \angles{[\tilP_{\geq s},H_\prin]u}{Q_{s_3}^{(\alpha)}\tilP_{\geq s}u}_{t,x} \bigr| \, \ds &\lesssim C(\ringa) \|u\|_{LE}^2,
\end{split}
\end{align}
where $C(\ringa)$ is as in \eqref{eq:Cringadef2} and $Q_{s_3}^{(\alpha)}$ is as in Definition~\ref{d:beta_high} for some fixed $s_3 \simeq 1$. Our argument for the proofs of Propositions~\ref{p:commP_sH_prin} and~\ref{p:commP_leqsH_prin} can also  be used to prove \eqref{eq:princompmaintemp1noQ} and \eqref{eq:princomplowmaintemp1noQ}.
\end{rem}

\begin{proof}[Proof of Proposition~\ref{p:commP_sH_prin}]
To simplify notation we write $Q$ instead of $Q^{(\alpha)}_s$ in this proof. We also use the notation $C_j(\bsz)$ introduced in \eqref{eq:Cjdef}. Note that
\begin{align*}
\begin{split}
[\tilP_s,H_\prin]&=H_{\prin}s(\Delta+\rho^2)e^{s(\Delta+\rho^2)}-s(\Delta+\rho^2)e^{s(\Delta+\rho^2)}H_\prin\\
&=se^{(s\Delta+\rho^2)}(H_{\prin}(\Delta+\rho^2)-(\Delta+\rho^2)H_\prin)\\
&\quad+s(H_\prin e^{s(\Delta+\rho^2)}-e^{s(\Delta+\rho^2)}H_\prin)(\Delta+\rho^2).
\end{split}
\end{align*}
Therefore, since $[e^{s\Delta},\Delta]=[\rho^2,H_\prin]=0$, we find that
\begin{align}\label{eq:Ps_Hp_decomp}
 [\tilP_s,H_\prin]u &=se^{s(\Delta+\rho^2)}([\Delta,\nabla_\mu\ringa^{\mu\nu}\nabla_\nu] u)+s[e^{s(\Delta+\rho^2)},\nabla_\mu\ringa^{\mu\nu}\nabla_\nu](\Delta+\rho^2) u \nonumber \\
 &=:I+II.
\end{align}
We start with the second term $II$. Let 
\begin{align*}
\begin{split}
  w(s) := [e^{s(\Delta+\rho^2)},\nabla_\mu \ringa^{\mu \nu} \nabla_\nu](\Delta+\rho^2) u.
\end{split}
\end{align*}
Then $w$ satisfies the shifted heat equation
\begin{align*}
\begin{split}
 (\partial_s-\Delta-\rho^2) w(s)=[\Delta,\nabla_\mu\ringa^{\mu\nu}\nabla_\nu]e^{s(\Delta+\rho^2)}(\Delta+\rho^2) u,\qquad w(0)=0.
\end{split}
\end{align*}
Writing 
\begin{align*}
\begin{split}
\nabla_\mu \ringa^{\mu\nu}\nabla_\nu=(\nabla_\mu \ringa^{\mu\nu})\nabla_\nu+\ringa^{\mu\nu}\nabla_\mu\nabla_\nu
\end{split}
\end{align*} 
and using the curvature formulas \eqref{eq:Rich}, we see that $[\Delta,\nabla_\mu\ringa^{\mu\nu}\nabla_\nu]$ is a linear combination of operators of the form
\begin{align}\label{eq:Delta_a_temp11}
\begin{split}
(\nabla^{(j)}\cdot\ringa)\cdot\nabla^{(k)},\qquad k+j\leq4,\qquad k,j\leq 3,
\end{split}
\end{align}
where the dots between $\nabla^{(j)}$ and $\ringa$ as well as between $(\nabla^{(j)}\cdot\ringa)$ and $\nabla^{(4-k)}$ denote metric contractions. The contributions of these terms can be estimated similarly and here we present the details only for the top order term $(\nabla \ringa)\cdot\nabla^{(3)}$. Then in view of Duhamel's formula for $w$ we let
\begin{align*}
\begin{split}
  \tilw(s) := -\int_0^se^{(s-s')(\Delta+\rho^2)}(\nabla\ringa)\cdot \nabla^{(3)} \tilP_{s'}u \dsp.
\end{split}
\end{align*}
Now
\begin{align*}
\begin{split}
 &\int_{0}^{2}  \sup_{\{\alpha\}\in\calA}| \angles{s\tilw(s)}{Q\tilP_su}_{t,x}|\,\ds \\
 &\leq\Big(\int_0^{2}s^{\frac{1}{2}}\|s\tilw(s)\|_{LE_s^\ast}^2\,\ds\Big)^{\frac{1}{2}}\Big(\int_0^{2}s^{-\frac{1}{2}} \sup_{\{\alpha\}\in\calA}\|Q\tilP_su\|_{LE_s}^2\,\ds\Big)^{\frac{1}{2}}.
\end{split}
\end{align*}
By Proposition~\ref{p:QsX} and Lemmas~\ref{l:Xal},~\ref{lem:LE_comparison}, and~\ref{lem:PsLEs1}
\begin{align*}
\begin{split}
\|Q\tilP_su\|_{LE_s}&\lesssim \sup_{\{\gamma_\ell\}\in\calA} \, \|Q\tilP_su\|_{X_{\gamma,s}}\lesssim \sup_{\{\gamma_\ell\}\in\calA} \, \|\tilP_{\frac{s}{2}}u\|_{X_{\gamma,s}} \lesssim \|\tilP_{\frac{s}{4}}u\|_{LE_{\frac{s}{4}}},
\end{split}
\end{align*}
uniformly in $\{\alpha_\ell\}\in\calA$, so using the change of variables $\frac{s}{4}\mapsto s$ we get 
\begin{align*}
\begin{split}
\Big(\int_0^{2}s^{-\frac{1}{2}} \sup_{\{\alpha\}\in\calA}\|Q\tilP_su\|_{LE_s}^2\,\ds\Big)^{\frac{1}{2}}&\lesssim \Big(\int_0^{\frac{1}{2}}s^{-\frac{1}{2}}\|\tilP_su\|_{LE_s}^2\,\ds\Big)^{\frac{1}{2}}\\
&\lesssim \|u\|_{LE}.
\end{split}
\end{align*}
Therefore, our desired estimate follows if we can show
\begin{align}\label{eq:highestcommtemp1}
\begin{split}
 \int_0^{2}s^{\frac{1}{2}}\|s\tilw(s)\|_{LE_s^\ast}^2\,\ds\leq C(\ringa)^2 \|u\|_{LE}^2.
\end{split}
\end{align}
Let us write
\begin{align}\label{eq:02_012temp1}
\begin{split}
 \int_0^{2}s^{\frac{1}{2}}\|s\tilw(s)\|_{LE_s^\ast}^2\,\ds= \int_0^{\frac{1}{2}}s^{\frac{1}{2}}\|s\tilw(s)\|_{LE_s^\ast}^2\,\ds+ \int_{\frac{1}{2}}^{2}s^{\frac{1}{2}}\|s\tilw(s)\|_{LE_s^\ast}^2\,\ds.
\end{split}
\end{align}
The first term on the right-hand side above is bounded by
\begin{align}\label{eq:para_temp1}
\begin{split}
 \int_0^{\frac{1}{2}}\Big(\int_0^s s^{\frac{5}{4}}\|e^{(s-s')(\Delta+\rho^2)}(\nabla\ringa)\cdot\nabla^{(3)}\tilP_{s'}u\|_{LE_s^\ast}\,\dsp\Big)^2\,\ds .
\end{split}
\end{align}
Since $s-s'\leq s$, by Lemmas~\ref{lem:LE_comparison},~\ref{lem:PsLEs1}, and~\ref{lem:multLEs1}, we have 
\begin{align*}
\begin{split}
\|e^{(s-s')(\Delta+\rho^2)}(\nabla\ringa)\cdot\nabla^{(3)}\tilP_{s'}u\|_{LE_s^\ast} &\lesssim \|(\nabla\ringa)\cdot\nabla^{(3)}\tilP_{s'}u\|_{LE_s^\ast}\\
&\lesssim C_0(\nabla\ringa) \|\nabla^{(3)}\tilP_{s'}u\|_{LE_{s'}}\\
&\lesssim C_0(\nabla\ringa) (s')^{-\frac{3}{2}}\|\tilP_{\frac{s'}{2}}u\|_{LE_{\frac{s'}{2}}}.
\end{split}
\end{align*}
Therefore the contribution of $s'\in[\frac{s}{2},s]$ to \eqref{eq:para_temp1} is bounded by
\begin{align*}
\begin{split}
C_0(\nabla\ringa)^2 \int_0^{\frac{1}{2}} \Big(\int_{\frac{s}{2}}^s(\frac{s}{s'})^{\frac{5}{4}}(s')^{-\frac{1}{4}}\|P_{\frac{s'}{2}}u\|_{LE_{\frac{s'}{2}}}\,\dsp\Big)^2\,\ds \lesssim C_0(\nabla\ringa)^2 \|u\|_{LE}^2,
\end{split}
\end{align*}
by Schur's test. By the same argument the contributions to this integral of the other terms in \eqref{eq:Delta_a_temp11} are bounded by
\begin{align*}
\begin{split}
\sum_{j=0}^3 C_0(\nabla^{(j)}\ringa)^2 \|u\|_{LE}^2.
\end{split}
\end{align*}
Next we consider the integral over $s'\in[0,\frac{s}{2}]$ in \eqref{eq:para_temp1}. Here by repeated applications of the product rule  and in view of \eqref{eq:Delta_a_temp11} we can write $[\Delta,\nabla_\mu\ringa^{\mu\nu}\nabla_\nu]\tilP_{s'}u$ as a linear combination of terms of the form
\begin{align}\label{eq:aderivativesfourtemp1}
\begin{split}
 \nabla^{(k)}\cdot((\nabla^{(j)}\cdot\ringa)\tilP_{s'}u ),\quad k\leq 3, \quad k+j\leq 4.
\end{split}
\end{align}
Again we treat only the case $k=3$, $j=1$ in detail. For this by Lemmas~\ref{lem:PsLEs1} and~\ref{lem:multLEs1}
\begin{align*}
\begin{split}
\|e^{(s-s')(\Delta+\rho^2)}\nabla^{(3)}\cdot((\nabla\ringa) \tilP_{s'}u)\|_{LE_{s}^\ast} &\lesssim (s-s')^{-\frac{3}{2}}\|(\nabla \ringa) \tilP_{s'}u\|_{LE_s^\ast} \\
&\lesssim C_0(\nabla\ringa)(s-s')^{-\frac{3}{2}}\|\tilP_{s'}u\|_{LE_{s'}}.
\end{split}
\end{align*}
It follows that the contribution of $s'\in[0,\frac{s}{2}]$ to \eqref{eq:para_temp1} is bounded by
\begin{align*}
\begin{split}
C_0(\nabla\ringa)^2 \int_0^{\frac{1}{2}} \Big(\int_0^{\frac{s}{2}}\frac{s^{\frac{5}{4}}(s')^{\frac{1}{4}}}{(s-s')^{\frac{3}{2}}}(s')^{-\frac{1}{4}}\|\tilP_{s'}u\|_{LE_{s'}}\,\dsp\Big)^2\,\ds \lesssim C_0(\nabla\ringa)^2 \|u\|_{LE}^2,
\end{split}
\end{align*}
by Schur's test. Similarly the contributions of the other terms in \eqref{eq:aderivativesfourtemp1} are bounded by
\begin{align*}
\begin{split}
\sum_{j=0}^4 C_0(\nabla^{(j)} \ringa)^2 \|u\|_{LE}^2.
\end{split}
\end{align*}
This completes the treatment of the first integral on the right-hand side of \eqref{eq:02_012temp1}. For the second integral we again decompose the $s'$ integration in the Duhamel representation of $w$ into integrals over $[0,\frac{s}{2}]$ and $[\frac{s}{2},s]$. For the latter, using similar arguments as above, and writing $\tilP_{s'}=-s'(\Delta+\rho^2)e^{\frac{s'}{2}(\Delta+\rho^2)}\tilP_{\geq\frac{s'}{2}}$, we have
\begin{align*}
\begin{split}
&\int_{\frac{1}{2}}^2s^{\frac{1}{2}}\Big(\int_{\frac{s}{2}}^s\|e^{(s-s')(\Delta+\rho^2)}(\nabla\ringa)\cdot\nabla^{(3)}\tilP_{s'}u\|_{LE^\ast_s}\,\dsp\Big)^2\,\ds\\
&\lesssim C_1(\nabla\ringa)^2 \int_{\frac{1}{2}}^2\Big(\int_{\frac{s}{2}}^s\|\tilP_{\geq\frac{s'}{2}}u\|_{LE_{\low}}\,\dsp\Big)^2\,\ds\\
&\lesssim C_1(\nabla\ringa)^2 \int_{\frac{1}{8}}^1\|\tilP_{\geq s'}u\|_{LE_\low}^2\,\dsp \lesssim C_1(\nabla\ringa)^2 \|u\|_{LE}^2,
\end{split}
\end{align*}
where to pass to the last line we used Cauchy-Schwarz and the change of variables $\frac{s'}{2}\mapsto s'$. Similarly the contributions of the other terms in \eqref{eq:Delta_a_temp11} are bounded by
\begin{align*}
\begin{split}
\sum_{j=0}^3 C_1(\nabla\ringa)^2 \|u\|_{LE}^2.
\end{split}
\end{align*}
For the region $s'\in[0,\frac{s}{2}]$ again arguing as above, and writing $\tilP_{s'}=2e^{\frac{s'}{2}(\Delta+\rho^2)}\tilP_{\frac{s'}{2}}$,
\begin{align*}
\begin{split}
&\int_{\frac{1}{2}}^2s^{\frac{1}{2}}\Big(\int_0^{\frac{s}{2}}\|e^{(s-s')(\Delta+\rho^2)}\nabla^{(3)}((\nabla\ringa)\tilP_{s'}u)\|_{LE^\ast_s}\,\dsp\Big)^2\,\ds\\
&\lesssim C_0(\nabla \ringa)^2 \int_{\frac{1}{2}}^2\Big(\int_0^{\frac{s}{2}}(s')^{\frac{1}{4}}(s')^{-\frac{1}{4}}\|\tilP_{\frac{s'}{2}}u\|_{LE_{\frac{s'}{2}}}\,\dsp\Big)^2\,\ds\\
&\lesssim C_0(\nabla\ringa)^2 \int_0^{\frac{1}{2}}(s')^{-\frac{1}{2}}\|\tilP_{s'}u\|_{LE_{s'}}^2\,\dsp \lesssim C_0(\nabla\ringa)^2 \|u\|_{LE}^2,
\end{split}
\end{align*}
where to pass to the last line we again used Cauchy-Schwarz and the change of variables $\frac{s'}{2}\mapsto s'$. Similarly the contributions of the other terms in \eqref{eq:aderivativesfourtemp1} are bounded by
\begin{align*}
\begin{split}
\sum_{j=0}^4 C_0(\nabla^{(j)}\ringa)^2 \|u\|_{LE}^2.
\end{split}
\end{align*}
This completes the treatment of the term $II$ on the right-hand side of \eqref{eq:Ps_Hp_decomp}.

Let us now turn to the first term $I$. We first consider the $s$ integration over $[0,\frac{1}{2}]$. As usual we decompose $u$ as
\begin{align*}
\begin{split}
u=\tilP_{\leq s}u+\tilP_{s\leq\cdot\leq s_0}u+\tilP_{\geq s_0}u,
\end{split}
\end{align*}
with $s_0\in[\frac{3}{4},1]$. Using the argument which led to \eqref{eq:highestcommtemp1}, it then suffices to  prove the estimates
\begin{align}
&\int_0^{\frac{1}{2}}s^{\frac{1}{2}}\|se^{s(\Delta+\rho^2)}[\Delta,\nabla_\mu\ringa^{\mu\nu}\nabla_\nu]\tilP_{\leq s}u\|_{LE_s^\ast}^2\,\ds\leq\ C(\ringa)^2 \|u\|_{LE}^2, \label{eq:Ps_Jp_I_temp1}\\
&\int_{\frac{3}{4}}^1\int_0^{\frac{1}{2}}s^{\frac{1}{2}}\|se^{s(\Delta+\rho^2)}[\Delta,\nabla_\mu\ringa^{\mu\nu}\nabla_\nu]\tilP_{s\leq \cdot\leq s_0}u\|_{LE_s^\ast}^2\,\ds\,\frac{\ud s_0}{s_0}\leq C(\ringa)^2 \|u\|_{LE}^2, \quad \quad \label{eq:Ps_Jp_I_temp2}\\
&\int_{\frac{3}{4}}^1\int_0^{\frac{1}{2}}s^{\frac{1}{2}}\|se^{s(\Delta+\rho^2)}[\Delta,\nabla_\mu\ringa^{\mu\nu}\nabla_\nu]\tilP_{\geq s_0}u\|_{LE_s^\ast}^2\,\ds\,\frac{\ud s_0}{s_0}\leq C(\ringa)^2 \|u\|_{LE}^2.\label{eq:Ps_Jp_I_temp3}
\end{align}
Starting with \eqref{eq:Ps_Jp_I_temp1} note that the commutator $[\Delta,\nabla_\mu \ringa^{\mu \nu} \nabla_\nu] \tilP_{\leq s}u$ is a linear combination of terms of the form
\begin{align}\label{eq:Delta_Hp_temp1}
\begin{split}
 \nabla^{(k)}\cdot((\nabla^{(j)}\cdot\ringa)\tilP_{\leq s'}u ),\quad k\leq3,\quad k+j\leq 4.
\end{split}
\end{align}
We only treat the case $j=1$, $k=3$ in detail. Estimate  \eqref{eq:Ps_Jp_I_temp1} then reduces to proving
\begin{align} \label{eq:Ps_Jp_I_temp4}
\begin{split}
\int_0^{\frac{1}{2}}s^{-\frac{1}{2}}\|s^{\frac{3}{2}}e^{s(\Delta+\rho^2)}\nabla^{(3)}\cdot((\nabla\ringa)\tilP_{\leq s}u)\|_{LE_s^\ast}^2\,\ds\leq C(\ringa)^2 \|u\|_{LE}^2.
\end{split}
\end{align}
Now by Lemmas~\ref{lem:PsLEs1} and~\ref{lem:multLEs1}
\begin{align*}
\begin{split}
\|s^{\frac{3}{2}}e^{s(\Delta+\rho^2)}\nabla^{(3)}\cdot((\nabla\ringa)\tilP_{\leq s}u)\|_{LE_s^\ast} &\lesssim \|(\nabla\ringa)\tilP_{\leq s}u\|_{LE_s^\ast}\\
&\lesssim \int_{0}^s \|(\nabla\ringa)\tilP_{s'}u\|_{LE_s^\ast}\,\dsp\\
&\leq C_0(\nabla\ringa) \int_{0}^s\|\tilP_{s'}u\|_{LE_{s'}}\,\dsp.
\end{split}
\end{align*}
It follows that the left-hand side of \eqref{eq:Ps_Jp_I_temp4} is bounded by
\begin{align*}
\begin{split}
C_0(\nabla\ringa)^2 \int_0^{\frac{1}{2}}\Big(\int_0^s(\frac{s'}{s})^{\frac{1}{4}}(s')^{-\frac{1}{4}}\|P_{s'}u\|_{LE_{s'}}\,\dsp\Big)^2\,\ds \lesssim C_0(\nabla\ringa)^2 \|u\|_{LE}^2,
\end{split}
\end{align*}
by Schur's test. Similarly the contributions of the other terms in \eqref{eq:Delta_Hp_temp1} are bounded by
\begin{align*}
\begin{split}
\sum_{j=0}^4 C_0(\nabla^{(j)}\ringa)^2 \|u\|_{LE}^2.
\end{split}
\end{align*}
For \eqref{eq:Ps_Jp_I_temp2} we go back to \eqref{eq:Delta_a_temp11} and only treat the most difficult case $j=1$, $k=3$ in detail. Then by Lemma~\ref{lem:PsLEs1} it suffices to prove
\begin{align} \label{eq:Ps_Jp_I_temp5}
\begin{split}
\int_{\frac{3}{4}}^1\int_0^{\frac{1}{2}}\Big(\int_s^{s_0}\|s^{\frac{5}{4}}(\nabla\ringa)\cdot\nabla^{(3)}\tilP_{s'}u\|_{LE_s^\ast}\,\dsp\Big)^2\,\ds\,\frac{\ud s_0}{s_0}\leq C(\ringa)^2 \|u\|_{LE}^2.
\end{split}
\end{align}
Now arguing as above the inner integral with respect to $s$ on the left-hand side of \eqref{eq:Ps_Jp_I_temp5} is bounded by
\begin{align*}
\begin{split}
C_0(\nabla\ringa)^2 \int_0^{\frac{1}{2}}\Big(\int_s^{s_0}\|s^{\frac{5}{4}}\nabla^{(3)}\tilP_{s'}u\|_{LE_{s'}}\,\dsp\Big)^2\,\ds,
\end{split}
\end{align*}
which, using Lemmas~\ref{lem:LE_comparison} and~\ref{lem:PsLEs1} and Schur's test, is in turn bounded by
\begin{align*}
\begin{split}
&C_0(\nabla\ringa)^2  \int_0^{\frac{1}{2}}\Big(\int_s^{s_0}(\frac{s}{s'})^{\frac{5}{4}}(s')^{-\frac{1}{4}}\|\tilP_{\frac{s'}{2}}u\|_{LE_{s'}}\,\dsp\Big)^2\,\ds\\
&\lesssim C_0(\nabla\ringa)^2 \int_0^{\frac{1}{2}}\Big(\int_s^{s_0}(\frac{s}{s'})^{\frac{5}{4}}(s')^{-\frac{1}{4}}\|\tilP_{\frac{s'}{2}}u\|_{LE_{\frac{s'}{2}}}\,\dsp\Big)^2\,\ds\\
&\lesssim C_0(\nabla\ringa)^2 \int_0^{\frac{1}{2}}\Big(\int_{\frac{s}{2}}^{\frac{s_0}{2}}(\frac{s}{s'})^{\frac{5}{4}}(s')^{-\frac{1}{4}}\|\tilP_{s'}u\|_{LE_{s'}}\,\dsp\Big)^2\,\ds\\
&\lesssim C_0(\nabla\ringa)^2 \|u\|_{LE}^2.
\end{split}
\end{align*}
Since the implicit constants are bounded uniformly in $s_0\in[\frac{3}{2},1]$, integrating both sides of this estimate in $s_0$ gives us the desired estimate. Similarly the contributions of the other terms in \eqref{eq:Delta_a_temp11} are bounded by
\begin{align*}
\begin{split}
\sum_{j=0}^3C_0(\nabla^{(j)}\ringa)^2 \|u\|_{LE}^2.
\end{split}
\end{align*}
For \eqref{eq:Ps_Jp_I_temp3}, arguing as above we will show that
\begin{align}\label{eq:Ps_Jp_I_temp6}
\begin{split}
\int_{\frac{3}{4}}^1\int_0^{\frac{1}{2}}\|s^{\frac{5}{4}}(\nabla\ringa)\cdot\nabla^{(3)}\tilP_{\geq s_0}u\|_{LE_s^\ast}^2\,\ds\,\frac{\ud s_0}{s_0} \lesssim C(\ringa)^2 \|u\|_{LE}^2.
\end{split}
\end{align}
Now by Lemma~\ref{lem:multLEs1} we have 
\begin{align*}
\begin{split}
\|(\nabla\ringa)\cdot\nabla^{(3)}\tilP_{\geq s_0}u\|_{LE_s^\ast}\lesssim C_1(\nabla\ringa)^2 \|\nabla^{(3)}\tilP_{\geq s_0}u\|_{LE_\low}.
\end{split}
\end{align*}
It follows that the left-hand side of \eqref{eq:Ps_Jp_I_temp6} is bounded by
\begin{align*}
\begin{split}
C_1(\nabla\ringa)^2 \int_{\frac{3}{4}}^{1}\|\nabla^{(3)}\tilP_{\geq s_0}u\|_{LE_\low}^2\int_0^{\frac{1}{2}}s^{\frac{5}{2}}\,\ds\frac{\ud s_0}{s_0}&\lesssim C_1(\nabla\ringa)^2 \int_{\frac{3}{8}}^\frac{1}{2}\|\tilP_{\geq s_0}u\|_{LE_\low}^2\frac{\ud s_0}{s_0} \\
&\lesssim C_1(\nabla\ringa)^2 \|u\|_{LE}^2.
\end{split}
\end{align*}
Similarly the contributions of the other terms in \eqref{eq:Delta_a_temp11} are bounded by
\begin{align*}
\begin{split}
\sum_{j=0}^3 C_1(\nabla^{(j)}\ringa)^2 \|u\|_{LE}^2.
\end{split}
\end{align*} 
It remains to consider the contributions of the $s$ integration over $[\frac{1}{2},2]$ to the term~$I$ on the right-hand side of \eqref{eq:Ps_Hp_decomp}. To this end we decompose $u$ as
\begin{align*}
\begin{split}
u=\tilP_{\leq \frac{s}{4}}u+\tilP_{\geq \frac{s}{4}}u.
\end{split}
\end{align*}
Then instead of \eqref{eq:Ps_Jp_I_temp1}--\eqref{eq:Ps_Jp_I_temp3}, it suffices to show
\begin{align}
\int_{\frac{1}{2}}^2s^{\frac{1}{2}}\|se^{s(\Delta+\rho^2)}[\Delta,\nabla_\mu\ringa^{\mu\nu}\nabla_\nu]\tilP_{\leq \frac{s}{4}}u\|_{LE_s^\ast}^2\,\ds &\lesssim C(\ringa)^2 \|u\|_{LE}^2,\label{eq:Ps_Jp_I_temp54}\\
\int_{\frac{1}{2}}^2s^{\frac{1}{2}}\|se^{s(\Delta+\rho^2)}[\Delta,\nabla_\mu\ringa^{\mu\nu}\nabla_\nu]\tilP_{\geq \frac{s}{4}}u\|_{LE_s^\ast}^2\,\ds &\lesssim C(\ringa)^2 \|u\|_{LE}^2.\label{eq:Ps_Jp_I_temp55}
\end{align}
For \eqref{eq:Ps_Jp_I_temp55} arguing as above, and using \eqref{eq:Delta_Hp_temp1} with $\tilP_{\leq s'}$ replaced by $\tilP_{\geq \frac{s}{4}}$ we have
\begin{align*}
\begin{split}
\int_{\frac{1}{2}}^{2}s^{\frac{1}{2}}\|se^{s(\Delta+\rho^2)}\nabla^{(3)}\cdot((\nabla\ringa)\tilP_{\geq\frac{s}{4}}u)\|_{LE_s^\ast}^2\,\ds &\lesssim\int_\frac{1}{2}^2\|(\nabla\ringa)\tilP_{\geq\frac{s}{4}}u\|_{LE_{s}^\ast}^2\,\ds \\
&\lesssim C_1(\nabla\ringa)^2 \int_{\frac{1}{2}}^2 \|\tilP_{\geq\frac{s}{4}}u\|_{LE_{\low}}^2 \, \ds \\
&\lesssim C_1(\nabla\ringa)^2 \|u\|_{LE}^2.
\end{split}
\end{align*}
Similarly the contributions of the other terms in \eqref{eq:Delta_Hp_temp1} with $\tilP_{\leq s'}$ replaced by $\tilP_{\geq \frac{s}{4}}$ are bounded by
\begin{align*}
\begin{split}
\sum_{j=0}^4 C_1(\nabla^{(j)}\ringa)^2 \|u\|_{LE}^2.
\end{split}
\end{align*}
For \eqref{eq:Ps_Jp_I_temp54} we have that
\begin{align*}
\begin{split}
&\int_{\frac{1}{2}}^{2}s^{\frac{1}{2}}\|se^{s(\Delta+\rho^2)}\nabla^{(3)}\cdot((\nabla\ringa)\tilP_{\leq\frac{s}{4}}u)\|_{LE_s^\ast}^2\,\ds\\
&\lesssim \int_{\frac{1}{2}}^2\|(\nabla\ringa)\tilP_{\leq\frac{s}{4}}u\|_{LE_s^\ast}^2\,\ds \\
&\lesssim C_0(\nabla\ringa)^2 \int_{\frac{1}{2}}^2\Big(\int_0^{\frac{s}{4}}(s')^{\frac{1}{4}}(s')^{-\frac{1}{4}}\|\tilP_{s'}u\|_{LE_{s'}}\,\dsp\Big)^2\,\ds\\
&\lesssim C_0(\nabla\ringa)^2 \int_0^{\frac{1}{2}}(s')^{-\frac{1}{2}}\|\tilP_{s'}u\|_{LE_{s'}}^2\,\dsp \lesssim C_0(\nabla\ringa)^2 \|u\|_{LE}^2.
\end{split}
\end{align*}
Similarly, the contributions of the other terms in \eqref{eq:Delta_Hp_temp1} with $\tilP_{\leq s'}$ replaced by $\tilP_{\leq \frac{s}{4}}$ are bounded by
\begin{align*}
\begin{split}
\sum_{j=0}^4 C_0(\nabla^{(j)}\ringa)^2 \|u\|_{LE}^2.
\end{split}
\end{align*}
This completes the proof of the proposition.
\end{proof}

\begin{proof}[Proof of Proposition~\ref{p:commP_leqsH_prin}]
To simplify notation we write $Q$ for $Q^{(\alpha)}$ and use the notation $C_j(\bsz)$ introduced in \eqref{eq:Cjdef}. Since we are in the low-frequency regime we do not need to exploit the commutator structure as carefully as in the previous proof to gain regularity. Let
\begin{align*}
\begin{split}
 v(s) := [\tilP_{\geq s},H_{\prin}]u.
\end{split}
\end{align*}
Then
\begin{align*}
\begin{split}
\int_\frac{1}{8}^{4} \sup_{\{\alpha\}\in\calA}|\angles{v}{Q\tilP_{\geq s}u}_{t,x}|\,\ds
&\lesssim \Big(\int_{\frac{1}{8}}^4 \|v\|_{LE_\low^\ast}^2\,\ds\Big)^{\frac{1}{2}}\Big(\int_{\frac{1}{8}}^4\sup_{\{\alpha\}\in\calA}\|Q\tilP_{\geq s}u\|_{LE_\low}^2\,\ds\Big)^{\frac{1}{2}}\\
&\lesssim \Big(\int_{\frac{1}{8}}^4 \|v\|_{LE_\low^\ast}^2\,\ds\Big)^{\frac{1}{2}}\Big(\int_{\frac{1}{8}}^4\|\tilP_{\geq \frac{s}{2}}u\|_{LE_\low}^2\,\ds\Big)^{\frac{1}{2}}.
\end{split}
\end{align*}
Now using the fundamental theorem of calculus to write
\begin{align*}
\begin{split}
\tilP_{\geq\frac{s}{2}}u=\tilP_{\geq s}u+\int_{\frac{s}{2}}^s\tilP_{s'}u\,\dsp,
\end{split}
\end{align*}
and by Lemmas~\ref{lem:LE_comparison},~\ref{lem:LEs_LElow_comparision}, and~\ref{lem:PsLEs1}, we can estimate
\begin{align*}
\begin{split}
\int_{\frac{1}{8}}^4\|\tilP_{\geq \frac{s}{2}}u\|_{LE_\low}^2\,\ds &\lesssim \|u\|_{LE}^2+\int_{\frac{1}{8}}^4\Big(\int_{\frac{s}{2}}^s\|\tilP_{s'}u\|_{LE_\low}\,\dsp\Big)^2\,\ds\\
&\lesssim \|u\|_{LE}^2+\int_{\frac{1}{16}}^{4}\|e^{\frac{7s'}{8}(\Delta+\rho^2)}\tilP_{\frac{s'}{8}}u\|_{LE_{s'}}^2\,\dsp\\
&\lesssim \|u\|_{LE}^2+\int_{0}^{\frac{1}{2}}\|\tilP_{s'}u\|_{LE_{s'}}^2\,\dsp\lesssim \|u\|_{LE}^2.
\end{split}
\end{align*}
Therefore, it suffices to show that
\begin{align}\label{eq:lowprintemp1}
\begin{split}
\int_{\frac{1}{8}}^4 \|[\tilP_{\geq s},H_{\prin}]u\|_{LE_\low^\ast}^2\,\ds \leq C(\ringa)^2 \|u\|_{LE}^2.
\end{split}
\end{align}
For this we simply write the commutator $[\tilP_{\geq s},H_{\prin}]u$ as $\tilP_{\geq s}\ringH_\prin u-\ringH_\prin \tilP_{\geq s}u$ where $\ringH_\prin$ is defined the same way as $H_\prin$ but with $\bsa$ replaced by $\ringa$. This is possible because $[\tilP_{\geq s},\Delta]=0$. Now we estimate the two terms $\tilP_{\geq s}\ringH_\prin u$ and $\ringH_\prin \tilP_{\geq s}u$ separately. For the latter, writing
\begin{align*}
\begin{split}
\ringH_\prin=-\ringa^{\mu\nu}\nabla_{\mu}\nabla_\nu+(\nabla_\mu\ringa^{\mu\nu})\nabla_\nu,
\end{split}
\end{align*}
and by applications of the product rule and using Lemmas~\ref{lem:PsLEs1} and ~\ref{lem:multLEs1}
\begin{align*}
\begin{split}
\|\ringH_\prin \tilP_{\geq s}u\|_{LE_\low^\ast} \lesssim C(\ringa)\|\tilP_{\geq \frac{s}{2}}u\|_{LE_\low}.
\end{split}
\end{align*}
Then arguing as above we get
\begin{align*}
\begin{split}
\int_{\frac{1}{8}}^4 \|\ringH_\prin \tilP_{\geq s}u\|_{LE_\low^\ast}^2\,\ds \lesssim C(\ringa)^2 \|u\|_{LE}^2.
\end{split}
\end{align*}
To estimate the contribution of $\tilP_{\geq s}\ringH_\prin u$ we decompose $u$ as
\begin{align*}
\begin{split}
u=\tilP_{\leq s}u+\tilP_{\geq s}u.
\end{split}
\end{align*}
For $\tilP_{\geq s}u$ we write
\begin{align*}
\begin{split}
\ringH_\prin=-\nabla_\mu\nabla_\nu\ringa^{\mu\nu}+\nabla_\mu(\nabla_\nu\ringa^{\mu\nu}),
\end{split}
\end{align*}
and by applications of the product rule and using Lemmas~\ref{lem:PsLEs1} and ~\ref{lem:multLEs1}, we obtain
\begin{align*}
\begin{split}
\|\tilP_{\geq s}\ringH_\prin \tilP_{\geq s}u\|_{LE_\low^\ast} \lesssim C(\ringa)\|\tilP_{\geq s}u\|_{LE_\low},
\end{split}
\end{align*}
and thus
\begin{align*}
\begin{split}
\int_{\frac{1}{8}}^4 \|\tilP_{\geq s}\ringH_\prin \tilP_{\geq s}u\|_{LE_\low^\ast}^2\,\ds \lesssim C(\ringa)^2 \|u\|_{LE}^2.
\end{split}
\end{align*}
Similarly for $\tilP_{\leq s}u$ writing $\ringH_\prin=-\nabla_\mu\nabla_\nu\ringa^{\mu\nu}+\nabla_\mu(\nabla_\nu\ringa^{\mu\nu})$ we have
\begin{align*}
\begin{split}
 &\|\tilP_{\geq s}\ringH_\prin \tilP_{\leq s}u\|_{LE_\low^\ast} \\
 &\quad \lesssim \|\ringa \tilP_{\leq s}u\|_{LE_\low^\ast}+\|(\nabla\ringa)\tilP_{\leq s}u\|_{LE_\low^\ast}\\
 &\quad \lesssim C(\ringa) \int_0^s\|\tilP_{s'}u\|_{LE_{s'}}\,\dsp\lesssim C(\ringa)\int_0^s\|e^{\frac{7s'}{8}(\Delta+\rho^2)}\tilP_{\frac{s'}{8}}u\|_{LE_{s'}}\,\dsp\\
 &\quad \lesssim C(\ringa) \int_{0}^{\frac{s}{8}}\|\tilP_{s'}u\|_{LE_{s'}}\,\dsp\lesssim C(\ringa)\Big(\int_{0}^{\frac{s}{8}}(s')^{-\frac{1}{2}}\|\tilP_{s'}u\|_{LE_{s'}}^2\,\dsp\Big)^{\frac{1}{2}}.
\end{split}
\end{align*}
It follows that
\begin{align*}
\begin{split}
\int_{\frac{1}{8}}^4\|\tilP_{\geq s}\ringH_\prin \tilP_{\leq s}u\|_{LE_\low^\ast}^2\,\ds&\lesssim C(\ringa)^2 \int_{\frac{1}{8}}^4\int_{0}^{\frac{s}{8}}(s')^{-\frac{1}{2}}\|\tilP_{s'}u\|_{LE_{s'}}^2\,\dsp\\
&\lesssim C(\ringa)^2 \int_{0}^{\frac{1}{2}}(s')^{-\frac{1}{2}}\|\tilP_{s'}u\|_{LE_{s'}}^2\,\dsp \lesssim C(\ringa)^2 \|u\|_{LE}^2,
\end{split}
\end{align*}
completing the proof of the proposition.
\end{proof}

\subsection{Estimates needed to handle $H_\prin+\Delta$} \label{ss:DeltaHdifference}

In the positive commutator estimate of Sections~\ref{s:low}--\ref{s:trans} the coercivity comes from the commutator of our low and high-frequency multipliers with $-\Delta$, \emph{not with $H_\prin$}. The estimates in this section are needed to handle the errors generated from the difference between $H_\prin$ and $-\Delta$. For low frequencies we will need the following result.

\begin{prop} \label{p:DeltaHpQlow}
For any $\{\alpha_\ell\}\in\calA$, let $Q^{(\alpha)}$ be as in Definition~\ref{d:beta_low}. Then we have 
\begin{equation}
 \int_{\frac{1}{8}}^4 \sup_{\{ \alpha_\ell \} \in\calA} \, \bigl| \Re \langle i (H_\prin+\Delta) \tilP_{\geq s} u, Q^{(\alpha)} \tilP_{\geq s} u \rangle_{t,x} \bigr| \, \ds \leq C \varepsilon_0 \|u\|_{LE}^2,
\end{equation}
where $\varepsilon_0>0$ is as in \eqref{eq:decay_assumptions-prin}.
\end{prop}

\begin{rem} \label{rem:DeltaHpQlow}
 The argument for the proof of Proposition~\ref{p:DeltaHpQlow} can also be used to establish the following closely related estimate: For any $\delta \in (0,1]$, $s_3 \simeq 1$, and $\{ \alpha_\ell \} \in \calA$, let $Q_{s_3}^{(\alpha)}$ be as in Definition~\ref{d:beta_high}. Then we have 
 \begin{equation}
  \int_{\frac{1}{8}}^4 \sup_{\{ \alpha_\ell \} \in \calA} \, \bigl| \Re \langle i (H_\prin+\Delta) \tilP_{\geq s} u, Q^{(\alpha)}_{s_3} \tilP_{\geq s} u \rangle_{t,x} \bigr| \, \ds \leq C \varepsilon_0 \|u\|_{LE}^2,   
 \end{equation}
 where $\varepsilon_0>0$ is as in \eqref{eq:decay_assumptions-prin}.
\end{rem}

The following is the analogous result for high frequencies.
\begin{prop} \label{p:DeltaHpQhigh}
For any $\delta \in (0,1]$, $s>0$, and $\{ \alpha_\ell \} \in \calA$, let $Q_s^{(\alpha)}$ and $\beta_s^{(\alpha)}$ be as in Definition~\ref{d:beta_high}. Then it holds that 
\begin{equation}
 \begin{aligned}
  &\int_{0}^2 \sup_{\{ \alpha_\ell \} \in\calA} \, \bigl| \Re \langle i (H_\prin+\Delta) \tilP_{s} u, Q_s^{(\alpha)} \tilP_{s} u \rangle_{t,x} \\
  &\qquad \qquad \qquad + 2 \langle \ringa^{\theta_a\theta_b}(\coth(r) \beta_s^{(\alpha)} - \partial_r \beta_s^{(\alpha)}) \partial_{\theta_a} \tilP_s u, \partial_{\theta_b} \tilP_s u \rangle_{t,x} \bigr| \, \ds \leq C \varepsilon_0 \|u\|_{LE}^2, 
 \end{aligned}
\end{equation}
where $\varepsilon_0>0$ is as in \eqref{eq:decay_assumptions-prin}.
\end{prop}
We only provide the proof of Proposition~\ref{p:DeltaHpQhigh} and leave the easier treatment of the low-frequency case in Proposition~\ref{p:DeltaHpQlow} to the reader.
\begin{proof}[Proof of Proposition~\ref{p:DeltaHpQhigh}]
 We begin by recalling that 
 \[
  H_{\prin} + \Delta = - \nabla_\mu \ringa^{\mu \nu} \nabla_\nu,
 \]
 where we use the notation $\ringa = \bsa - \bsh^{-1}$. Then we first integrate by parts in the term $\Re \langle i (H_\prin+\Delta) \tilP_{s} u, Q_s^{(\alpha)} \tilP_{s} u \rangle_{t,x}$ so that only one derivative falls on $\tilP_s u$,
 \begin{align*}
  &\Re \langle i (H_\prin+\Delta) \tilP_{s} u, Q_s^{(\alpha)} \tilP_{s} u \rangle_{t,x} \\ 
  &\quad = \Re \langle \nabla_\mu \ringa^{\mu \nu} \nabla_\nu \tilP_s u, 2 \bsbeta_s^{(\alpha), \lambda} \nabla_\lambda \tilP_s u + (\nabla_\lambda \bsbeta_s^{(\alpha), \lambda}) \tilP_s u \rangle_{t,x} \\
  &\quad = -2 \Re \langle \ringa^{\mu \nu} \nabla_\nu \tilP_s u, \bsbeta_s^{(\alpha), \lambda} \nabla_\mu \nabla_\lambda \tilP_s u \rangle_{t,x} \\
  &\quad \quad \, -2 \Re \langle \ringa^{\mu \nu} \nabla_\nu \tilP_s u, (\nabla_\mu \bsbeta_s^{(\alpha), \lambda}) \nabla_\lambda \tilP_s u \rangle_{t,x} \\
  &\quad \quad \, - \Re \langle \ringa^{\mu \nu} \nabla_\nu \tilP_s u, (\nabla_\mu \nabla_\lambda \bsbeta_s^{(\alpha), \lambda}) \tilP_s u \rangle_{t,x} \\
  &\quad \quad \, - \Re \langle \ringa^{\mu \nu} \nabla_\nu \tilP_s u, (\nabla_\lambda \bsbeta_s^{(\alpha), \lambda}) \nabla_\mu \tilP_s u \rangle_{t,x}.
 \end{align*}
 Integrating by parts once more in the first term on the right-hand side and using that $\ringa$ is symmetric and real-valued, we find 
 \begin{equation} \label{equ:Hprin_plus_Delta_rewritten1}
  \begin{aligned}
   &\Re \langle i (H_\prin+\Delta) \tilP_{s} u, Q_s^{(\alpha)} \tilP_{s} u \rangle_{t,x} \\
   &\quad = \Re \langle (\nabla_\lambda \ringa^{\mu \nu}) \nabla_\nu \tilP_s u, \bsbeta_s^{(\alpha), \lambda} \nabla_\mu \tilP_s u \rangle_{t,x} \\
   &\quad \quad \, -2 \Re \bigl\langle \ringa^{\mu \nu} \nabla_\nu \tilP_s u, (\nabla_\mu \bsbeta_s^{(\alpha), \lambda}) \nabla_\lambda \tilP_s u \bigr\rangle_{t,x} \\
   &\quad \quad \, - \Re \bigl\langle \ringa^{\mu \nu} \nabla_\nu \tilP_s u, (\nabla_\mu \nabla_\lambda \bsbeta_s^{(\alpha), \lambda}) \tilP_s u \bigr\rangle_{t,x}.
  \end{aligned}
 \end{equation}
 Here we note that due to the radial symmetry of the vector field $\bsbeta_s^{(\alpha)}$, the only non-zero components of $\nabla_\mu \bsbeta_s^{(\alpha), \lambda}$ are given by
 \begin{align*}
  \nabla_r \bsbeta_s^{(\alpha), r} = \partial_r \beta_s^{(\alpha)} \quad \text{and} \quad \nabla_{\theta_a} \bsbeta_s^{(\alpha), \theta_a} = \coth(r) \beta_s^{(\alpha)} \text{ for } a = 1, \ldots, d-1.
 \end{align*}
 It follows that the second term on the right-hand side of~\eqref{equ:Hprin_plus_Delta_rewritten1} is given by 
 \begin{equation} \label{eq:metricIIdivision}
  \begin{aligned}
   &-2 \Re \bigl\langle \ringa^{\mu \nu} \nabla_\nu \tilP_s u, (\nabla_\mu \bsbeta_s^{(\alpha), \lambda}) \nabla_\lambda \tilP_s u \bigr\rangle_{t,x} \\
   &\quad = -2 \Re \angles{\ringa^{r\nu}\nabla_\nu\tilP_su}{ (\partial_r\beta_s^{(\alpha)}) \nabla_r \tilP_su}_{t,x} \\
   &\quad \quad - 2 \Re\angles{\ringa^{\theta_a \nu} \nabla_\nu \tilP_su}{\coth(r) \beta_s^{(\alpha)} \nabla_{\theta_a} \tilP_s u}_{t,x} \\
   &\quad = - 2 \Re \angles{\ringa^{\mu\nu} (\partial_r \beta_s^{(\alpha)}) \nabla_\mu \tilP_su}{\nabla_\nu\tilP_su}_{t,x}\\
   &\quad \quad-2\Re\angles{\ringa^{\theta_a \theta_b}(\coth(r) \beta_s^{(\alpha)}-\partial_r\beta_s^{(\alpha)})\partial_{\theta_a}\tilP_su}{\partial_{\theta_b}\tilP_su}_{t,x}\\
   &\quad \quad-2\Re\angles{\ringa^{\theta_a r}(\coth(r) \beta_s^{(\alpha)}-\partial_r\beta_s^{(\alpha)})\partial_{\theta_a}\tilP_su}{\partial_r\tilP_su}_{t,x}.
  \end{aligned}
 \end{equation}
 Combining the previous identities, we obtain that
 \begin{equation}
  \begin{aligned}
   &\Re \langle i (H_\prin+\Delta) \tilP_{s} u, Q_s^{(\alpha)} \tilP_{s} u \rangle_{t,x} + 2 \langle \ringa^{\theta_a\theta_b} (\coth(r) \beta_s^{(\alpha)} - \partial_r \beta_s^{(\alpha)}) \partial_{\theta_a} \tilP_s u, \partial_{\theta_b} \tilP_s u \rangle_{t,x} \\
   &= \Re \langle (\nabla_\lambda \ringa^{\mu \nu}) \nabla_\mu \tilP_s u, \bsbeta_s^{(\alpha), \lambda} \nabla_\nu \tilP_s u \rangle_{t,x} \\
   &\quad - 2 \langle \ringa^{\mu \nu} (\partial_r \beta_s^{(\alpha)}) \nabla_\mu \tilP_s u, \nabla_\nu \tilP_s u \rangle_{t,x} \\  
   &\quad - 2 \Re \angles{\ringa^{\theta_a r}(\coth(r) \beta_s^{(\alpha)}-\partial_r\beta_s^{(\alpha)})\partial_{\theta_a}\tilP_su}{\partial_r\tilP_su}_{t,x} \\
   &\quad - \Re \langle \ringa^{\mu \nu} \nabla_\mu \tilP_s u, (\nabla_\nu \nabla_\lambda \bsbeta_s^{(\alpha), \lambda}) \tilP_s u \rangle_{t,x} \\
   &=: I + II + III + IV.
  \end{aligned}
 \end{equation}
 In order to finish the proof of Proposition~\ref{p:DeltaHpQhigh} we now show  
 \begin{equation} \label{equ:left_todo_prop_DeltaHpQhigh}
  \begin{aligned}
   &\int_0^2 \sup_{\{ \alpha_\ell \} \in \calA} |I + \ldots + IV| \, \ds \\ 
   &\quad \lesssim \biggl(  \bigl\| |\ringa^{r\cdot}|_{\slashed{\bsh}} \coth(r) \bigr\|_{L^\infty(\bbR \times A_{\leq 0})}  + \sum_{k=0}^1 \, \Bigl( \| \nabla^{(k)} \ringa \|_{L^\infty_{t,x}} + \sum_{\ell \geq 0} \| r \nabla^{(k)} \ringa \|_{L^\infty(\bbR \times A_\ell)} \Bigr) \biggr) \|u\|_{LE}^2,
  \end{aligned}
 \end{equation}
 where we recall the notation
 \begin{equation}
   |\ringa^{r\cdot}|_{\slashed{\bsh}}^2 := \bsh_{\theta_a\theta_b} \ringa^{r\theta_a} \ringa^{r\theta_b}. 
 \end{equation}
Note that the term involving $|\ringa^{r\cdot}|_{\slashed{\bsh}} \coth(r)$ is controlled by the vanishing condition \eqref{eq:vanishing_assumption-prin} and Taylor expansion at $r = 0$.

 We start off with the estimate for the term $I$. Using the uniform pointwise bound $\beta_s^{(\alpha)}(r) \lesssim s^{\frac{1}{2}}$ and Lemma~\ref{lem:PsLEs1}, it is straightforward to conclude that
 \begin{equation}
  \begin{aligned}
   \sup_{\{ \alpha_\ell \} \in \calA} \, |I| &\lesssim \Bigl( \| \nabla \ringa \|_{L^\infty_{t,x}} + \sum_{\ell \geq 0} \| r \nabla \ringa \|_{L^\infty(\bbR \times A_\ell)} \Bigr) s^{-\frac{1}{2}} \| s^{\frac{1}{2}} \nabla \tilP_s u \|_{LE_s}^2 \\
   &\lesssim \Bigl( \| \nabla \ringa \|_{L^\infty_{t,x}} + \sum_{\ell \geq 0} \| r \nabla \ringa \|_{L^\infty(\bbR \times A_\ell)} \Bigr) s^{-\frac{1}{2}} \| \tilP_{\frac{s}{4}} u \|_{LE_{\frac{s}{4}}}^2. 
  \end{aligned}
 \end{equation}
 Integrating over this inequality in $\ds$ from $0$ to $2$ establishes~\eqref{equ:left_todo_prop_DeltaHpQhigh} for the contribution of the term $I$. For the term $II$ we have by Lemma~\ref{lem:betaLErelation} and Lemma~\ref{lem:PsLEs1} that 
 \begin{equation}
  \begin{aligned}
   \sup_{\{ \alpha_\ell \} \in \calA} \, |II| &\lesssim \|\ringa\|_{L^\infty_{t,x}} \sup_{\{ \alpha_\ell \} \in \calA} \, s^{-1} \langle (\partial_r \beta_s^{(\alpha)}) |s^{\frac{1}{2}} \nabla \tilP_s u|, | s^{\frac{1}{2}} \nabla \tilP_s u| \rangle_{t,x} \\
   &\simeq \|\ringa\|_{L^\infty_{t,x}}  s^{-\frac{1}{2}} \bigl\| s^{\frac{1}{2}} \nabla \tilP_s u \bigr\|_{LE_s}^2 \\
   &\lesssim \|\ringa\|_{L^\infty_{t,x}}  s^{-\frac{1}{2}} \| \tilP_{\frac{s}{4}} u \bigr\|_{LE_{\frac{s}{4}}}^2. 
  \end{aligned}
 \end{equation}
 Then integrating in $\ds$ also furnishes~\eqref{equ:left_todo_prop_DeltaHpQhigh} for the contribution of the term~$II$. 
 Next, we turn to the term $III$ and first decompose it into
 \begin{align*}
  III &= - 2 \Re \angles{\ringa^{\theta_a r}(\coth(r) \beta_s^{(\alpha)}-\partial_r\beta_s^{(\alpha)})\partial_{\theta_a}\tilP_su}{\partial_r\tilP_su}_{t,x} \\
  &= - 2 \Re \angles{\ringa^{\theta_a r} \coth(r) \beta_s^{(\alpha)} \partial_{\theta_a}\tilP_su}{\partial_r\tilP_su}_{t,x} \\
  &\quad \, + 2 \Re \angles{\ringa^{\theta_a r} ( \partial_r\beta_s^{(\alpha)} ) \partial_{\theta_a}\tilP_su}{\partial_r\tilP_su}_{t,x} \\
  &=: III_1 + III_2.
 \end{align*}
 Here the contribution of the term $III_2$ to~\eqref{equ:left_todo_prop_DeltaHpQhigh} can be estimated in exactly the same manner as that of the term $II$. Moreover, using the uniform pointwise bounds $\beta_s^{(\alpha)}(r) \lesssim s^{\frac{1}{2}}$ for any $r > 0$ and $\coth(r) \lesssim 1$ for $r \gtrsim 1$ as well as Lemma~\ref{lem:PsLEs1}, one readily verifies that
 \begin{align*}
  \sup_{\{ \alpha_\ell \} \in \calA} \, |III_1| &\lesssim \Bigl( \bigl\| |\ringa^{r\cdot}|_{\slashed{\bsh}} \coth(r) \bigr\|_{L^\infty(\bbR \times A_{\leq 0})} + \sum_{\ell \geq 0} \| r \ringa \|_{L^\infty(\bbR \times A_\ell)} \Bigr) s^{-\frac{1}{2}} \| \tilP_{\frac{s}{4}} u \|_{LE_{\frac{s}{4}}}^2.
 \end{align*}
 Then integrating in $\ds$ also yields~\eqref{equ:left_todo_prop_DeltaHpQhigh} for the contribution of the term~$III_1$. 
 
 Finally, in order to estimate the contribution of the term $IV$, we first observe that due to the radial symmetry of the vector field $\bsbeta_s^{(\alpha)}$, the only non-zero component of $\nabla_\nu \nabla_\lambda \bsbeta_s^{(\alpha), \lambda}$ is
 \begin{align*}
  \nabla_r \nabla_\lambda \bsbeta_s^{(\alpha), \lambda} = \partial_r^2 \beta_s^{(\alpha)} + 2 \rho \bigl( \coth(r) \partial_r \beta_s^{(\alpha)} - \sinh^{-2}(r) \beta_s^{(\alpha)} \bigr).
 \end{align*}
 Correspondingly, we write 
 \begin{align*}
  IV &= - \Re \langle \ringa^{\mu r} \nabla_\mu \tilP_s u, (\partial_r^2 \beta_s^{(\alpha)}) \tilP_s u \rangle_{t,x} \\
  &\qquad - 2 \rho \Re \langle \ringa^{\mu r}\nabla_\mu \tilP_s u, ( \coth(r) \partial_r \beta_s^{(\alpha)} - \sinh^{-2}(r) \beta_s^{(\alpha)} ) \tilP_s u \rangle_{t,x} \\
  &=:IV_1 + IV_2.
 \end{align*}
 Using the following uniform pointwise bound
 \begin{equation}
  | (\partial_r^2 \beta_s^{(\alpha)})(r) | \lesssim \delta^2 s^{-\frac{1}{2}} \frac{\alpha(\delta s^{-\frac{1}{2}} r)}{\langle \delta s^{-\frac{1}{2}} r \rangle^2} \lesssim s^{-\frac{1}{2}} (\partial_r \beta_s^{(\alpha)})(r),
 \end{equation}
 we obtain for the term $IV_1$ by Lemma~\ref{lem:betaLErelation} and Lemma~\ref{lem:PsLEs1} that 
 \begin{align*}
  \sup_{\{\alpha_\ell\} \in \calA} \, |IV_1| &\lesssim \| \ringa \|_{L^\infty_{t,x}} \sup_{\{ \alpha_\ell \} \in \calA} \, \langle |\nabla \tilP_s u|, s^{-\frac{1}{2}} (\partial_r \beta_s^{(\alpha)}) |\tilP_s u| \rangle_{t,x} \\
  &\lesssim \| \ringa \|_{L^\infty_{t,x}} \sup_{\{ \alpha_\ell \} \in \calA} \, s^{-1} \Bigl( \bigl\| (\partial_r \beta_s^{(\alpha)})^{\frac{1}{2}} s^{\frac{1}{2}} \nabla \tilP_s u \bigr\|_{L^2_{t,x}}^2 + \bigl\| (\partial_r \beta_s^{(\alpha)})^{\frac{1}{2}} \tilP_s u \bigr\|_{L^2_{t,x}}^2 \Bigr) \\
  &\lesssim \| \ringa \|_{L^\infty_{t,x}} s^{-\frac{1}{2}} \bigl( \| s^{\frac{1}{2}} \nabla \tilP_s u \|_{LE_s}^2 + \| \tilP_s u \|_{LE_s}^2 \bigr) \\
  &\lesssim \| \ringa \|_{L^\infty_{t,x}} s^{-\frac{1}{2}} \| \tilP_{\frac{s}{4}} u \|_{LE_{\frac{s}{4}}}^2,
 \end{align*}
 which upon integrating in $\ds$ proves~\eqref{equ:left_todo_prop_DeltaHpQhigh} for the contribution of $IV_1$. Moreover, by invoking the uniform pointwise bounds
 \begin{equation}
  \bigl| \coth(r) (\partial_r \beta_s^{(\alpha)})(r) - \sinh^{-2}(r) \beta_s^{(\alpha)}(r) \bigr| \lesssim \begin{cases} s^{-\frac{1}{2}} & \text{for } r \leq s^{\frac{1}{2}}, \\ s^{\frac{1}{2}} r^{-2} & \text{for } s^{\frac{1}{2}} \leq r \lesssim 1, \\ 1 & \text{for } r \gtrsim 1, \end{cases}
 \end{equation}
 it follows easily that
 \begin{align*}
  \sup_{\{\alpha_\ell\} \in \calA} \, |IV_2| &\lesssim \biggl( \| \ringa \|_{L^\infty_{t,x}} + \sum_{\ell \geq 0} \| r \ringa \|_{L^\infty(\bbR \times A_\ell)} \biggr) s^{-\frac{1}{2}} \bigl( \| s^{\frac{1}{2}} \nabla \tilP_s u \|_{LE_s}^2 + \| \tilP_s u \|_{LE_s}^2 \bigr) \\
  &\lesssim \biggl( \| \ringa \|_{L^\infty_{t,x}} + \sum_{\ell \geq 0} \| r \ringa \|_{L^\infty(\bbR \times A_\ell)} \biggr) s^{-\frac{1}{2}} \| \tilP_{\frac{s}{4}} u \|_{LE_{\frac{s}{4}}}^2,
 \end{align*}
 which after integrating in $\ds$ also yields~\eqref{equ:left_todo_prop_DeltaHpQhigh} for the contribution of $IV_2$. This finishes the proof of Proposition~\ref{p:DeltaHpQhigh}.
 \end{proof}

\subsection{An elliptic regularity estimate}\label{ss:elliptic_reg}
Our goal in this section is to prove the following elliptic regularity estimate which will be used in Section~\ref{s:error}.

\begin{lem}[Elliptic regularity in $\tilLE_0$] \label{l:ell-LE}
Suppose $z\in \bbC$ satisfies $|z|\leq M$ for some fixed $M$. Assume also that $H$ as in~\eqref{eq:H_def} is symmetric and stationary with $H_\prin = -\Delta$ and that the potentials $\bsb, V$ satisfy~\eqref{equ:LE-H_thm_additional_assumption}. Then for sufficiently large $R>0$, and with $\chi_{\geq R}$ denoting a smooth cutoff to the region $\{r\geq R\}\subseteq \bbH^d$,
\begin{align*}
\begin{split}
&\|\chi_{\geq R}u\|_{\tilLE_0^1}\lesssim_M \|u\|_{\tilLE_0}+\|(H-z)v\|_{\tilLE_0^\ast}.
\end{split}
\end{align*}
\end{lem}

It turns out that in the proof it is more convenient to work with the modified space $\tilLE_{0,\ext}^1$ introduced in Section~\ref{s:prelim}, so we start with the following lemma.

\begin{lem}\label{lem:LE_LEext2}
If $R>1$ is sufficiently large, then for any function $v$
\begin{align*}
\begin{split}
&\|\chi_{\geq R}v\|_{\tilLE_0^1}\simeq\|\chi_{\geq R}v\|_{\tilLE_{0,\ext}^1}.
\end{split}
\end{align*}
\end{lem}
\begin{proof}
That $\|\chi_{\geq R}v\|_{\tilLE_0^1}\gtrsim\|\chi_{\geq R}v\|_{\tilLE_{0,\ext}^1}$ follows from Lemma~\ref{lem:LE_LEext1}. For the other direction, since the low-frequency components of $\tilLE_{0,\ext}^1$ and $\tilLE_{0}^1$ are the same, it suffices to prove
\begin{align}\label{eq:LEextgoal1}
\begin{split}
\int_0^{\frac{1}{2}}s^{-\frac{3}{2}}\|\tilP_s\chi_{\geq R}u\|_{\tilLE_{0,s}}^2\,\ds\lesssim \|\chi_{\geq R}u\|_{\tilLE_{0,\ext}^1}^2.
\end{split}
\end{align}
The left-hand side above is bounded by
\begin{align*}
\begin{split}
\int_0^{\frac{1}{2}}s^{-\frac{3}{2}-\frac{1}{2}}\|\tilP_s\chi_{\geq R}u\|_{L^2(A_{\leq-k_s})}^2 \, \ds &+\int_0^{\frac{1}{2}}s^{-\frac{3}{2}} \Bigl( \sup_{-k_s\leq\ell\leq0}2^{-\frac{\ell}{2}}\|\tilP_s\chi_{\geq R}u\|_{L^2(A_\ell)} \Big)^2\,\ds\\
&+\int_0^{\frac{1}{2}}s^{-\frac{3}{2}}\|r^{-\frac{3}{2}-\nsigma}\tilP_s\chi_{\geq R}u\|_{L^2(A_\geq0)}^2\,\ds\\
&=:I+II+III.
\end{split}
\end{align*}
By definition, $III\lesssim \|u\|_{\tilLE_{0,\ext}^1}^2$, so it remains to show that $I,II\lesssim \|u\|_{\tilLE_{0,\ext}^1}^2$. For $II$ note that for any $\ell$ with $-k_s\leq \ell\leq0$, and any $s_0\in[\frac{5}{16},\frac{7}{16}]$
\begin{align*}
\begin{split}
2^{-\frac{\ell}{2}}\|\tilP_s\chi_{\geq R}u\|_{L^2(A_\ell)}&=2^{-\frac{\ell}{2}}\|\tilP_s\chi_{\geq\frac{R}{2}}\chi_{\geq R}u\|_{L^2(A_\ell)}\\
&\leq 2^{-\frac{\ell}{2}}\|\tilP_s\chi_{\geq\frac{R}{2}}\tilP_{\geq s_0}\chi_{\geq R}u\|_{L^2(A_\ell)}\\
&\quad+2^{-\frac{\ell}{2}}\|\tilP_s\chi_{\geq\frac{R}{2}}\tilP_{\leq s_0}\chi_{\geq R}u\|_{L^2(A_\ell)}=:II_1+II_2.
\end{split}
\end{align*}
Therefore,
\begin{align*}
\begin{split}
\Big(\sup_{-k_s\leq\ell\leq0}2^{-\frac{\ell}{2}}\|\tilP_s\chi_{\geq R}u\|_{L^2(A_\ell)}\Big)^2\lesssim \int_{\frac{5}{16}}^{\frac{7}{16}}\Big(\big(\sup_{-k_s\leq\ell\leq0}II_1\big)^2+\big(\sup_{-k_s\leq\ell\leq0}II_2\big)^2\Big)\,\dsnot.
\end{split}
\end{align*}
Now by localized parabolic regularity, if $R$ is sufficiently large, and with $\chi_\ell$ a cutoff adapted to $A_\ell$,
\begin{align*}
\begin{split}
II_1&\lesssim 2^{-\frac{\ell}{2}}\|\chi_\ell\tilP_s\chi_{\geq \frac{R}{2}}\tilP_{\geq s_0}\chi_{\geq R}u\|_{L^2}\lesssim \sum_{2^m\geq \frac{R}{2}}2^{-\frac{\ell}{2}}\|\chi_\ell\tilP_s\chi_m\tilP_{\geq s_0}\chi_{\geq R}u\|_{L^2}\\
&\lesssim_N \sum_{2^m\geq \frac{R}{2}}2^{-\frac{\ell}{2}}2^{-\frac{m}{2}(N-3-2\nsigma)}s^{N}\|r^{-\frac{3}{2}-\nsigma}\chi_m\tilP_{\geq s_0} \chi_{\geq R} u\|_{L^2}\\
&\lesssim_N \|r^{-\frac{3}{2}-\nsigma}\tilP_{\geq s_0}u\|_{L^2(A_{\geq 0})} s^{N-\frac{1}{4}}\sup_{2^m\geq\frac{R}{2}}2^{-(\frac{N-3-2\nsigma}{2})m}\\
&\lesssim_N s^{N-\frac{1}{4}}R^{-(\frac{N-3-2\nsigma}{2})}\|\tilP_{\geq s_0}\chi_{\geq R}u\|_{\tilLE_{0,\low}}.
\end{split}
\end{align*}
It follows that
\begin{align*}
\begin{split}
\int_0^{\frac{1}{2}}\int_{\frac{5}{16}}^{\frac{7}{16}}s^{-\frac{3}{2}} \Big( \sup_{-k_s\leq\ell\leq0}II_1 \Big)^2 \,\dsnot\,\ds&\lesssim \int_{\frac{5}{16}}^{\frac{7}{16}}\int_0^{\frac{1}{2}}s^{2N-2}\|\tilP_{\geq s_0}\chi_{\geq R}u\|_{\tilLE_{0,\low}}^2\,\ds\,\dsnot\\
&\lesssim \int_{\frac{5}{16}}^{\frac{7}{16}}\|\tilP_{\geq s_0}\chi_{\geq R}u\|_{\tilLE_{0,\low}}^2\,\dsnot\\
&\lesssim \|\chi_{\geq R}u\|_{\tilLE_{0,\ext}^1}^2.
\end{split}
\end{align*}
Similarly, for $R$ sufficiently large
\begin{align*}
\begin{split}
II_2&\lesssim s^{-\frac{1}{4}}\int_0^{s_0}\|\chi_\ell\tilP_s\chi_{\geq\frac{R}{2}}\tilP_{s'}\chi_{\geq R}u\|_{L^2}\,\dsp\\
&\lesssim_N s^{N-\frac{1}{4}}\int_0^{s_0}\sum_{2^m\geq\frac{R}{2}}2^{-\frac{m}{2}(N-3-2\nsigma)}\|r^{-\frac{3}{2}-\nsigma}\chi_m\tilP_{s'}\chi_{\geq R}u\|_{L^2}\,\dsp\\
&\lesssim_N s^{N-\frac{1}{4}}R^{-(\frac{N-3-2\nsigma}{2})}\int_0^{s_0}\|\tilP_{s'}\chi_{\geq R}u\|_{\tilLE_{0,1}}\,\dsp\\
&\lesssim_N s^{N-\frac{1}{4}}R^{-(\frac{N-3-2\nsigma}{2})}s_0^{\frac{3}{4}}\|\chi_{\geq R}u\|_{\tilLE_{0,\ext}^1},
\end{split}
\end{align*}
so
\begin{align*}
\begin{split}
\int_0^{\frac{1}{2}}\int_{\frac{5}{16}}^{\frac{7}{16}}s^{-\frac{3}{2}}\big(\sup_{-k_s\leq\ell\leq0}II_2\big)^2\,\dsnot\,\ds&\lesssim\int_{\frac{5}{16}}^{\frac{7}{16}}\int_0^{\frac{1}{2}}s^{2N-2}\|\chi_{\geq R}u\|_{\tilLE_{0,\ext}^1}^2\,\ds\,\dsnot\\
&\lesssim \|\chi_{\geq R}u\|_{\tilLE_{0,\ext}^1}^2.
\end{split}
\end{align*}
This shows that $II\lesssim \|\chi_{\geq R}u\|_{\tilLE_{0,\ext}^1}^2$. The proof for $I$ is similar. In fact note that in estimating $II$ we simply bounded $2^{-\frac{\ell}{2}}$ by $s^{-\frac{1}{4}}$ so the same exact proof as above shows that $I\lesssim \|\chi_{\geq R}u\|_{\tilLE_{0,\ext}^1}^2$, completing the proof of \eqref{eq:LEextgoal1}.
\end{proof}

We can now prove Lemma~\ref{l:ell-LE}.

\begin{proof}[Proof of Lemma~\ref{l:ell-LE}]
Since $\| \chi_{\geq R} u \|_{\tilLE_0^1} \simeq \|\chi_{\geq R} u \|_{\tilLE_{0,\ext}^1}$ by Lemma~\ref{lem:LE_LEext2}, it suffices to prove 
 \begin{equation}
  \| \chi_{\geq R} u\|_{\tilLE_{0,\ext}^1}^2 \lesssim_M \|u\|_{\tilLE_0}^2 + \|(H-z) u\|_{\tilLE_0^\ast}^2.
 \end{equation}
 Moreover, since the low-frequency components of $\tilLE_{0,\ext}^1$ and $\tilLE_0$ coincide and since $\chi_{\geq R}$ is bounded in $\tilLE_0$ according to Lemma~\ref{lem:boundedness_LE_spatial_cutoff}, it suffices to consider the high-frequency component of $\tilLE_{0,\ext}^1$ and to prove that 
 \begin{equation}
  \int_0^{\frac{1}{2}} s^{-\frac{3}{2}} \| \tilP_s (\chi_{\geq R} u) \|_{\tilLE_{0,1}}^2 \, \ds \lesssim \| u \|_{\tilLE_0}^2 + \| (H-z) u \|_{\tilLE_0^\ast}^2.
 \end{equation}
 We use that 
 \begin{align*}
  \int_0^{\frac{1}{2}} s^{-\frac{3}{2}} \| \tilP_s ( \chi_{\geq R} u) \|_{\tilLE_{0,1}}^2 \, \ds \simeq \int_0^{\frac{1}{2}} s^{-\frac{3}{2}} \bigl\| \langle r \rangle^{-\frac{3}{2}-\delta} \tilP_s (\chi_{\geq R} u ) \bigr\|_{L^2}^2 \, \ds.
 \end{align*}
 Then we have 
 \begin{align*}
  &\frac{\ud}{\ud s} \Bigl( s^{-\frac{1}{2}} \bigl\| \langle r \rangle^{-\frac{3}{2}-\delta} \tilP_s ( \chi_{\geq R} u ) \bigr\|_{L^2}^2 \Bigr) \\
  &= \frac{3}{2} s^{-\frac{3}{2}} \bigl\| \langle r \rangle^{-\frac{3}{2}-\delta} \tilP_s ( \chi_{\geq R} u ) \bigr\|_{L^2}^2 \\
  &\quad + s^{-\frac{1}{2}} 2 \Re \bigl\langle \langle r \rangle^{-\frac{3}{2}-\delta} (\Delta + \rho^2) \tilP_s (\chi_{\geq R} u), \langle r \rangle^{-\frac{3}{2}-\delta} \tilP_s (\chi_{\geq R} u) \bigr\rangle,
 \end{align*}
 and thus
 \begin{align*}
  &\frac{3}{2} s^{-\frac{3}{2}} \bigl\| \langle r \rangle^{-\frac{3}{2}-\delta} \tilP_s ( \chi_{\geq R} u ) \bigr\|_{L^2}^2 + 2 \rho^2 s^{-\frac{1}{2}} \bigl\| \langle r \rangle^{-\frac{3}{2}-\delta} \tilP_s (\chi_{\geq R} u) \bigr\|_{L^2}^2 \\
  &= \frac{\ud}{\ud s} \Bigl( s^{-\frac{1}{2}} \bigl\| \langle r \rangle^{-\frac{3}{2}-\delta} \tilP_s (\chi_{\geq R} u) \bigr\|_{L^2}^2 \Bigr)\\
  &\quad - 2 s^{-\frac{1}{2}} \Re \bigl\langle \langle r \rangle^{-\frac{3}{2}-\delta} \Delta \tilP_s ( \chi_{\geq R} u ), \langle r \rangle^{-\frac{3}{2}-\delta} \tilP_s (\chi_{\geq R} u) \bigr\rangle.
 \end{align*}
 Integrating in $\ds$ (and dropping the term with $\rho^2$ which has a favorable sign), we obtain for any $s_0 \in [\frac{1}{2}, 1]$ that
 \begin{align*}
  &\frac{3}{2} \int_0^{s_0} s^{-\frac{3}{2}} \bigl\| \langle r \rangle^{-\frac{3}{2}-\delta} \tilP_s ( \chi_{\geq R} u ) \bigr\|_{L^2}^2 \, \ds \\
  &\leq s_0^{-\frac{1}{2}} \bigl\| \langle r \rangle^{-\frac{3}{2}-\delta} \tilP_{s_0} (\chi_{\geq R} u) \bigr\|_{L^2}^2 \\
  &\quad + 2\int_0^{s_0} s^{-\frac{1}{2}} \bigl| \bigl\langle \langle r \rangle^{-\frac{3}{2}-\delta} \Delta \tilP_s (\chi_{\geq R} u), \langle r \rangle^{-\frac{3}{2}-\delta} \tilP_s (\chi_{\geq R} u) \bigr\rangle \bigr| \, \ds .
 \end{align*}
 Averaging over $s_0 \in [\frac{1}{2}, 1]$ in $\frac{\ud s_0}{s_0}$ gives
 \begin{align*}
  &\int_0^{\frac{1}{2}} s^{-\frac{3}{2}}  \bigl\| \langle r \rangle^{-\frac{3}{2}-\delta} \tilP_s ( \chi_{\geq R} u ) \bigr\|_{L^2}^2 \, \ds \\
  &\lesssim \int_{\frac{1}{2}}^1 s_0^{-\frac{1}{2}} \bigl\| \langle r \rangle^{-\frac{3}{2}-\delta} \tilP_{s_0} ( \chi_{\geq R} u ) \bigr\|_{L^2}^2 \, \frac{\ud s_0}{s_0} \\
  &\quad + \int_0^1 s^{-\frac{1}{2}} \bigl| \bigl\langle \langle r \rangle^{-\frac{3}{2}-\delta} \Delta \tilP_s (\chi_{\geq R} u), \langle r \rangle^{-\frac{3}{2}-\delta} \tilP_s (\chi_{\geq R} u) \bigr\rangle \bigr| \, \ds \\
  &=: I + II.
 \end{align*}
 For term $I$ we use the boundedness of $\frac{s_0}{2}(\Delta+\rho^2) e^{\frac{s_0}{2}(\Delta+\rho^2)}$ in $\tilLE_{0,\low}$ to conclude
 \begin{align*}
  I &\lesssim \int_{\frac{1}{4}}^{\frac{1}{2}} \bigl\| \langle r \rangle^{-\frac{3}{2}-\delta} \tilP_{\geq s_0} ( \chi_{\geq R} u ) \bigr\|_{L^2}^2 \, \frac{\ud s_0}{s_0}\\
  & \lesssim \int_{\frac{1}{8}}^4 \| \tilP_{\geq s_0} (\chi_{\geq R} u) \bigr\|_{\tilLE_{0,\low}}^2 \, \frac{\ud s_0}{s_0} \lesssim \|u\|_{\tilLE_0}^2.
 \end{align*}
 For term $II$ we insert
 \[
  \Delta = -(H-z) + H_\lot - z,
 \]
 and correspondingly obtain
 \begin{align*}
  II &\lesssim \int_0^1 s^{-\frac{1}{2}} \bigl| \bigl\langle \langle r \rangle^{-\frac{3}{2}-\delta} \tilP_s \bigl( (H-z) \chi_{\geq R} u \bigr), \langle r \rangle^{-\frac{3}{2}-\delta} \tilP_s \bigl( \chi_{\geq R} u \bigr) \bigr\rangle \bigr| \, \ds \\
  &\quad + \int_0^1 s^{-\frac{1}{2}} \bigl| \bigl\langle \langle r \rangle^{-\frac{3}{2}-\delta} \tilP_s \bigl( H_\lot \chi_{\geq R} u \bigr), \langle r \rangle^{-\frac{3}{2}-\delta} \tilP_s \bigl( \chi_{\geq R} u \bigr) \bigr\rangle \bigr| \, \ds \\
  &\quad + |z| \int_0^1 s^{-\frac{1}{2}} \bigl| \bigl\langle \langle r \rangle^{-\frac{3}{2}-\delta} \tilP_s \bigl( \chi_{\geq R} u \bigr), \langle r \rangle^{-\frac{3}{2}-\delta} \tilP_s \bigl( \chi_{\geq R} u \bigr) \bigr\rangle \bigr| \, \ds \\
  &=: II_1 + II_2 + II_3.
 \end{align*}
 In order to estimate term $II_1$, we take the absolute values inside, drop the weights and then bound the inner product by pairing $\tilLE_{0,s}^\ast$ and $\tilLE_{0,s}$ to find that
 \begin{align*}
  II_1 &\lesssim \int_0^1 s^{-\frac{1}{2}} \bigl\langle  |\tilP_s \bigl( (H-z) \chi_{\geq R} u \bigr) |, | \tilP_s \bigl( \chi_{\geq R} u \bigr) | \bigr\rangle \, \ds \\
  &\lesssim \int_0^1 s^{\frac{1}{4}} \bigl\| \tilP_s \bigl( (H-z) \chi_{\geq R} u \bigr) \bigr\|_{\tilLE_{0,s}^\ast} s^{-\frac{3}{4}} \bigl\| \tilP_s \bigl( \chi_{\geq R} u \bigr) \bigr\|_{\tilLE_{0,s}} \, \ds \\
  &\lesssim C_\varepsilon \int_0^1 s^{\frac{1}{2}} \bigl\| \tilP_s \bigl( (H-z) \chi_{\geq R} u \bigr) \bigr\|_{\tilLE_{0,s}^\ast}^2 \, \ds + \varepsilon \int_0^1 s^{-\frac{3}{2}} \bigl\| \tilP_s \bigl( \chi_{\geq R} u \bigr) \bigr\|_{\tilLE_{0,s}}^2 \, \ds \\
  &\lesssim C_\varepsilon \int_0^1 s^{\frac{1}{2}} \bigl\| \tilP_{\frac{s}{2}} \bigl( (H-z) \chi_{\geq R} u \bigr) \bigr\|_{\tilLE_{0,\frac{s}{2}}^\ast}^2 \, \ds + \varepsilon \int_0^1 s^{-\frac{3}{2}} \bigl\| \tilP_{\frac{s}{2}} \bigl( \chi_{\geq R} u \bigr) \bigr\|_{\tilLE_{0,\frac{s}{2}}}^2 \, \ds \\
  &\lesssim C_\varepsilon \int_0^{\frac{1}{2}} s^{\frac{1}{2}} \bigl\| \tilP_s \bigl( (H-z) \chi_{\geq R} u \bigr) \bigr\|_{\tilLE_{0,s}^\ast}^2 \, \ds + \varepsilon \int_0^{\frac{1}{2}} s^{-\frac{3}{2}} \bigl\| \tilP_s \bigl( \chi_{\geq R} u \bigr) \bigr\|_{\tilLE_{0,s}}^2 \, \ds \\
  &\lesssim C_\varepsilon \| (H-z) (\chi_{\geq R} u) \|_{\tilLE_0^\ast}^2  + \varepsilon \| \chi_{\geq R} u \|_{\tilLE_0^1}^2.
 \end{align*}
 Then to reabsorb the second term on the right-hand side we again use that 
 \[
  \| \chi_{\geq R} u \|_{\tilLE_0^1}^2 \simeq \| \chi_{\geq R} u \|_{\tilLE_{0, \ext}^1}^2.
 \]
 Finally, we further estimate the first term on the right-hand side by
 \begin{align*}
  \| (H-z) (\chi_{\geq R} u) \|_{\tilLE_0^\ast} &\lesssim \| \chi_{\geq R} (H-z) u \|_{\tilLE_0^\ast} + \| [H, \chi_{\geq R}] u \|_{\tilLE_0^\ast} \\
  &\lesssim \| (H-z) u \|_{\tilLE_0^\ast} + C(\bsb, R) \|u\|_{\tilLE_0}.
 \end{align*}
 We estimate the term $II_2$ analogously to the term $II_1$ and obtain that
 \begin{align*}
  II_2 &\lesssim C_\varepsilon \| H_\lot (\chi_{\geq R} u) \|_{\tilLE_0^\ast}^2 + \varepsilon \| \chi_{\geq R} u \|_{\tilLE_0^1}^2 \lesssim \|u\|_{\tilLE_0}^2+\varepsilon \| \chi_{\geq R} u \|_{\tilLE_0^1}^2.
 \end{align*}
  Finally, we bound $II_3$ as 
 \begin{align*}
  II_3 &= |z| \int_0^1 s^{-\frac{1}{2}} \bigl\| \langle r \rangle^{-\frac{3}{2}-\delta} \tilP_s \bigl( \chi_{\geq R} u \bigr) \bigr\|_{L^2}^2 \, \ds \\
  &\lesssim_M \int_0^1 s^{-\frac{1}{2}} \bigl\| \tilP_s \bigl( \chi_{\geq R} u \bigr) \bigr\|_{\tilLE_{0,1}}^2 \, \ds \\
  &\lesssim \int_0^1 s^{-\frac{1}{2}} \bigl\| \tilP_s \bigl( \chi_{\geq R} u \bigr) \bigr\|_{\tilLE_{0,s}}^2 \, \ds \\
  &\lesssim \int_0^1 s^{-\frac{1}{2}} \bigl\| \tilP_{\frac{s}{2}} \bigl( \chi_{\geq R} u \bigr) \bigr\|_{\tilLE_{0,\frac{s}{2}}}^2 \, \ds \\
  &\lesssim \int_0^{\frac{1}{2}} s^{-\frac{1}{2}} \bigl\| \tilP_s \bigl( \chi_{\geq R} u \bigr) \bigr\|_{\tilLE_{0,s}}^2 \, \ds \\
  &\lesssim \|\chi_{\geq R} u\|_{\tilLE_0}^2 \lesssim \|u\|_{\tilLE_0}^2.
 \end{align*}
\end{proof}

In Section~\ref{s:error} we will need to estimate $b\partial_ru$ in $\tilLE_0$ for certain bounded radial functions~$b$. The following lemma allows us to use the elliptic estimate above for this purpose. Here it is crucial that we work with the spaces $\tilLE_0$ and $\tilLE_0^1$ instead of $LE_0$ and $LE_0^1$.

\begin{lem}\label{l:drLE}
Suppose $b$ is a bounded radial function in $\bbH^d$ supported in $\{r\geq R\}$ with bounded first and second order derivatives, where $R>1$ is sufficiently large. Then 
\begin{align*}
\begin{split}
\|b\partial_r u\|_{\tilLE_0}\lesssim \|u\|_{\tilLE_0^1}.
\end{split}
\end{align*}
\end{lem}
\begin{proof}
 The proof is similar to that of Proposition~\ref{p:Hlotbound}, so we omit some of the details. We need to prove the following two estimates
 \begin{align*}
  \int_0^{\frac{1}{2}}s^{-\frac{1}{2}}\|\tilP_s b\partial_ru\|_{\tilLE_{0,s}}^2\,\ds \lesssim \|u\|_{\tilLE_0^1}^2,\\
  \int_{\frac{1}{8}}^4\|\tilP_{\geq s}b\partial_r u\|_{\tilLE_{0,\low}}^2\,\ds \lesssim \|u\|_{\tilLE_0^1}^2.
 \end{align*}
 We provide the details only for the first estimate. Using an averaging argument as in the proof of Proposition~4.16, for any $s_0 \in [\frac{3}{4}, 1]$ we decompose into
 \begin{align*}
  \int_0^{\frac{1}{2}}s^{-\frac{1}{2}}\|\tilP_s b\partial_ru\|_{\tilLE_{0,s}}^2\,\ds&\lesssim\int_0^{\frac{1}{2}}s^{-\frac{1}{2}}\|\tilP_s b\partial_r\tilP_{\leq s}u\|_{\tilLE_{0,s}}^2\,\ds\\
  &\quad +\int_{\frac{3}{4}}^1\int_0^{\frac{1}{2}}s^{-\frac{1}{2}}\|\tilP_s b\partial_r\tilP_{s\leq\cdot\leq s_0}u\|_{\tilLE_{0,s}}^2\,\ds\,\dsnot\\
  &\quad +\int_{\frac{3}{4}}^1\int_0^{\frac{1}{2}}s^{-\frac{1}{2}}\|\tilP_s b\partial_r\tilP_{\geq s_0}u\|_{\tilLE_{0,s}}^2\,\ds\,\dsnot\\
  &=:I+II+III.
 \end{align*}
 We provide the details only for estimating $I$ and $III$. For term $III$ note that since $b$ is supported in $\{r\geq R\}$, we have 
 \begin{align*}
  \|(\Delta+\rho^2) b\partial_r\tilP_{\geq s_0}u\|_{\tilLE_{0,s}}=\|(\Delta+\rho^2) b\partial_r\tilP_{\geq s_0}u\|_{\tilLE_{0,\low}},
 \end{align*} 
 and therefore 
 \begin{align*}
   s^{-\frac{1}{2}}\|\tilP_sb\partial_r\tilP_{\geq s_0}u\|_{\tilLE_{0,s}}^2 &= s^{\frac{1}{2}}\|\tilP_{\geq s} (\Delta + \rho^2) b \partial_r \tilP_{\geq s_0}u\|_{\tilLE_{0,s}}^2 \\
   &\lesssim s^{\frac{1}{2}}\| (\Delta + \rho^2) b\partial_r\tilP_{\geq s_0}u\|_{\tilLE_{0,s}}^2\\
   &\simeq s^{\frac{1}{2}} \| (\Delta+\rho^2) b \partial_r \tilP_{\geq s_0} u \|_{\tilLE_{0,\low}}^2 \\
   &\lesssim_{\bsb, s_0} s^{\frac{1}{2}}\|\tilP_{\geq \frac{s_0}{2}}u\|_{\tilLE_{0,\low}}^2.
 \end{align*}
 It follows that
 \begin{align*}
  III \lesssim \int_{\frac{3}{8}}^{\frac{1}{2}}\|\tilP_{\geq s_0}u\|_{\tilLE_{0,\low}}^2\,\dsnot\leq \|u\|_{\tilLE_0^1}^2.
 \end{align*}
 In order to treat term $I$, again in view of the spatial support of $b$, for any $0 < s' \leq s \leq \frac{1}{2}$ we have
 \begin{align*}
  \|b\partial_r\tilP_{s'}u\|_{\tilLE_{0,s}}=\|b\partial_r\tilP_{s'}u\|_{\tilLE_{0,s'}},
 \end{align*}
 and thus,
 \begin{align*}
  I &\lesssim \int_0^{\frac{1}{2}} \Big(\int_0^s \bigl( \frac{s'}{s}\bigr)^{\frac{1}{4}} (s')^{-\frac{1}{4}} \|b\partial_r\tilP_{s'}u\|_{\tilLE_{0,s}}\,\ds\Big)^2\,\ds \\
  &\lesssim \int_0^{\frac{1}{2}} \Big( \int_0^s \bigl( \frac{s'}{s} \bigr)^{\frac{1}{4}} (s')^{-\frac{3}{4}} \| (s')^{\frac{1}{2}}\nabla\tilP_{s'}u\|_{\tilLE_{0,s'}}\,\ds\Big)^2\,\ds \\
  &\lesssim \int_0^{\frac{1}{2}} \Big( \int_0^s \bigl( \frac{s'}{s} \bigr)^{\frac{1}{4}} (s')^{-\frac{3}{4}} \|\tilP_{\frac{s'}{2}}u\|_{\tilLE_{0,\frac{s'}{2}}}\,\ds\Big)^2\,\ds \\
  &\lesssim \|u\|_{\tilLE_0^1}^2,
 \end{align*}
where in the last step we used Schur's test. 
\end{proof}

\subsection{Estimates needed to handle time-dependent perturbations}\label{ss:pert-reg}

Here we collect a number of estimates that are needed to treat time-dependent perturbations of a stationary operator. More precisely, assume that $H$ is of the form
\begin{equation*}
	H = H_{\stat} + H_{\pert}
\end{equation*}
where $H_{\stat}$ is a symmetric, stationary Hamiltonian 
\begin{equation*}
	H_{\stat} u = - \Delta + \frac{1}{i} (\bsb^{\mu} \nb_{\mu} u + \nb_{\mu} (\bsb^{\mu} u)) + Vu,
\end{equation*}
and $H_{\pert}$ is a perturbation of the form
\begin{equation*}
	H_{\pert} u = - \nb_{\mu} (\bsa_{\pert}^{\mu \nu} \nb_{\nu} u) + \frac{1}{i} (\bsb_{\pert}^{\mu} \nb_{\mu} u + \nb_{\mu} (\bsb_{\pert}^{\mu} u)) + V_{\pert} u
\end{equation*}
where $\bsa_{\pert}, \bsb_{\pert}, V_{\pert}$ are possibly complex-valued, time-dependent coefficients satisfying 
\EQ{ \label{eq:pert-decay-sp}
 &\sum_{k=0}^4 \|\jap{r}^{3+2\nsigma}\nabla^{(k)}\bsa_{\pert}\|_{L^\infty_{t,x}}  \leq \kpp, \\
 &\sum_{k=0}^4 \|\jap{r}^{3+2\nsigma}\nabla^{(k)}\bsb_{\pert} \|_{L^\infty_{t,x}} \leq \kpp, \\
 &\sum_{k=0}^4 \|\jap{r}^{3+2\nsigma}\nabla^{(k)}V_{\pert}\|_{L^\infty_{t,x}} \leq \kpp,
}
with size $\kpp > 0$.

\begin{lem}

Suppose $H_{\stat}$ and $H_{\pert}$ are as above and that the bounds \eqref{eq:pert-decay-sp} hold. Then for sufficiently small $0 < s_0 < \frac{1}{2}$
\begin{align}
\nrm{[\tilP_{\geq s_{0}}, H_{\stat}] u}_{\tilLE^{\ast}} &\aleq s_{0}^{\frac{1}{2}} \nrm{u}_{\tilLE}. \label{eq:comm-stat}\\
\nrm{\tilP_{\geq s_{0}} H_{\pert} u}_{\tilLE^{\ast}} &\aleq \kpp (1+s_{0}^{-1}) \nrm{u}_{\tilLE}. \label{eq:smooth-pert}\\
\nrm{\tilP_{\geq s_{0}} F}_{\tilLE^{\ast}} &\aleq \nrm{F}_{\tilLE^{\ast}}.\label{eq:LEs-heat}
\end{align}
\end{lem}
\begin{proof}
We begin with \eqref{eq:LEs-heat}. Starting from the definition of $\tilLE^\ast$, we have 
 \begin{align*}
  \| \tilP_{\geq s_0} F \|_{\tilLE^\ast}^2 &= \int_0^{\frac{1}{2}} s^{\frac{1}{2}} \| \tilP_{\geq s_0} \tilP_{s} F \|_{\tilLE^\ast_s}^2 \, \ds + \int_{\frac{1}{8}}^4 \| \tilP_{\geq s_0} \tilP_{\geq s} F \|_{\tilLE_\low^\ast}^2 \, \ds \\
  &= \int_0^{s_0} s^{\frac{1}{2}} \| \tilP_{\geq s_0} \tilP_{s} F \|_{\tilLE^\ast_s}^2 \, \ds + \int_{s_0}^{\frac{1}{2}} s^{\frac{1}{2}} \| \tilP_{\geq s_0} \tilP_{s} F \|_{\tilLE^\ast_s}^2 \, \ds \\
  &\quad+ \int_{\frac{1}{8}}^4 \| \tilP_{\geq s_0} \tilP_{\geq s} F \|_{\tilLE_\low^\ast}^2 \, \ds \\
  &\lesssim \int_0^{s_0} s^{\frac{1}{2}} \| \tilP_{\geq s_0} \tilP_{s} F \|_{\tilLE^\ast_s}^2 \, \ds + \int_{s_0}^{\frac{1}{2}} s^{\frac{1}{2}} \| \tilP_{s} F \|_{\tilLE^\ast_s}^2 \, \ds\\
  &\quad + \int_{\frac{1}{8}}^4 \| \tilP_{\geq s} F \|_{\tilLE_\low^\ast}^2 \, \ds,
 \end{align*}
 where in the last step in the last two integrals we just used the boundedness of $\tilP_{\geq s_0}$ on $\tilLE^\ast_s$ (for $s \geq s_0$) and on $\tilLE^\ast_\low$ as already proved in~Lemma~\ref{lem:PsLEs1}. For the first integral on the right-hand side, note that
 \[
  \tilP_{\geq s_0} \tilP_{s} = \frac{s}{s+s_0} \tilP_{s+s_0}
 \]
 and that by the definition of $\tilLE_s^\ast$, 
 \begin{align}\label{eq:LEs-LEss0comp}
  \| G \|_{\tilLE_s^\ast} \lesssim \| G \|_{\tilLE_{s+s_0}^\ast}.
 \end{align}
 Then we obtain by a simple change of variables ($s' := s + s_0$) that
 \begin{align*}
  \int_0^{s_0} s^{\frac{1}{2}} \| \tilP_{\geq s_0} \tilP_{s} F \|_{\tilLE^\ast_s}^2 \, \ds &= \int_0^{s_0} s^{\frac{1}{2}} \Bigl( \frac{s}{s+s_0} \Bigr)^2 \| \tilP_{s+s_0} F \|_{\tilLE_s^\ast}^2 \, \ds \\
  &\lesssim \int_0^{s_0} s^{\frac{1}{2}} \Bigl( \frac{s}{s+s_0} \Bigr)^2 \| \tilP_{s+s_0} F \|_{\tilLE_{s+s_0}^\ast}^2 \, \ds \\
  &\simeq \int_{s_0}^{2s_0} (s'-s_0)^{\frac{1}{2}}  \Bigl( \frac{s'-s_0}{s'} \Bigr)^2 \| \tilP_{s'} F \|_{\tilLE_{s'}^\ast}^2 \, \frac{\ud s'}{s'-s_0} \\
  &\simeq \int_{s_0}^{2s_0} \Bigl( \frac{s'-s_0}{s'} \Bigr)^{\frac{3}{2}} (s')^{\frac{1}{2}} \| \tilP_{s'} F \|_{\tilLE_{s'}^\ast}^2 \, \frac{\ud s'}{s'} \\
  &\lesssim \int_{s_0}^{2s_0} (s')^{\frac{1}{2}} \| \tilP_{s'} F \|_{\tilLE_{s'}^\ast}^2 \, \frac{\ud s'}{s'},
 \end{align*}
 which suffices and finishes the proof.
 
We turn to \eqref{eq:smooth-pert}. First we observe that by the boundedness of $\tilP_{\geq s_0}$ on $\tilLE^\ast$ according to \eqref{eq:LEs-heat}, for the lower order terms we can just invoke Proposition~\ref{p:Hlotbound} and Proposition~\ref{p:Hlotbound_for_V} to obtain that (here $B_{\pert}$ denotes the first order part of $H_{\pert}$ and $C$ is as in Proposition~\ref{p:Hlotbound} and Proposition~\ref{p:Hlotbound_for_V})
 \begin{align*}
  \| \tilP_{\geq s_0} ( B_{\pert} u + V_{\pert} u ) \|_{\tilLE^\ast} &\lesssim \| B_{\pert} u + V_{\pert} u \|_{\tilLE^\ast} \\
  &\lesssim \bigl( C(B_{\pert}) + C(V_{\pert}) \bigr) \|u\|_{\tilLE} \\
  &\lesssim \kappa \|u\|_{\tilLE}.
 \end{align*}
 It therefore remains to deal with the principal part $- \nabla_\mu ( \bsa_{\pert}^{\mu \nu} \nabla_\nu u )$. To simplify notation we let $\ringH_{\prin}$ denote this principal part. Then it suffices to prove the estimates
\begin{align*}
\begin{split}
\int_{\frac{1}{8}}^4 \|\tilP_{\geq s+s_0}\ringH_{\prin}u\|_{\tilLE_\low^*}^2\,\ds &\lesssim \kappa \|u\|_{\tilLE}^2,\\
\int_{0}^{\frac{1}{2}} s^{\frac{1}{2}}\|\tilP_{s}\tilP_{\geq s_0}\ringH_\prin u\|_{\tilLE_{s}^\ast}^2\,\ds &\lesssim \kappa (1+s_0^{-1})\|u\|_{\tilLE}^2.
\end{split}
\end{align*}
Now note that $\|\tilP_{\geq s+s_0}\ringH_{\prin}u\|_{\tilLE_\low^\ast}\lesssim \|P_{\geq s}\ringH_{\prin}u\|_{\tilLE_\low^\ast}$, so the first estimate follows similarly to the proof of \eqref{eq:lowprintemp1} (note that the commutator structure was not used in the proof of \eqref{eq:lowprintemp1}). For the second estimate first note that using \eqref{eq:LEs-LEss0comp}
\begin{align*}
\begin{split}
s^{\frac{1}{2}}\|\tilP_{s}\tilP_{\geq s_0}\ringH_\prin u\|_{\tilLE_{s}^\ast}^2&= (\frac{s}{s+s_0})^{\frac{5}{2}}(s+s_0)^{\frac{1}{2}}\|\tilP_{s+s_0}\ringH_{\prin}u\|_{\tilLE_s^\ast}^2\\
&\lesssim (s+s_0)^{\frac{1}{2}}\|\tilP_{s+s_0}\ringH_{\prin}u\|_{\tilLE_{s+s_0}^\ast}^2\\
&\lesssim (s+s_0)^{\frac{1}{2}}\|\tilP_{\frac{s+s_0}{2}}\ringH_{\prin}u\|_{\tilLE_{\frac{s+s_0}{2}}^\ast}^2\\
&\lesssim (s+s_0)^{\frac{1}{2}}\|\ringH_{\prin}\tilP_{\frac{s+s_0}{2}}u\|_{\tilLE_{\frac{s+s_0}{2}}^\ast}^2\\
&\quad +(s+s_0)^{\frac{1}{2}}\|[\tilP_{\frac{s+s_0}{2}},\ringH_{\prin}]u\|_{\tilLE_{\frac{s+s_0}{2}}^\ast}^2\\
&\lesssim \kappa(s+s_0)^{-\frac{3}{2}}\|\tilP_{\frac{s+s_0}{4}}u\|_{\tilLE_{\frac{s+s_0}{4}}}^2\\
&\quad +(s+s_0)^{\frac{1}{2}}\|[\tilP_{\frac{s+s_0}{2}},\ringH_{\prin}]u\|_{\tilLE_{\frac{s+s_0}{2}}^\ast}^2.
\end{split}
\end{align*}
Noting that $(s+s_0)^{-\frac{3}{2}}\leq s_0^{-1}(s+s_0)^{-\frac{1}{2}}$ it follows that
\begin{align*}
\begin{split}
\int_{0}^{\frac{1}{2}} s^{\frac{1}{2}}\|\tilP_{s}\tilP_{\geq s_0}\ringH_\prin u\|_{\tilLE_{s}^\ast}^2\,\ds&\lesssim \kappa s_0^{-1}\int_{s_0}^{\frac{1}{2}}s^{-\frac{1}{2}}\|\tilP_su\|_{\tilLE_s}^2\,\ds\\
&\quad +\int_{s_0}^{\frac{1}{2}}s^{\frac{1}{2}}\|[\tilP_s,\ringH_{\prin}]u\|_{\tilLE_s^\ast}^2\,\ds.
\end{split}
\end{align*}
The first term on the right is bounded by $\kappa s_0^{-1}\|u\|_{\tilLE}^2$.  Similarly as  in the proof of Proposition~\ref{p:commP_sH_prin}, the second term is bounded $\kappa \|u\|_{\tilLE}^2$. This completes the proof of \eqref{eq:smooth-pert}.
 
Finally we prove \eqref{eq:comm-stat}.  Let
\begin{align*}
\begin{split}
w(s):=[\tilP_{\geq s},H_\lot]u,
\end{split}
\end{align*}
where $H_\lot$ denotes the lower order part of $H_\stat$, that is $H_\stat =-\Delta+H_\lot$. Since $\tilP_{\geq s_0}$ commutes with $\Delta$, the estimate we want to prove is
\begin{align*}
\begin{split}
\|w(s_0)\|_{\tilLE^\ast}\lesssim s_0^{\frac{1}{2}}\|u\|_{\tilLE}.
\end{split}
\end{align*}
A direct computation shows that $w$ satisfies the shifted heat equation
\begin{align*}
\begin{split}
\partial_s w(s)-(\Delta+\rho^2)w(s)= [\Delta,H_{\lot}]\tilP_{\geq s}u.
\end{split}
\end{align*}
Let $\tilB = [\Delta,H_{\lot}]$ and note that $\tilB$ is a differential operator of order two with coefficients that satisfy similar bounds to \eqref{eq:pert-decay-sp}. By Duhamel's principle
\begin{align*}
\begin{split}
w(s_0)=\int_0^{s_0}e^{(s_0-s')(\Delta+\rho^2)}s'\tilB\tilP_{\geq s'}u\,\dsp,
\end{split}
\end{align*}
and
\begin{align*}
\begin{split}
\|w(s_0)\|_{\tilLE^\ast}^2&=\int_{0}^{\frac{1}{2}}s^{\frac{1}{2}}\Big\|\int_{0}^{s_0}\tilP_s\tilP_{\geq s_0-s'}s'\tilB\tilP_{\geq s'}u\,\dsp\Big\|_{\tilLE_s^\ast}^2\,\ds\\
&\quad +\int_{\frac{1}{8}}^{4}\Big\|\int_{0}^{s_0}\tilP_{\geq s}\tilP_{\geq s_0-s'}s'\tilB\tilP_{\geq s'}u\,\dsp\Big\|_{\tilLE_\low^\ast}^2\,\ds=:I+II.
\end{split}
\end{align*}
Below we will repeatedly use the fact that by successive applications of the product rule we may write $\tilB$ in either of the generic forms
\begin{align*}
\begin{split}
\tilB=\tilde{\bsb}_1^{\mu\nu}\nabla_\mu\nabla_\nu+\tilde{\bsb}_2^\mu\nabla_\mu+\tilV_1
\end{split}
\end{align*}
or
\begin{align*}
\begin{split}
\tilB=\nabla_\mu\nabla_\nu\tilde{\bsb}_3^{\mu\nu}+\nabla_\mu\tilde{\bsb}_4^\mu+\tilV_2.
\end{split}
\end{align*}
To treat $I$ we use the fundamental theorem of calculus to write
\begin{align}\label{eq:9.1temp1}
\begin{split}
\tilP_{\geq s'}u=\tilP_{\geq s'+\frac{1}{4}}u+\int_{s'}^{s'+\frac{1}{4}}\tilP_{s''}u\dspp.
\end{split}
\end{align}
In view of \eqref{eq:LEs-LEss0comp}, the contribution of the integral on the right-hand side above to $I$ is bounded by
\begin{align*}
\begin{split}
&\int_0^{\frac{1}{2}}s^{\frac{1}{2}}\Big(\int_0^{s_0}\int_{s'}^{s'+\frac{1}{4}}\|\tilP_s\tilP_{\geq s_0-s'}s'\tilB \tilP_{s''}u\|_{\tilLE_{s}^\ast}\,\dspp\,\dsp\Big)^2\,\ds\\
&\lesssim \int_0^{\frac{1}{2}}s^{\frac{1}{2}}\Big(\int_0^{s_0}\int_{s'}^{s'+\frac{1}{4}}\frac{ss'}{s+s_0-s'}\|\tilP_{s +s_0-s'}\tilB \tilP_{s''}u\|_{\tilLE_{s+s_0-s'}^\ast}\,\dspp\,\dsp\Big)^2\,\ds\\
&\lesssim\int_0^{\frac{1}{2}}\Big(\int_0^{s_0}\int_{\frac{s'}{2}}^{\frac{s'}{2}+\frac{1}{8}}\frac{s^{\frac{5}{4}}s'\min\{(s+s_0-s')^{-1},(s'')^{-1}\}}{s+s_0-s'}\| \tilP_{s''}u\|_{\tilLE_{s''}}\,\dspp\,\dsp\Big)^2\,\ds\\
&=s_0\int_0^{\frac{1}{2}}\Big(\int_0^{\frac{s_0}{2}+\frac{1}{8}}k(s,s'')(s'')^{-\frac{1}{4}}\|\tilP_{s''}u\|_{\tilLE_{s''}}\,\dspp\Big)^2\,\ds,
\end{split}
\end{align*}
where
\begin{align*}
\begin{split}
k(s,s'')=\int_0^{\min\{2s'',s_0\}}\frac{s^{ \frac{5}{4} }s' s''^{ \frac{1}{4} } }{ s_0^{\frac{1}{2}}(s+s_0-s')}\min\{(s+s_0-s')^{-1},(s'')^{-1}\}\,\dsp.
\end{split}
\end{align*}
By Schur's test the contribution of this term is under control if we can show that
\begin{align*}
\begin{split}
\int_0^{\frac{1}{2}}k(s,s'')\,\ds+\int_0^{\frac{s_0}{2}+\frac{1}{8}}k(s,s'')\,\dspp\lesssim 1
\end{split}
\end{align*}
independently of $s_0$. We start with the $\dspp$ integral. Noting that $$\min\{a,b\}\leq a^{\alpha}b^{1-\alpha}$$ for any $\alpha\in[0,1]$ we bound
\begin{align*}
\begin{split}
\int_{0}^{\frac{s_0}{2}}k(s,s'')\,\dspp&\leq\int_0^{\frac{s_0}{2}}\int_0^{2s''}\frac{s^{\frac{5}{4}}\,\ud s'\,\ud s''}{s_0^{\frac{1}{2}}(s'')^{\frac{3}{4}}(s+s_0-s')^2}\\
&\lesssim \int_0^{\frac{s_0}{2}}\frac{s^{\frac{5}{4}}(s'')^{\frac{1}{4}}\,\ud s''}{s_0^{\frac{1}{2}}(s+s_0)(s+s_0-2s'')}\\
&\lesssim s_0^{-\frac{5}{4}}\int_0^{\frac{s_0}{2}}(s'')^{\frac{1}{4}}\,\ud s''\lesssim1.
\end{split}
\end{align*}
Next, 
\begin{align*}
\begin{split}
\int_{\frac{s_0}{2}}^{s+s_0}k(s,s'')\,\dspp&\leq \int_{\frac{s_0}{2}}^{s+s_0}\int_0^{s_0}\frac{s^{\frac{5}{4}}\,\ud s'\,\ud s''}{s_0^{\frac{1}{2}}(s'')^{\frac{3}{4}}(s+s_0-s')^2}\\
&\lesssim \int_{\frac{s_0}{2}}^{s+s_0}\frac{s^{\frac{1}{4}}s_0^{\frac{1}{2}}\ud s''}{(s+s_0)(s'')^{\frac{3}{4}}}\lesssim \frac{s^{\frac{1}{4}}s_0^{\frac{1}{2}}}{(s+s_0)^{\frac{3}{4}}}\lesssim 1.
\end{split}
\end{align*}
Finally, if $s+s_0\leq \frac{s_0}{2}+\frac{1}{8}$ then
\begin{align*}
\begin{split}
\int_{s+s_0}^{\frac{s_0}{2}+\frac{1}{8}}k(s,s'')\,\dspp&\leq \int_{s+s_0}^{\frac{s_0}{2}+\frac{1}{8}}\int_0^{s_0}\frac{s^{\frac{5}{4}}\,\ud s'\,\ud s''}{s_0^{\frac{1}{2}}(s+s_0-s')^{\frac{5}{4}}(s'')^{\frac{3}{2}}}\\
&\leq \int_{s+s_0}^{\frac{s_0}{2}+\frac{1}{8}}\int_0^{s_0}\frac{\ud s'\,\ud s''}{s_0^{\frac{1}{2}}(s'')^{\frac{3}{2}}}\lesssim\int_{s+s_0}^{\frac{s_0}{2}+\frac{1}{8}}\frac{s_0^{\frac{1}{2}}\,\ud s''}{(s'')^{\frac{3}{2}}}\lesssim 1.
\end{split}
\end{align*}
The argument for the $\ds$ integration of $k$ is similar. First
\begin{align*}
\begin{split}
\int_{s_0}^{\frac{1}{2}}\int_0^{2s''}\frac{\chi_{\{2s''\leq s_0\}}s^{\frac{1}{4}}(s'')^{\frac{1}{4}}\,\ud s' \,\ud s}{s_0^{\frac{1}{2}}(s+s_0-s')^2}&\lesssim \int_{s_0}^{\frac{1}{2}}\frac{\chi_{\{2s''\leq s_0\}}s^{\frac{1}{4}}s''\,\ud s}{s_0^{\frac{1}{4}}(s+s_0)(s+s_0-2s'')}\\
&\lesssim \int_{s_0}^{\frac{1}{2}}s^{-\frac{3}{4}}s_0^{\frac{3}{4}}\,\ds\lesssim 1.
\end{split}
\end{align*}
Next
\begin{align*}
\begin{split}
\int_{s_0-2s''}^{s_0}\int_0^{2s''}\frac{\chi_{\{2s''\leq s_0\}}s^{\frac{1}{4}}(s'')^{\frac{1}{4}}\,\ud s'\,\ud s}{s_0^{\frac{1}{2}}(s+s_0-s')^2}&\lesssim \int_{s_0-2s''}^{s_0}\frac{\chi_{\{2s''\leq s_0\}}s^{\frac{1}{4}}s''\,\ud s}{s_0^{\frac{1}{4}}(s+s_0)(s+s_0-2s'')}\\
&\lesssim \int_{s_0-2s''}^{s_0}s^{-\frac{3}{4}}s_0^{-\frac{1}{4}} \, \ud s \lesssim 1,
\end{split}
\end{align*}
\begin{align*}
\begin{split}
\int_0^{s_0-2s''}\int_0^{2s''}\frac{\chi_{\{2s''\leq s_0\}}s^{\frac{1}{4}}(s'')^{\frac{1}{4}}\,\ud s'\,\ud s}{s_0^{\frac{1}{2}}(s+s_0-s')^2}&\lesssim \int_0^{s_0-2s''}\frac{\chi_{\{2s''\leq s_0\}}s^{\frac{1}{4}}s''\,\ud s}{s_0^{\frac{1}{4}}(s+s_0)(s+s_0-2s'')}\\
&\lesssim\int_0^{s_0-2s''}\frac{\ud s}{s_0-2s''}\lesssim 1.
\end{split}
\end{align*}
In the region $s_0\leq 2s''$ first we estimate
\begin{align*}
\begin{split}
\int_{s''}^{\frac{1}{2}}\int_0^{s_0}\frac{\chi_{\{s_0\leq 2s''\}}s^{\frac{1}{4}}(s'')^{\frac{1}{4}}\,\ud s'\,\ud s}{s_0^{\frac{1}{2}}(s+s_0-s')^2}\lesssim \int_{s''}^{\frac{1}{2}}\frac{s_0^{\frac{1}{2}}(s'')^{\frac{1}{4}}\,\ud s}{s^{\frac{3}{4}}(s+s_0)}\lesssim \int_{s''}^{\frac{1}{2}}(s'')^{\frac{1}{4}}s^{-\frac{5}{4}}\,\ud s\lesssim1.
\end{split}
\end{align*}
Moreover, we have 
\begin{align*}
\begin{split}
\int_0^{s''}\int_0^{s_0}\frac{\chi_{\{s_0\leq 2s''\}}s^{\frac{1}{4}}\,\ud s' \,\ud s}{s_0^{\frac{1}{2}}(s'')^{\frac{3}{4}}(s+s_0-s')}&\lesssim\int_0^{s''}\int_0^{s_0}\frac{\chi_{\{s_0\leq 2s''\}}\,\ud s' \,\ud s}{s_0^{\frac{1}{2}}(s'')^{\frac{3}{4}}s^{\frac{3}{4}}}\\
&\lesssim \chi_{\{s_0\leq 2s''\}}s_0^{\frac{1}{2}}(s'')^{-\frac{1}{2}}\lesssim 1.
\end{split}
\end{align*}
This completes the proof of the fact that $k$ is a Schur kernel. 

Next we consider the contribution of $\tilP_{\geq s'+\frac{1}{4}}u$ to $I$ above. This is bounded by
\begin{align*}
\begin{split}
&\int_0^{\frac{1}{2}}\Big(\int_0^{s_0}s^{\frac{1}{4}}\|\tilP_s\tilP_{\geq s_0-s'}s'\tilB\tilP_{\geq s'+\frac{1}{4}}u\|_{\tilLE_s^\ast}\,\dsp\Big)^2\,\ds\\
&\lesssim \int_0^{\frac{1}{2}}\Big(\int_0^{s_0}\frac{s^{\frac{5}{4}}s'}{s+s_0-s'}\|\tilP_{s+s_0-s'}\tilB\tilP_{\geq s'+\frac{1}{4}}u\|_{\tilLE_{s+s_0-s'}^\ast}\,\dsp\Big)^2\,\ds\\
&\lesssim \int_0^{\frac{1}{2}}s^{\frac{1}{2}}\Big(\int_0^{s_0}\|\tilP_{\geq\frac{s'}{2}+\frac{1}{8}}u\|_{\tilLE_\low}\,\ud s'\Big)^2\,\ds\\
&\lesssim \int_0^{\frac{1}{2}}s^{\frac{1}{2}}\Big(\int_{\frac{1}{8}}^{\frac{s_0}{2}+\frac{1}{8}}\|\tilP_{\geq s'}u\|_{\tilLE_\low}^2\,\dsp\Big)s_0\,\ds\lesssim s_0\|u\|_{\tilLE}^2.
\end{split}
\end{align*}
The estimate for $II$ above is similar. Again we decompose $\tilP_{\geq s'}u$ as in \eqref{eq:9.1temp1}. The contribution of the integral is bounded by
\begin{align*}
\begin{split}
&\int_{\frac{1}{8}}^{4}\Big(\int_0^{s_0}\int_{s'}^{s'+\frac{1}{4}}\|\tilP_{\geq s}\tilP_{\geq s_0-s'}\tilB\tilP_{s''}u\|_{\tilLE_\low^\ast}\,\dspp\,\ud s'\Big)^2\,\ds\\
&\lesssim \int_{\frac{1}{8}}^{4}\Big(s_0\int_0^{\frac{s_0}{2}+\frac{1}{4}}\|\tilP_{s''}u\|_{\tilLE_{s''}}\,\dspp\Big)^2\,\ds\\
&\lesssim s_0^2\int_0^{\frac{s_0}{2}+\frac{1}{4}}(s'')^{-\frac{1}{2}}\|\tilP_{s''}u\|_{\tilLE_{s''}}^2\,\dspp\lesssim s_0^2\|u\|_{\tilLE}^2.
\end{split}
\end{align*}
Finally the contribution of $\tilP_{\geq s'+\frac{1}{4}}u$ in \eqref{eq:9.1temp1} to $II$ is bounded by
\begin{align*}
\begin{split}
&\int_{\frac{1}{8}}^4\Big(\int_0^{s_0}\|\tilP_{\geq s}\tilP_{\geq s_0-s'}\tilB\tilP_{\geq s'+\frac{1}{4}}u\|_{\tilLE_{\low}^\ast}\,\ud s'\Big)^2\,\ds\\
&\lesssim\int_{\frac{1}{8}}^{4}\Big(\int_0^{s_0}\|\tilP_{\geq s'+\frac{1}{4}}u\|_{\tilLE_\low}\,\ud s'\Big)^2\,\ds\\
&\lesssim s_0\int_{\frac{1}{4}}^{s_0+\frac{1}{4}}\|\tilP_{\geq s'}u\|_{\tilLE_\low}^2\,\dsp \lesssim s_0\|u\|_{\tilLE}^2.
\end{split}
\end{align*}
\end{proof}

\section{Smoothing for Low Frequencies} \label{s:low} 

In this section we start the proof of Theorem~\ref{t:LE1} by considering the low-frequency regime. Our goal is to implement the strategy laid out in Section~\ref{s:mult}, using the main multiplier identity from Lemma~\ref{lem:Q}. Recall that for any $s > 0$ the low-frequency projection $\tilP_{\geq s} u$ of a solution $u(t)$ to the linear Schr\"odinger equation~\eqref{eq:S} satisfies the equation 
 \begin{equation} \label{eq:P>s_H}
  (\partial_t - i \Delta) \tilP_{\geq s} u = i \tilP_{\geq s} F - i \tilP_{\geq s} H_\lot u - i [ \tilP_{\geq s}, H_\prin ] u - i (H_\prin + \Delta) \tilP_{\geq s} u
 \end{equation}
with initial value $\tilP_{\geq s} u(0) = \tilP_{\geq s} u_0$.
Given a slowly varying function $\alpha$ associated with a slowly varying sequence $\{ \alpha_\ell \} \in \calA$ as in Definition~\ref{d:alpha}, we recall the definition of the radial function
\begin{equation} \label{equ:definition_beta_low_freq}
 \beta^{(\alpha)} := \beta^{(\alpha)}(r) = \int_0^r \frac{\alpha(y)}{\langle y \rangle} \, \ud y
\end{equation}
and the associated self-adjoint operator
\begin{equation} \label{equ:definition_Q_low_freq}
 Q^{(\alpha)} := \frac{1}{i} \bigl( \beta^{(\alpha)} \partial_r - \partial_r^\ast \beta^{(\alpha)} \bigr)
\end{equation}
from Definition~\ref{d:beta_low}. The following proposition is our precise goal in this section.
\begin{prop} \label{prop:low_frequency_local_smoothing}
 Let $d \geq 2$. There exists an absolute constant $R \geq 1$ such that for all solutions $u(t)$ to~\eqref{eq:S} and for all $s > 0$ it holds that
 \begin{equation} \label{equ:low_frequency_local_smoothing_estimate}
  \begin{aligned}
   \| \tilP_{\geq s} u \|_{LE_\low}^2 &\lesssim \sup_{ \{ \alpha_\ell \} \in \calA } \, \sup_{t \in \bbR} \, \bigl| \langle \tilP_{\geq s} u(t), Q^{(\alpha)} \tilP_{\geq s} u (t) \rangle \bigr| \\
   &\qquad + \sup_{\{ \alpha_\ell \} \in \calA } \, \bigl| \Re \langle (\partial_t - i\Delta) \tilP_{\geq s} u, Q^{(\alpha)} \tilP_{\geq s} u \rangle_{t,x} \bigr| \\
   &\qquad + \| \tilP_{\geq s} u\|_{L^2(\bbR\times\{r \leq R\})}^2.
  \end{aligned}
 \end{equation}
\end{prop}

The first difficulty is already evident by inspecting \eqref{eq:Q1} for $H=-\Delta$ in the radial case. Indeed, to obtain the desired coercivity for $\Re \langle  i H \tilP_{\geq s} u, Q \tilP_{\geq s} u \rangle$ we hope to make the right hand side of \eqref{eq:Q1} positive. However, the terms
\begin{align*}
\begin{split}
 \langle (\partial_r \beta) \tilP_{\geq s} u, \tilP_{\geq s} u \rangle \qquad \mathrm{and}\qquad \langle (\partial_r \beta) \partial_r \tilP_{\geq s} u, \partial_r \tilP_{\geq s} u \rangle
\end{split}
\end{align*}
on the right-hand side of \eqref{eq:Q1} appear with the opposite signs, and while the former can be placed in $L^2( \bbR \times \{r \leq R\})$ for $r \leq R$, the error for $r\geq R$ needs to be treated differently. Moreover, in view of the fact that $\mathrm{Spec} (-\Delta)=[\rho^2,\infty)$, the balance between these terms with opposing signs is especially delicate for low frequencies, i.e., near the bottom edge of the spectrum of $-\Delta$. This motivates the weighted Hardy-Poincar\'e estimate below\footnote{For a related estimate without weights see \cite{ManSan1, BGG1}.}, which gives us a positive lower bound for the difference $\langle w\nabla v,\nabla v\rangle-\rho^2\langle w v,v\rangle$ for arbitrary $v$ and appropriate weight functions~$w$. This estimate is the reason for the different weights in the definition of our local smoothing norms for low and high frequencies. Note that the regime where $r\gg 1$ and $u$ is at low frequencies is precisely where we expect to see the effects of the geometry of the hyperbolic space most prominently. 

Since in Section~\ref{s:error} we will need a slight variant of the weighted Hardy-Poincar\'e estimate needed in the current section, below we state a more general version which contains both formulations as special cases.

\begin{lem}[Weighted Hardy-Poincar\'e Estimate] \label{lem:HP}
 Let $d \geq 2$ and let $w \colon \bbH^d \to \bbR$ be a non-negative, smooth, and bounded radial weight function with bounded first derivative, and such that $r^{-1}\partial_rw$ is bounded at the origin. Then for any $u\in H^1(\bbH^d)$ and any constant $m > 0$ we have
\begin{align}\label{equ:weighted_hardy_poincare}
\begin{split}
&\angles{w\partial_ru}{\partial_ru}-\rho^2\angles{wu}{u}-\rho\angles{(\partial_rw)\coth ru}{u}\\
&=\angles{w(\partial_r+\rho\coth r)u}{(\partial_r+\rho\coth r)u}+\rho(\rho-1)\angles{w\sinh^{-2}ru}{u}\\
&=\angles{w(\partial_r+\rho\coth r+mr^{-1})u}{(\partial_r+\rho\coth r+mr^{-1})u}+m(1-m)\angles{2r^{-2}u}{u}\\
&\quad+\rho(\rho-1)\angles{w\sinh^{-2}ru}{u}-m\angles{(\partial_rw)r^{-1}u}{u}.
\end{split}
\end{align}
In particular with $m=\frac{1}{2}$ and
 \begin{align}\label{eq:HP_U_def}
  U:= &\begin{cases} \frac{1}{4}\frac{1}{r^2} + \frac{(d-1)(d-3)}{4} \frac{1}{\sinh^2r}, &d\geq3 \\ \\ \frac{1}{4}\left(\frac{1}{r^2}-\frac{1}{\sinh^2r}\right), &d=2 \end{cases},
 \end{align}
 we have
 \begin{equation} \label{equ:weighted_hardy_poincare_whole}
  \begin{aligned}
   &\int_{\bbH^d} w \bigl( |\nabla u|^2 - \rho^2 |u|^2 \bigr) \, \dh - \int_{\bbH^d} \rho \coth(r) (\partial_r w) |u|^2 \, \dh \\
   &\quad \geq \int_{\bbH^d} U w |u|^2 \, \dh - \int_{\bbH^d} \frac{\partial_r w}{2r} |u|^2 \, \dh.
  \end{aligned}
 \end{equation} 
\end{lem}

\begin{proof}
Estimate \eqref{equ:weighted_hardy_poincare_whole} is a direct consequence of \eqref{equ:weighted_hardy_poincare} so we concentrate on the latter. First
\begin{align*}
\begin{split}
&\int_{\bbH^d}w(|\partial_ru|^2-\rho^2\coth^2r|u|^2)\,\dh\\
&=\int_{\bbH^d}w|\partial_ru-\rho\coth r u|^2\,\dh-2\rho^2\int_{\bbH^d}w\coth^2r|u|^2\,\dh\\
&\quad-\rho\int_{\bbH^d}w\coth r\partial_r|u|^2\,\dh\\
&=\int_{\bbH^d}w|\partial_ru-\rho\coth r u|^2\,\dh+\rho\int_{\bbH^d}\partial_r(w\coth r)|u|^2\,\dh\\
&=\int_{\bbH^d}w|\partial_ru-\rho\coth r u|^2\,\dh-\rho\int_{\bbH^d}w\sinh^{-2}r|u|^2\,\dh\\
&\quad+\rho\int_{\bbH^d}(\partial_rw)\coth r|u|^2\,\dh.
\end{split}
\end{align*}
Noting that $\rho^2\coth^2r=\rho^2+\rho^2\sinh^{-2}r$ and rearranging we get the first identity in \eqref{equ:weighted_hardy_poincare}. For the second identity it suffices to observe that using integration by parts
\begin{align*}
\begin{split}
&\int_{\bbH^d}w|\partial_ru-\rho\coth r u|^2\,\dh=\int_{\bbH^d}w|\partial_ru-\rho\coth r u-mr^{-1}u+mr^{-1}u|^2\,\dh\\
&=\int_{\bbH^d}w|\partial_ru-\rho\coth r u-mr^{-1}u|^2\,\dh+m(1-m)\int_{\bbH^d}\frac{w}{r^2}|u|^2\,\dh\\
&\quad-m\int_{\bbH^d}\frac{\partial_rw}{r}|u|^2\,\dh.
\end{split}
\end{align*}
\end{proof}

 Next we apply the weighted Hardy-Poincar\'e estimate~\eqref{equ:weighted_hardy_poincare_whole} to our main multiplier identity~\eqref{eq:Q1}. The resulting positive commutator type estimate is at the heart of the proof of our low-frequency local smoothing estimate~\eqref{equ:low_frequency_local_smoothing_estimate}.
 
 \begin{lem} \label{l:weighted_HP_applied_to_main_multiplier_identity}
  Let $d \geq 2$. There exist constants $C \geq 1$ and $R \geq 1$ with the following property: For any slowly varying sequence $\{ \alpha_\ell \} \in \calA$ with associated slowly varying function~$\alpha$ as in Definition~\ref{d:alpha}, let $\beta^{(\alpha)}$ and $Q^{(\alpha)}$ be defined as in~\eqref{equ:definition_beta_low_freq} and~\eqref{equ:definition_Q_low_freq}, respectively. Then  for all $u \in H^2(\bbH^d)$ 
  \begin{equation} \label{equ:positive_commutator_lower_bound_from_HP}
   \Re \langle i \Delta u, Q^{(\alpha)} u \rangle \geq \langle U (\partial_r \beta^{(\alpha)}) u, u \rangle - C \| u \|_{L^2(\{ r \leq R\})}^2,
  \end{equation}
  where $U$ is defined in~\eqref{eq:HP_U_def}. Moreover, 
  \begin{equation} \label{equ:additional_bounds_positive_commutator_from_HP}
   \big\langle \frac{\partial_r \beta^{(\alpha)}}{\langle r \rangle^2} u, u \big\rangle + \big| \Re\langle (\partial_r \beta^{(\alpha)}) (\Delta + \rho^2) u, u \rangle \bigr| \lesssim \Re \langle i \Delta u, Q^{(\alpha)} u \rangle + \|u\|_{L^2(\{r \leq R\})}^2.
  \end{equation}
 \end{lem}
  
 \begin{rem} \label{rem:weighted_HP_applied_to_main_multiplier_identity_s3version}
  In combining the low and high-frequency estimates in Section~\ref{s:trans} we will need an analogue of  Lemma~\ref{l:weighted_HP_applied_to_main_multiplier_identity} when the multiplier is defined as in Section~\ref{s:high} for high frequencies. Therefore, for future reference we remark that, by inspection, our proof also yields the following result:  Let $d \geq 2$ and let $s_3 \simeq 1$. There exist constants $C \geq 1$ and ${R \geq 1}$ with the following property: For any slowly varying sequence $\{ \alpha_\ell \} \in \calA$ with associated slowly varying function~$\alpha$ as in Definition~\ref{d:alpha}, let $\beta_{s_3}^{(\alpha)}$ and $Q_{s_3}^{(\alpha)}$ be defined as in Defintion~\ref{d:beta_high}. Then for all $u \in H^2(\bbH^d)$
  \begin{equation} \label{equ:positive_commutator_lower_bound_from_HP_s3version}
   \Re \langle i \Delta u, Q_{s_3}^{(\alpha)} u \rangle \geq \langle U (\partial_r \beta_{s_3}^{(\alpha)}) u, u \rangle - C \| u \|_{L^2(\{ r \leq R\})}^2,
  \end{equation}
  where $U$ is defined in~\eqref{eq:HP_U_def}. Moreover, 
  \begin{equation} \label{equ:additional_bounds_positive_commutator_from_HP_s3version}
   \big\langle \frac{\partial_r \beta_{s_3}^{(\alpha)}}{\langle r \rangle^2} u, u \big\rangle + \big|\Re \langle (\partial_r \beta_{s_3}^{(\alpha)}) (\Delta + \rho^2) u, u \rangle \bigr| \lesssim \Re \langle i \Delta u, Q_{s_3}^{(\alpha)} u \rangle + \|u\|_{L^2(\{r \leq R\})}^2.
  \end{equation}
 \end{rem}
 
 Before we can turn to the proof of Lemma~\ref{l:weighted_HP_applied_to_main_multiplier_identity} we need the following technical lemma.
 \begin{lem} \label{l:technical_lemma_low_freq}
  There exist constants $C \geq 1$ and $R \geq 1$ with the following property: For any slowly varying sequence $\{ \alpha_\ell \} \in \calA$ with associated slowly varying function~$\alpha$ as in Definition~\ref{d:alpha}, let $\beta^{(\alpha)}$ be defined as in~\eqref{equ:definition_beta_low_freq}. Then we have 
  \begin{align}
   \partial_r^2 \beta^{(\alpha)} &\leq 0, \label{equ:second_derivative_beta_nonnegative} \\
   \beta^{(\alpha)} \coth(r) - \partial_r \beta^{(\alpha)} &\geq 0, \label{equ:beta_coth_minus_partial_beta_nonnegative} \\
   |\partial_r^3 \beta^{(\alpha)}| &\leq C \frac{\partial_r \beta^{(\alpha)}}{\langle r \rangle^2}, \label{equ:third_derivative_beta_bound} 
  \end{align}
  and the following lower bounds hold
  \begin{align}
   \frac{1}{2} U \partial_r \beta^{(\alpha)} - \frac{\partial_r^2 \beta^{(\alpha)}}{2 r} - \frac{\partial_r^3 \beta^{(\alpha)}}{4} &\geq 0, \label{equ:Ubeta_sum_3d_lower_bound} \\ 
   \frac{1}{2} U \partial_r \beta^{(\alpha)} - \frac{\partial_r^2 \beta^{(\alpha)}}{2 r} - \frac{\partial_r^3 \beta^{(\alpha)}}{4} - \frac{ \beta^{(\alpha)} \coth(r) - \partial_r \beta^{(\alpha)} }{4 \sinh^2(r)} &\geq - C \chi_{\{ r \leq R\}}, \label{equ:Ubeta_sum_2d_lower_bound} \\
   \biggl| \frac{\beta^{(\alpha)} \coth(r) - \partial_r \beta^{(\alpha)}}{\sinh^2(r)} \biggr| &\lesssim \frac{\partial_r \beta^{(\alpha)}}{\langle r \rangle^2} + \chi_{\{ r \leq R\}}, \quad \quad \label{eq:extrabetacothrbetarbound}
  \end{align}
  where $\chi_{\{ r \leq R\}}$ denotes a cut-off function to the region $\{ r \leq R \}$.
 \end{lem}
 \begin{proof}
  In what follows we drop the superscript from $\beta^{(\alpha)}$. In view of the definition~\eqref{equ:definition_beta_low_freq} of $\beta(r)$ and the fact that the function $y \mapsto \frac{\alpha(y)}{\langle y \rangle}$ is non-increasing by Definition~\ref{d:alpha}, the property~\eqref{equ:second_derivative_beta_nonnegative} is immediate and the inequality~\eqref{equ:beta_coth_minus_partial_beta_nonnegative} follows readily from the identity
  \begin{align*}
   &\beta(r) \coth(r) - (\partial_r \beta)(r) \\
   &\quad = \biggl( \int_0^r \Bigl( \frac{\alpha(y)}{\langle y \rangle} - \frac{\alpha(r)}{\langle r \rangle} \Bigr) \, dy \biggr) \coth(r) + (r \coth(r) - 1) (\partial_r \beta)(r) \geq 0,
  \end{align*}
  since $r \coth(r) - 1 \geq 0$ for all $r > 0$. In order to prove~\eqref{equ:third_derivative_beta_bound}--\eqref{equ:Ubeta_sum_2d_lower_bound} we first compute 
  \begin{align*}
   \partial_r \beta = \frac{\alpha}{\langle r \rangle}, \quad \partial_r^2 \beta = \frac{\alpha'}{\langle r \rangle} - \frac{r \alpha}{\langle r \rangle^3}, \quad \partial_r^3 \beta = \frac{\alpha''}{\langle r\rangle} - \frac{2 r \alpha'}{\langle r \rangle^3} + \frac{(2 r^2 - 1) \alpha}{\langle r \rangle^5}.
  \end{align*}
  Then the bound~\eqref{equ:third_derivative_beta_bound} is a consequence of the symbol-type bounds $|\alpha^{(j)}(r)| \lesssim \frac{\alpha(y)}{\langle y \rangle^j}$ for the slowly varying function $\alpha$ as in Definition~\ref{d:alpha}. Moreover, we see that
  \begin{equation} \label{equ:Ubeta_sum_small_terms}
   - \frac{\partial_r^2 \beta}{2r} - \frac{\partial_r^3 \beta}{4} = \frac{3 \alpha}{4 \langle r \rangle^5} - \frac{\alpha'}{2 r \langle r \rangle^3} - \frac{\alpha''}{4 \langle r \rangle}.
  \end{equation}
  Now recall from Definition~\ref{d:alpha} that the slowly varying function $\alpha$ satisfies $\alpha'(r) = 0$ for all $r \leq 1$. In view of~\eqref{equ:Ubeta_sum_small_terms} we therefore only have to verify~\eqref{equ:Ubeta_sum_3d_lower_bound} for $r > 1$. Since by Definition~\ref{d:alpha} we also have for $r > 1$ that
  \begin{align*}
   |\alpha'(r)| \leq C \eta \frac{\alpha(r)}{r}, \quad |\alpha''(r)| \leq C \eta \frac{\alpha(r)}{r^2}
  \end{align*}
  for some absolute constants $C \geq 1$ and $0 < \eta \ll 1$, we conclude that for $r > 1$
  \begin{equation} \label{equ:Ubeta_sum_2d_lower_bound_prelim}
   \begin{aligned}
    \frac{1}{2} U \partial_r \beta - \frac{\partial_r^2 \beta}{2r} - \frac{\partial_r^3 \beta}{4} &= \frac{1}{2} U \frac{\alpha}{\langle r \rangle} + \frac{3 \alpha}{4 \langle r \rangle^5} - \frac{\alpha'}{2 r \langle r \rangle^3} - \frac{\alpha''}{4 \langle r \rangle} \\
    &\geq \frac{1}{16} \frac{\alpha}{r^3} - C \eta \frac{\alpha}{r^5} - C \eta \frac{\alpha}{r^3} \\
    &\geq \Bigl( \frac{1}{16} - 2 C \eta \Bigr) \frac{\alpha}{r^3}.
   \end{aligned}
  \end{equation}
  The right-hand side is non-negative for sufficiently small $0 < \eta \ll 1$, which completes the proof of~\eqref{equ:Ubeta_sum_3d_lower_bound}. Finally, we prove~\eqref{equ:Ubeta_sum_2d_lower_bound}. To this end we first note that Definition~\ref{d:alpha} of a slowly varying function and the definition~\eqref{equ:definition_beta_low_freq} of $\beta$ imply that  
  \begin{equation} \label{equ:growth_bounds_beta_and_alpha}
   \beta(r) \lesssim 1 \quad \text{ and } \quad \alpha(r) \gtrsim \frac{1}{\langle r \rangle^\eta} \quad \forall \, r \geq 0
  \end{equation}
  uniformly for all slowly varying functions $\alpha$ associated with a sequence $\{ \alpha_\ell \} \in \calA$. Then using $\beta''(0)=0$ one easily verifies that 
  \[
   \lim_{r \searrow 0} - \frac{\beta \coth(r) - \partial_r \beta}{4 \sinh^2(r)} = - \frac{\alpha(0)}{12}.
  \]
  Hence, given any $R \geq 1$, by continuity there exists $C \equiv C(R)$ such that 
  \begin{equation} \label{equ:Ubeta_sum_2d_lower_bound_extra_term_near_origin}
   - \frac{\beta \coth(r) - \partial_r \beta}{4 \sinh^2(r)} \geq - C(R) \quad \forall \, 0 < r \leq R
  \end{equation}
  uniformly for all slowly varying functions $\alpha$ associated with a sequence $\{ \alpha_\ell \} \in \calA$. Moreover, we infer from~\eqref{equ:Ubeta_sum_2d_lower_bound_prelim} and~\eqref{equ:growth_bounds_beta_and_alpha} for any $r > 1$ that
  \begin{equation}
   \begin{aligned}
    \frac{1}{2} U \partial_r \beta - \frac{\partial_r^2 \beta}{2r} - \frac{\partial_r^3 \beta}{4} - \frac{\beta \coth(r) - \partial_r \beta}{4 \sinh^2(r)} &\geq \Bigl( \frac{1}{16} - 2 C \eta \Bigr) \frac{\alpha}{r^3} - \frac{\beta \coth(r)}{4 \sinh^2(r)} \\
    &\geq \Bigl( \frac{1}{16} - 2 C \eta \Bigr) \frac{\alpha}{r^3} - C \frac{ \langle r \rangle^\eta \, \alpha}{\sinh^2(r)},
   \end{aligned}
  \end{equation}
  which is non-negative for all $r \geq R$ with $R > 1$ sufficiently large by the exponential decay of $\sinh^{-2}(r)$. Combining this last observation with~\eqref{equ:Ubeta_sum_2d_lower_bound_extra_term_near_origin} finishes the proof of~\eqref{equ:Ubeta_sum_2d_lower_bound} as well as~\eqref{eq:extrabetacothrbetarbound}.
 \end{proof}

 \begin{rem}
  We expect that it should be possible to show non-negativity in~\eqref{equ:Ubeta_sum_2d_lower_bound} by arguing more carefully.
 \end{rem}

 We are now prepared to prove Lemma~\ref{l:weighted_HP_applied_to_main_multiplier_identity}.
 
 \begin{proof}[Proof of Lemma~\ref{l:weighted_HP_applied_to_main_multiplier_identity}]
  We first rewrite our main multiplier identity \eqref{eq:Q1} in a form which is more amenable to application of the weighted Hardy-Poincar\'e estimate~\eqref{equ:weighted_hardy_poincare_whole}:
  \begin{equation} \label{equ:main_multiplier_identity_rewritten}
   \begin{aligned}
    \frac{1}{2} \Re &\langle i \Delta u, Q^{(\alpha)} u \rangle \\
    &= \int_{\bbH^d} \bigl( (\partial_r \beta^{(\alpha)}) ( |\nabla u|^2 - \rho^2 |u|^2 ) - \rho (\partial_r^2 \beta^{(\alpha)}) \coth(r) |u|^2 \bigr) \, \dh \\
    &\quad + \int_{\bbH^d} \bigl( \beta^{(\alpha)} \coth(r) - \partial_r \beta^{(\alpha)} \bigr) \frac{1}{\sinh^2r}|\slashed{\nabla} u|^2 \, \dh \\
    &\quad + \int_{\bbH^d} \Bigl( -\frac{ \partial_r^3 \beta^{(\alpha)}}{4} + \frac{\rho (\rho - 1)}{\sinh^2(r)} \bigl( \beta^{(\alpha)} \coth(r) - \partial_r \beta^{(\alpha)} \bigr) \Bigr) |u|^2 \, \dh. 
   \end{aligned}
  \end{equation}
  Noting that the weight function $w := \partial_r \beta^{(\alpha)}$ satisfies the requirements of Lemma~\ref{lem:HP}, we then apply the weighted Hardy-Poincar\'e estimate~\eqref{equ:weighted_hardy_poincare_whole} to the first term on the right-hand side of the above equation to obtain that
  \begin{equation}
   \begin{aligned}
    \frac{1}{2} \Re &\langle i \Delta u, Q^{(\alpha)} u \rangle \\
    &\geq \int_{\bbH^d} U (\partial_r \beta^{(\alpha)}) |u|^2 \, \dh - \int_{\bbH^d} \frac{\partial_r^2 \beta^{(\alpha)}}{2 r} |u|^2 \, \dh \\
    &\quad + \int_{\bbH^d} \bigl( \beta^{(\alpha)} \coth(r) - \partial_r \beta^{(\alpha)} \bigr) \frac{1}{\sinh^2r}|\slashed{\nabla} u|^2 \, \dh \\
    &\quad + \int_{\bbH^d} \Bigl( -\frac{ \partial_r^3 \beta^{(\alpha)}}{4} + \frac{\rho (\rho - 1)}{\sinh^2(r)} \bigl( \beta^{(\alpha)} \coth(r) - \partial_r \beta^{(\alpha)} \bigr) \Bigr) |u|^2 \, \dh \\
    &= \int_{\bbH^d} \Bigl( U \partial_r \beta^{(\alpha)} - \frac{\partial_r^2 \beta^{(\alpha)}}{2 r} - \frac{\partial_r^3 \beta^{(\alpha)}}{4} \Bigr) |u|^2 \, \dh \\
    &\quad + \int_{\bbH^d} \bigl( \beta^{(\alpha)} \coth(r) - \partial_r \beta^{(\alpha)} \bigr) \frac{1}{\sinh^2r}|\slashed{\nabla} u|^2 \, \dh \\
    &\quad + \int_{\bbH^d} \frac{\rho (\rho - 1)}{\sinh^2(r)} \bigl( \beta^{(\alpha)} \coth(r) - \partial_r \beta^{(\alpha)} \bigr) |u|^2 \, \dh.
   \end{aligned}
  \end{equation}
  By Lemma~\ref{l:technical_lemma_low_freq} it then follows that
  \begin{equation}
   \frac{1}{2} \Re \langle i \Delta u, Q^{(\alpha)} u \rangle \geq \frac{1}{2} \int_{\bbH^d} U (\partial_r \beta^{(\alpha)}) |u|^2 \, \dh - C \| u \|_{L^2(\{r \leq R\})}^2,
  \end{equation}
  which completes the proof of~\eqref{equ:positive_commutator_lower_bound_from_HP}. Next we observe that for all space dimensions $d \geq 2$, the function $U(r)$ defined in~\eqref{eq:HP_U_def} satisfies 
  \begin{equation}
   U(r) \gtrsim \frac{1}{\langle r \rangle^2} \quad \forall \, r > 0.
  \end{equation}
  Combining this estimate with~\eqref{equ:positive_commutator_lower_bound_from_HP} yields the desired upper bound on the first term on the left-hand side of~\eqref{equ:additional_bounds_positive_commutator_from_HP}, namely
  \begin{equation} \label{equ:partial_beta_over_r_squared_upper_bound}
   \big\langle \frac{\partial_r \beta^{(\alpha)}}{\langle r \rangle^2} u, u \big\rangle \lesssim \Re \langle i \Delta u, Q^{(\alpha)} u \rangle + \|u\|_{L^2(\{r \leq R\})}^2.
  \end{equation}
  In order to also establish this upper bound for the second term on the left-hand side of~\eqref{equ:additional_bounds_positive_commutator_from_HP}, we introduce as an auxiliary tool the quadratic form 
  \begin{equation}
   \calB(u) := \int_{\bbH^d} (\partial_r \beta^{(\alpha)}) ( |\nabla u|^2 - \rho^2 |u|^2) \, \dh - \int_{\bbH^d} \rho \coth(r) (\partial_r^2 \beta^{(\alpha)}) |u|^2 \, \dh.
  \end{equation}
  An application of the weighted Hardy-Poincar\'e estimate~\eqref{equ:weighted_hardy_poincare_whole} with $w := \partial_r \beta^{(\alpha)}$ shows that $\calB(u)$ is non-negative:
  \begin{equation}
   \calB(u) \geq \langle U (\partial_r \beta^{(\alpha)}) u, u \rangle - \bigl\langle \frac{\partial_r^2 \beta^{(\alpha)}}{2r} u, u \bigr\rangle \geq 0
  \end{equation}
  since $-\partial_r^2 \beta^{(\alpha)} \geq 0$ by Lemma~\ref{l:technical_lemma_low_freq}. Then we observe that the first term on the right-hand side of the identity~\eqref{equ:main_multiplier_identity_rewritten} is exactly $\calB(u)$. Using also~\eqref{equ:beta_coth_minus_partial_beta_nonnegative}, we may therefore conclude that
  \begin{align*}
   \calB(u) \lesssim  \Re \langle i \Delta u, Q^{(\alpha)} u \rangle  + \bigl| \langle (\partial_r^3 \beta^{(\alpha)}) u, u \rangle \bigr| + \biggl| \bigl\langle \frac{ \beta^{(\alpha)} \coth(r) - \partial_r \beta^{(\alpha)} }{ \sinh^{2}(r) } u, u \bigr\rangle \biggr|. 
  \end{align*}
Now, by \eqref{eq:extrabetacothrbetarbound} there exists a sufficiently large $R \geq 1$ such that 
  \begin{equation}
   \biggl| \frac{\beta^{(\alpha)} \coth(r) - \partial_r \beta^{(\alpha)}}{\sinh^2(r)} \biggr| \lesssim \frac{\partial_r \beta^{(\alpha)}}{\langle r \rangle^2} + \chi_{\{ r \leq R\}} \quad \forall \, r > 0.
  \end{equation}
  Thus, using also~\eqref{equ:third_derivative_beta_bound} and~\eqref{equ:partial_beta_over_r_squared_upper_bound}, we infer that
  \begin{equation} \label{equ:quadratic_form_B_upper_bound}
   \calB(u) \lesssim \Re \langle i \Delta u, Q^{(\alpha)} u \rangle + \|u\|_{L^2(\{r \leq R\})}^2.
  \end{equation}
  Finally, the desired upper bound on $\bigl| \Re \langle (\partial_r \beta) (\Delta + \rho^2) u, u \rangle \bigr|$ follows from the estimates~\eqref{equ:third_derivative_beta_bound}, \eqref{equ:partial_beta_over_r_squared_upper_bound}, \eqref{equ:quadratic_form_B_upper_bound} and the identity
  \begin{equation}
   \Re \langle (\partial_r \beta^{(\alpha)}) (\Delta + \rho^2) u, u \rangle = - \calB(u) + \frac{1}{2} \langle (\partial_r^3 \beta^{(\alpha)}) u, u \rangle.
  \end{equation}
  \end{proof}
Finally, we are in the position to provide the proof of the low-frequency local smoothing estimate of Proposition~\ref{prop:low_frequency_local_smoothing}.

\begin{proof}[Proof of Proposition~\ref{prop:low_frequency_local_smoothing}]
 Let $s > 0$. We first recall from Lemma~\ref{l:Xal} that 
 \begin{equation} \label{equ:proof_low_freq_est1}
  \| \tilP_{\geq s} u \|_{LE_\low} \simeq \sup_{\{ \alpha_\ell \} \in \calA} \| \tilP_{\geq s} u \|_{X_{\low, \alpha}}
 \end{equation}
 and from Lemma~\ref{lem:betaLErelation} that 
 \begin{equation} \label{equ:proof_low_freq_est2}
  \| \tilP_{\geq s} u \|_{X_{\low, \alpha}}^2 \simeq \bigl\langle \frac{\alpha}{\langle r \rangle^3} \tilP_{\geq s} u, \tilP_{\geq s} u \bigr\rangle_{t,x} \lesssim \langle U (\partial_r \beta^{(\alpha)}) \tilP_{\geq s} u, \tilP_{\geq s} u \rangle_{t,x}.
 \end{equation}
 Moreover, we recall the identity (see \eqref{eq:dtQu})
 \begin{equation} \label{equ:proof_low_freq_est3}
  \Re \langle i \Delta \tilP_{\geq s} u, Q^{(\alpha)} \tilP_{\geq s} u \rangle = \frac{1}{2} \frac{\ud}{\ud t} \langle \tilP_{\geq s} u, Q^{(\alpha)} \tilP_{\geq s} u \rangle - \Re \langle (\partial_t - i \Delta) \tilP_{\geq s} u, Q^{(\alpha)} \tilP_{\geq s} u \rangle
 \end{equation}
 and note that by Lemma~\ref{l:weighted_HP_applied_to_main_multiplier_identity} we have 
 \begin{equation} \label{equ:low_frequency_smoothing_derivation1}
  \langle U (\partial_r \beta^{(\alpha)}) \tilP_{\geq s} u, \tilP_{\geq s} u \rangle \leq \Re \langle i \Delta \tilP_{\geq s} u, Q^{(\alpha)} \tilP_{\geq s} u \rangle + C \| \tilP_{\geq s} u \|_{L^2(\{ r \leq R\})}^2.
 \end{equation}
 Thus, integrating in time over~\eqref{equ:low_frequency_smoothing_derivation1} and taking the supremum over all slowly varying sequences $\{ \alpha_\ell \} \in \calA$ of the resulting inequality, we get the asserted estimate~\eqref{equ:low_frequency_local_smoothing_estimate} by combining with \eqref{equ:proof_low_freq_est1}, \eqref{equ:proof_low_freq_est2},  and \eqref{equ:proof_low_freq_est3}.
\end{proof}


\section{Smoothing for High Frequencies} \label{s:high} 


In this section we establish the high-frequency part of the local smoothing estimate in Theorem~\ref{t:LE1}. We note that for any $s > 0$ the frequency projection $\tilP_s u$ of a solution $u(t)$ to the linear Schr\"odinger equation~\eqref{eq:S} satisfies the equation
\begin{equation} \label{eq:Ps_H}
 (\partial_t - i \Delta) \tilP_s u = i \tilP_s F - i \tilP_s H_\lot u - i [ \tilP_s, H_\prin ] u - i (H_\prin + \Delta) \tilP_s u
\end{equation}
with initial value $\tilP_s u(0) = \tilP_s u_0$.
Given a heat time $s > 0$, a small positive number $0 < \delta \leq 1$, and a slowly varying sequence $\{ \alpha_\ell \} \in \calA$ with associated slowly varying function $\alpha$ as in Definition~\ref{d:alpha}, we recall the definition of the radial function 
\begin{equation} \label{equ:definition_beta_high_freq}
 \beta_s^{(\alpha)} := \beta_s^{(\alpha)}(r) := s^{\frac{1}{2}} \int_0^{\delta s^{-\frac{1}{2}} r} \frac{\alpha(y)}{\langle y \rangle} \, \ud y 
\end{equation}
and the self-adjoint operator
\begin{equation} \label{equ:definition_Q_high_freq}
 Q_s^{(\alpha)} := \frac{1}{i} \bigl( \beta_s^{(\alpha)} \partial_r - \partial_r^\ast \beta_s^{(\alpha)} \bigr)
\end{equation}
 from Definition~\ref{d:beta_high}. Before stating the main result of this section we decompose the right-hand side of \eqref{eq:Ps_H} a bit more. Since for high frequencies our local smoothing spaces distinguish between annuli of very small radii, we will not be able to treat the contribution of $(H_\prin + \Delta) \tilP_s u$ as an error term as crudely as in the low-frequency regime. In particular to absorb the contribution of 
\begin{align}\label{eq:tilFstemp1}
\begin{split}
\nabla_{\theta_a}(\bsh^{-1}-\bsa)^{\theta_a\theta_b}\nabla_{\theta_b}\tilP_su,
\end{split}
\end{align}
we will need to make use of the positive term
\begin{align}\label{eq:tilFstemp2}
\begin{split}
\angles{(\beta_s^{(\alpha)}\coth(r) - \partial_r\beta_s^{(\alpha)})\sinh^{-2}(r) \snabla\tilP_su}{\snabla \tilP_su}
\end{split}
\end{align}
in our main multiplier identity (see \eqref{eq:Q1}) which we simply discarded in Section~\ref{s:low}. For this reason we introduce the notation
\begin{align}\label{eq:tilFsdef}
\begin{split}
\tilF_s^{(\alpha)}&:=-\Re \langle (\partial_t - i \Delta) \tilP_s u, Q_s^{(\alpha)} \tilP_s u \rangle_{t,x}\\
&\quad \quad + 2 \angles{\ringa^{\theta_a\theta_b}(\coth(r) \beta_s^{(\alpha)} - \partial_r \beta_s^{(\alpha)})\partial_{\theta_a}\tilP_su}{\partial_{\theta_b}\tilP_su}_{{t,x}},
\end{split}
\end{align}
where we recall that
\begin{align*}
\begin{split}
\ringa:=\bsa-\bsh^{-1}.
\end{split}
\end{align*}
Note that in \eqref{eq:tilFsdef} we  have subtracted precisely the contribution of \eqref{eq:tilFstemp1} which needs to be absorbed by \eqref{eq:tilFstemp2}, and that this is consistent with the estimates derived in Section~\ref{ss:DeltaHdifference}. The following is the analogue of Proposition~\ref{prop:low_frequency_local_smoothing} in the high-frequency regime.
\begin{prop} \label{prop:high_frequency_local_smoothing}
 Let $d \geq 2$. Fix a dyadic heat time $s_2 \in 2^{\bbZ}$ and let $\tilF_s^{(\alpha)}$ be as in \eqref{eq:tilFsdef}. Then there exists a constant $R \equiv R(s_2) \geq 1$ such that for all solutions $u(t)$ to~\eqref{eq:S} it holds that
 \begin{equation} \label{equ:high_frequency_local_smoothing}
  \begin{aligned}
   \int_0^{\frac{s_2}{4}} s^{-\frac{1}{2}} \| \tilP_s u \|_{LE_s}^2 \, \ds &\lesssim_{s_2} \sup_{\{ \alpha_\ell \} \in \calA } \, \int_{\frac{s_2}{8}}^{s_2} \frac{1}{s} \langle (\partial_r \beta_s^{(\alpha)}) \tilP_s u, \tilP_s u \rangle_{t,x} \, \ds \\
   &\qquad \quad + \int_0^{s_2} \sup_{\{ \alpha_\ell \} \in \calA } \, \sup_{t \in \bbR} \, \bigl| \langle \tilP_s u(t), Q_s^{(\alpha)} \tilP_s u(t) \rangle \bigr| \, \ds \\
    &\qquad \quad + \int_0^{s_2} \sup_{\{ \alpha_\ell \} \in \calA } \, | \tilF_s^{(\alpha)} | \, \ds \\
   &\qquad \quad + \int_0^{s_2} \| \tilP_s u \|_{L^2(\bbR \times \{ r \leq R \})}^2 \, \ds.
  \end{aligned}
 \end{equation}
\end{prop}

\begin{rem} \label{rem:no-bernstein}
Note that unlike the low-frequency estimate in Proposition~\ref{prop:low_frequency_local_smoothing}, the estimate above is not at a fixed frequency, but rather it is integrated in $s$. As we will see (cf. the proof of Lemma~\ref{lem:high_freq_multiplier_identity_integrated} below) a main technical difficulty in the high-frequency regime is the lack of a Bernstein type inequality 
\begin{align*}
\begin{split}
 s^{-\frac{1}{2}} \| \tilP_{s}u\|_{L^2} \lesssim \|\nabla \tilP_s u\|_{L^2}.
\end{split}
\end{align*}
This should be compared with the available estimate $\|\nabla \tilP_s u\|_{L^2}\lesssim s^{-\frac{1}{2}}\|\tilP_{\frac{s}{2}}u\|_{L^2}$, which is the analogue of the other direction of the standard Bernstein inequalities. This issue did not arise in the low-frequency regime because there we used the weighted Hardy-Poincar\'e estimate to bound the difference
\begin{align*}
\begin{split}
\angles{(\partial_r\beta^{(\alpha)})\nabla \tilP_{\geq s}u}{\nabla\tilP_{\geq s}u}-\rho^2\angles{(\partial_r\beta^{(\alpha)})\tilP_{\geq s}u}{\tilP_{\geq s}u}
\end{split}
\end{align*}
from below. In the high-frequency regime in this section, where $\tilP_{\geq s}u$ is replaced by $\tilP_{s}u$, it is crucial for our smoothing estimate to gain regularity, so we cannot resort to the weighted Hardy-Poincar\'e estimate anymore. On the other hand, heuristically we expect that at high frequencies we should be able to absorb the contribution of $\tilP_su$ by that of $\nabla\tilP_su$. The integration in $s$ is introduced precisely to accomplish this by compensating for the aforementioned missing direction of Bernstein's inequality (cf. Lemma~\ref{lem:integrated_bernstein_substitute}). 
\end{rem}

We start by recording an integration identity in the heat variable $s$ which, as discussed above, compensates for the failure of Bernstein's estimates.

\begin{lem} \label{lem:integrated_bernstein_substitute}
 Let $w \colon (0,\infty) \times \bbH^d \to \bbR$ be a radially symmetric weight function that also depends on the heat time $s$. For any $0 < \sigma_1 < \sigma_2 < \infty$ it holds that
 \begin{equation} \label{equ:integrated_bernstein_substitute}
  \begin{aligned}
   &2 \int_{\sigma_1}^{\sigma_2} \bigl( \langle w \nabla \tilP_s u, \nabla \tilP_s u \rangle -\rho^2\angles{w\tilP_su}{\tilP_su} - \rho \langle \coth(r) (\partial_r w) \tilP_s u, \tilP_s u \rangle \bigr) \, \ds \\
   &= \frac{1}{\sigma_1} \langle w (\sigma_1) \tilP_{\sigma_1} u, \tilP_{\sigma_1} u \rangle - \frac{1}{\sigma_2} \langle w(\sigma_2) \tilP_{\sigma_2} u, \tilP_{\sigma_2} u \rangle \\
   &\quad + \int_{\sigma_1}^{\sigma_2} \frac{1}{s} \langle w \tilP_s u, \tilP_s u \rangle \, \ds + \int_{\sigma_1}^{\sigma_2} \langle (\partial_s w) \tilP_s u, \tilP_s u \rangle \, \ds \\
   &\quad + \int_{\sigma_1}^{\sigma_2} \langle (\partial_r^2 w) \tilP_s u, \tilP_s u \rangle \, \ds.
  \end{aligned}
 \end{equation}
\end{lem}

\begin{proof}
 Integrating by parts in space we obtain that
 \begin{align*}
  \int_{\sigma_1}^{\sigma_2} \langle w \nabla \tilP_s u, \nabla \tilP_s u \rangle \, \ds &= \frac{1}{2} \int_{\sigma_1}^{\sigma_2} \langle \Delta w \tilP_s u, \tilP_s u \rangle \, \ds - \Re \int_{\sigma_1}^{\sigma_2} \langle w \Delta \tilP_s u, \tilP_s u \rangle \, \ds \\
  &= \frac{1}{2} \int_{\sigma_1}^{\sigma_2} \langle \Delta w \tilP_s u, \tilP_s u \rangle \, \ds + \rho^2 \int_{\sigma_1}^{\sigma_2} \langle w \tilP_s u, \tilP_s u \rangle \, \ds \\
  &\quad - \Re \int_{\sigma_1}^{\sigma_2} \langle w (\Delta + \rho^2) \tilP_s u, \tilP_s u \rangle \, \ds.
 \end{align*}
 In the last term on the right-hand side of the previous equation we insert the identity
 \[
  (\Delta + \rho^2) \tilP_s u = \frac{\ud}{\ud s} \tilP_s u - \frac{1}{s} \tilP_s u,
 \]
 which follows from the heat flow definition \eqref{eq:tilPdef} of the frequeny projection $\tilP_s$. Integrating by parts with respect to the heat time $s$, we then find that
 \begin{equation}
  \begin{aligned}
   \int_{\sigma_1}^{\sigma_2} \langle w \nabla \tilP_s u, \nabla \tilP_s u \rangle \, \ds &= \frac{1}{2} \int_{\sigma_1}^{\sigma_2} \langle \Delta w \tilP_s u, \tilP_s u \rangle \, \ds + \rho^2 \int_{\sigma_1}^{\sigma_2} \langle w \tilP_s u, \tilP_s u \rangle \, \ds \\
   &\quad + \frac{1}{2} \int_{\sigma_1}^{\sigma_2} \frac{1}{s} \langle w \tilP_s u, \tilP_s u \rangle \, \ds + \frac{1}{2} \int_{\sigma_1}^{\sigma_2} \langle (\partial_s w) \tilP_s u, \tilP_s u \rangle \, \ds \\
   &\quad + \frac{1}{2 \sigma_1} \langle w(\sigma_1) \tilP_{\sigma_1} u, \tilP_{\sigma_1} u \rangle - \frac{1}{2 \sigma_2} \langle w(\sigma_2) \tilP_{\sigma_2} u, \tilP_{\sigma_2} \rangle.
  \end{aligned}
 \end{equation}
 Since the weight function $w$ is assumed to be radially symmetric, we have 
 \[
  \Delta w = \partial_r^2 w + 2 \rho \coth(r) \partial_r w.
 \]
 Inserting this expression in the first term on the right-hand side of the previous equation, we arrive at the desired identity~\eqref{equ:integrated_bernstein_substitute} upon rearranging.
\end{proof}

For high frequencies the following technical lemma plays the role of Lemma~\ref{l:technical_lemma_low_freq}.

\begin{lem} \label{lem:technical_lemma_high_freq}
 Let $0 < \delta \leq 1$ and let $\{ \alpha_\ell \} \in \calA$ be a slowly varying sequence with associated slowly varying function $\alpha$ as in Definition~\ref{d:alpha}. For any $s > 0$, let $\beta_s^{(\alpha)}$ be defined as in~\eqref{equ:definition_beta_high_freq}. Then we have uniformly for all $s > 0$ and for all $\{ \alpha_\ell \} \in \calA$ that
 \begin{align}
  \beta_s^{(\alpha)} \coth(r) - \partial_r \beta_s^{(\alpha)} &\geq 0, \label{equ:betas_coth_minus_partial_betas_nonnegative} \\
  \partial_s \partial_r \beta_s^{(\alpha)} &\geq 0, \label{equ:mixed_second_derivative_beta_s_nonnegative} \\
  | \partial_r^3 \beta_s^{(\alpha)} | &\lesssim \delta^2 s^{-1} \partial_r \beta_s^{(\alpha)}. \label{equ:third_derivative_betas_bound}
 \end{align}
 Moreover, let $0 < \delta \ll 1$ be fixed sufficiently small (independently of $s_2$) and let $s_2 > 0$ be given. Then there exist constants $C \equiv C(\delta, s_2) \geq 1$ and $R \equiv R(\delta, s_2) \geq 1$ such that uniformly for all $0 < s \leq s_2$ and for all $\{ \alpha_\ell \} \in \calA$ it holds that
 \begin{equation} \label{equ:upper_bound_betas_coth_difference_2D}
  0 \leq \frac{\beta_s^{(\alpha)} \coth(r) - \partial_r \beta_s^{(\alpha)}}{\sinh^2(r)} \leq \delta s^{-1} \partial_r \beta_s^{(\alpha)} + C \chi_{\{ r \leq R\}},
 \end{equation}
 where $\chi_{\{r \leq R\}}$ is a cut-off function adapted to the region $\{ r \leq R\}$.
\end{lem}

\begin{proof}
 In what follows we drop the superscript from $\beta_s^{(\alpha)}$ and recall that $\partial_r \beta_s = \delta \frac{\alpha(\delta s^{-\frac{1}{2}} r)}{\langle \delta s^{-\frac{1}{2}} r \rangle}$. The lower bounds~\eqref{equ:betas_coth_minus_partial_betas_nonnegative} and \eqref{equ:mixed_second_derivative_beta_s_nonnegative} are a consequence of the fact that the function $y \mapsto \frac{\alpha(y)}{\langle y \rangle}$ is non-increasing. Specifically, \eqref{equ:betas_coth_minus_partial_betas_nonnegative} follows exactly as in the proof of the lower bound~\eqref{equ:beta_coth_minus_partial_beta_nonnegative} in Lemma~\ref{l:technical_lemma_low_freq}, while for \eqref{equ:mixed_second_derivative_beta_s_nonnegative} we compute that
 \[
  \partial_s \partial_r \beta_s = - \frac{1}{2} s^{-\frac{3}{2}} \delta^2 r \frac{\ud}{\ud y} \Bigl( \frac{\alpha(y)}{\langle y \rangle} \Bigr) \bigg|_{y = \delta s^{-\frac{1}{2}} r} \geq 0.
 \]
 Moreover, using the symbol-type bounds $|\alpha^{(j)}(y)| \lesssim \frac{\alpha(y)}{\langle y \rangle^j}$, we find that
 \[
  |\partial_r^3 \beta_s| \lesssim \delta^3 s^{-1} \frac{\alpha(\delta s^{-\frac{1}{2}} r)}{\langle \delta s^{-\frac{1}{2}} r \rangle^3} \lesssim \delta^2 s^{-1} \partial_r \beta_s,
 \]
 which proves~\eqref{equ:third_derivative_betas_bound}. We now turn to the proof of the estimate~\eqref{equ:upper_bound_betas_coth_difference_2D}. The lower bound already follows from~\eqref{equ:betas_coth_minus_partial_betas_nonnegative} and so we focus on proving the delicate upper bound, for which we distinguish several cases.
 
 \medskip 
 
 \noindent \underline{{\it Case 1: $r \leq 1$ and $\delta s^{-\frac{1}{2}} r \leq 1$.}} Since $r \leq 1$ we have by Taylor expansion that
 \[
  \coth(r) = r^{-1} + O(r).
 \]
 In addition, recalling that $\alpha(0) \simeq 1$ and $\alpha'(0) = 0$ for any slowly varying function~$\alpha$, since $\delta s^{-\frac{1}{2}} r \leq 1$ we also obtain by Taylor expansion that
 \begin{align*}
  \beta_s = s^{\frac{1}{2}} \int_0^{\delta s^{-\frac{1}{2}} r} \frac{\alpha(y)}{\langle y \rangle} \, \ud y = s^{\frac{1}{2}} \Bigl( \alpha(0) (\delta s^{-\frac{1}{2}} r) + O\bigl( (\delta s^{-\frac{1}{2}} r)^3 \bigr) \Bigr) 
 \end{align*}
 and 
 \begin{align*}
  \partial_r \beta_s = \delta \frac{\alpha( \delta s^{-\frac{1}{2}} r )}{\langle \delta s^{-\frac{1}{2}} r \rangle} = \delta \Bigl( \alpha(0) + O\bigl( (\delta s^{-\frac{1}{2}} r)^2 \Bigr).
 \end{align*}
 Hence, in this case we have 
 \begin{align*}
  \beta_s \coth(r) - \partial_r \beta_s = \delta \alpha(0) O(r^2) + \delta O\bigl( (\delta s^{-\frac{1}{2}} r)^2 \bigr) + \delta r O\bigl( (\delta s^{-\frac{1}{2}} r)^2 \bigr)
 \end{align*}
 and it follows that 
 \begin{equation} \label{equ:technical_lemma_high_freq_case1_derivation1}
  \frac{\beta_s \coth(r) - \partial_r \beta_s}{\sinh^2(r)} \leq C_1 \delta + C_1 \delta^3 s^{-1}.
 \end{equation}
 for some absolute constant $C_1 \geq 1$. On the other hand, since $\delta s^{-\frac{1}{2}} r \leq 1$, we have
 \begin{equation*}
  \delta s^{-1} \partial_r \beta_s = \delta^2 s^{-1} \frac{\alpha(\delta s^{-\frac{1}{2}} r)}{\langle \delta s^{-\frac{1}{2}} r \rangle} \simeq \delta^2 s^{-1}.
 \end{equation*}
 Thus, by spending one power of $\delta$ for sufficiently small $0 < \delta \ll 1$, we may bound the second term on the right-hand side of~\eqref{equ:technical_lemma_high_freq_case1_derivation1} by $\delta s^{-1} \partial_r \beta_s$ and conclude for this case that
 \begin{equation}
  \frac{\beta_s \coth(r) - \partial_r \beta_s}{\sinh^2(r)} \leq \delta s^{-1} \partial_r \beta_s + C \chi_{\{ r \leq 1\}}.
 \end{equation}
 
 \medskip 
 
 \noindent \underline{{\it Case 2: $r \leq 1$ and $\delta s^{-\frac{1}{2}} r > 1$.}} We first observe that by the uniform lower bound 
 \[
  \alpha(y) \gtrsim \frac{1}{\langle y \rangle^\eta} \quad \forall \, y \geq 0
 \]
 for any slowly varying function as in Definition~\ref{d:alpha}, we obtain in this case that 
 \begin{equation} \label{equ:technical_lemma_high_freq_case2_derivation1}
  \delta s^{-1} \partial_r \beta_s = \delta^2 s^{-1} \frac{\alpha(\delta s^{-\frac{1}{2}} r)}{\langle \delta s^{-\frac{1}{2}} r \rangle} \gtrsim \delta^2 s^{-1} (\delta s^{-\frac{1}{2}} r)^{-(1+\eta)} \simeq \delta^{1-\eta} s^{-\frac{1}{2} + \frac{1}{2} \eta} r^{-(1+\eta)}.
 \end{equation}
 On the other hand, dropping the negative contribution of $-\partial_r \beta_s$ and using that in this case $\beta_s \lesssim s^{\frac{1}{2}}$, $\coth(r) \lesssim r^{-1}$, $\sinh^{-2}(r) \lesssim r^{-2}$, and $r^{-1} \leq \delta s^{-\frac{1}{2}}$ we find that
 \begin{align*}
  \frac{\beta_s \coth(r) - \partial_r \beta_s}{\sinh^2(r)} \leq \frac{\beta_s \coth(r)}{\sinh^2(r)} \lesssim s^{\frac{1}{2}} r^{-3} &\lesssim s^{\frac{1}{2}} (\delta s^{-\frac{1}{2}})^{2-\eta} r^{-(1+\eta)} \\
  &\simeq \delta \bigl( \delta^{1-\eta} s^{-\frac{1}{2} + \frac{1}{2} \eta} r^{-(1+\eta)} \bigr) \\
  &\lesssim \delta \bigl( \delta s^{-1} \partial_r \beta_s \bigr),
 \end{align*}
 where in the last line we inserted the bound~\eqref{equ:technical_lemma_high_freq_case2_derivation1}. Hence, by spending one factor of $\delta$ for sufficiently small $0 < \delta \ll 1$, we obtain for this case the estimate
 \begin{equation}
  \frac{\beta_s \coth(r) - \partial_r \beta_s}{\sinh^2(r)} \leq \delta s^{-1} \partial_r \beta_s.
 \end{equation}

 \medskip 
 
 \noindent \underline{{\it Case 3: $r > 1$ and $\delta s^{-\frac{1}{2}} r \leq 1$.}} Since $\delta s^{-\frac{1}{2}} r \leq 1$ and $\alpha(y) \simeq 1$ for $0 \leq y \leq 1$ for any slowly varying function $\alpha$, we have in this case that
 \[
  \beta_s = s^{\frac{1}{2}} \int_0^{\delta s^{-\frac{1}{2}} r} \frac{\alpha(y)}{\langle y \rangle} \, dy \lesssim s^{\frac{1}{2}} (\delta s^{-\frac{1}{2}} r) \lesssim \delta r.
 \]
 Dropping the negative contribution of $-\partial_r \beta_s$ we therefore obtain that
 \begin{equation}
  \frac{\beta_s \coth(r) - \partial_r \beta_s}{\sinh^2(r)} \lesssim \delta \frac{r}{\sinh^2(r)},
 \end{equation}
 which is uniformly bounded by $\lesssim 1$ for all $r > 1$. On the other hand, since $\delta s^{-\frac{1}{2}} r \leq 1$ and since $s^{-1} \geq s_2^{-1}$ for any $0 < s \leq s_2$, we have that 
 \begin{equation}
  \delta s^{-1} \partial_r \beta_s = \delta^2 s^{-1} \frac{\alpha(\delta s^{-\frac{1}{2}} r)}{\langle \delta s^{-\frac{1}{2}} r \rangle} \gtrsim \delta^2 s_2^{-1}.
 \end{equation}
 Combining the previous two estimates we infer by the exponential decay of $\sinh^{-2}(r)$ that there exists a constant $R \equiv R(\delta, s_2) \gg 1$ such that 
 \begin{equation}
  \frac{\beta_s \coth(r) - \partial_r \beta_s}{\sinh^2(r)} \leq \delta s^{-1} \partial_r \beta_s \quad \forall \, r \geq R.
 \end{equation}
 In conclusion, we have established for this case the desired estimate
 \begin{equation}
  \frac{\beta_s \coth(r) - \partial_r \beta_s}{\sinh^2(r)} \leq \delta s^{-1} \partial_r \beta_s + C \chi_{\{r \leq R\}}.
 \end{equation}

 \medskip 
 
 \noindent \underline{{\it Case 4: $r > 1$ and $\delta s^{-\frac{1}{2}} r > 1$.}} Using that $\coth(r) \lesssim 1$ for $r > 1$ and dropping the negative contribution of $-\partial_r \beta_s$, we find in this case that
 \begin{equation}
  \frac{\beta_s \coth(r) - \partial_r \beta_s}{\sinh^2(r)} \lesssim \frac{s^{\frac{1}{2}}}{\sinh^2(r)} \lesssim \frac{s_2^{\frac{1}{2}}}{\sinh^2(r)},
 \end{equation}
 which is trivially bounded by $\lesssim s_2^{\frac{1}{2}}$ for all $r > 1$. On the other hand, we obtain by proceeding as in~\eqref{equ:technical_lemma_high_freq_case2_derivation1} of Case 2 the lower bound 
 \begin{equation}
  \begin{aligned}
   \delta s^{-1} \partial_r \beta_s = \delta^2 s^{-1} \frac{\alpha(\delta s^{-\frac{1}{2}} r)}{\langle \delta s^{-\frac{1}{2}} r \rangle} \gtrsim \delta^2 s^{-1} ( \delta s^{-\frac{1}{2}} r )^{-(1+\eta)} \gtrsim \delta^{1-\eta} s_2^{-\frac{1}{2} + \frac{1}{2} \eta} r^{-(1+\eta)}.    
  \end{aligned}
 \end{equation}
 Combining the previous two estimates we again conclude by the exponential decay of $\sinh^{-2}(r)$ that there exists a constant $R \equiv R(\delta, s_2) \gg 1$ such that 
 \begin{equation}
  \frac{\beta_s \coth(r) - \partial_r \beta_s}{\sinh^2(r)} \leq \delta s^{-1} \partial_r \beta_s \quad \forall \, r \geq R,
 \end{equation}
 which yields the desired estimate~\eqref{equ:upper_bound_betas_coth_difference_2D} also for this case and finishes the proof of the lemma.
\end{proof}

\begin{lem} \label{lem:high_freq_multiplier_identity_integrated}
 Let $s_2 > 0$ and let $\{ \alpha_\ell \} \in \calA$ be a slowly varying sequence with associated slowly varying function $\alpha$ as in Definition~\ref{d:alpha}. For $0 < s \leq s_2$ let $\beta_s^{(\alpha)}$, $Q_s^{(\alpha)}$, and $\tilF_s^{(\alpha)}$ be defined as in~\eqref{equ:definition_beta_high_freq},~\eqref{equ:definition_Q_high_freq}, and~\eqref{eq:tilFsdef} respectively. There exist constants $C \equiv C(s_2) \geq 1$ and $R \equiv R(s_2) \geq 1$ such that for all solutions $u(t)$ to~\eqref{eq:S} and for any $0 < \sigma_1 < \sigma_2 \leq s_2$ it holds that 
  \begin{equation} \label{equ:high_freq_multiplier_identity_integrated}
  \begin{aligned}
   &\frac{1}{\sigma_1} \langle (\partial_r \beta_{\sigma_1}^{(\alpha)}) \tilP_{\sigma_1} u, \tilP_{\sigma_1} u \rangle_{t,x} + \frac{1}{2} \int_{\sigma_1}^{\sigma_2} \frac{1}{s} \langle (\partial_r \beta_s^{(\alpha)}) \tilP_s u, \tilP_s u \rangle_{t,x} \, \ds \\
   &\leq \frac{1}{\sigma_2} \langle (\partial_r \beta_{\sigma_2}^{(\alpha)}) \tilP_{\sigma_2} u, \tilP_{\sigma_2} u \rangle_{t,x} + C \int_{\sigma_1}^{\sigma_2} \sup_{t \in \bbR} \, \bigl| \langle \tilP_s u(t), Q_s^{(\alpha)} \tilP_s u(t) \rangle \bigr| \, \ds \\
   &\quad + \int_{\sigma_1}^{\sigma_2} | \tilF_s^{(\alpha)}| \, \ds + C \int_{\sigma_1}^{\sigma_2} \| \tilP_s u \|_{L^2(\bbR \times \{ r \leq R \})}^2 \, \ds.
  \end{aligned}
 \end{equation}
 \end{lem}
\begin{proof}
 To simplify notation we write
\begin{align*}
\begin{split}
&|\snabla\tilP_su|_{\slashed{\ringa}}^2:=\ringa^{\theta_a\theta_b}\partial_{\theta_a}\tilP_su\partial_{\theta_b}\tilP_su,\qquad |\snabla\tilP_su|_{\slashed{\bsa}}^2:=\bsa^{\theta_a\theta_b}\partial_{\theta_a}\tilP_su\partial_{\theta_b}\tilP_su.
\end{split}
\end{align*}
 We begin by writing our main multiplier identity~\eqref{eq:Q1} applied with $Q_s^{(\alpha)}$ defined in~\eqref{equ:definition_Q_high_freq} as 
 \begin{equation}
  \begin{aligned}
   &\Re \langle i \Delta \tilP_s u, Q_s^{(\alpha)} \tilP_s u \rangle \\
   &\quad \quad = 2 \int_{\bbH^d} \bigl( (\partial_r \beta_s^{(\alpha)}) \bigl( |\nabla \tilP_s u|^2 - \rho^2 |\tilP_s u|^2 \bigr) - \rho (\partial_r^2 \beta_s^{(\alpha)}) \coth(r) |\tilP_s u|^2 \bigr) \, \dh \\
   &\quad \quad \quad + 2 \int_{\bbH^d} \bigl( \beta_s^{(\alpha)} \coth(r) - \partial_r \beta_s^{(\alpha)} \bigr) \frac{|\slashed{\nabla} \tilP_s u|^2}{\sinh^2r} \, \dh \\
   &\quad \quad \quad + 2 \int_{\bbH^d} \Bigl( -\frac{\partial_r^3 \beta_s^{(\alpha)}}{4} + \frac{\rho (\rho - 1)}{\sinh^2(r)} \bigl( \beta_s^{(\alpha)} \coth(r) - \partial_r \beta_s^{(\alpha)} \bigr) \Bigr) |\tilP_s u|^2 \, \dh. 
  \end{aligned}
 \end{equation}
 Integrating this equation with respect to the heat time $\ds$ from $\sigma_1$ to $\sigma_2$ and applying Lemma~\ref{lem:integrated_bernstein_substitute} with the weight $w := \partial_r \beta_s^{(\alpha)}$ to the first integral on the right-hand side of the resulting integrated equation, we obtain that
 \begin{equation} \label{equ:key_identity_applied_to_Ps_integrated}
  \begin{aligned}
   &\int_{\sigma_1}^{\sigma_2} \Re \langle i \Delta \tilP_s u, Q_s^{(\alpha)} \tilP_s u \rangle \, \ds + 2 \int_{\sigma_1}^{\sigma_2} \int_{\bbH^d} \bigl( \beta_s^{(\alpha)} \coth(r) - \partial_r \beta_s^{(\alpha)} \bigr) |\slashed{\nabla} \tilP_s u|_{\slashed{\ringa}}^2 \, \dh \, \ds\\
   &\quad \quad = \frac{1}{\sigma_1} \langle (\partial_r \beta_{\sigma_1}^{(\alpha)}) \tilP_{\sigma_1} u, \tilP_{\sigma_1} u \rangle - \frac{1}{\sigma_2} \langle (\partial_r \beta_{\sigma_2}^{(\alpha)}) \tilP_{\sigma_2} u, \tilP_{\sigma_2} u \rangle \\
   &\quad \quad \quad + \int_{\sigma_1}^{\sigma_2} \frac{1}{s} \langle (\partial_r \beta_s^{(\alpha)}) \tilP_s u, \tilP_s u \rangle \, \ds + \int_{\sigma_1}^{\sigma_2} \langle (\partial_s \partial_r \beta_{s}^{(\alpha)}) \tilP_s u, \tilP_s u \rangle \, \ds \\
   &\quad \quad \quad + \frac{1}{2} \int_{\sigma_1}^{\sigma_2} \langle (\partial_r^3 \beta_s^{(\alpha)}) \tilP_s u, \tilP_s u \rangle \, \ds \\
   &\quad \quad \quad + 2 \int_{\sigma_1}^{\sigma_2} \int_{\bbH^d} \bigl( \beta_s^{(\alpha)} \coth(r) - \partial_r \beta_s^{(\alpha)} \bigr) |\slashed{\nabla} \tilP_s u|_{\slashed{\bsa}}^2 \, \dh \, \ds \\
   &\quad \quad \quad + 2 \int_{\sigma_1}^{\sigma_2} \int_{\bbH^d} \frac{\rho (\rho - 1)}{\sinh^2(r)} \bigl( \beta_s^{(\alpha)} \coth(r) - \partial_r \beta_s^{(\alpha)} \bigr) |\tilP_s u|^2 \, \dh \, \ds.
  \end{aligned}
 \end{equation}
 Now we note that $\rho (\rho - 1) \geq 0$ for dimensions $d \geq 3$, while $\rho (\rho -1) = -\frac{1}{4}$ for $d=2$. Correspondingly, using Lemma~\ref{lem:technical_lemma_high_freq}, for dimensions $d \geq 3$ we find that
 \begin{equation} 
  \begin{aligned}
   &\int_{\sigma_1}^{\sigma_2} \Re \langle i \Delta \tilP_s u, Q_s^{(\alpha)} \tilP_s u \rangle \, \ds+ 2 \int_{\sigma_1}^{\sigma_2} \int_{\bbH^d} \bigl( \beta_s^{(\alpha)} \coth(r) - \partial_r \beta_s^{(\alpha)} \bigr) |\slashed{\nabla} \tilP_s u|_{\slashed{\ringa}}^2 \, \dh \, \ds \\
   &\quad \quad \geq \frac{1}{\sigma_1} \langle (\partial_r \beta_{\sigma_1}^{(\alpha)}) \tilP_{\sigma_1} u, \tilP_{\sigma_1} u \rangle - \frac{1}{\sigma_2} \langle (\partial_r \beta_{\sigma_2}^{(\alpha)}) \tilP_{\sigma_2} u, \tilP_{\sigma_2} u \rangle \\
   &\quad \quad \quad + \int_{\sigma_1}^{\sigma_2} \frac{1}{s} \langle (\partial_r \beta_s^{(\alpha)}) \tilP_s u, \tilP_s u \rangle \, \ds + \frac{1}{2} \int_{\sigma}^{s_0} \langle (\partial_r^3 \beta_s^{(\alpha)}) \tilP_s u, \tilP_s u \rangle \, \ds \\
   &\quad \quad \geq \frac{1}{\sigma_1} \langle (\partial_r \beta_{\sigma_1}^{(\alpha)}) \tilP_{\sigma_1} u, \tilP_{\sigma_1} u \rangle - \frac{1}{\sigma_2} \langle (\partial_r \beta_{\sigma_2}^{(\alpha)}) \tilP_{\sigma_2} u, \tilP_{\sigma_2} u \rangle \\
   &\quad \quad \quad + (1 - C \delta^2) \int_{\sigma_1}^{\sigma_2} \frac{1}{s} \langle (\partial_r \beta_s^{(\alpha)}) \tilP_s u, \tilP_s u \rangle \, \ds,
  \end{aligned}
 \end{equation}
 while for dimension $d = 2$ we obtain 
 \begin{equation} 
  \begin{aligned}
   &\int_{\sigma_1}^{\sigma_2} \Re \langle i \Delta \tilP_s u, Q_s^{(\alpha)} \tilP_s u \rangle \, \ds + 2 \int_{\sigma_1}^{\sigma_2} \int_{\bbH^d} \bigl( \beta_s^{(\alpha)} \coth(r) - \partial_r \beta_s^{(\alpha)} \bigr) |\slashed{\nabla} \tilP_s u|_{\slashed{\ringa}}^2 \, \dh \, \ds\\
   &\quad \quad \geq \frac{1}{\sigma_1} \langle (\partial_r \beta_{\sigma_1}^{(\alpha)}) \tilP_{\sigma_1} u, \tilP_{\sigma_1} u \rangle - \frac{1}{\sigma_2} \langle (\partial_r \beta_{\sigma_2}^{(\alpha)}) \tilP_{\sigma_2} u, \tilP_{\sigma_2} u \rangle \\
   &\quad \quad \quad + (1 - C \delta^2) \int_{\sigma_1}^{\sigma_2} \frac{1}{s} \langle (\partial_r \beta_s^{(\alpha)}) \tilP_s u, \tilP_s u \rangle \, \ds \\
   &\quad \quad \quad -\frac{1}{2} \int_{\sigma_1}^{\sigma_2} \int_{\bbH^d} \frac{ \beta_s^{(\alpha)} \coth(r) - \partial_r \beta_s^{(\alpha)}}{\sinh^2(r)} |\tilP_s u|^2 \, \dh \, \ds \\
   &\quad \quad \geq \frac{1}{\sigma_1} \langle (\partial_r \beta_{\sigma_1}^{(\alpha)}) \tilP_{\sigma_1} u, \tilP_{\sigma_1} u \rangle - \frac{1}{\sigma_2} \langle (\partial_r \beta_{\sigma_2}^{(\alpha)}) \tilP_{\sigma_2} u, \tilP_{\sigma_2} u \rangle \\
   &\quad \quad \quad + (1 - C \delta^2 - {\textstyle \frac{1}{2}} \delta) \int_{\sigma_1}^{\sigma_2} \frac{1}{s} \langle (\partial_r \beta_s^{(\alpha)}) \tilP_s u, \tilP_s u \rangle \, \ds \\
   &\quad \quad \quad - C \int_{\sigma_1}^{\sigma_2} \| \tilP_s u \|_{L^2(\bbR \times \{ r \leq R \})}^2 \, \ds.
  \end{aligned}
 \end{equation}
 In either case, choosing $0 < \delta \ll 1$ sufficiently small, we conclude that
 \begin{equation}
  \begin{aligned}
   &\frac{1}{\sigma_1} \langle (\partial_r \beta_{\sigma_1}^{(\alpha)} \tilP_{\sigma_1} u, \tilP_{\sigma_1} u \rangle + \frac{1}{2} \int_{\sigma_1}^{\sigma_2} \frac{1}{s} \langle (\partial_r \beta_s^{(\alpha)}) \tilP_s u, \tilP_s u \rangle \, \ds \\
   &\quad \quad \leq \frac{1}{\sigma_2} \langle (\partial_r \beta_{\sigma_2}^{(\alpha)}) \tilP_{\sigma_2} u, \tilP_{\sigma_2} u \rangle  + \int_{\sigma_1}^{\sigma_2} \Re \langle i \Delta \tilP_s u, Q_s^{(\alpha)} \tilP_s u \rangle \, \ds \\
   &\quad \quad \quad  + 2 \int_{\sigma_1}^{\sigma_2} \int_{\bbH^d} \bigl( \beta_s^{(\alpha)} \coth(r) - \partial_r \beta_s^{(\alpha)} \bigr) |\slashed{\nabla} \tilP_s u|_{\slashed{\ringa}}^2 \, \dh \, \ds  \\
   &\quad \quad \quad + C \int_{\sigma_1}^{\sigma_2} \|\tilP_s u\|_{L^2(\bbR\times\{r\leq R\})}^2 \, \ds.
  \end{aligned}
 \end{equation}
 The lemma now follows from the previous inequality upon integrating the following identity (see \eqref{eq:dtQu}) in time $\ud t$ as well as in heat time $\ds$ from $\sigma_1$ to $\sigma_2$ and rearranging
 \begin{equation}
  \Re \langle i \Delta \tilP_s u, Q_s^{(\alpha)} \tilP_s u \rangle = \frac{1}{2} \frac{\ud}{\ud t} \langle \tilP_s u, Q_s^{(\alpha)} \tilP_s u \rangle - \Re \langle (\partial_t - i \Delta) \tilP_s u, Q_s^{(\alpha)} \tilP_s u \rangle. 
 \end{equation}
\end{proof}

We are now prepared to prove Proposition~\ref{prop:high_frequency_local_smoothing}

\begin{proof}[Proof of Proposition~\ref{prop:high_frequency_local_smoothing}]
Let $0 < s \leq \frac{s_2}{2}$ be arbitrary. We apply the estimate \eqref{equ:high_freq_multiplier_identity_integrated} from Lemma~\ref{lem:high_freq_multiplier_identity_integrated} with $\sgm_1$ and $\sigma_2$ replaced by $s$ and $2s$, respectively, to get
\begin{equation} \label{eq:dyadic-s} 
 \begin{aligned}
  &\frac{1}{s} \langle (\partial_r \beta_s^{(\alpha)}) \tilP_{s} u, \tilP_{s} u \rangle_{t,x} + \frac{1}{2} \int_{s}^{2 s} \frac{1}{s'} \langle (\partial_r \beta_{s'}^{(\alpha)})  \tilP_{s'} u, \tilP_{s'} u \rangle_{t,x} \, \frac{\ud s'}{s'} \\
  &\quad \leq \frac{1}{2 s} \langle (\partial_r \beta_{2s}^{(\alpha)}) \tilP_{2 s} u, \tilP_{2 s} u \rangle_{t,x} + C \int_{s}^{2s} \sup_{t \in \bbR} \, \bigl| \langle \tilP_{s'} u(t), Q_{s'}^{(\alpha)} \tilP_{s'} u(t) \rangle \bigr| \, \frac{\ud s'}{s'} \\
  &\quad \quad + \int_{s}^{2s}  |\tilF_{s'}^{(\alpha)}|  \, \frac{\ud s'}{s'}  + C \int_{s}^{2s} \| \tilP_{s'} u \|_{L^2(\bbR\times\{r\leq R\})}^2 \frac{\ud s'}{s'}.
 \end{aligned}
\end{equation}
Naively, we would like to take the supremum in $\alp$ and then sum up in $s \in 2^{-\bbN}$. An immediate problem is dealing with the first term on the right-hand side. The hope is that we can perform a ``telescopic sum'', i.e., that the first term on the right-hand side of \eqref{eq:dyadic-s} at $s$ is controlled by \eqref{eq:dyadic-s} at $2 s$, and so on. However, there is still the issue that the supremum of the whole left-hand side may not control the supremum of the first term on the left-hand side. A simple way out is to average \eqref{eq:dyadic-s} in $s$, which makes the two terms on the left-hand side of \eqref{eq:dyadic-s} effectively comparable. 

Let $N \geq 1 $ be a constant to be fixed sufficiently large later in the proof. For $\sigma > 0$, consider the weight function
\begin{equation*}
 \chi_{\sgm}(s) = \left\{
  \begin{array}{ll}
   ( \frac{s}{\sigma} )^N & \hbox{ when } s \leq \sgm, \\
   0 & \hbox{ when } s > \sgm.
  \end{array} \right.
\end{equation*}
We now multiply \eqref{eq:dyadic-s} by $\chi_{\sgm}(s)$ and integrate in $\frac{\ud s}{s}$. To simplify the notation we define
\begin{equation*}
 g(s) =  \frac{1}{s} \langle (\partial_r \beta_s^{(\alpha)}) \tilP_s u, \tilP_s u \rangle_{t,x}.
\end{equation*}
Thus, the left-hand side of \eqref{eq:dyadic-s} equals $g(s) + \frac{1}{2} \int_{s}^{2s} g(s') \, \frac{\ud s'}{s'}$, and the first term on the right-hand side of \eqref{eq:dyadic-s} equals $g(2s)$. Making a change of variables and using the simple relation $\chi_{a \sgm} (s)= \chi_{\sgm}(a^{-1} s)$ for any $a > 0$, we arrive at the estimate
\begin{equation} \label{eq:dyadic-avg-sgm} 
 \begin{aligned}
  &\int_{0}^{\infty} \chi_{\sgm}(s) g(s) \, \frac{\ud s}{s} + \frac{1}{2} \int_{\sgm}^{2 \sgm} \int_{0}^{\infty} \chi_{\sgm'}(s) g(s) \, \frac{\ud s}{s} \, \frac{\ud \sgm'}{\sgm'} \\
  &\leq \int_{0}^{\infty} \chi_{2 \sgm}(s) g(s) \, \frac{\ud s}{s} + C \int_{\sgm}^{2 \sgm} \int_{0}^{\infty} \chi_{\sgm'}(s) \sup_{t \in \bbR} \, \bigl| \langle \tilP_{s} u(t), Q_s^{(\alpha)} \tilP_{s} u(t) \rangle \bigr| \, \frac{\ud s}{s} \frac{\ud\sgm'}{\sgm'} \\
  &\quad + \int_{\sgm}^{2 \sgm} \int_{0}^{\infty} \chi_{\sgm'}(s) |\tilF_s^{(\alpha)}| \, \frac{\ud s}{s} \, \frac{\ud \sgm'}{\sgm'} \\
  &\quad + C \int_{\sgm}^{2 \sgm} \int_{0}^{\infty}\chi_{\sgm'}(s)\|\tilP_s u\|_{L^2(\bbR\times\{r\leq R\})}^2 \, \ds \, \frac{\ud\sgm'}{\sgm'},
 \end{aligned}
\end{equation}
which is valid for any $0 < \sigma \leq \frac{s_2}{2}$. Now observe that 
\begin{equation*}
 \chi_{\sgm_{1}} \leq 2^N \chi_{\sgm_{2}} \quad \hbox{ for } \sgm \leq \sgm_{1} \leq \sgm_{2} \leq 2 \sgm.
\end{equation*}
One consequence is that the left-hand side of \eqref{eq:dyadic-avg-sgm} is bounded from below by 
\begin{equation*}
 (1 + \mu_N) \int_{0}^{\infty} \chi_{\sgm}(s) g(s) \, \frac{\ud s}{s}, \quad \mu_N = 2^{-N-1} \log 2 > 0.
\end{equation*}
On the other hand, the last three terms on the right-hand side of \eqref{eq:dyadic-avg-sgm} are bounded, respectively, by 
\begin{gather*}
 C_{N} \int_{0}^{\infty} \chi_{2\sgm}(s) \sup_{t \in \bbR} \, \bigl| \langle \tilP_{s} u(t), Q_s^{(\alpha)} \tilP_{s} u(t) \rangle \bigr| \, \frac{\ud s}{s}, \\
 C_{N} \int_{0}^{\infty} \chi_{2 \sgm}(s)|\tilF_s^{(\alpha)}|  \, \frac{\ud s}{s}, \\
 C_{N} \int_{0}^{\infty} \chi_{2\sgm}(s) \| \tilP_s u\|_{L^2(\bbR\times\{r\leq R\})}^2 \, \frac{\ud s}{s},
\end{gather*}
where $C_N = 2^NC\log 2$. It follows that
\begin{equation} \label{eq:dyadic-avg-sgm-0} 
 \begin{aligned}
  &(1+\mu_N) \int_{0}^{\infty} \chi_{\sgm}(s) g(s) \, \frac{\ud s}{s} \\
  &\qquad \qquad \leq \int_{0}^{\infty} \chi_{2 \sgm}(s) g(s) \, \frac{\ud s}{s} \\
  &\qquad \qquad \quad + C_N \int_{0}^{\infty} \chi_{2 \sgm}(s) \sup_{t \in \bbR} \, \bigl| \langle \tilP_{s} u(t), Q_s^{(\alpha)} \tilP_{s} u(t) \rangle \bigr| \, \frac{\ud s}{s} \\
  &\qquad \qquad \quad + C_N \int_{0}^{\infty} \chi_{2 \sgm}(s) |\tilF_s^{(\alpha)}|\, \frac{\ud s}{s} \\
  &\qquad \qquad \quad + C_N \int_{0}^{\infty} \chi_{2 \sgm}(s) \|\tilP_s u\|_{L^2(\bbR \times \{r\leq R\})}^2 \ds.
 \end{aligned}
\end{equation}
Next we take the supremum over all $\{ \alpha_\ell \} \in \calA$ of~\eqref{eq:dyadic-avg-sgm-0} applied for $\sigma = 2^\ell$ for every integer $\ell \leq -2 k_{s_2}-1$. Summing up the resulting inequalities over all integers $\ell \leq - 2 k_{s_2} - 1$ and exploiting the cancellation between part of the term on the left-hand side of~\eqref{eq:dyadic-avg-sgm-0} and the first term on the right-hand side of~\eqref{eq:dyadic-avg-sgm-0} for consecutive values of $\ell$ we find that
\begin{equation*} 
 \begin{aligned}
  &\mu_N \sum_{\ell=-\infty}^{-2k_{s_{2}} -1} \sup_{ \{ \alpha_j \} \in \calA} \int_{0}^{\infty} \chi_{2^{\ell}}(s) g(s) \, \ds \\
  &\qquad \leq \sup_{ \{ \alpha_j \} \in \calA} \int_{0}^{\infty} \chi_{2^{-2k_{s_2}}}(s) g(s) \, \ds \\
  &\qquad \qquad + C_N \sum_{\ell=-\infty}^{-2k_{s_{2}} } \sup_{ \{ \alpha_j \} \in \calA} \int_{0}^{\infty} \chi_{2^{\ell}}(s) \sup_{t \in \bbR} \, \bigl| \langle \tilP_{s} u(t), Q_s^{(\alpha)} \tilP_{s} u(t) \rangle \bigr| \, \ds \\
  &\qquad \qquad + C_N \sum_{\ell=-\infty}^{-2k_{s_{2}} } \sup_{ \{ \alpha_j \} \in \calA} \int_{0}^{\infty} \chi_{2^{\ell}}(s) |\tilF_s^{(\alpha)}| \, \ds \\
  &\qquad \qquad + C_N \sum_{\ell=-\infty}^{-2k_{s_{2}}} \sup_{\{ \alpha_j \} \in \calA} \int_{0}^{\infty} \chi_{2^{\ell}}(s) \|\tilP_s u\|_{L^2(\bbR \times \{r \leq R\})}^2 \, \ds.
 \end{aligned}
\end{equation*}
Let us now remove the weight $\chi_{2^{\ell}}(s)$. For the left-hand side, we note that
\begin{equation*}
 \chi_{2^{\ell}}(s) \geq 2^{-N} \hbox{ for } 2^{\ell-1} \leq s \leq 2^{\ell}.
\end{equation*}
For the first term on the right-hand side we use, recalling that $2^{-2 k_{s_2}} = s_2$,
\begin{equation*}
	\chi_{ 2^{-2 k_{s_2}} }(s) 
	\left\{
	\begin{array}{ll}
	\leq 2^{-3 N} & \hbox{for } s \leq \frac{s_2}{8}, \\
	\leq 1 & \hbox{for } \frac{s_2}{8} < s \leq s_2, \\
	=0 & \hbox{for } s > s_{2}.
	\end{array}
	\right.	
\end{equation*}
On the other hand, for the last three terms on the right-hand side, we simply take the supremum over $\{ \alpha_j\} \in \calA$ and the summation over $\ell$ inside the $s$-integral, and observe that
\begin{equation*}
	\sum_{\ell=-\infty}^{-2k_{s_{2}} } \chi_{2^{\ell}}(s) 
	\left\{
	\begin{array}{ll}
	\aleq_N 1 & \hbox{for } s \leq s_{2} , \\
	= 0 & \hbox{for } s >  s_{2}.
	\end{array}
	\right.
\end{equation*}
Thus, we obtain the simplified estimate 
\begin{equation} \label{eq:dyadic-avg-sgm-1}
\begin{aligned}
 &\mu_N 2^{-N} \sum_{\ell=-\infty}^{-2k_{s_{2}} -1} \sup_{ \{ \alpha_j \} \in \calA} \int_{2^{\ell-1}}^{2^{\ell}} g(s) \, \frac{\ud s}{s} \\
 &\qquad \leq 2^{- 3N} \int_{0}^{\frac{s_2}{8}} \sup_{ \{ \alpha_j \} \in \calA} g(s) \, \frac{\ud s}{s} + \sup_{ \{ \alpha_\ell \} \in \calA} \int_{\frac{s_2}{8}}^{s_2} g(s) \, \frac{\ud s}{s} \\
 &\qquad \quad + C_N \int_{0}^{s_{2}} \sup_{ \{ \alpha_j \} \in \calA} \sup_{t \in \bbR} \, \bigl| \langle \tilP_{s} u(t), Q_s^{(\alpha)} \tilP_{s} u(t) \rangle \bigr| \, \frac{\ud s}{s} \\
 &\qquad \quad + C_N \int_{0}^{s_{2}} \sup_{ \{ \alpha_j \} \in \calA} |\tilF_s^{(\alpha)}|  \, \frac{\ud s}{s}  + C_N \int_{0}^{s_{2}} \| \tilP_s u \|_{L^2(\bbR \times \{ r \leq R\} )}^2 \, \frac{\ud s}{s},
\end{aligned}
\end{equation}
where $C_N$ are now different from before. We want to simplify the left-hand side further. To this end we observe that by the mean-value theorem, for each integer $\ell \leq -2k_{s_2}-1$ there exists $\tils \in [2^{\ell-1}, 2^{\ell}]$ such that $g(\tils) = \frac{1}{\log(2)} \int_{2^{\ell-1}}^{2^\ell}g(s)\ds$. Now for each $\ell \leq - 2 k_{s_2} -1$ we once more invoke the estimate~\eqref{equ:high_freq_multiplier_identity_integrated} from Lemma~\ref{lem:high_freq_multiplier_identity_integrated} with $\sigma_1 \in [2^{\ell-2}, 2^{\ell-1}]$ and $\sigma_2 = \tils$ to obtain after taking the supremum over all $\{ \alpha_j \} \in \calA$ that
\begin{align*}
 \sup_{ \{ \alpha_j \} \in \calA } \, \sup_{\sigma_1 \in [2^{\ell-2}, 2^{\ell-1}]} g(\sigma_1) &\leq C \sup_{ \{ \alpha_j \} \in \calA } \int_{2^{\ell-1}}^{2^{\ell}} g(s) \, \frac{\ud s}{s} \\
  &\quad + C \int_{2^{\ell-2}}^{2^{\ell}} \sup_{ \{ \alpha_j \} \in \calA} \, \sup_{t \in \bbR} \, \bigl| \langle \tilP_s u(t), Q_s^{(\alpha)} \tilP_s u(t) \rangle \bigr| \, \frac{\ud s}{s} \\
  &\quad + C \int_{2^{\ell-2}}^{2^{\ell}} \sup_{ \{ \alpha_j\} \in \calA} \, |\tilF_s^{(\alpha)}|  \, \frac{\ud s}{s} \\
  &\quad + C \int_{2^{\ell-2}}^{2^{\ell}} \| \tilP_s u \|_{L^2(\bbR \times \{ r\leq R \})}^2 \, \frac{\ud s}{s}.
\end{align*}
Of course, the left-hand side is further bounded from below by 
\begin{equation*}
 \frac{1}{\log(2)} \int_{2^{\ell-2}}^{2^{\ell-1}} \sup_{ \{ \alpha_j \} \in \calA} g(s) \, \frac{\ud s}{s}.
\end{equation*}
Combined with \eqref{eq:dyadic-avg-sgm-1}, this gives (where again $C_N$ can be different from above)
\begin{equation} \label{eq:dyadic-avg-sgm-2}
 \begin{aligned}
  &\mu_N 2^{-N} \int_{0}^{\frac{s_{2}}{4}} \sup_{ \{ \alpha_j \} \in \calA} g(s) \, \frac{\ud s}{s} \\
  &\qquad \leq C 2^{-3 N} \int_{0}^{\frac{s_2}{8}} \sup_{ \{ \alpha_j \} \in \calA} g(s) \, \frac{\ud s}{s} + C \sup_{ \{ \alpha_j \} \in \calA} \int_{\frac{s_2}{8}}^{s_2} g(s) \, \frac{\ud s}{s} \\
  &\qquad \quad + C_N \int_{0}^{s_{2}} \sup_{ \{ \alpha_j \} \in \calA} \sup_{t \in \bbR} \, \bigl| \langle \tilP_{s} u(t), Q_s^{(\alpha)} \tilP_{s} u(t) \rangle \bigr| \, \frac{\ud s}{s} \\
  &\qquad \quad + C_N \int_{0}^{s_{2}} \sup_{ \{ \alpha_j \} \in \calA} \, |\tilF_s^{(\alpha)}|  \, \frac{\ud s}{s}+ C_N \int_{0}^{s_{2}} \|\tilP_s u \|_{L^2(\bbR \times \{ r \leq R\} )}^2 \, \frac{\ud s}{s}.
 \end{aligned}
\end{equation}
Recalling that $\mu_N = 2^{-N-1} \log(2)$, we may absorb the first term on the right-hand side into the left-hand side by choosing $N$ large enough. Also recalling the definition of $g(s)$ and surpassing the $N$ dependency of constants, we arrive at the estimate 
\begin{equation} \label{eq:dyadic-final s0-tils0}
\begin{aligned}
&\int_{0}^{\frac{s_2}{4}} \sup_{ \{ \alpha_\ell \} \in \calA} \frac{1}{s} \langle (\partial_r \beta_s^{(\alpha)}) \tilP_{s} u, \tilP_{s} u \rangle_{t,x} \, \frac{\ud s}{s} \\
&\qquad \leq C \sup_{ \{ \alpha_\ell \} \in \calA} \int_{\frac{s_2}{8}}^{s_2} \frac{1}{s} \langle (\partial_r \beta_s^{(\alpha)}) \tilP_{s} u, \tilP_{s} u \rangle_{t,x} \, \frac{\ud s}{s} \\
&\qquad \quad + C \int_{0}^{s_2} \sup_{ \{ \alpha_\ell \} \in \calA} \sup_{t \in \bbR} \, \bigl| \langle \tilP_{s} u(t), Q_s^{(\alpha)} \tilP_{s} u(t) \rangle \bigr| \, \frac{\ud s}{s} \\
&\qquad \quad + C \int_{0}^{s_2} \sup_{\{ \alpha_\ell \} \in \calA} \, |\tilF_s^{(\alpha)}| \, \frac{\ud s}{s}+ C \int_{0}^{s_2} \|\tilP_s u\|_{L^2(\bbR \times \{ r \leq R\} )}^2 \, \frac{\ud s}{s}.
\end{aligned}
\end{equation}
Finally, we insert the relation (see Lemma~\ref{lem:betaLErelation})
\[
 \sup_{ \{ \alpha_\ell \} \in \calA} \frac{1}{s} \langle (\partial_r \beta_s^{(\alpha)}) \tilP_{s} u, \tilP_{s} u \rangle_{t,x} \simeq \sup_{ \{ \alpha_\ell \} \in \calA } s^{-\frac{1}{2}} \| \tilP_s u \|_{X_{s,\alpha}}^2 \simeq s^{-\frac{1}{2}} \| \tilP_s u \|_{LE_s}^2
\]
for the integrand on the left-hand side of the previous estimate, which completes the proof of~\eqref{equ:high_frequency_local_smoothing}.
\end{proof}


\section{Transitioning Estimate and Proof of Theorems~\ref{t:LE1-sym} and \ref{t:LE1}} \label{s:trans} 


In this section we complete the proof of Theorem~\ref{t:LE1} (and therefore its immediate corollary, Theorem~\ref{t:LE1-sym}) by combining the low-frequency and the high-frequency estimates of the previous two sections. The main additional ingredient for the proof is the following \emph{transitioning estimate} which enables us to transition between the different spatial weights appearing in the definitions of our low-frequency and our high-frequency local smoothing spaces $LE_\low$ and $LE_s$. In its proof we will make crucial use of Lemma~\ref{l:weighted_HP_applied_to_main_multiplier_identity} and Remark~\ref{rem:weighted_HP_applied_to_main_multiplier_identity_s3version} from Section~\ref{s:low}.

\begin{lem} \label{lem:transitioning_estimate}
 Let $d \geq 2$. Let $I = [s_3, s_4]$ be a heat time interval with $s_3, s_4 \simeq 1$. Let $\beta_s^{(\alpha)}$ for $s \in I$ be defined as in~\eqref{equ:definition_beta_high_freq} and let $Q_{s_3}^{(\alpha)}$ be defined as in~\eqref{equ:definition_Q_high_freq}. Then there exists an absolute constant $R \geq 1$ such that for all solutions $u(t)$ to~\eqref{eq:S}
 \begin{equation} \label{equ:transitioning_estimate}
  \begin{aligned}
   &\sup_{ \{ \alpha_\ell \} \in \calA} \, \int_I \langle (\partial_r \beta_s^{(\alpha)}) \tilP_s u, \tilP_s u \rangle_{t,x} \, \ds \\
   &\qquad \lesssim \int_{2I} \sup_{ \{ \alpha_\ell \} \in \calA } \, \sup_{t \in \bbR} \, \bigl| \langle \tilP_{\geq s} u(t), Q_{s_3}^{(\alpha)} \tilP_{\geq s} u(t) \rangle \bigr| \, \ds \\
   &\qquad \quad + \int_{2I} \sup_{ \{ \alpha_\ell \} \in \calA } \, \bigl| \Re \langle (\partial_t - i \Delta) \tilP_{\geq s} u, Q_{s_3}^{(\alpha)} \tilP_{\geq s} u \rangle_{t,x} \bigr| \, \ds \\
   &\qquad \quad + \int_{2I} \| \tilP_{\geq s} u \|_{L^2(\bbR \times \{ r \leq R \})}^2 \, \ds,
  \end{aligned}
 \end{equation}
 where we use the notation $2I := [ \frac{s_3}{2}, 2 s_4 ]$. 
\end{lem}
\begin{proof}
 In order to simplify the notation in what follows, we define 
 \[
  h := \tilP_{\geq s} u = e^{s(\Delta + \rho^2)} u
 \]
 and note that 
 \[
  \frac{1}{s} \tilP_s u = - (\Delta + \rho^2) e^{s(\Delta + \rho^2)} u = - \partial_s h.              
 \]
 Moreover, we have uniformly for all heat times $s \in I = [s_3, s_4]$ and for all slowly varying functions $\alpha$ associated with a slowly varying sequence $\{ \alpha_\ell \} \in \calA$ that
 \begin{equation}
  \partial_r \beta_s^{(\alpha)} = \delta \frac{\alpha(\delta s^{-\frac{1}{2}} r)}{\langle \delta s^{-\frac{1}{2}} r \rangle} \simeq \delta \frac{\alpha(\delta s_{3}^{-\frac{1}{2}} r)}{\langle \delta s_{3}^{-\frac{1}{2}} r \rangle} = \partial_r \beta_{s_3}^{(\alpha)},
 \end{equation}
 since $s_3, s_4 \simeq 1$ by assumption and since $\alpha$ is slowly varying according to Definition~\ref{d:alpha}. Hence, we have
 \begin{equation}
  \begin{aligned}
   \int_I \langle (\partial_r \beta_s^{(\alpha)}) \tilP_s u, \tilP_s u \rangle_{t,x} \, \ds &\lesssim \int_I \langle (\partial_r \beta_{s_3}^{(\alpha)}) \tilP_s u, \tilP_s u \rangle_{t,x} \, \ds \\
   &\lesssim \int_I \langle (\partial_r \beta_{s_3}^{(\alpha)}) \partial_s h, \partial_s h \rangle_{t,x} \, \ud s,
  \end{aligned}  
 \end{equation}
 where in the last estimate we again used that $s_3, s_4 \simeq 1$. We let $\chi = \chi(s)$ be a smooth bump function with support in $2I$ and such that $\chi(s) = 1$ for $s \in I$. Trivially, it holds that
 \begin{equation} \label{equ:proof_transitioning_estimate_trivial_L}
  \int_I \langle (\partial_r \beta_{s_3}^{(\alpha)}) \partial_s h, \partial_s h \rangle_{t,x} \, \ud s \leq \int_{2I} \chi \langle (\partial_r \beta_{s_3}^{(\alpha)}) \partial_s h, \partial_s h \rangle_{t,x} \, \ud s =: L.
 \end{equation}
 Now on the one hand, integration by parts in space yields that
 \begin{align*}
  L &= \Re \int_{2I} \chi \langle (\partial_r \beta_{s_3}^{(\alpha)}) \partial_s h, \partial_s h \rangle_{t,x} \, \ud s \\
  &= \Re \int_{2I} \chi \langle (\partial_r \beta_{s_3}^{(\alpha)}) \partial_s h, (\Delta + \rho^2) h \rangle_{t,x} \, \ud s \\
  &= \Re \int_{2I} \chi \langle \Delta (\partial_r \beta_{s_3}^{(\alpha)}) \partial_s h,  h \rangle_{t,x} \, \ud s + 2 \, \Re \int_{2I} \chi \langle \nabla (\partial_r \beta_{s_3}^{(\alpha)}) \cdot \nabla \partial_s h, h \rangle_{t,x} \, \ud s \\ 
  &\quad + \Re \int_{2I} \chi \langle (\partial_r \beta_{s_3}^{(\alpha)}) (\Delta + \rho^2) \partial_s h, h \rangle_{t,x} \, \ud s \\
  &= - \Re \int_{2I} \chi \langle \Delta (\partial_r \beta_{s_3}^{(\alpha)}) \partial_s h, h \rangle_{t,x} \, \ud s - 2 \, \Re \int_{2I} \chi \langle (\partial_r^2 \beta_{s_3}^{(\alpha)}) \partial_s h, \partial_r h \rangle_{t,x} \, \ud s \\ 
  &\quad + \Re \int_{2I} \chi \langle (\partial_r \beta_{s_3}^{(\alpha)}) (\Delta + \rho^2) \partial_s h, h \rangle_{t,x} \, \ud s. 
 \end{align*}
 On the other hand, integrating by parts in the heat time $s$ we find that
 \begin{align*}
  L &=  \Re \int_{2I} \chi \langle (\partial_r \beta_{s_3}^{(\alpha)}) \partial_s h, (\Delta + \rho^2) h \rangle_{t,x} \, \ud s \\
  &= - \Re \int_{2I} \chi \langle (\partial_r \beta_{s_3}^{(\alpha)}) h, (\Delta + \rho^2) \partial_s h \rangle_{t,x} \, \ud s - \Re \int_{2I} (\partial_s \chi) \langle (\partial_r \beta_{s_3}^{(\alpha)}) h, (\Delta + \rho^2) h \rangle_{t,x} \, \ud s. 
 \end{align*}
 Adding the last two identities we obtain that
 \begin{equation} \label{equ:proof_transitioning_estimate_2L}
  \begin{aligned}
   2 L &= -\Re \int_{2I} \chi \langle \Delta (\partial_r \beta_{s_3}^{(\alpha)}) \partial_s h, h \rangle_{t,x} \, \ud s - 2 \, \Re \int_{2I} \chi \langle (\partial_r^2 \beta_{s_3}^{(\alpha)}) \partial_s h, \partial_r h \rangle_{t,x} \, \ud s \\ 
   &\quad - \Re \int_{2I} (\partial_s\chi) \langle (\partial_r \beta_{s_3}^{(\alpha)}) h, (\Delta + \rho^2) h \rangle_{t,x} \, \ud s \\
   &=: I + II + III.
  \end{aligned}
 \end{equation}
 We now estimate the terms $I$, $II$, and $III$ separately. For the first term $I$ we use Cauchy-Schwarz and the fact that 
 \[
  \biggl| \frac{\Delta (\partial_r \beta_{s_3}^{(\alpha)})}{ \partial_r \beta_{s_3}^{(\alpha)} } \biggr| \lesssim \frac{1}{\langle r \rangle}
 \]
 uniformly for all $\{ \alpha_\ell \} \in \calA$, to obtain for any $\varepsilon > 0$ that
 \begin{align*}
  |I| \leq \varepsilon L + C_{\varepsilon} \int_{2I} \chi \bigl\langle \langle r \rangle^{-2} (\partial_r \beta_{s_3}^{(\alpha)}) h, h \bigr\rangle_{t,x} \, \ud s. 
 \end{align*}
 Then for the second term on the right-hand side we invoke the estimate~\eqref{equ:additional_bounds_positive_commutator_from_HP_s3version} from Remark~\ref{rem:weighted_HP_applied_to_main_multiplier_identity_s3version}  to conclude that
 \begin{equation} \label{equ:proof_transitioning_estimate_bound_on_I}
  |I| \leq \varepsilon L +  C_{\varepsilon} \int_{2I} \bigl| \Re \langle i \Delta h, Q_{s_3}^{(\alpha)} h \rangle_{t,x} \bigr| \, \ud s + C_\varepsilon \int_{2I} \| h \|_{L^2(\bbR \times \{ r \leq R \})}^2 \, \ud s.
 \end{equation}
 For the term $III$ we again use the estimate~\eqref{equ:additional_bounds_positive_commutator_from_HP_s3version} from Remark~\ref{rem:weighted_HP_applied_to_main_multiplier_identity_s3version} to find that
 \begin{equation} \label{equ:proof_transitioning_estimate_bound_on_III}
  \begin{aligned}
   |III| &\leq \int_{2I} |\partial_s \chi| \, \bigl| \Re\langle (\partial_r \beta_{s_3}^{(\alpha)}) h, (\Delta + \rho^2) h \rangle_{t,x} \bigr| \, \ud s \\
   &\lesssim \int_{2I} \bigl| \Re \langle i \Delta h, Q_{s_3}^{(\alpha)} h \rangle_{t,x} \bigr| \, \ud s + \int_{2I} \| h \|_{L^2(\bbR\times\{r\leq R\})}^2 \, \ud s.
  \end{aligned}
 \end{equation}
 Finally for the term $II$ we argue as we did for the term $I$, using that 
 \[
  \biggl| \frac{\partial_r^2 \beta_{s_3}^{(\alpha)}}{\partial_r \beta_{s_3}^{(\alpha)}} \biggr| \lesssim \frac{1}{\langle r \rangle}
 \]
 uniformly for all $\{ \alpha_\ell \} \in \calA$, to estimate 
 \begin{equation} 
  |II| \leq \varepsilon L + C_{\varepsilon} \int_{2I} \chi \bigl\langle \langle r \rangle^{-2} (\partial_r \beta_{s_3}^{(\alpha)}) \nabla h, \nabla h \rangle_{t,x} \, \ud s.
 \end{equation}
 In view of how we estimated the term $I$ using the estimate~\eqref{equ:additional_bounds_positive_commutator_from_HP_s3version} from Remark~\ref{rem:weighted_HP_applied_to_main_multiplier_identity_s3version}, it suffices to prove 
 \begin{equation} \label{equ:transitioning_local_parabolic_regularity}
  \int_{2I} \chi \bigl\langle \langle r \rangle^{-2} (\partial_r \beta_{s_3}^{(\alpha)}) \nabla h, \nabla h \rangle_{t,x} \, \ud s \lesssim \int_{2I} \bigl\langle \langle r \rangle^{-2} (\partial_r \beta_{s_3}^{(\alpha)}) h, h \bigr\rangle_{t,x} \, \ud s
 \end{equation}
 to then infer that
 \begin{equation} \label{equ:proof_transitioning_estimate_bound_on_II}
  |II| \leq \varepsilon L +  C_{\varepsilon} \int_{2I} \bigl| \Re \langle i \Delta h, Q_{s_3}^{(\alpha)} h \rangle_{t,x} \bigr| \, \ud s + C_\varepsilon \int_{2I} \| h \|_{L^2(\bbR \times \{ r \leq R \})}^2 \, \ud s.
 \end{equation}
 For the proof of~\eqref{equ:transitioning_local_parabolic_regularity} recall that by the definition of $h = e^{s(\Delta + \rho^2)} u$ we have 
 \[
  (\partial_s - \Delta -\rho^2) h = 0.
 \]
 It follows that
 \begin{align*}
 0 &= \Re \int_{2I} \chi \bigl\langle \langle r \rangle^{-2} (\partial_r \beta_{s_3}^{(\alpha)}) h, (\partial_s - \Delta - \rho^2) h \bigr\rangle_{t,x} \, \ud s \\
 &= \int_{2I} \chi \bigl\langle \langle r \rangle^{-2} (\partial_r \beta_{s_3}^{(\alpha)}) \nabla h, \nabla h \bigr\rangle_{t,x} \, \ud s - \rho^2 \int_{2I} \chi \bigl\langle \langle r \rangle^{-2} (\partial_r \beta_{s_3}^{(\alpha)}) h, h \bigr\rangle_{t,x} \, \ud s \\
 &\quad \quad - \frac{1}{2} \int_{2I} (\partial_s \chi) \bigl\langle \langle r \rangle^{-2} (\partial_r \beta_{s_3}^{(\alpha)}) h, h \bigr\rangle \, \ud s + \int_{2I} \chi \bigl\langle \partial_r \bigl( \langle r \rangle^{-2} (\partial_r \beta_{s_3}^{(\alpha)}) \bigr) h, \partial_r h \bigr\rangle_{t,x} \, \ud s.
\end{align*}
Upon rearranging and using the estimate 
\begin{align*}
 \Bigl| \partial_r \bigl( \langle r \rangle^{-2} (\partial_r \beta_{s_3}^{(\alpha)}) \bigr) \Bigr| \lesssim \frac{\partial_r \beta_{s_3}^{(\alpha)}}{\langle r \rangle^3} \lesssim \frac{\partial_r \beta_{s_3}^{(\alpha)}}{\langle r \rangle^2},
\end{align*}
we get
\begin{align*}
 \int_{2I} \chi \bigl\langle \langle r \rangle^{-2} (\partial_r \beta_{s_3}^{(\alpha)}) \nabla h, \nabla h \bigr\rangle_{t,x} \, \ud s &\leq C \int_{2I} \bigl\langle \langle r \rangle^{-2} (\partial_r \beta_{s_3}^{(\alpha)}) h, h \bigr\rangle_{t,x} \, \ud s \\
 &\quad + C \int_{2I} \chi \bigl\langle \langle r \rangle^{-2} (\partial_r \beta_{s_3}^{(\alpha)})  |\nabla h|, |h| \bigr\rangle_{t,x} \, \ud s \\
 &\leq C_{\epsilon_1} \int_{2I} \bigl\langle \langle r \rangle^{-2} (\partial_r \beta_{s_3}^{(\alpha)}) h, h \bigr\rangle_{t,x} \, \ud s \\
 &\quad + \epsilon_1 \int_{2I} \chi \bigl\langle \langle r \rangle^{-2} (\partial_r \beta_{s_3}^{(\alpha)}) \nabla h, \nabla h \bigr\rangle_{t,x} \, \ud s,
\end{align*}
where in the last estimate we used Cauchy-Schwarz. Taking $\epsilon_1>0$ sufficiently small to absorb the last term on the right-hand side into the left-hand side, we obtain the desired estimate~\eqref{equ:transitioning_local_parabolic_regularity}.

In conclusion, by combining the identity~\eqref{equ:proof_transitioning_estimate_2L} with the estimates~\eqref{equ:proof_transitioning_estimate_trivial_L}, \eqref{equ:proof_transitioning_estimate_bound_on_I}, \eqref{equ:proof_transitioning_estimate_bound_on_III}, and \eqref{equ:proof_transitioning_estimate_bound_on_II}, we obtain upon choosing $\varepsilon > 0$ sufficiently small that 
\begin{equation}
 \int_{I} \bigl\langle (\partial_r \beta_{s_3}^{(\alpha)}) \partial_s h, \partial_s h \bigr\rangle_{t,x} \, \ud s \lesssim \int_{2I} \bigl| \Re \langle i \Delta h, Q_{s_3}^{(\alpha)} h \rangle_{t,x} \bigr| \, \ds + \int_{2I} \| h \|_{L^2(\bbR\times\{ r \leq R \})}^2 \, \ud s. 
\end{equation}
Finally, recalling that $h = \tilP_{\geq s} u$, $\partial_s h = - \frac{1}{s} \tilP_s u$ and that $s \simeq 1$ for $s \in 2I$, the asserted transitioning estimate~\eqref{equ:transitioning_estimate} follows from the previous estimate upon integrating in time over the identity (see \eqref{eq:dtQu})
\begin{align*}
 \Re \langle i \Delta \tilP_{\geq s} u, Q_{s_3}^{(\alpha)} \tilP_{\geq s} u \rangle = \frac{1}{2} \frac{\ud}{\ud t} \langle \tilP_{\geq s} u, Q_{s_3}^{(\alpha)} \tilP_{\geq s} u \rangle - \Re \langle (\partial_t - i \Delta) \tilP_{\geq s} u, Q_{s_3}^{(\alpha)} \tilP_{\geq s} u \rangle
\end{align*}
and taking the supremum over all slowly varying sequences $\{ \alpha_\ell \} \in \calA$.
\end{proof}

We are finally in the position to complete the proof of Theorem~\ref{t:LE1}.

\begin{proof}[Proof of Theorem~\ref{t:LE1}]
In what follows we will apply the high-frequency local smoothing estimate~\eqref{equ:high_frequency_local_smoothing} from Section~\ref{s:high} with $s_2 = 2$, the low-frequency local smoothing estimate~\eqref{equ:low_frequency_local_smoothing_estimate} from Section~\ref{s:low}, and the transitioning estimate from Lemma~\ref{lem:transitioning_estimate} with $s_4=s_2$ and $s_3=\frac{s_2}{8}$. For the sake of readability we keep the notation $s_2$ instead of making these heat times explicit, and allow the constants to depend on $s_2$. 
Moreover, we use the low-frequency $\beta^{(\alpha)}$, $Q^{(\alpha)}$ as defined in \eqref{equ:definition_beta_low_freq} and \eqref{equ:definition_Q_low_freq}, while $\beta_s^{(\alpha)}$, $Q_s^{(\alpha)}$ denote the high-frequency analogues as defined in  \eqref{equ:definition_beta_high_freq} and \eqref{equ:definition_Q_high_freq}. Finally, $\tilF_s^{(\alpha)}$ is as in \eqref{eq:tilFsdef}. 

Starting with the high-frequency local smoothing estimate~\eqref{equ:high_frequency_local_smoothing}, we have 
\begin{equation} \label{equ:proof_combined_smoothing_est1}
 \begin{aligned}
  &\int_0^{\frac{s_2}{4}} s^{-\frac{1}{2}} \| \tilP_s u \|_{LE_s}^2 \, \ds \\
  &\qquad \quad \lesssim \sup_{\{ \alpha_\ell \} \in \calA } \, \int_{\frac{s_2}{8}}^{s_2} \frac{1}{s} \langle (\partial_r \beta_s^{(\alpha)}) \tilP_s u, \tilP_s u \rangle_{t,x} \, \ds \\
  &\qquad \qquad \quad + \int_0^{s_2} \sup_{\{ \alpha_\ell \} \in \calA } \, \sup_{t \in \bbR} \, \bigl| \langle \tilP_s u(t), Q_s^{(\alpha)} \tilP_s u(t) \rangle \bigr| \, \ds \\
  &\qquad \qquad \quad + \int_0^{s_2} \sup_{\{ \alpha_\ell \} \in \calA } \, |\tilF_s^{(\alpha)}| \, \ds + \int_0^{s_2} \| \tilP_s u \|_{L^2(\bbR \times \{ r \leq R \})}^2 \, \ds.  
 \end{aligned}
\end{equation}
For the first term on the right-hand side of~\eqref{equ:proof_combined_smoothing_est1} we note that $\frac{1}{s} \simeq 1$ for $s \in [\frac{s_2}{8}, s_2]$ and insert the transitioning estimate~\eqref{equ:transitioning_estimate} from Lemma~\ref{lem:transitioning_estimate} with $s_3 = \frac{s_2}{8}$ and $s_4 = s_2$ to obtain that
\begin{equation} \label{equ:proof_combined_smoothing_est2}
 \begin{aligned}
  &\int_0^{\frac{s_2}{4}} s^{-\frac{1}{2}} \| \tilP_s u \|_{LE_s}^2 \, \ds \\
  &\qquad \quad \lesssim \int_{\frac{s_2}{16}}^{2 s_2} \sup_{ \{ \alpha_\ell \} \in \calA } \, \sup_{t \in \bbR} \, \bigl| \langle \tilP_{\geq s} u(t), Q_{s_3}^{(\alpha)} \tilP_{\geq s} u(t) \rangle \bigr| \, \ds \\
  &\qquad \quad \quad + \int_{\frac{s_2}{16}}^{2 s_2} \sup_{ \{ \alpha_\ell \} \in \calA } \, \bigl| \Re \langle (\partial_t - i \Delta) \tilP_{\geq s} u, Q_{s_3}^{(\alpha)} \tilP_{\geq s} u \rangle_{t,x} \bigr| \, \ds \\
  &\qquad \quad \quad + \int_{\frac{s_2}{16}}^{2 s_2} \| \tilP_{\geq s} u \|_{L^2(\bbR \times \{ r \leq R \})}^2 \, \ds \\
  &\qquad \quad \quad + \int_0^{s_2} \sup_{\{ \alpha_\ell \} \in \calA } \, \sup_{t \in \bbR} \, \bigl| \langle \tilP_s u(t), Q_s^{(\alpha)} \tilP_s u(t) \rangle \bigr| \, \ds \\
  &\qquad \quad \quad + \int_0^{s_2} \sup_{\{ \alpha_\ell \} \in \calA } \, |\tilF_s^{(\alpha)}| \, \ds  + \int_0^{s_2} \| \tilP_s u \|_{L^2(\bbR \times \{ r \leq R \})}^2 \, \ds.  
 \end{aligned}
\end{equation}
Integrating over the low-frequency local smoothing estimate~\eqref{equ:low_frequency_local_smoothing_estimate} in $\ds$ from $\frac{s_2}{16}$ to $2 s_2$ yields 
\begin{equation} \label{equ:proof_combined_smoothing_est3}
 \begin{aligned}
  &\int_{\frac{s_2}{16}}^{2 s_2} \| \tilP_{\geq s} u \|_{LE_\low}^2 \, \ds \\
  &\qquad \quad \lesssim \int_{\frac{s_2}{16}}^{2 s_2} \sup_{ \{ \alpha_\ell \} \in \calA } \, \sup_{t \in \bbR} \, \bigl| \langle \tilP_{\geq s} u(t), Q^{(\alpha)} \tilP_{\geq s} u(t) \rangle \bigr| \, \ds \\
  &\qquad \quad \quad + \int_{\frac{s_2}{16}}^{2 s_2} \sup_{ \{ \alpha_\ell \} \in \calA } \, \bigl| \Re \langle (\partial_t - i \Delta) \tilP_{\geq s} u, Q^{(\alpha)} \tilP_{\geq s} u \rangle_{t,x} \bigr| \, \ds \\
  &\qquad \quad \quad + \int_{\frac{s_2}{16}}^{2 s_2} \| \tilP_{\geq s} u \|_{L^2(\bbR \times \{ r \leq R \})}^2 \, \ds.
 \end{aligned}
\end{equation}
Adding~\eqref{equ:proof_combined_smoothing_est2} and \eqref{equ:proof_combined_smoothing_est3} we obtain that
\begin{equation} \label{equ:proof_combined_smoothing_est4}
 \begin{aligned}
  &\int_{\frac{s_2}{16}}^{2 s_2} \| \tilP_{\geq s} u \|_{LE_\low}^2 \, \ds + \int_0^{\frac{s_2}{4}} s^{-\frac{1}{2}} \| \tilP_s u \|_{LE_s}^2 \, \ds  \\
  &\qquad \quad \lesssim \int_{\frac{s_2}{16}}^{2 s_2} \sup_{\{ \alpha_\ell \} \in \calA } \, \sup_{t \in \bbR} \, \bigl| \langle \tilP_{\geq s} u(t), Q^{(\alpha)} \tilP_{\geq s} u(t) \rangle \bigr| \, \ds \\
  &\qquad \qquad \quad + \int_{\frac{s_2}{16}}^{2 s_2} \sup_{\{ \alpha_\ell \} \in \calA } \, \sup_{t \in \bbR} \, \bigl| \langle \tilP_{\geq s} u(t), Q_{s_3}^{(\alpha)} \tilP_{\geq s} u(t) \rangle \bigr| \, \ds \\
  &\qquad \qquad \quad + \int_0^{s_2} \sup_{\{ \alpha_\ell \} \in \calA } \, \sup_{t \in \bbR} \, \bigl| \langle \tilP_s u(t), Q_s^{(\alpha)} \tilP_s u(t) \rangle \bigr| \, \ds \\
  &\qquad \qquad \quad + \int_{\frac{s_2}{16}}^{2 s_2} \sup_{\{ \alpha_\ell \} \in \calA } \, \bigl| \Re \langle (\partial_t - i\Delta) \tilP_{\geq s} u, Q^{(\alpha)} \tilP_{\geq s} u \rangle_{t,x} \bigr| \, \ds \\
  &\qquad \qquad \quad + \int_{\frac{s_2}{16}}^{2 s_2} \sup_{\{ \alpha_\ell \} \in \calA } \, \bigl| \Re \langle (\partial_t - i\Delta) \tilP_{\geq s} u, Q_{s_3}^{(\alpha)} \tilP_{\geq s} u \rangle_{t,x} \bigr| \, \ds \\
  &\qquad \qquad \quad + \int_0^{s_2} \sup_{\{ \alpha_\ell \} \in \calA } \,|\tilF_s^{(\alpha)}| \, \ds \\
  &\qquad \qquad \quad + \int_{\frac{s_2}{16}}^{2 s_2} \| \tilP_{\geq s} u\|_{L^2(\bbR\times\{r \leq R\})}^2 \, \ds + \int_0^{s_2} \| \tilP_s u \|_{L^2(\bbR \times \{ r \leq R \})}^2 \, \ds \\
  &\qquad \quad =: I + II + III + IV + V + VI + VII + VIII.
 \end{aligned}
\end{equation}
Here we first observe that since we chose $s_2=2$, the left-hand side of~\eqref{equ:proof_combined_smoothing_est4} coincides exactly with the local smoothing norm $\|u\|_{LE}^2$ of the solution~$u(t)$ to the Schr\"odinger equation~\eqref{eq:S}. 
Next, given any $\varepsilon > 0$, by sufficiently enlarging $R \gg 1$ we obtain from Lemma~\ref{l:wL2LE} and Lemma~\ref{l:L2R} that
\begin{equation}
 VII + VIII \leq \varepsilon \| u \|_{LE}^2 + C_{\varepsilon} \| u \|_{L^2(\bbR\times\{ r \leq R\})}^2.
\end{equation}
Recall the definition of $\tilF_s^{(\alpha)}$ from \eqref{eq:tilFsdef}, and that the low and high-frequency projections of the solution~$u(t)$ to the linear Schr\"odinger equation~\eqref{eq:S} satisfy the equations
\begin{align*}
 (\partial_t - i \Delta) \tilP_{\geq s} u &= i \tilP_{\geq s} F - i \tilP_{\geq s} H^{\antisym}_\lot u - i \tilP_{\geq s} H^{\sym}_\lot u \\
 &\phantom{=} - i [ \tilP_{\geq s}, H_\prin ] u - i (H_\prin + \Delta) \tilP_{\geq s} u, \\
 (\partial_t - i \Delta) \tilP_s u &= i \tilP_s F - i \tilP_s H^{\antisym}_\lot u - i \tilP_s H^{\sym}_\lot u \\
  &\phantom{=} - i [ \tilP_s, H_\prin ] u - i (H_\prin + \Delta) \tilP_s u.
\end{align*}
where we recall the decomposition $H_{\lot} = H_{\lot}^{\antisym} + H_{\lot}^{\sym}$ with
\begin{align*}
	H^{\antisym}_{\lot} u &= \Im \bsb^{\mu} \nb_{\mu} u + \nb_{\mu} (\Im \bsb^{\mu} u) + i \Im V u, \\
	H^{\sym}_{\lot} u &= \frac{1}{i} (\Re \bsb^{\mu} \nb_{\mu} u + \nb_{\mu} (\Re \bsb^{\mu} u)) + \Re V u.
\end{align*}
We also recall that the smallness assumptions on $\bsa-\bsh^{-1}$, $\Im \bsb$ and $\Im V$ in \eqref{eq:decay_assumptions-prin} and \eqref{eq:decay_assumptions-lot} are formulated in terms of the small constant $\varepsilon_0>0$. It follows from Propositions~\ref{p:FQlowu}, \ref{p:Hlot_antisym_Q}, \ref{p:commutator_Hlot_Q_low}, \ref{p:commP_leqsH_prin}, \ref{p:DeltaHpQlow}, and Remarks~\ref{rem:FQlowu}, \ref{rem:commutator_Hlot_Q_low}, \ref{rem:commP_leqsH_prin}, \ref{rem:DeltaHpQlow} that for sufficiently large $R \gg 1$ 
\begin{equation}
 IV + V \lesssim \varepsilon_0 \| u \|_{LE}^2 + C_{\varepsilon_0} \| F \|_{LE^\ast}^2 + C_{\varepsilon_0} \| u \|_{L^2(\bbR\times\{ r \leq R\})}^2.
\end{equation}
Similarly, it follows from Propositions~\ref{p:FQsu}, \ref{p:Hlot_antisym_Q}, \ref{p:commutator_Hlot_Q_high}, \ref{p:commP_sH_prin}, \ref{p:DeltaHpQhigh} that for sufficiently large $R \gg 1$
\begin{equation}
VI \lesssim \varepsilon_0 \| u \|_{LE}^2 + C_{\varepsilon_0} \| F \|_{LE^\ast}^2 + C_{\varepsilon_0} \| u \|_{L^2(\bbR\times\{ r \leq R\})}^2.
\end{equation}

It therefore remains to estimate the terms $I$, $II$ and $III$ for which we use the approximate mass conservation \eqref{equ:approx_mass_conservation} of solutions to \eqref{eq:S}. We begin with the terms $I$ and $II$ which can be treated in the same manner. By Cauchy-Schwarz as well as the $L^2$-boundedness of $Q^{(\alpha)} \tilP_{\geq s}$ and $Q^{(\alpha)}_{s_3} \tilP_{\geq s}$ according to Proposition~\ref{p:QlowL2} 
\begin{equation}
 \begin{aligned}
  I + II &= \int_{\frac{s_2}{16}}^{2 s_2} \sup_{ \{ \alpha_\ell \} \in \calA } \, \sup_{t \in \bbR} \, \bigl| \langle \tilP_{\geq s} u(t), Q^{(\alpha)} \tilP_{\geq s} u(t) \rangle \bigr| \, \ds \\
  &\qquad + \int_{\frac{s_2}{16}}^{2 s_2} \sup_{ \{ \alpha_\ell \} \in \calA } \, \sup_{t \in \bbR} \, \bigl| \langle \tilP_{\geq s} u(t), Q_{s_3}^{(\alpha)} \tilP_{\geq s} u(t) \rangle \bigr| \, \ds \\ 
  &\lesssim \int_{\frac{s_2}{16}}^{2 s_2} \sup_{ \{ \alpha_\ell \} \in \calA } \, \sup_{t \in \bbR} \, \bigl( \| \tilP_{\geq s} u(t) \|_{L^2}^2 + \| Q^{(\alpha)} \tilP_{\geq s} u(t) \|_{L^2}^2 + \| Q_{s_3}^{(\alpha)} \tilP_{\geq s} u(t) \|_{L^2}^2\bigr) \, \ds \\
  &\lesssim \int_{\frac{s_2}{16}}^{2 s_2} \sup_{t \in \bbR} \, \bigl( \| \tilP_{\geq s} u(t) \|_{L^2}^2 + \| \tilP_{\geq \frac{s}{2}} u(t) \|_{L^2}^2\bigr) \, \ds.
 \end{aligned}
\end{equation}
Disposing of the low-frequency projections $\tilP_{\geq s}$ and $\tilP_{\geq \frac{s}{2}}$ using Lemma~\ref{l:preg}, we obtain that
\begin{equation}
 I + II \lesssim \sup_{t \in \bbR} \, \| u(t) \|_{L^2}^2.
\end{equation}
From the approximate mass conservation \eqref{equ:approx_mass_conservation}, that is,
\begin{equation}
 \frac{\ud}{\ud t} \, \| u(t) \|_{L^2}^2  = 2 \, \Re \langle i F, u \rangle,
\end{equation}
we conclude that
\begin{equation}
 \sup_{t \in \bbR} \, \| u(t) \|_{L^2}^2 \lesssim \|u_0\|_{L^2}^2 + \bigl| \langle F, u \rangle_{t,x} \bigr| \lesssim \|u_0\|_{L^2}^2 + \|F\|_{LE^\ast} \|u\|_{LE}.
\end{equation}
Hence, we obtain for any $\varepsilon>0$ that
\begin{equation}
 I + II \lesssim \|u_0\|_{L^2}^2 + C_\varepsilon \|F\|_{LE^\ast}^2 + \varepsilon \|u\|_{LE}^2.
\end{equation}
Next we deal with the term $III$. By Cauchy-Schwarz and the $L^2$-boundedness of $Q_s^{(\alpha)} \tilP_s$ thanks to Proposition~\ref{p:QsL2}, we have 
\begin{align*}
 III &= \int_0^{s_2} \sup_{\{ \alpha_\ell \} \in \calA } \, \sup_{t \in \bbR} \, \bigl| \langle \tilP_s u(t), Q_s^{(\alpha)} \tilP_s u(t) \rangle \bigr| \, \ds \\
 &\lesssim \int_0^{s_2} \sup_{\{ \alpha_\ell \} \in \calA } \, \sup_{t \in \bbR} \, \bigl(  \| \tilP_s u(t) \|_{L^2}^2 + \| \tilP_{\frac{s}{2}} u(t) \|_{L^2}^2 \bigr) \, \ds \\
 &\lesssim \int_0^{s_2} \sup_{t \in \bbR} \, \| \tilP_s u(t) \|_{L^2}^2 \, \ds,
\end{align*}
where for the last estimate we used the change of variables $\frac{s}{2} \mapsto s$. By approximate mass conservation and the fact that $H_\prin$ is self-adjoint, we infer from
\begin{equation}
 \frac{\ud}{\ud t} \| \tilP_s u(t) \|_{L^2}^2 = 2 \, \Re \langle \partial_t \tilP_s u, \tilP_s u \rangle = 2 \, \Re \langle (\partial_t + i H_{\prin}) \tilP_s u, \tilP_s u \rangle
\end{equation}
that 
\begin{equation}
 \sup_{t \in \bbR} \, \| \tilP_s u(t) \|_{L^2}^2 \leq \| \tilP_s u_0 \|_{L^2}^2 + C \bigl| \Re \langle (\partial_t + i H_{\prin}) \tilP_s u, \tilP_s u \rangle_{t,x} \bigr|.
\end{equation}
Hence, we obtain from \eqref{eq:extraPsHlotuPsuestimate} and \eqref{eq:princompmaintemp1noQ} that for any $\varepsilon>0$
\begin{equation} \label{equ:proof_combined_smoothing_est5}
 \begin{aligned}
  III &\lesssim \int_0^{s_2} \sup_{t \in \bbR} \, \| \tilP_s u(t) \|_{L^2}^2 \, \ds \\
  &\lesssim \int_0^{s_2} \| \tilP_s u_0 \|_{L^2}^2 \, \ds + \int_0^{s_2} \bigl| \Re \langle (\partial_t + i H_{\prin}) \tilP_s u, \tilP_s u \rangle_{t,x} \bigr| \, \ds \\
  &\lesssim \| u_0 \|_{L^2}^2 + \varepsilon \| u \|_{LE}^2 + C_{\varepsilon} \| F \|_{LE^\ast}^2 + C_{\varepsilon} \| u \|_{L^2(\bbR\times\{ r \leq R \})}^2.
 \end{aligned}
\end{equation}

Finally, combining the estimate~\eqref{equ:proof_combined_smoothing_est4} with all of the above bounds on the terms $I$--$VIII$, we arrive at the estimate
\[
 \|u\|_{LE}^2 \lesssim \|u_0\|_{L^2}^2 + \varepsilon_0 \|u\|_{LE}^2 + C_{\varepsilon_0} \| F \|_{LE^\ast}^2 + C_{\varepsilon_0} \|u\|_{L^2(\bbR \times \{ r \leq R \})}^2,
\]
which upon choosing $\varepsilon_0 > 0$ in \eqref{eq:decay_assumptions-prin} sufficiently small finishes the proof of Theorem~\ref{t:LE1}.
\end{proof}

\section{Local smoothing estimate in the stationary, symmetric case and its perturbations} \label{s:error} 

The bulk of this section is devoted to the proof of Theorem~\ref{t:LE-H}, where we assume that $H$ contains no metric perturbations and is \emph{symmetric} and \emph{stationary}, i.e., $\bsb, V$ are real-valued and independent of $t$.

The idea of the proof is to make use of the well-known equivalence of the local smoothing estimate with the limiting absorption principle for the resolvents $(\tau - H)^{-1}$ in the time-independent self-adjoint (i.e., symmetric and stationary) case. Our main multiplier estimate (Theorem~\ref{t:LE1-sym} or \ref{t:LE1}) implies that the desired resolvent bound holds for all large $\Re \tau$. Moreover, by a compactness argument, such bounds hold unless $H$ has a resonance in $[\rho^{2}, \infty)$; see Proposition~\ref{p:LE-res}. The argument is concluded by showing the absence of resonances embedded in $(\rho^{2}, \infty)$ (which was essentially known before in the case $\bsb = 0$); see Proposition~\ref{p:no-emb-res}. Our execution of this scheme is modeled on the one in \cite[Section 4]{MMT}. Below we will use the notation
\begin{align*}
\begin{split}
(f)_{\pm} = \max\set{\pm f, 0}.
\end{split}
\end{align*}

This section is organized as follows. In Section~\ref{ss:LE-compute}, we record the key multiplier identities required in the present section. In Section~\ref{ss:LE-res}, we show that if \eqref{eq:le-H} fails, then $H$ has a resonance at some $\tau \in [\rho^{2}, \infty)$ (see Definition~\ref{def:emb-res}). The proof of Theorem~\ref{t:LE-H} is completed in Section~\ref{ss:no-emb-res}, where we rule out the existence of any resonances at $\tau \in (\rho^{2}, \infty)$. In Section~\ref{ss:th-res} we establish Proposition~\ref{p:no-th-res} and Corollary~\ref{c:no-th-res}. Finally, in Section~\ref{ss:LE2-pf} we prove Theorem~\ref{t:LE2}, which extends the validity of the global-in-time local smoothing estimate to small but possibly non-symmetric and non-stationary perturbations of a symmetric, stationary Hamiltonian $H_{\stat}$ considered in Theorem~\ref{t:LE-H}.

\subsection{Key multiplier identities} \label{ss:LE-compute}
Here we collect basic integration by parts identities which are useful for analyzing the equation $(\rho^{2} + \kpp^{2} + \lap) v = 0$, as well as its perturbation $(\rho^{2} + \kpp^{2} - H) v = 0$. 

We start with basic multiplier identities for $\Re \brk{(\rho^{2} + \kpp^{2} + \lap) v , \gmm v}$.
\begin{lem}[Multiplier identity] \label{l:mult}
For any $v \in C^{\infty}_{0}(\bbH^{d})$, a real-valued smooth radial function $\gmm$ and $\kpp \in \bbR$, we have
\begin{align} 
&- \Re \brk{(\rho^{2} + \kpp^{2} + \lap) v , \gmm v} \nonumber \\
&= \int_{\bbH^d} \gmm \left(\abs{(\rd_{r} + \rho \coth r - i \kpp) v}^{2} + \frac{1}{\sinh^{2} r} \abs{\slashed{\nb} v}^{2} + \rho (\rho-1) \sinh^{-2} r \abs{v}^{2} \right)\,\dh \nonumber \\
&\phantom{=} + 2 \kpp \Re \brk{i v  , \gmm (\rd_{r} + \rho \coth r - i \kpp) v} - \int_{\bbH^d} \frac{1}{2} \gmm'' \abs{v}^{2}\,\dh. \label{eq:mult}
\end{align}
Alternatively, we have
\EQ{\label{eq:mult-far}
&- \Re \brk{(\rho^{2} + \kpp^{2} + \lap) v , \gmm v} \\
&= \int_{\bbH^d} \gmm \left(\abs{(\rd_{r} + \rho - i \kpp) v}^{2} + \frac{1}{\sinh^{2} r} \abs{\slashed{\nb} v}^{2} + 2 \rho^{2} (\coth r - 1) \abs{v}^{2} \right) \,\dh \\
&\phantom{=} + 2 \kpp \Re \brk{i v  ,, \gmm (\rd_{r} + \rho - i \kpp) v} - \int_{\bbH^d} \left( \frac{1}{2} \gmm'' + \rho (\coth r - 1) \gmm' \right) \abs{v}^{2}\,\dh.
}
\end{lem}

\begin{proof}
We only prove \eqref{eq:mult}, leaving the similar proof of \eqref{eq:mult-far} to the reader. By a simple integration by parts, the left-hand side of \eqref{eq:mult} equals
\begin{align*}
\int_{\bbH^d} \gmm \left(\abs{\rd_{r} v}^{2} + \frac{1}{\sinh^{2} r} \abs{\slashed{\nb} v}^{2} - \rho^{2} \abs{v}^{2} - \kpp^{2} \abs{v}^{2} \right)  \,\dh
 + \int_{\bbH^d} \gmm' \Re(v \overline{\rd_{r} v}) \,\dh.
\end{align*}
Hence, it suffices to show that
\EQ{ \label{eq:mult-sq}
& \int_{\bbH^d} \gmm \left(\abs{\rd_{r} v}^{2} - \rho^{2} \abs{v}^{2} - \kpp^{2} \abs{v}^{2}\right) \,\dh+ \int_{\bbH^d} \gmm' \Re(v \overline{\rd_{r} v}) \,\dh \\
& =  \int_{\bbH^d} \gmm \abs{(\rd_{r} + \rho \coth r - i \kpp) v}^{2}  \,\dh
 + 2 \kpp \Re \brk{i v  , \gmm (\rd_{r} + \rho \coth r - i \kpp) v}  \\
& \phantom{=} + \int_{\bbH^d} \left(- \frac{1}{2} \gmm'' + \rho (\rho-1) \sinh^{-2} r \gmm \right)  \abs{v}^{2}\,\dh.
}
Indeed, after completing the square, the left-hand side equals
\begin{align*}
& \int_{\bbH^d} \gmm \abs{(\rd_{r} + \rho \coth r - i \kpp) v}^{2} \,\dh \\
& + 2 \int_{\bbH^d} \Re(i \kpp v \overline{\gmm (\rd_{r} + \rho \coth r) v}) \,\dh- 2 \int_{\bbH^d} \kpp^{2} \gmm \abs{v}^{2}\,\dh\\
& + \int_{\bbH^d} (\gmm' - 2 \rho \coth r \gmm) \Re(\rd_{r} v \bar{v})\,\dh - \int_{\bbH^d} \rho^{2} \coth^{2} r \gmm \abs{v}^{2}\,\dh \\
&- \int_{\bbH^d} \rho^{2} \gmm \abs{v}^{2}\,\dh.
\end{align*}
For the second line, we easily have
\begin{align*}
2 \int_{\bbH^d}( \Re(i \kpp v \overline{\gmm (\rd_{r} + \rho \coth r) v}) - \kpp^{2} \abs{v}^{2})\,\dh
= 2 \kpp \Re \brk{i v  , \gmm (\rd_{r} + \rho \coth r - i \kpp) v} .
\end{align*}
For the last two lines, we rewrite $2 \Re(\rd_{r} v \bar{v}) = \rd_{r} \abs{v}^{2}$, integrate by parts and simplify to obtain:
\begin{align*}
&\int_{\bbH^d} (\gmm' - 2 \rho \coth r \gmm) \Re(\rd_{r} v \bar{v}) \,\dh- \int_{\bbH^d} \rho^{2} \coth^{2} r \gmm \abs{v}^{2}\,\dh - \int_{\bbH^d} \rho^{2} \gmm \abs{v}^{2} \,\dh\\
&= \int_{\bbH^d} \left( - \frac{1}{2} \gmm'' + \rho (\rho -1) \sinh^{-2} r \gmm\right) \abs{v}^{2}\,\dh.
\end{align*}
Putting all computations together, the desired identify follows. \qedhere
\end{proof}

Next, we consider the counterpart for $\Re \brk{i (\rho^{2} + \kpp^{2} + \lap) v , \gmm v}$, which is useful when $\kpp \neq 0$. We omit the obvious proof.
\begin{lem}[Charge identity] \label{l:charge}
For any $v \in C^{\infty}_{0}(\bbH^{d})$, a real-valued smooth radial function $\gmm$ and $\kpp \in \bbR$, we have
\EQ{ \label{eq:charge}
\Re \brk{i (\rho^{2} + \kpp^{2} + \lap) v ,  \gmm v}
=& \Re \brk{i \lap v , \gmm v}
= \Re \brk{i \rd_{r} v , \gmm' v} \\
=& - \kpp \brk{\gmm' v , v} + \Re \brk{i (\rd_{r} + \rho \coth r - i \kpp) v , \gmm' v} .
}
\end{lem}

As in Lemma~\ref{lem:Q}, given a real-valued smooth radial function $\bt$, define
$Q = \bt D_{r} + D_{r}^{\ast} \bt = \frac{1}{i} (\bt \rd_{r} + \rd_{r}^{\ast} \bt)$. Combining the key \emph{positive commutator} identity (Lemma~\ref{lem:Q}) with the preceding identities, we obtain the following identity.
\begin{lem} [Positive commutator] \label{l:pos-comm}
For any $v \in C^{\infty}_{0}(\bbH^{d})$, a real-valued smooth radial function $\bt$ and $\kpp \in \bbR$, we have
\EQ{ \label{eq:pos-comm}
&\Re \brk{i (\rho^{2} + \kpp^{2} + \lap) v , Q v - 2 \bt \kpp v}\\
& = \int_{\bbH^d} \rd_{r} \bt \abs{(\rd_{r} + \rho \coth r - i \kpp) v}^{2}\,\dh   \\
& \phantom{=} + \int_{\bbH^d} (2 \bt \coth r - \rd_{r} \bt ) \left( \sinh^{-2} r \abs{\slashed{\nb} v}^{2} + \rho (\rho - 1) \sinh^{-2} r \abs{v}^{2} \right)\,\dh\\
& \phantom{=} - \Re \brk{(\rho^{2} + \kpp^{2} + \lap) v , \rd_{r} \bt v}.
}
\end{lem}

\begin{proof}
By Lemma~\ref{l:charge}, we have
\begin{align*}
- 2 \kpp \Re \brk{i \lap v , \bt v}
= - 2 \kpp^{2} \brk{\rd_{r} \bt v , v} + 2 \kpp \Re \brk{i (\rd_{r} + \rho \coth r - i \kpp) v , \rd_{r} \bt v}.
\end{align*}
Hence, combined with the key positive commutator identity (Lemma~\ref{lem:Q}) for $\Re \brk{i \lap v , Q v}$, we obtain
\begin{align*}
	\Re \brk{i \lap v , Q v - 2 \bt \kpp v}
	= & 2 \int_{\bbH^d} \rd_{r} \bt \left( \abs{\rd_{r} v}^{2} - (\rho^{2} + \kpp^{2}) \abs{v}^{2} \right)\,\dh\\
	& - \int_{\bbH^d} (\lap \rd_{r} \bt) \abs{v} \,\dh + 2 \int_{\bbH^d} \bt \coth r  \frac{\abs{\slashed{\nb} v}^{2}}{\sinh^2r}\,\dh\\
	&+ 2 \kpp \Re \brk{i (\rd_{r} + \rho \coth r - i \kpp) v , \rd_{r} \bt v} \\
	& + \frac{1}{2} \int_{\bbH^d} (\lap \rd_{r} \bt) \abs{v}\,\dh\\
	& - \rho \int_{\bbH^d} (\lap (\bt \coth r) - 2 \rho \rd_{r} \bt) \abs{v}^{2}\,\dh.
\end{align*}
Note that $\int (\lap \rd_{r} \bt) \abs{v} = - \int \rd_{r}^{2} \bt \rd_{r} \abs{v}^{2} =  - 2 \int \rd_{r}^{2} \bt \Re(\rd_{r} v \bar{v})$.
Thus, applying \eqref{eq:mult-sq} to the first two terms on the right-hand side, we see that $\Re \brk{i \lap v , Q v - 2 \bt \kpp v}$ equals to
\begin{align*}
& 2 \int_{\bbH^d} \rd_{r} \bt \left(\abs{(\rd_{r} + \rho \coth r - i \kpp) v}^{2} + \rho (\rho - 1) \sinh^{-2} r \abs{v}^{2}\right)\,\dh \\
&- \int_{\bbH^d} \rd_{r}^{3} \bt \abs{v}^{2}   \,\dh + 2 \int_{\bbH^d} \bt \coth r \sinh^{-2} r \abs{\slashed{\nb} v}^{2}\,\dh\\
& + 2 \kpp \Re \brk{i v  , \rd_{r} \bt (\rd_{r} + \rho \coth r - i \kpp) v}  \\
&+ \frac{1}{2} \int_{\bbH^d} (\lap \rd_{r} \bt) \abs{v}^{2} \,\dh
 - \rho \int_{\bbH^d} (\lap (\bt \coth r) - 2 \rho \rd_{r} \bt) \abs{v}^{2}\,\dh.
\end{align*}
Finally, applying \eqref{eq:mult} to the third term, we can write $\Re \brk{i \lap v , Q v - 2 \bt \kpp v}$ as
\begin{align*}
 & \int_{\bbH^d} \rd_{r} \bt \abs{(\rd_{r} + \rho \coth r - i \kpp) v}^{2} \,\dh \\
 & + \int_{\bbH^d} (2 \bt \coth r - \rd_{r} \bt ) \sinh^{-2} r \abs{\slashed{\nb} v}^{2} \,\dh - \Re \brk{(\rho^{2} + \kpp^{2} + \lap) v , \rd_{r} \bt v} \\
&+ \rho (\rho-1)  \int_{\bbH^d} \rd_{r} \bt \sinh^{-2} r  \abs{v}^{2} \,\dh- \frac{1}{2} \int_{\bbH^d} \rd_{r}^{3} \bt \abs{v}^{2}\,\dh \\
&+ \frac{1}{2} \int_{\bbH^d} (\lap \rd_{r} \bt) \abs{v}^{2} \,\dh - \rho \int_{\bbH^d} (\lap (\bt \coth r) - 2 \rho \rd_{r} \bt) \abs{v}^{2}\,\dh.
\end{align*}
Simplifying the last two lines, we arrive at the desired identity.
\end{proof}

The following computation is useful for proving weighted Hardy--Poincar\'e-type estimates:
\begin{lem} [Hardy--Poincar\'e computation] \label{l:thr-hardy}
For any $v \in C^{\infty}_{0}(\bbH^{d})$ and a real-valued smooth radial function $\gmm$, we have
\EQ{ \label{eq:thr-hardy}
\int_{\bbH^d} \gmm \abs{(\rd_{r} + \rho \coth r) v}^{2}\,\dh
= & \int_{\bbH^d} \gmm \abs{(\rd_{r} + \rho \coth r - m r^{-1}) v}^{2}\,\dh \\
& + m \int_{\bbH^d} \left( 1 - m - r (\log \gmm)' \right) \gmm r^{-2}  \abs{v}^{2}\,\dh.
}
\end{lem}
\begin{proof}
We complete the square and integrate by parts as follows:
\begin{align*}
\int_{\bbH^d} \gmm \abs{(\rd_{r} + \rho \coth r) v}^{2}\,\dh
= & \int_{\bbH^d} \gmm \abs{(\rd_{r} + \rho \coth r - m r^{-1}) v}^{2}\,\dh\\
&- \int_{\bbH^d} m^{2} r^{-2} \gmm \abs{v}^{2} \,\dh
 + \int_{\bbH^d} m r^{-1} \gmm \rd_{r} \abs{v}^{2}\,\dh\\
&+ \int_{\bbH^d} 2 \rho m r^{-1} \coth r  \gmm \abs{v}^{2} \,\dh\\
= & \int_{\bbH^d} \gmm \abs{(\rd_{r} + \rho \coth r - m r^{-1}) v}^{2} \,\dh\\
& + \int_{\bbH^d} \left(m (1-m) r^{-2} \gmm - m r^{-1} \gmm' \right) \abs{v}^{2}\,\dh. 
\end{align*}
After rearranging terms, we obtain \eqref{eq:thr-hardy}.
\end{proof}

\begin{rem} \label{rem:bdry}
Although all identities in this subsection were derived under the assumption that $v \in C_{0}^{\infty}(\bbH^{d})$, they remain valid for more general $v$, provided that it has sufficient regularity and decay. For instance, in the applications in Section~\ref{ss:LE-res}, $v$ will satisfy $v, \nb v \in L^{2}$, which is enough for bounded $\gmm, \bt$. 
\end{rem}

\subsection{Extraction of a resonance} \label{ss:LE-res}
Here, our aim is to show that the local smoothing estimate \eqref{eq:le-H} holds unless $H$ has a resonance at $\tau \in [\rho^{2}, \infty)$. The definition of a resonance at the threshold $\tau = \rho^{2}$ was given in Definition~\ref{def:th-res}. In the case $\tau \in (\rho^{2}, \infty)$, the precise definition is as follows.
\begin{defn} [Embedded resonance] \label{def:emb-res}
For $\kpp > 0$, we say that $w$ is a \emph{outgoing (resp. incoming) embedded resonance} at $\tau = \rho^{2} + \kpp^{2}$ if it is a nontrivial solution to the problem
\begin{equation*}
	(\rho^{2} + \kpp^{2} - H) w = 0, \quad w \in \tilLE_{0},
\end{equation*}
satisfying the outgoing (resp. incoming) radiation condition
\begin{equation} \label{eq:orc}
	\nrm{r^{-\frac{1}{2}} (\rd_{r} + \rho - i \kpp) w}_{L^{2} (A_{j})} \to 0 \quad
	(\hbox{resp. } \nrm{r^{-\frac{1}{2}} (\rd_{r} + \rho + i \kpp) w}_{L^{2} (A_{j})} \to 0).
\end{equation}
\end{defn}
The remainder of this subsection is devoted to the proof of the following result:
\begin{prop}  \label{p:LE-res} 
Suppose that $H$ has no threshold resonance (as in Definition~\ref{def:th-res}), nor an embedded resonance at $\tau \in (\rho^{2}, \infty)$ (as in Definition~\ref{def:emb-res}).
Then for any $u_{0} \in L^{2}(\bbH^d)$ and $F \in \tilLE^{\ast}$, the solution $u(t)$ to the linear Schr\"odinger equation
\EQ{ \label{eq:sch-H}
 (-i \p_t + H) u = F,
}
with the initial value $u(0) = u_{0}$, obeys the local smoothing estimate \eqref{eq:le-H}.
\end{prop} 
\begin{proof}

\pfstep{Step~1: Reduction to forward solutions}
We begin by showing that it suffices to establish \eqref{eq:le-H} just for \emph{forward solutions}, i.e., solutions $u$ to \eqref{eq:sch-H} satisfying
\begin{equation} \label{eq:fsol}
	u = F = 0 \quad \hbox{ for sufficiently negative } t.
\end{equation}
By time reversal symmetry, having \eqref{eq:le-H} for forward solutions is clearly equivalent to having it for \emph{backward solutions}, i.e., solutions to \eqref{eq:sch-H} obeying $u = F = 0$ for sufficiently positive $t$. 

To see why it suffices to prove \eqref{eq:le-H} for only forward (or equivalently, backward) solutions, let $F$ be smooth and compactly supported in $t, x$. Then a solution to the initial value problem differs from the forward solution only by a homogeneous solution; thus it suffices to prove \eqref{eq:le-H} in the homogeneous case $F = 0$. By the $T^{\ast} T$ principle applied to $T : L^{2} \to 
\tilLE, \, u_{0} \mapsto e^{- i t H} u_{0}$, the desired estimate is equivalent to
\begin{equation*}
	\nrm{\int_{-\infty}^{\infty} e^{-i (t - s) H} F(s) \, \ud s}_{\tilLE} \aleq \nrm{F}_{\tilLE^{\ast}}.
\end{equation*}
But the expression on the left-hand side is 
\begin{equation*}
\int_{-\infty}^{\infty} e^{-i (t - s) H} F(s) \, \ud s = \int_{-\infty}^{t} e^{-i (t - s) H} F(s) \, \ud s - \int_{t}^{\infty} e^{-i (t - s) H} F(s) \, \ud s,
\end{equation*}
which is, by Duhamel's principle, precisely the sum of the forward and backward solutions.

\pfstep{Step~2: Reformulation as resolvent bounds}
The next step is to reformulate both our starting point (Theorem~\ref{t:LE1-sym}) and our goal \eqref{eq:le-H} as estimates for the resolvents of $H$ (i.e., coercivity bounds for $\tau - H$). The rough idea is to take the Fourier transform in $t$ and apply Plancherel's theorem. However, we cannot proceed directly for two reasons: First, $u$ is a-priori not square integrable in $t$, and second, $\tilLE$ and $\tilLE^{\ast}$ are not Hilbert spaces. We address these (standard) difficulties in Steps~2.a and 2.b, respectively, and achieve the desired reformulation in Step~2.c.

\pfstep{Step~2.a: Damping forward in time}
In order to apply the Fourier transform, we consider a forward-in-time damping $e^{-\eps t} u$. Since we only consider forward solutions $u$, $e^{-\eps t} u$ exhibits good decay as $\abs{t} \to \infty$ (this point is the key reason for restricting to forwards solutions!). Moreover, as the next lemma shows, Theorem~\ref{t:LE1-sym} transfers nicely to $e^{-\eps t} u$.

\begin{lem} \label{lem:LEdamp}
Let $R > 0$ be defined as in \eqref{eq:LE1} of Theorem~\ref{t:LE1-sym}. Then, for any $\eps > 0$, any forward solution $u$ to \eqref{eq:sch-H} satisfies
\begin{equation} \label{eq:LEdamp}
	\nrm{e^{-\eps t} u}_{\tilLE} \aleq \nrm{e^{-\eps t} F}_{\tilLE^{\ast}} + \nrm{e^{- \eps t} u}_{L^{2}(\bbR \times \set{r \leq R})}.
\end{equation}
Conversely, if \eqref{eq:LEdamp} holds for every $\eps > 0$ and forward solutions, then \eqref{eq:LE1} holds with $LE$ and $LE^\ast$ replaced by $\tilLE$ and $\tilLE^\ast$, respectively.
\end{lem}
\begin{proof}
To prove \eqref{eq:LEdamp} from \eqref{eq:LE1}, we partition the time axis $\bbR$ into intervals of length $\eps^{-1}$, i.e., $\bbR = \cup \set{I_{j} = [j \eps^{-1}, (j+1) \eps^{-1})}_{j \in \bbZ}$. For each $j$, we introduce the notation $t_{j} = j \eps^{-1}$, so that $I_{j} = [t_{j}, t_{j+1})$. Note that a weaker version of Theorem~\ref{t:LE1-sym} resulting from replacing $LE$ and $LE^\ast$ by $\tilLE$ and $\tilLE^\ast$, respectively, (and more precisely, its time localized version on $(-\infty, t_{j+1})$, which is immediate from the proof) for the forward solution $u$ yields 
\EQ{ \label{eq:LEdamp-piece}
\nrm{e^{-\eps t} u}^{2}_{\tilLE(I_{j})}
\aleq & e^{- 2 \eps t_{j+1}} \left( \nrm{F}_{\tilLE^{\ast}((-\infty, t_{j+1}))}^{2} + \nrm{u}_{L^{2}((-\infty, t_{j+1}) \times \set{r \leq R})}^{2} \right) \\
\aleq & \sum_{k \leq j} e^{- 2 (j - k)} \left( \nrm{e^{-\eps t} F}_{\tilLE^{\ast}(I_{k})}^{2} + \nrm{e^{-\eps t} u}_{L^{2}(I_{k} \times \set{r \leq R})}^{2}\right).
}
We emphasize that the implicit constants are \emph{independent} of $\eps$.
The exponential factors arise from the simple fact that $e^{\pm 2 \eps t} \simeq e^{\pm 2\eps t_{\ell}} \simeq e^{\pm 2 \ell}$ on $I_{\ell}$; note that the length $\eps^{-1}$ has been chosen to optimize this bound. 

The desired estimate \eqref{eq:LEdamp} now follows by summing up \eqref{eq:LEdamp-piece} in $j$ using Schur's test, where we use square summability of the expression inside the parentheses on the right-hand side.

Finally, \eqref{eq:LEdamp} implies \eqref{eq:LE1} (with $LE$ and $LE^\ast$ replaced by $\tilLE$ and $\tilLE^\ast$, respectively) by a simple application of Fatou's lemma; we leave the details to the reader.\qedhere
\end{proof}

As a simple corollary of the proof, we see that our goal \eqref{eq:le-H} is also equivalent to appropriate uniform estimates for the damped solution.
\begin{cor} \label{cor:LEeps}
Estimate \eqref{eq:le-H} holds if and only if, for any $\eps > 0$, any forward solution $u$ to \eqref{eq:sch-H} satisfies
\begin{equation} \label{eq:LEeps}
	\nrm{e^{-\eps t} P_{c} u}_{\tilLE} \aleq \nrm{e^{-\eps t} P_{c} F}_{\tilLE^{\ast}}.
\end{equation} 
\end{cor}

\pfstep{Step~2.b: $\calX^{0}_{\alp}$ spaces}
The Fourier transform in time is well-defined for the damped solution $e^{-\eps t} u$. In order to apply Plancherel, however, we face a minor nuisance that our spaces $\tilLE$ and $\tilLE^{\ast}$ are not of the form $L^{2}_{t} X_{0}$ for some Hilbert space $X_{0}$. To fix this issue, we use the auxiliary Hilbert spaces $\calX_\alpha$ and $\calX_\alpha^0$ introduced in Definition~\ref{def:X0spaces}. Then according to Lemma~\ref{l:albeLE} we have the equivalences
\begin{align*}
\|u\|_{\tilLE}\simeq\sup_{\alpha}\|u\|_{\calX_\alpha}, \quad \| F \|_{\tilLE^\ast} \simeq \inf_{\al} \|F\|_{  \calX_{\alpha}^\ast  }
\end{align*}
where the supremum and infimum are taken over all functions $\alpha \colon (0,4] \to\calA$ as in Definition~\ref{def:X0spaces} and Lemma~\ref{l:albeLE}.
The main advantage here is that for each fixed $\al \colon (0,4] \to \calA$,  $\calX_{\al}$, resp., $\calX_{\al}^\ast$, is a Hilbert space, and as in Definition~\ref{def:X0spaces}, we can write 
\EQ{
\| u \|_{\calX_{\al}}=: \| u \|_{L^2_t \calX_{\al}^0}, \quad   \| F \|_{\calX_{\al}^\ast} =: \| F \|_{L^2_t (\calX_{\al}^0)^\ast} 
}
where the norms $\calX_{\al}^0$, resp., $(\calX_{\al}^0)^\ast$ are the spatial components of $\calX_{\al}$, resp., $\calX_{\al}^\ast$. 

\pfstep{Step~2.c: Application of Plancherel's theorem}
At this point we exploit the Hilbert space structure of the spaces  $\calX_\al^0$, $(\calX_{\be}^0)^\ast$ and $L^2(\set{r \leq R})$. By applying Plancherel's theorem in time, the
 estimate established in Lemma~\ref{lem:LEdamp} is equivalent to the estimate 
\EQ{\label{eq:L2tau1}
\| \ha {e^{-\eps t} u} \|_{L_\tau^2 \calX_{\al}^0} \lesssim  \| \ha {e^{-\eps t} F} \|_{L^2_\tau (\calX_{\be}^0)^\ast} + \| \ha {e^{-\eps t} u} \|_{L^2_{\tau} L^2_x(r \le R)} 
}
Similarly, the desired estimate \eqref{eq:LEeps} (see Corollary~\ref{cor:LEeps}) is equivalent to showing that 
\EQ{ \label{eq:L2tau2} 
\| \ha {e^{-\eps t}P_c u} \|_{L_\tau^2 \calX_{\al}^0} \lesssim  \| \ha {e^{-\eps t} P_cF} \|_{L^2_\tau (\calX_{\be}^0)^\ast}
}
uniformly in $\eps>0$, and $\al, \be \colon (0, 4] \to \calA$. We claim that the previous two estimates are equivalent to the following fixed $\tau$ bounds holding uniformly in $\tau \in \R$. 
\begin{lem}
The estimate~\eqref{eq:L2tau1} is equivalent to the estimate 
\EQ{\label{eq:fixedtau1} 
 \| v(\tau) \|_{ \tilLE_0} \lesssim  \| (\tau + i\eps - H) v (\tau) \|_{ \tilLE_0^\ast} + \| v(\tau) \|_{L^2(r \le R)} ,
}
holding uniformly in $\tau \in \R$, and $\eps>0$. Similarly, the estimate~\eqref{eq:L2tau2} is equivalent to the estimate 
\EQ{\label{eq:fixedtau2} 
\| P_{c} v(\tau) \|_{ \tilLE_0} \lesssim    \| (\tau + i\eps - H) P_{c} v (\tau) \|_{ \tilLE_0^\ast} ,
}
holding uniformly in $\tau \in \bbR$ and $\eps>0$. 
\end{lem}  
\begin{proof} 
We only provide a detailed proof of the equivalence between \eqref{eq:L2tau1} and \eqref{eq:fixedtau1}; the case of \eqref{eq:L2tau2} and \eqref{eq:fixedtau2} is similar, and is left to the reader.

We consider a forward solution $u_{\eps}$ to the damped equation
\begin{equation} \label{eq:sch-H-damp}
- i \rd_{t} u_{\eps} - i \eps u_{\eps} + H u_{\eps} = G.
\end{equation}
If $u$ is a forward solution to \eqref{eq:sch-H}, then $u_{\eps} = e^{-\eps t} u$ obeys \eqref{eq:sch-H-damp} with $G = e^{-\eps t} F$.
Hence, the uniform in $\tau$ estimates~\eqref{eq:fixedtau1} and~\eqref{eq:fixedtau2} directly imply the estimates~\eqref{eq:L2tau1} and~\eqref{eq:L2tau2} by taking $G = e^{-\eps t} F$, square integrating in $\tau$, and noting that\footnote{Here, we use a nonstandard definition of the Fourier transform
\begin{equation*}
	\hat{f}(\tau) = \int f(t) e^{i t \tau} \, \ud t
\end{equation*}
to match our choice of signs in the Schr\"odinger operator $- i \rd_{t} + H$.} 
\EQ{
\ha {e^{-\eps t} F} =  \ha{ -i \p_t u_\eps} - i \eps \ha{u_{\eps}} + \ha{ H u_\eps} = - (\tau + i\eps - H) \ha u_{\eps}.
}

For the opposite direction, by modulation symmetry, it suffices to prove the implication for $\tau =0$. Let $\ha \phi \in C^\infty_0(\R)$ be such that $\ha \phi(\tau) = 1$ for $\tau \in [-1, 1]$ and $\supp \ha \phi \in B(0, 2)$ and $\int_{\R} \ha \phi^2(\tau) \, \ud \tau  = 1$. Then set $\ha \phi_n(\tau) = \sqrt{n} \ha \phi (n \tau)$.  For any continuous function $\ha f$, we have 
\EQ{ \label{eq:phin} 
\int_\R \ha\phi_n^2(\tau) \ha f^2(\tau) \, d \tau \to f^2(0) \mas n \to \infty.
}
For any forward solution $u_\eps$ to \eqref{eq:sch-H-damp}, the time-convolution $u_\eps  \ast \phi_n$ solves 
\EQ{
(-i \p_t - i\eps + H) u_\eps  \ast \phi_n = G \ast \phi_n.
}
Applying the estimate~\eqref{eq:L2tau1} to the above yields 
\EQ{
\| \ha u_\eps  \ha \phi_n \|_{L_\tau^2 \calX_{\al}^0} \lesssim  \| \ha G \ha \phi_n \|_{L^2_\tau (\calX_{\be}^0)^\ast} + \| \ha u_\eps \ha \phi_n \|_{L^2_{\tau} L^2_x(r \le R)} 
}
Letting $n \to \infty$ and using~\eqref{eq:phin} yields~\eqref{eq:fixedtau1}. \qedhere
\end{proof} 

\pfstep{Step~3: Compactness argument}
Our goal now has been reduced to proving that~\eqref{eq:fixedtau2} holds, assuming~\eqref{eq:fixedtau1}. At this stage we begin a contradiction argument, which is a minor variant of the proof of the Fredholm alternative theorem. Its aim is to show that failure of \eqref{eq:fixedtau2} implies existence of a resonance at some $\tau \in [\rho^{2}, \infty)$; for a precise definition, see Definition~\ref{def:th-res} in the case $\tau = \rho^{2}$, and Definition~\ref{def:emb-res} in the case $\tau \in (\rho^{2}, \infty)$.

\pfstep{Step~3.a: Beginning of the argument}
Assume the estimate~\eqref{eq:fixedtau2} fails to hold uniformly in $\eps>0$, and $\tau \in \R$. In view of \eqref{eq:fixedtau1}, we can find sequences~$\{\tau_n\} \subset \R$, $\eps_n > 0$ and $v_n \in \tilLE_0$ such that 
\EQ{ \label{eq:ch}
 \| P_{c} v_n \|_{L^2(r \le R)}  =1, \mand \|(\tau_n + i \eps_n - H) P_{c} v_n \|_{\tilLE_0^\ast} \to 0 \mas n \to \infty.
}
Redefining $P_{c} v_{n}$ to be $v_{n}$, we may remove the projection $P_{c}$ in the previous statement, and add the property that $v_{n} = P_{c} v_{n}$.
By the obvious resolvent bound
\begin{equation*}
	\nrm{v}_{L^{2}}^{2} \aleq \eps^{-1} \nrm{(\tau + i \eps - H) v}_{L^{2}}^{2},
\end{equation*}
which follows by applying Cauchy--Schwarz to $\Re \brk{(\tau + i \eps - H) v, i v}$, we know that for each fixed $\eps>0$, ~\eqref{eq:fixedtau2} holds uniformly in~$\tau \in \R$. Thus, passing to subsequences we can assume that $\eps_{n} \to 0$ and $\tau_n \to \tau \in [- \infty, \infty]$. In the case that $\tau \in \R$ we can ensure, passing to a further subsequence,  that 
\EQ{
v_n \rightharpoonup v_\infty \quad \textrm{weakly* in} \, \, \tilLE_0.
}
By Lemma~\ref{l:H12loc}, we have $\tilLE_0 \subset H^{\frac{1}{2}}_{\loc}$ and thus we have the strong $L^2_{\loc}$ convergence, 
\EQ{
v_n \to v_\infty \quad \textrm{in} \, \, L^2_{\loc}.
}
Hence in the case $\tau_n \to  \tau \in \R$, we can produce a function $v_\infty$ with the properties 
\EQ{ \label{eq:vinfty} 
v_\infty \in \tilLE_0, \quad H v_\infty = \tau v_\infty, \quad  \| v_\infty \|_{L^2( r \le R)} = 1, \quad P_{c} v_{\infty} = v_{\infty}.
}
In the remainder of this step, we make the initial reduction to the case $\tau \in (\rho^{2}, \infty)$.

\pfstep{Step~3.b: Ruling out the case of large $\abs{\tau}$}
Here, we reduce to the case that $\tau_n \to \tau  \in [-M, M]$ as $n \to \infty$ for $M$ sufficiently large; the idea is to prove a uniform bound for large enough $\tau$. To do this we establish the fixed-$\tau$ inequality
\EQ{ \label{eq:taubig} 
\abs{\tau}^{\frac{1}{4}}  \| v \|_{L^2( r \le 2R)} \lesssim  \|v \|_{\tilLE_0} + \| (\tau  + i \eps - H) v \|_{\tilLE^\ast_0}.
}
Assuming for the moment that~\eqref{eq:taubig} is true, we can insert the above into~\eqref{eq:fixedtau1} yielding 
\EQ{
 \| v \|_{ \tilLE_0} \le   C \| (\tau + i\eps - H) v \|_{ \tilLE_0^\ast} +  C\abs{\tau}^{-\frac{1}{4}}\|v \|_{\tilLE_0},
}
and we can find $M$ large enough so that for $\abs{\tau}>M$ the last term above can be absorbed into the left-hand side, yielding the desired estimate 
\EQ{
 \| v \|_{ \tilLE_0} \lesssim \| (\tau + i\eps - H) v \|_{ \tilLE_0^\ast}.
}
The above allows us to  assume $\tau_n \to  \tau \in [-M, M]$   in the context of our contradiction argument. 

We now prove~\eqref{eq:taubig}. Let $v_R(x) := v(x) \chi_{R}(x)$ where $\chi_{R}(x) = 1$ if $\abs{x} \le R$ and is smooth and compactly supported in the ball of radius $2R$. From Lemmas~\ref{l:H12loc} and~\ref{l:HminusLEstar}, such a function $v_{R}$ obeys
\EQ{
 \| v_R \|_{H^{\frac{1}{2}}} &\le C_R\| v \|_{\tilLE_0}, \\ 
 \| (\tau + i \eps - H) v_R \|_{H^{-\frac{3}{2}}} &\le C_R \| (\tau + i \eps - H) v \|_{\tilLE^\ast_0}. 
}
These estimates imply that~\eqref{eq:taubig} follows from the estimate 
\EQ{ \label{eq:bt12} 
\abs{\tau}^{\frac{1}{4}}  \| v \|_{L^2( r \le 2R)} \lesssim  \| v_R \|_{H^{\frac{1}{2}}} +  \| (\tau + i \eps - H) v_R \|_{H^{-\frac{3}{2}}}.
}
To prove~\eqref{eq:bt12} we interpolate between the trivial estimate 
\EQ{
\| v_R \|_{H^{\frac{1}{2}}} \le \| v_R \|_{H^{\frac{1}{2}}} +  \| (\tau + i \eps - H) v_R \|_{H^{-\frac{3}{2}}},
}
and the estimate 
\EQ{
\abs{\tau} \| v_R \|_{H^{-\frac{3}{2}}} \le \| v_R \|_{H^{\frac{1}{2}}} +  \| (\tau + i \eps - H) v_R \|_{H^{-\frac{3}{2}}}.
}
To prove the latter, write 
\EQ{
-\tau v_R = (H - \tau - i \eps) v_R - H v_{R} + i \eps v_R .
}
Then choosing $\tau$ sufficiently large relative to $\eps \in (0, 1]$ we have
\EQ{
\abs{\tau} \| v_R \|_{H^{-\frac{3}{2}}}  &\lesssim     \| (\tau + i \eps - H) v_R \|_{H^{-\frac{3}{2}}} + \| H v_{R} \|_{H^{-\frac{3}{2}}}  \\
&\lesssim   \| (\tau + i \eps - H) v_R \|_{H^{-\frac{3}{2}}} +  \| v_R \|_{H^{\frac{1}{2}}} .
}

\pfstep{Step~3.c: Ruling out the case $\tau < \rho^{2}$}
Using the fact that $v_{n} = P_{c} v_{n}$, we have
\EQ{
0 = \lim_{n \to \infty} \Re \ang{(H - \tau_n-i\eps_n) v_{n} , v_{n} }  \ge \lim_{n \to \infty} ( \rho^2 - \tau_n) \| v_{n} \|_{L^2}^2.
}
Since $\tau< \rho^2$, this implies that $\lim_{n \to \infty} \nrm{v_n}_{L^{2}} = 0$, which contradicts \eqref{eq:vinfty} and the strong $L^{2}_{loc}$ convergence $v_{n} \to v_{\infty}$. 

\pfstep{Step~4: Extraction of an embedded resonance in the case $\tau > \rho^{2}$}
The goal here is to show that when $\tau > \rho^{2}$, $v_{\infty}$ is an \emph{embedded resonance} in the sense that $v_{\infty} \in \tilLE_{0}$, $(\tau - H) v_{\infty} = 0$ and $v_{\infty}$ satisfies the \emph{outgoing radiation condition}: 
\begin{equation} \label{eq:orc-vinfty}
	\lim_{j \to \infty} \nrm{r^{-\frac{1}{2}} (\rd_{r} + \rho - i \sqrt{\tau - \rho^{2}}) v_{\infty}}_{L^{2}(A_{j})} = 0.
\end{equation}

To begin, we fix $\tau> \rho^2$, and write $\tau = \rho^{2} + \kpp^{2}$ for $\kpp > 0$. 

\begin{proof}[Proof of \eqref{eq:orc-vinfty}]
We begin by splitting $\rho^{2} + \kpp^{2} + i \eps - H = (\rho^{2} + \kpp^{2} + \lap) + i \eps - H_{\lot}$ and applying Lemma~\ref{l:pos-comm}:
\begin{align*}
&\Re \brk{i (\rho^{2} + \kpp^{2} + i \eps - H) v , Q v - 2 \bt \kpp v} + \Re \brk{(\rho^{2} + \kpp^{2} + i \eps - H) v , \rd_{r} \bt v}.\\
& = - \eps \Re \brk{v , Q v - 2 \bt \kpp v} + \int_{\bbH^d} \rd_{r} \bt \abs{(\rd_{r} + \rho \coth r - i \kpp) v}^{2}  \,\dh \\
& \phantom{=} + \int_{\bbH^d} (2 \bt \coth r - \rd_{r} \bt ) \left( \sinh^{-2} r \abs{\slashed{\nb} v}^{2} + \rho (\rho - 1) \sinh^{-2} r \abs{v}^{2} \right) \,\dh\\
& \phantom{=} - \Re \brk{i H_{\lot} v , Q v - 2 \bt \kpp v} - \Re \brk{H_{\lot} v , \rd_{r} \bt v}.
\end{align*}
To treat the first term on the right-hand side, we use Lemma~\ref{l:mult} and proceed as follows:
\begin{align*}
& \kpp \Re \brk{v , Q v - 2 \bt \kpp v} \\
& =  2\kpp \Re \brk{i v , \bt(\rd_{r} + \rho \coth r - i\kpp) v}  \\
&=  - \int_{\bbH^d} \bt \left(\abs{(\rd_{r} + \rho \coth r - i \kpp) v}^{2} + \frac{1}{\sinh^{2} r} \abs{\slashed{\nb} v}^{2} + \rho (\rho-1) \frac{\abs{v}^{2}}{\sinh^{2} r}  \right)\,\dh \\
&\phantom{=} - \Re \brk{(\rho^{2} + \kpp^{2} + i \eps + \lap) v , \bt v}+ \int_{\bbH^d} \frac{1}{2} \rd_{r}^{2} \bt \abs{v}^{2}\,\dh \\
& \leq  \int_{\bbH^d} \bt (\rho(\rho-1))_{-} \sinh^{-2} r \abs{v}^{2} \,\dh+ \int_{\bbH^d} \frac{1}{2} \rd_{r}^{2} \bt \abs{v}^{2}\,\dh \\
& \phantom{\leq} - \Re \brk{(\rho^{2} + \kpp^{2} + i \eps - H) v , \bt v} + \Re \brk{H_{\lot} v , \bt v}.
\end{align*}
Combining these two expressions, using the observation
\begin{align*}
\begin{split}
\Re\angles{(\rho^2+\kappa^2-H)v}{\partial_r\beta v}=\Re\angles{(\rho^2+\kappa^2+i\eps-H)v}{\partial_r\beta v},
\end{split}
\end{align*}
and rearranging terms, we arrive at the following key inequality:
\EQ{ \label{eq:orc-key}
&\int_{\bbH^d} \rd_{r} \bt \abs{(\rd_{r} + \rho \coth r - i \kpp) v}^{2}\,\dh \\
&\leq \Re \brk{(\rho^{2} + \kpp^{2} + i \eps - H) v , (-i) Q v + (2 i \bt \kpp + \rd_{r} \bt - \frac{\eps}{\kpp} \bt) v} \\
& \phantom{\leq} - \Re \brk{H_{\lot} v , iQ v - (2 i \bt \kpp + \rd_{r} \bt +\frac{\eps}{\kpp} \bt) v } \\
& \phantom{\leq} - \int_{\bbH^d} (2 \bt \coth r - \rd_{r} \bt ) \sinh^{-2} r \abs{\slashed{\nb} v}^{2} \,\dh\\
&\phantom{\leq} + \frac{\eps}{\kpp}(\rho(\rho-1))_{-}\int_{\bbH^d} \bt  \sinh^{-2} r \abs{v}^{2} \,\dh\\
& \phantom{\leq} - \rho (\rho - 1) \int_{\bbH^d} (2 \bt \coth r - \rd_{r} \bt )  \sinh^{-2} r \abs{v}^{2}\,\dh
+ \frac{\eps}{2 \kpp} \int_{\bbH^d} \rd_{r}^{2} \bt \abs{v}^{2} \,\dh.
}

We now apply \eqref{eq:orc-key} with $(v, \kpp, \eps) = (v_{n}, \kpp_{n}, \eps_{n})$ where $\kappa_n:=\tau_n-\rho^2$. Our aim is to choose $\bt$ adequately in order to produce the estimate
\EQ{ \label{eq:orc-goal}
& 2^{-j} \nrm{(\rd_{r} + \rho \coth r - i \kpp_{n}) v_{n}}_{L^{2}(A_{j})}^{2} \\
& \aleq \nrm{(\rho^{2} + \kpp_{n}^{2} + i \eps_{n} - H) v_{n}}_{\tilLE_{0}^{\ast}} \left( \nrm{\chi_{>R} \nb v_{n}}_{\tilLE_{0}} + \left( 1+\frac{\eps_{n}}{\kpp_{n}} \right) \nrm{v_{n}}_{\tilLE_{0}} \right) \\
&\phantom{\aleq} + \sum_{k \geq 0} 2^{-\dlt (j-k)_{+}} (2^{-k} + \nrm{ r^{3+2\sigma}(\bsb, \nb \bsb, V)}_{L^{\infty}(A_{k})}) \nrm{ \brk{r}^{-\frac{3+2\sigma}{2}} (\nb v_{n}, v_{n})}_{L^{2}(A_{k})}^{2} \\
&\phantom{\aleq} + \frac{\eps_{n}}{\kpp_{n}} 2^{j} \nrm{\frac{1}{\brk{r}^{3}} v_{n}}_{L^{2}(\set{\frac{1}{2} R < r < C 2^{j}}}^{2}.
}
As we will discuss below, under the assumption that
\begin{equation} \label{eq:orc-eps-kpp}
	\frac{\eps_{n}}{\kpp_{n}} = o(1),
\end{equation}
which is obvious when $\kpp > 0$, we have $\lim_{j \to \infty} \limsup_{n \to\infty} \hbox{(right-hand side)} = 0$, thanks to the decay assumptions on $\bsb, V$ as well as the uniform $\tilLE_{0}$ and $\tilLE_0^1$ bounds for $v_n$, and $\chi_{>R} v_{n}$ respectively (see Lemma~\ref{l:ell-LE}). Note that the $\liminf_{n \to \infty}$ of the left-hand side is bounded below by $2^{-j} \nrm{(\rd_{r} + \rho \coth r - i \kpp) v_{\infty}}_{L^{2}(A_{j})}^{2}$, by Fatou's lemma and strong $L^{2}_{\loc}$ convergence of $v_{n}$. Thus, 
\begin{equation*}
\lim_{j \to \infty} 2^{-j} \nrm{(\rd_{r} + \rho \coth r - i \kpp) v_{\infty}}_{L^{2}(A_{j})}^{2} = 0,
\end{equation*}
which implies \eqref{eq:orc-vinfty} since the difference $\rho \coth r - \rho$ decays exponentially.

It remains to prove \eqref{eq:orc-goal}. For a sufficiently small (universal) number $\dlt > 0$, and any $j > 0$, let $\alp_{(>j)}(r)$ be a slowly varying, nondecreasing radial function satisfying  
\EQ{ \label{alp->j}
	\alp_{(>j)}(r) \simeq \min \set{(r 2^{-j})^{\dlt}, 1}, \quad
	\alp_{(>j)}(r) = 1 \quad \hbox{ for } r > 2^{j+10}, \\
	\alp'_{(>j)}(r) \ageq \dlt 2^{-j} \quad \hbox{ for } r \in A_{j}, \quad
	\abs{\alp_{(>j)}^{(k)}(r)} \aleq_{k} \frac{\dlt}{r^{k}} \alp_{(>j)}(r) .
}
In particular, $r \abs{(\log \alp_{(>j)}(r))'} \aleq \dlt$. We make the choice
\begin{equation*}
	\bt(r) = \dlt^{-1} \chi_{>R}^{2}(r) \alp_{(>j)}(r),
\end{equation*}
where $R > 1$ is fixed so that Lemma~\ref{l:ell-LE} holds, and so that
\EQ{ \label{eq:orc-ang-coeff}
0 \leq 2 \bt \coth r - \rd_{r} \bt \leq C \dlt^{-1} \chi_{(>j)} \quad \hbox{ for } r > 10 R, \\
\abs{2 \bt \coth r - \rd_{r} \bt} \aleq_{R} \dlt^{-1} 2^{-\dlt j} \quad \hbox{ for } r \leq 10R.
}

Note that $\bt$ is nondecreasing and $\rd_{r} \bt \ageq 2^{-j}$ on $A_{j}$, so that the left-hand side of \eqref{eq:orc-key} is bounded below by that of \eqref{eq:orc-goal}, i.e.,
\begin{equation*}
	2^{-j} \nrm{(\rd_{r} + \rho \coth r - i \kpp_{n}) v_{n}}_{L^{2}(A_{j})}^{2}
	\aleq \int_{\bbH^d} \rd_{r} \bt \abs{(\rd_{r} + \rho \coth r - i \kpp_n) v_n}^{2}\,\dh.
\end{equation*}

We now treat the right-hand side of \eqref{eq:orc-key}. By passing to a subsequence, we may assume that $\kpp_{n} \leq 2M$. In what follows, we suppress the dependence of constants on $R$, $\dlt$ and $M$.

For the first term on the right-hand side of \eqref{eq:orc-key}, we estimate
\begin{align*}
& \abs{\Re \brk{(\rho^{2} + \kpp_{n}^{2} + i \eps_{n} - H) v_{n} , (-i) Q v_{n} + (2 i \bt \kpp_{n} + \rd_{r} \bt - \frac{\eps_{n}}{\kpp_{n}} \bt) v_{n}}} \\
& \aleq \nrm{(\rho^{2} + \kpp_{n}^{2} + i \eps_{n} - H) v_{n}}_{\tilLE_{0}^{\ast}} \left( \nrm{\beta\partial_r v_{n}}_{\tilLE_{0}} + \left( 1+\frac{\eps_{n}}{\kpp_{n}} \right) \nrm{v_{n}}_{\tilLE_{0}} \right).
\end{align*}
Now in view of Lemmas~\ref{l:ell-LE} and~\ref{l:drLE}, the right-hand side vanishes as $n \to \infty$, provided $\frac{\eps_{n}}{\kpp_{n}} = O(1)$.

Next, for the second term, we estimate  
\begin{align*}
& \abs{\Re \brk{H_{\lot} v_n , iQ v_n - (2 i \bt \kpp_n + \rd_{r} \bt + \frac{\eps_n}{\kpp_n} \bt) v_n } }\\
& \aleq \sum_{k \geq 0} 2^{-\dlt (j-k)_{+}} \nrm{ r^{3+2\sigma} (\bsb, \nb \bsb, V)}_{L^{\infty}(A_{k})} \nrm{ \brk{r}^{-\frac{3+2\sigma}{2}} \chi_{> R}(\nb v_{n}, v_{n})}_{L^{2}(A_{k})}^{2},
\end{align*}
which decays as $j \to \infty$ uniformly in $n$ thanks to $\nrm{ r^{3+2\sigma}(\bsb, \nb \bsb, V)}_{L^{\infty}(A_{k})}$, the uniform $\tilLE_{0}$ and $\tilLE_0^1$ bounds for $ v_{n}$ and $\chi_{r> R}v_{n}$ respectively, as well as Lemma~\ref{l:wL2LE}.

For the third term, note that the integrand has the favorable sign for $r > 10 R$ thanks to \eqref{eq:orc-ang-coeff}. Hence,
\begin{align*}
 &- \int_{\bbH^d} (2 \bt \coth r - \rd_{r} \bt ) \sinh^{-2} r \abs{\slashed{\nb} v_n}^{2} \,\dh\\
& \leq C 2^{-\dlt j} \int_{\set{r \leq 10R}}  \sinh^{-2} r \abs{\slashed{\nb} v_n}^{2}\,\dh,
\end{align*}
which is acceptable. 

Finally, we treat the last three terms on the right-hand side of \eqref{eq:orc-key}. For the first two of these terms, we use \eqref{eq:orc-ang-coeff} and the obvious inequality $\sinh^{-2} r \aleq \brk{r}^{-4}$ (any power greater than $3$ would do) to bound 
\begin{align*}
	&\abs{\int_{\bbH^d} (2 \bt \coth r - \rd_{r} \bt) \sinh^{-2} r \abs{v_{n}}^{2} \,\dh} +\abs{\int_{\bbH^d} \bt  \sinh^{-2} r \abs{v_n}^{2}}\,\dh \\
	&\aleq  \int_{\set{r > 10 R}} \frac{\alp_{(>j)}(r)}{\brk{r}} \frac{1}{\brk{r}^{3}} \abs{v_{n}}^{2} \,\dh+ 2^{-\dlt j} \nrm{v_{n}}_{L^{2}(\set{r < 10 R})} \\
	&\aleq  \sum_{k \geq 0} 2^{-\dlt(j-k)_{+}} 2^{-k} \nrm{\brk{r}^{-\frac{3}{2}} v_{n}}_{L^{2}(A_{k})}^{2}.
\end{align*}
On the other hand,
\begin{equation*}
\frac{\eps_{n}}{2 \kpp_{n}} \abs{\int \rd_{r}^{2} \bt \abs{v_{n}}^{2}}
\aleq \frac{\eps_{n}}{\kpp_{n}} \nrm{\frac{2^{j}}{\brk{r}^{3}} v_{n}}_{L^{2}(\set{\frac{1}{2} R < r < C 2^{j}}}^{2}
\end{equation*}
which vanishes as $n \to \infty$ for each fixed $j$, provided that $\frac{\eps_{n}}{\kpp_{n}} = o(1)$. \qedhere
\end{proof}

\pfstep{Step~5: Extraction of a threshold resonance in the case $\tau = \rho^{2}$}
We need to prove \eqref{eq:th-res}:
\begin{equation*} 
	\nrm{r^{-\frac{1}{2}} (\rd_{r} + \rho) v_{\infty}}_{L^{2}(A_{j})} \to 0.
\end{equation*}
Note that $v_{n}$ exhibits different behaviors depending on how $\tau_{n} + i \eps_{n}$ approaches $\rho^{2}$; our proof will reflect this phenomenon. By passing to a subsequence, we may assume that one of the two scenarios hold:
\begin{enumerate}
\item $\eps_{n} \aleq \tau_{n} - \rho^{2}$.
\item $(\tau_{n} - \rho^{2})_{+} \aleq \eps_{n}$
\end{enumerate}
We argue differently in each case.

\pfstep{Step~5.a: The case $\eps_{n} \aleq \tau_{n} - \rho^{2}$}
Here, we simply note that the same proof as \eqref{eq:orc-vinfty} works (see, in particular, \eqref{eq:orc-goal}), since $\frac{\eps_{n}}{\kpp_{n}} \aleq \frac{\tau_{n} - \rho^{2}}{\kpp_{n}} = \kpp_{n} \to 0$ as $n \to \infty$.

\pfstep{Step~5.b: The case $(\tau_{n} - \rho^{2})_{+} \aleq \eps_{n}$}
Here, the idea is to perturb off the $\lap + \rho^{2}$ estimates. We write
\begin{equation*}
	\Re \brk{(\tau_{n} + i \eps_{n} - H) v_{n} , v_{n}}
	= \brk{(\rho^{2} + \lap) v_{n} , v_{n}} -  \Re \brk{H_{\lot} v_{n} , v_{n}} + (\tau_{n} - \rho^{2}) \brk{v_{n} , v_{n}},
\end{equation*}
The last term is good if $\tau_{n} - \rho^{2} \leq 0$, but problematic otherwise. However, note that
\begin{equation*}
	\Im \brk{(\tau_{n} + i \eps_{n} - H) v_{n} , v_{n}}
	= \eps_{n} \brk{v_{n} , v_{n}},
\end{equation*}
so that, after rearranging terms,
\begin{align*}
	\Re \brk{- (\rho^{2} + \lap) v_{n} , v_{n}} &+ (\tau_{n} - \rho^{2})_{-} \brk{v_{n} , v_{n}}
	\\=& - \Re \brk{(\tau_{n} + i \eps_{n} - H) v_{n} , v_{n}} - \Re \brk{H_{\lot} v_{n} , v_{n}} \\
	& + \frac{(\tau_{n} - \rho^{2})_{+}}{\eps_{n}} \Im \brk{(\tau_{n} + i \eps_{n} - H) v_{n} , v_{n}}.
\end{align*}
By Lemma~\ref{l:mult} with $\gmm = 1$ and $\kpp=0$, it follows that
\begin{align*}
	\nrm{(\rd_{r} + \rho) v_{n}}_{L^{2}}^{2}
	\leq & \abs{\brk{(\tau_{n} + i \eps_{n} - H) v_{n} , v_{n}}} + \abs{\brk{H_{\lot} v_{n} , v_{n}}} \\
	& + \frac{(\tau_{n} - \rho^{2})_{+}}{\eps_{n}} \abs{\Im \brk{(\tau_{n} + i \eps_{n} - H) v_{n} , v_{n}}}.
\end{align*}
This estimate implies that $(\rd_{r} + \rho) v_{\infty} \in L^{2}$, which is better than \eqref{eq:th-res}. 
%
%
\end{proof}

\subsection{Absence of embedded resonances} \label{ss:no-emb-res}
In this subsection, we establish the absence of embedded resonances of $H$:
\begin{prop} \label{p:no-emb-res}
Let $\kpp > 0$. Let $w$ be a solution to $(\rho^{2} + \kpp^{2} - H) w = 0$ satisfying $w \in \tilLE_{0}$ and the outgoing radiation condition \eqref{eq:orc}.
Then $w = 0$.
\end{prop}  
Combined with Proposition~\ref{p:LE-res}, this result implies Theorem~\ref{t:LE-H}.

\begin{proof}
The proof proceeds in three steps: First, we show that \eqref{eq:orc} implies arbitrary polynomial decay; second, we show that arbitrary polynomial decay implies vanishing outside a large ball; finally, appealing to standard unique continuation results, we conclude that $w =0$.
\pfstep{Step~1: Outgoing radiation condition implies faster than polynomial decay}
The aim of this step is to prove that, for any $N > 0$,
\begin{equation} \label{eq:poly-decay}
	r^{N} w \in L^{2}.
\end{equation}
The main point is that the outgoing radiation condition is used to justify various integration by parts in multiplier arguments for $w$. Indeed, \eqref{eq:orc} allows us to deduce  the following corollary. 
\begin{lem} \label{l:res-decay} Let $w$ be as above. Then 
\EQ{ \label{eq:res-decay}
\|r^{-\frac{1}{2}} w \|_{L^2( r \simeq 2^j)} + \| r^{-\frac{1}{2}} \na w \|_{L^2(r \simeq 2^j)} \to 0 \mas j \to \infty.
}
\end{lem} 
\begin{proof} 
We first split $(\rho^{2} + \kpp^{2} - H) w = \lap w + (\rho^{2} + \kpp^{2} - H_{\lot}) w$, and take the inner product with $i \gmm w$. For the contribution of $\lap w$, we apply Lemma~\ref{l:charge}. For the rest, we see that
\begin{align*}
\Re\brk{(\rho^{2} + \kpp^{2} - H_{\lot}) w , i \gmm w}
=& -\Re\brk{\frac{1}{i} (\bsb^{\mu} \nb_{\mu} w + \nb_{\mu} (\bsb^{\mu} w)) , i \gmm w} \\
=& \brk{(\bsb^{\mu} \nb_{\mu} \gmm) w , w}.
\end{align*}
Hence, we arrive at
\begin{align*}
	-\Re\brk{(\rho^{2} + \kpp^{2} - H) w , i \gmm w}
	=& \kpp \brk{\gmm' w , w} - \Re\brk{i (\rd_{r} + \rho - i \kpp) w , \gmm' w} \\
	& - \brk{(\bsb^{\mu} \nb_{\mu} \gmm) w , w}.
\end{align*}
Choosing $\gmm'$ to be a radial bump function adapted to $\set{r \simeq 2^{j}}$ and using the outgoing radiation condition and the decay of $\bsb$, we obtain $\|r^{-\frac{1}{2}} w \|_{L^2( r \simeq 2^j)} \to 0$. Then $\|r^{-\frac{1}{2}} \nb w \|_{L^2( r \simeq 2^j)} \to 0$ follows by elliptic regularity.
\end{proof} 

As a consequence of \eqref{eq:res-decay}, we may justify integration by parts identities involving bounded multipliers.
In particular, we claim that:
\begin{lem} \label{l:poly-decay-key}
Let $w \in \tilLE_{0}$ satisfy $(\tau - H) w = 0$ for $\tau > \rho^2$, as well as the conclusions of \eqref{l:res-decay}. Then for any large enough $N$ and $R_{1}$, we have 
\begin{equation} \label{eq:poly-decay-key}
	(\tau - \rho^{2}) \nrm{\brk{r}^{N} \chi_{< R_{1}} w}_{L^{2}}^{2}
	\aleq_{N} \nrm{\brk{r}^{N-1} (\nb w, w)}_{L^{2}}^{2},
\end{equation}
where the constant is independent of $R_{1}$.
\end{lem}
Let us postpone the proof of Lemma~\ref{l:poly-decay-key}, and describe an iteration argument which leads to the desired conclusion \eqref{eq:poly-decay}.
First, by \eqref{eq:res-decay}, note that $\brk{r}^{-\frac{1}{2}-\dlt} v \in L^{2}$ and $\brk{r}^{-\frac{1}{2}-\dlt} \nb v \in L^{2}$. Applying Lemma~\ref{l:poly-decay-key} with $N = \frac{1}{2}-\dlt$ (for any $0 < \dlt < \frac{1}{2}$) and taking $R_{1} \to \infty$, it follows that $\brk{r}^{\frac{1}{2} - \dlt} w \in L^{2}$. By elliptic regularity, it follows that $\brk{r}^{\frac{1}{2} - \dlt} \nb w \in L^{2}$ as well (for this, we only need boundedness of $\bsb, V$). Clearly, we may repeat this process with $N = \frac{3}{2} -\dlt, \frac{5}{2} -\dlt, \ldots$, which establishes \eqref{eq:poly-decay}.

It remains to prove the preceding lemma.
\begin{proof}[Proof of Lemma~\ref{l:poly-decay-key}]
The multiplier argument uses the exponential volume growth of spheres, in that $\bt \coth r - \rd_{r} \bt$ remains positive for large $r$ for any polynomially growing $\bt$.

For now, we let $\bt$ be a smooth real-valued radial function, such that $\bt$ and its derivatives vanish at zero and are \emph{bounded} at infinity. Let $w \in H^{2}_{\loc}$ satisfy the decay condition \eqref{eq:res-decay}. Then Lemma~\ref{lem:Q} applies, and we have
\begin{align*}
\Re\brk{i(\tau - H) w , Q w}
= & \Re\brk{i \lap w , Q w} - \Re \brk{i H_{\lot} w , Q w} \\
= & 2 \int_{\bbH^d} \rd_{r} \bt \left( \abs{\nb w}^{2} - \rho^{2} \abs{w}^{2} \right) \,\dh\\
& + 2 \int_{\bbH^d} (\bt \coth r - \rd_{r} \bt) \sinh^{-2} r \abs{\slashed{\nb} w}^{2}\,\dh \\
& + \int_{\bbH^d} \left(-\frac{1}{2} \lap \rd_{r} \bt - \rho \lap (\bt \coth r) + 2 \rho^{2} \rd_{r} \bt \right) \abs{w}^{2} \,\dh \\
& - \Re \brk{i H_{\lot} w , Q w},
\end{align*}
where we have omitted the cumbersome cutoffs $\chi_{j}$. By the same splitting of $\tau - H$ and a simple integration by parts (see the proof of Lemma~\ref{l:mult}), which is also justified thanks to \eqref{eq:res-decay}, we also have
\begin{align*}
\Re \brk{(\tau - H) w , \rd_{r} \bt w}
= & \Re \brk{(\tau + \lap) w , \rd_{r} \bt w} - \Re \brk{H_{\lot} w , \rd_{r} \bt w} \\
= & - \int_{\bbH^d} \rd_{r} \bt \left( \abs{\nb w}^{2} - \rho^{2} \abs{w}^{2} \right)\,\dh\\
&+ (\tau - \rho^{2}) \int_{\bbH^d} \rd_{r} \bt \abs{w}^{2} \,\dh\\
&+ \frac{1}{2} \int_{\bbH^d} (\lap \rd_{r} \bt) \abs{w}^{2} \,\dh- \Re \brk{H_{\lot} w , \rd_{r} \bt w}.
\end{align*}
Combining the previous two identities, we obtain
\begin{align*}
&\Re \brk{(\tau - H) w , (-i) Q w + 2\rd_{r} \bt w}\\
&=  2 (\tau - \rho^{2}) \int_{\bbH^d} \rd_{r} \bt \abs{w}^{2}  \,\dh+ 2 \int_{\bbH^d} (\bt \coth r - \rd_{r} \bt) \sinh^{-2} r \abs{\slashed{\nb} w}^{2} \,\dh\\
& \quad+ \int_{\bbH^d} \left(\frac{1}{2} \lap \rd_{r} \bt - \rho \lap (\bt \coth r) + 2 \rho^{2} \rd_{r} \bt \right) \abs{w}^{2} \,\dh\\
& \quad- \Re \brk{H_{\lot} w , (-i)Q w + 2 \rd_{r} \bt w}.
\end{align*}
Simplifying the third line and rearranging terms, we arrive at the key identity
\EQ{ \label{eq:poly-decay-mult}
&2 (\tau - \rho)^{2} \int_{\bbH^d} \rd_{r} \bt \abs{w}^{2} \,\dh\\
 &= \Re \brk{(\tau - H) w , (-i) Q w + 2\rd_{r} \bt w}  
  - 2 \int_{\bbH^d} (\bt \coth r - \rd_{r} \bt) \sinh^{-2} r \abs{\slashed{\nb} w}^{2} \,\dh \\
& \quad - \int_{\bbH^d} \left(\frac{1}{2} \rd_{r}^{3} \bt + 2 \rho (\rho-1) (\bt \coth r - \rd_{r} \bt) \sinh^{-2} r \right) \abs{w}^{2} \,\dh\\
& \quad+ \Re \brk{H_{\lot} w , (-i)Q w + 2 \rd_{r} \bt w}.
}

We apply \eqref{eq:poly-decay-mult} to a solution $(\tau - H) w = 0$, so that the first term on the right-hand side vanishes. Our aim is to now choose $\bt$ adequately so that the right-hand side is bounded from above by the right-hand side of \eqref{eq:poly-decay-key}.
We choose
\begin{equation*}
	\rd_{r} \bt = \chi_{>r_{1}}(r) \chi_{<R_{1}}(r) r^{2N}, \quad \bt (r) = \int_{0}^{y} \rd_{r} \bt (y) \, \ud y.
\end{equation*}
Note that $\bt$ and its derivatives vanish near zero and are bounded. We fix $r_{1}$ sufficiently large (independent of $R_{1}$, but depending on $N$) so that
\begin{equation*}
	\bt \coth r - \rd_{r} \bt \geq 0 \quad \hbox{ on } \set{r > 2 r_{1}}.
\end{equation*}
Thus the second term on the right-hand side obeys
\begin{align*}
 -2 \int_{\bbH^d} (\bt \coth r - \rd_{r} \bt) \sinh^{-2} r \abs{\slashed{\nb} w}^{2}\,\dh
 \leq C_{N} \nrm{\slashed{\nb} w}_{L^{2}(r < 2 r_{1})}^{2}
\end{align*}
for some positive constant $C_{N} > 0$, which is okay. For the third term, we simply estimate
\begin{align*}
\abs{\int_{\bbH^d} \rd_{r}^{3} \bt \abs{w}^{2}\,\dh} \aleq_{N} & \nrm{\brk{r}^{N-1} w}_{L^{2}}^{2} \\
\abs{\int_{\bbH^d} (\bt \coth r - \rd_{r} \bt) \sinh^{-2} r \abs{w}^{2} \,\dh}\aleq_{N} & \nrm{\brk{r}^{N-1} w}_{L^{2}}^{2},
\end{align*}
where we used the exponential decay of $\sinh^{-2} r$ for the second estimate. Finally, for the last term on the right-hand side of \eqref{eq:poly-decay-mult}, from our decay assumptions on $\bsb, V$, we have 
\begin{align*}
\abs{\Re \brk{H_{\lot} w , (-i)Q w + 2 \rd_{r} \bt w}}
&\aleq \int_{\bbH^d} \brk{r}^{-3} \brk{r}^{2N+1} (\abs{\nb w}^{2} + \abs{w}^{2})\,\dh\\
&\aleq \nrm{\brk{r}^{N-1} (\nb w, w)}_{L^{2}}^{2},
\end{align*}
which is acceptable. \qedhere
\end{proof}

\pfstep{Step~2: Faster than polynomial decay implies compact support}
As mentioned in Section~\ref{ss:LE-H-outline}, in the presence of first order potentials this part of the proof is much more involved in the Euclidean case. Here the exponential decay of angular derivatives allows us to give a more uniform treatment of zero and first order potentials. More precisely, in the original proof for zero order potentials (see \cite{BorMar1, Donnelly1}), which is modeled after the Euclidean case,  one considers $(r^{2} L_{m})'$ instead of $(r^{4} L_{m})'$ (in the notation below) in the contradiction argument, and for first order potentials this yields problematic terms with weights of $m$ which cannot be treated. In the present setting we consider $(r^{4} L_{m})'$ which allows us to use the term $m(m+1) r^{-2}$ to overcome this difficulty. In the Euclidean case, this does \emph{not} work, as the angular derivatives always come with the unfavorable sign. Indeed, the argument for the Euclidean case is much more involved in the presence of a first order perturbation; compare \cite[Ch.~XIV, Thm.~14.7.2]{Hor2} with \cite[Ch.~XVII, Thm.~17.2.8]{Hor3}.

For $w_{m} = e^{\rho r} r^{m} v$, consider the quantity
\begin{equation*}
	L_{m}(r): = \nrm{\rd_{r} w_{m}}_{L^{2}(\bbS^{d-1})}^{2} - \frac{\nrm{\slashed{\nb} w_{m}}_{L^{2}(\bbS^{d-1})}^{2}}{\sinh^{2} r}  + \left(\kpp^{2} + \frac{m(m+1)}{r^{2}} \right) \nrm{w_{m}}_{L^{2}(\bbS^{d-1})}^{2}.
\end{equation*}
We claim that there exist large constants $m_{1}, R_{1}$ such that
\begin{equation} \label{eq:Lm-mono}
	(r^{4} L_{m})'(r) \geq 0 \quad \hbox{ provided that $m \geq m_{1}$ and $r \geq R_{1}$.} 
\end{equation}
Since $r^{m+2} v, r^{m+2} \nb v \in L^{2}$ by Step~1, it follows that $r^{4} L_{m} (r) \to 0$ as $r \to \infty$. Thus $L_{m}(r) \leq 0$ for sufficiently large $m \geq m_{1}$ and $r \geq R_{1}$. On the other hand,
\begin{equation*}
	L_{m}(r) \geq \frac{m(m+1)}{r^{2}} r^{2m} \nrm{e^{\rho r} v(r)}_{L^{2}(\bbS^{d-1})}^{2} - \frac{r^{2m}}{\sinh^{2} r} \nrm{e^{\rho r} \slashed{\nb} v(r)}_{L^{2}(\bbS^{d-1})}
\end{equation*}
so that if $\nrm{e^{\rho r} v(r)}_{L^{2}(\bbS^{d-1})}^{2} \neq 0$, then $L_{m}(r) > 0$ for a sufficiently large $m$. This contradiction shows that $v \equiv 0$ for $r \geq R_{1}$.

It remains to establish \eqref{eq:Lm-mono}. Note that
\begin{align*}
	(r^{4} L_{m})'
	= & 4r^{3} \nrm{\rd_{r} w_{m}}^{2}_{L^{2}(\bbS^{d-1})} - \left(\frac{r^{4}}{\sinh^{2} r} \right)' \nrm{\slashed{\nb} w_{m}}_{L^{2}(\bbS^{d-1})}^{2} \\
	&+ \left( 4 \kpp^{2} r^{3} + 2 m (m+1) r \right)  \nrm{w_{m}}_{L^{2}(\bbS^{d-1})}^{2} \\
	& + 2\Re \brk{\rd_{r}^{2} w_{m} + \sinh^{-2} r \slashed{\lap} w_{m} + (\kpp^{2} + \tfrac{m(m+1)}{r^{2}}) w_{m} , r^{4} \rd_{r} w_{m}}_{\bbS^{d-1}}
\end{align*}
A straightforward computation yields (recall that $\tau=\rho^2+\kpp^2)$
\begin{align*}
	\rd_{r}^{2} w_{m} + \sinh^{-2} r \slashed{\lap} w_{m} 
	=& e^{\rho r} r^{m} (\lap v + \tau v) + \left( \frac{2 m}{r} + 2 \rho (1 - \coth r) \right) \rd_{r} w_{m} \\ 
	& - \left(\kpp^{2} + \frac{m(m+1)}{r^{2}} + 2 \rho (1-\coth r) \left(\rho + \frac{m}{r} \right)\right) w_{m}.
\end{align*}
Thus,
\EQ{ \label{eq:Lm-mono-exp}
	&(r^{4} L_{m})'\\
	 &= (4m + 4) r^{3} \nrm{\rd_{r} w_{m}}_{L^{2}(\bbS^{d-1})}^{2}
	+ 2 \frac{r^{4}}{\sinh^{2} r} \left( \coth r - \frac{2}{r} \right) \nrm{\slashed{\nb} w_{m}}_{L^{2}(\bbS^{d-1})}^{2} \\
	&\quad+ \left( 4 \kpp^{2} r^{3} + 2m (m+1) r \right) \nrm{w_{m}}_{L^{2}(\bbS^{d-1})}^{2} \\
	&\quad + 4 \rho r^{4} (1-\coth r) \nrm{\rd_{r} w_{m}}_{L^{2}(\bbS^{d-1})}^{2}
	- 4 \rho \left( (1 - \coth r) (\rho r^{4} + m r^{3}) \right) \brk{w_{m} , \rd_{r} w_{m}}_{\bbS^{d-1}} \\
	&\quad + 2 r^{4} \brk{e^{\rho r} r^{m} H_{\lot} (e^{-\rho r} r^{-m} w_{m}) , \rd_{r} w_{m}}_{\bbS^{d-1}}.
}
Taking $r \geq R_{1}$ with $R_{1}$ sufficiently large, we may insure that the first three terms are bounded from below by
\begin{equation} \label{eq:Lm-mono-key}
	m r^{3} \nrm{\rd_{r} w_{m}}_{L^{2}(\bbS^{d-1})}^{2}
	+ \frac{r^{4}}{\sinh^{2} r} \nrm{\slashed{\nb} w_{m}}_{L^{2}(\bbS^{d-1})}^{2} 
	+ \left( \kpp^{2} r^{3} + m^{2} r \right) \nrm{w_{m}}_{L^{2}(\bbS^{d-1})}^{2}.
\end{equation}
Using the exponential decay of $1-\coth r$, it is not difficult to arrange (by taking $R_{1}$ large) so that the absolute value of the fourth line in \eqref{eq:Lm-mono-exp} is smaller than \eqref{eq:Lm-mono-key} for $r \geq R_{1}$. Finally, note that
\begin{align*}
& r^{4} \abs{\brk{e^{\rho r} r^{m} H_{\lot} (e^{-\rho r} r^{-m} w_{m}) , \rd_{r} w_{m}}_{\bbS^{d-1}}} \\
& \leq r^{4} \left(\abs{\bsb^{r}} \nrm{\rd_{r} w_{m}}_{L^{2}(\bbS^{d-1})} + (\abs{\bsb^{r}}(\rho + m r^{-1}) + \abs{V}) \nrm{w_{m}}_{L^{2}(\bbS^{d-1})} \right) \nrm{\rd_{r} w_{m}}_{L^{2}(\bbS^{d-1})} \\
& \phantom{\leq} + \frac{r^{4}}{\sinh^{2} r} \abs{\slashed{\bsb}} \nrm{\slashed{\nb} w_{m}}_{L^{2}(\bbS^{d-1})} \nrm{\rd_{r} w_{m}}_{L^{2}(\bbS^{d-1})}
\end{align*}
Applying Cauchy--Schwarz, and using the decay of $\bsb, V$, we may estimate the preceding line by \eqref{eq:Lm-mono-key} for sufficiently large $r \geq R_{1}$ and $m \geq m_{1}$, as desired. 

\pfstep{Step~3: Completion of proof}
At this point, we know that $v$ is a compactly supported solution to $(\tau - H) v = 0$, which is regular (say, $v \in H^{2}$).
Applying the standard unique continuation result for elliptic PDEs (see, for instance, the strong unique continuation theorem \cite[Thm.~17.2.6]{Hor3}), we conclude that $v = 0$ as desired.
\end{proof}

\subsection{Equivalent formulations of the threshold nonresonance condition} \label{ss:th-res}

We first establish \eqref{eq:no-th-res-H1} and \eqref{eq:no-th-res-LE} for $H = -\Delta$.
\begin{prop} \label{p:no-thr-res-lap}
Estimates \eqref{eq:no-th-res-H1} and \eqref{eq:no-th-res-LE} hold for $H = - \lap$.
\end{prop}
\begin{proof}
By density, it suffices to establish both estimates for $v \in C^{\infty}_{0}(\bbH^{d})$. 

\pfstep{Step~1.a: Proof of \eqref{eq:no-th-res-H1} in $d \geq 3$} 
Here, note that $\rho (\rho-1) = \frac{1}{4} (d-1)(d-3) \geq 0$. Therefore, by Lemma~\ref{l:mult} with $\gmm = 1$ and $\kpp = 0$, we already obtain
\begin{equation*}
	\int_{\bbH^d} \left(\abs{(\rd_{r} + \rho \coth r) v}^{2} + \frac{1}{\sinh^{2} r} \abs{\slashed{\nb} v}^{2}\right) \,\dh 
	\leq - \Re \brk{(\rho^{2} + \lap) v , v}
\end{equation*}
On the other hand, by Lemma~\ref{l:thr-hardy} with $\gmm = 1$ and $m = \frac{1}{2}$, we have
\begin{equation*}
	\frac{1}{4} \int_{\bbH^d} r^{-2} \abs{v}^{2} \,\dh 
	\leq - \Re \brk{(\rho^{2} + \lap) v , v}.
\end{equation*}
Combining the previous two estimates, and controlling $\rho (\coth r - 1) v$ by the preceding inequality, we see that
\begin{equation*}
\nrm{v}_{H^{1}_{thr}}^{2} \aleq \abs{\Re \brk{(\rho^{2} + \lap) v , v}}.
\end{equation*}
The desired estimate now follows by duality and Cauchy--Schwarz.

\pfstep{Step~1.b: Proof of \eqref{eq:no-th-res-H1} in $d = 2$}
The previous proof breaks down as $\rho (\rho - 1)$ now changes sign; hence we need to be more careful. 
Combining Lemma~\ref{l:mult} (with $\gmm =1$, $\kpp = 0$) and Lemma~\ref{l:thr-hardy} (with $\gmm = 1$, $m = \rho = \frac{1}{2}$), we have
\EQ{\label{eq:mult-hp}
&- \Re \brk{(\rho^{2} + \lap) v , v} \\
&= \int_{\bbH^2} \left(\abs{(\rd_{r} + \rho (\coth r - r^{-1})) v}^{2} + \frac{1}{\sinh^{2} r} \abs{\slashed{\nb} v}^{2}\right)  \,\dh \\
&\quad
+ \int_{\bbH^2} \frac{1}{4} \left( r^{-2} - \sinh^{-2} r  \right)\abs{v}^{2}  \,\dh \\
& \geq c_{0} \int_{\bbH^2}  \brk{r}^{-2} \abs{v}^{2} \,\dh ,
}
for some universal constant $c_{0} > 0$. On the other hand, by a similar proof as in Lemma~\ref{l:mult}, we may also prove
\EQ{\label{eq:mult-2d}
&- \Re \brk{(\rho^{2} + \lap) v ,  v} \\
&= \int_{\bbH^2} \left(\abs{(\rd_{r} + \rho) v}^{2} + \frac{1}{\sinh^{2} r} \abs{\slashed{\nb} v}^{2} + 2 \rho^{2} (\coth r - 1) \abs{v}^{2} \right) \,\dh  \\
& \geq \int_{\bbH^2} \left(\abs{(\rd_{r} + \rho) v}^{2} + \frac{1}{\sinh^{2} r} \abs{\slashed{\nb} v}^{2}\right)  \,\dh .
}
The rest of the proof is as before.

\pfstep{Step~2.a: Proof of \eqref{eq:no-th-res-LE} up to an interior bound}
Using the identities in Section~\ref{ss:LE-compute}, we first prove \eqref{eq:no-th-res-LE} up to a lower order error on a compact region: For some $R > 0$, we have
\begin{equation} \label{eq:no-th-res-LE-ext}
	\nrm{v}_{LE_{thr}^{1}} \aleq \nrm{(\rho^{2} + \lap) v}_{LE_{thr}^{\ast}} + \nrm{v}_{L^{2}(\set{r \leq R})}.
\end{equation}
We borrow the previous argument in Section~\ref{s:low}. Let $\alpha$ be as in Definition~\ref{d:alpha} 
and define $\bt \equiv \bt^{(\alp)}$ and $Q\equiv Q^{(\alpha)}$ according to Definition~\ref{d:beta_low}. Keeping only the first term on the left-hand side of \eqref{equ:additional_bounds_positive_commutator_from_HP}, we have
\begin{align*}
	\int_{\bbH^d} \frac{\rd_{r} \bt}{\brk{r}^{2}} \abs{v}^{2}  \,\dh
	\aleq \Re \brk{i (\rho^{2} + \lap) v , Q v} + \nrm{v}_{L^{2} \set{r \leq R}}^{2}.
\end{align*}
Moreover, applying Lemma~\ref{l:mult} (with $\gmm = \rd_{r} \bt^{(\alp)}$ and $\kpp = 0$) to the second term on the left-hand side of \eqref{equ:additional_bounds_positive_commutator_from_HP}, and recalling Lemma~\ref{l:technical_lemma_low_freq}, we have
\begin{align*}
	& \int_{\bbH^d} \rd_{r} \bt \left( \abs{(\rd_{r} + \rho) v}^{2} + \frac{1}{\sinh^{2} r} \abs{\slashed{\nb} v}^{2} \right) \,\dh\\
	& \aleq \Re \brk{i (\rho^{2} + \lap) v , Q v} + \nrm{v}_{L^{2} \set{r \leq R}}^{2} + \int_{\bbH^d} \rd_{r} \bt \frac{1}{\brk{r}^{2}} \abs{v}^{2}\,\dh.
\end{align*}
Combining these two bounds, and taking the supremum over slowly varying sequences in $\calA$ in Definition~\ref{d:alpha} of $\alpha$, we obtain
\begin{equation*}
	\nrm{v}_{LE_{thr}^{1}}^{2} \aleq \sup_{\alpha} \abs{\Re \brk{i (\rho^{2} + \lap) v , Q^{(\alp)} v}} + \nrm{v}_{L^{2} \set{r \leq R}}^{2}.
\end{equation*}
By duality, we estimate
\begin{align*}
\abs{\Re \brk{i (\rho^{2} + \lap) v , Q^{(\alp)} v}}
&\aleq \nrm{(\rho^{2} + \lap) v}_{LE_{thr}^{\ast}} \nrm{Q^{(\alp)} v}_{LE_{thr}} \\
&\aleq \nrm{(\rho^{2} + \lap) v}_{LE_{thr}^{\ast}} \nrm{v}_{LE_{thr}^{1}},
\end{align*}
which concludes the proof.

\pfstep{Step~2.b: Proof of an interior bound}
Finally, we complete the proof of \eqref{eq:no-th-res-LE} by proving an interior bound: For any $R > 0$, we have
\begin{equation} \label{eq:no-th-res-LE-int}
	\nrm{v}_{L^{2}(\set{r \leq R})} \aleq_{R} \nrm{(\rho^{2} + \lap) v}_{LE_{thr}^{\ast}}.
\end{equation}
Let $\boldsymbol{\Gmm}$ be a fundamental solution to $\rho^{2} + \lap$. By translation and rotation symmetries, we may assume that
\begin{equation*}
	\boldsymbol{\Gmm}(x, y) = \Gmm(\bfd(x, y))
\end{equation*}
for some radial function $\Gmm$. Using the ODE satisfied by $\Gmm(r)$, we see\footnote{This property is equivalent to the threshold nonresonance condition for the radial Laplacian.} that $\Gmm(r) = c_{\Gmm} e^{-\rho r} + o(e^{-\rho r})$ as $r \to \infty$ for some $c_{\Gmm} \neq 0$. For $G \in C^{\infty}_{0}(\bbH^{d})$, consider
\begin{equation*}
	v (x)= \int_{\bbH^d} \boldsymbol{\Gmm}(x, y) G(y) \, \dh(y).
\end{equation*}
By the asymptotics of $\boldsymbol{\Gmm}$, it follows that $v$ is a solution to $(\rho^{2} + \lap) v = G$ with $v \in H_{thr}^{1}$. By Step~1 and the Lax--Milgram theorem, such a $v$ is unique, so it suffices to prove \eqref{eq:no-th-res-LE-int} for this $v$.

For any $G$ supported in $A_{j}$ for large $j$, so that $2^{j} \gg R$, and for $x$ in $\{r\leq R\}$, we have 
\begin{align*}
\abs{\int_{\bbH^d} \boldsymbol{\Gmm}(x, y) G(y) \, \dh(y)}
\aleq & \int_{\set{r \simeq 2^{j}}} e^{-\frac{d-1}{2} (r-C R)}  \nrm{G(r \omg)}_{L^{1}(\bbS^{d-1})} \sinh^{d-1} r \, \ud r \\
\aleq & e^{\frac{d-1}{2} C R} 2^{\frac{1}{2} j} \nrm{G(y)}_{L^{2}(A_{j})}.
\end{align*}
Integrating this pointwise bound on $\set{r \leq R}$ and combining with the $H_{thr}^{-1} \to H_{thr}^{1}$ bound \eqref{eq:no-th-res-H1} for $G$ supported in $A_{j}$ in the remaining range of $j$, and then adding up in $j$, it follows that
\begin{equation*}
	\nrm{\int_{\bbH_d} \boldsymbol{\Gmm}(x, y) G(y) \, \dh(y)}_{L^{2}(\set{r \leq R})} \aleq_{R} \nrm{G}_{LE_{thr}^{\ast}}. \qedhere
\end{equation*}
\end{proof}
\begin{rem} \label{rem:alt-lap}
In the case $d \geq 3$, \eqref{eq:no-th-res-LE} can be proved alternatively using Lemma~\ref{l:pos-comm} in the case $\kpp = 0$. We leave the details to the interested reader.
\end{rem}

By a minor modification of the preceding argument, Corollary~\ref{c:no-th-res} follows as well.
\begin{proof}[Proof of Corollary~\ref{c:no-th-res}]
The idea is to repeat Case~1 in the proof of Proposition~\ref{p:no-thr-res-lap} with covariant derivatives $\bfD_{\mu} = \rd_{\mu} + i \bsb_{\mu}$. We only sketch the tricky case $d = 2$,  where $\rho=\frac{1}{2}$, and leave the rest to the reader. 

Analogous to \eqref{eq:mult-2d}, we have
\EQ{
\int_{\bbH^d} \left(\abs{(\bfD_{r} + \rho) v}^{2} + \frac{1}{\sinh^{2} r} \abs{\slashed{\bfD} v}^{2}\right)\,\dh
&\leq - \Re \brk{(\rho^{2} + \lap_{\bsb}) v , v}\\
& \leq  - \Re \brk{(\rho^{2} + \lap_{\bsb} - V) v , v}.
}
On the other hand, using the diamagnetic inequality $\abs{\nb \abs{v}} \leq \abs{\bfD v}$ and repeating the proof of \eqref{eq:mult-hp}, we obtain
\begin{equation*}
	c_{0} \int_{\bbH^d} \frac{1}{\brk{r}^{2}} \abs{v}^{2}\,\dh \leq - \Re \brk{(\rho^{2} + \lap_{\bsb}) v ,  v} \leq -\Re \brk{(\rho^{2} + \lap_{\bsb} - V) v , v}.
\end{equation*}
Splitting $(\bfD_{r}, \slashed{\bfD}) = (\rd_{r}, \slashed{\nb}) + (\bsb_{r}, \slashed{\bsb})$ and using the preceding Hardy--Poincar\'e type bound to control the contribution of the last term, it follows that
\begin{equation*}
	\nrm{v}_{H^{1}_{thr}}^{2} \leq C_{0} \abs{\brk{(\rho^{2} + \lap_{\bsb} - V) v , v}},
\end{equation*}
 where $C_{0}$ depends only on $c_{0}$ and $\nrm{\brk{r} \bsb}_{L^{\infty}}$. \qedhere
\end{proof}

Finally, we establish Proposition~\ref{p:no-th-res}.
\begin{proof}[Proof of Proposition~\ref{p:no-th-res}]
Key to the proof are the following results concerning small perturbations of $\lap$. Let $\chi$ be a smooth nondecreasing function supported in $\set{r \geq 1}$, which equals $1$ on $\set{r \geq 2}$. Consider
\begin{equation*}
	H_{>R} = -\lap + \chi_{>R} H_{\lot}.
\end{equation*}
By our decay assumptions on $\bsb$ and $V$, we have
\begin{equation*}
\nrm{\chi_{>R} H_{\lot} v}_{H_{thr}^{-1}} \aleq o_{R \to \infty}(1) \nrm{v}_{H_{thr}^{1}}, \quad
\nrm{\chi_{>R} H_{\lot} v}_{LE_{thr}^{\ast}} \aleq o_{R \to \infty}(1) \nrm{v}_{LE_{thr}^{1}}.
\end{equation*}
Combined with Proposition~\ref{p:no-thr-res-lap}, we see that, for large enough $R$,
\begin{equation} \label{eq:no-thr-res-lap-pert}
	\nrm{v}_{H_{thr}^{1}} \aleq \nrm{(\rho^{2} - H_{>R}) v}_{H_{thr}^{-1}}, \quad 
	\nrm{v}_{LE_{thr}^{1}} \aleq \nrm{(\rho^{2} - H_{>R}) v}_{LE_{thr}^{\ast}}.
\end{equation}
Henceforth, we fix $R > 0$ and suppress the dependence of constants on $R$.

\pfstep{(1) $\imp$ (2)} 
This argument is a variant of the proof of the Fredholm alternative theorem, as well as that of Proposition~\ref{p:LE-res}. We begin by writing
\begin{equation*}
(\rho^{2} - H_{>R}) v = (1-\chi_{<R}) H_{\lot} v + (\rho^{2} - H) v.
\end{equation*}
Note that
\begin{equation*}
\nrm{(1-\chi_{<R}) H_{\lot} v}_{H_{thr}^{-1}} \aleq \nrm{v}_{L^{2}(\set{r \leq 10 R})}.
\end{equation*}
Therefore, combined with \eqref{eq:no-thr-res-lap-pert}, it follows that
\begin{equation} \label{eq:no-thr-res-H1-ext}
	\nrm{v}_{H_{thr}^{1}} \aleq \nrm{(\rho^{2} - H) v}_{H_{thr}^{-1}} + \nrm{v}_{L^{2}(\set{r \leq 10 R})}
\end{equation}

Now assume, for the purpose of contradiction, that \eqref{eq:no-th-res-H1} fails. Then there exists a sequence $v_{n} \in H_{thr}^{1}$ such that
\begin{equation*}
	\nrm{v_{n}}_{H_{thr}^{1}} = 1, \quad \nrm{(\rho^{2} - H) v_{n}}_{H_{thr}^{-1}} \to 0.
\end{equation*}
In view of \eqref{eq:no-thr-res-H1-ext}, we see that $\nrm{v_{n}}_{L^{2}(r \leq 10 R)} \ageq 1$. Passing to a subsequence, we obtain a weak limit $v_{\infty} \in H_{thr}^{1}$ such that $(\rho^{2} - H) v_{\infty} = 0$ and $\nrm{v_{\infty}}_{L^{2}(r \leq 10 R)} \ageq 1$; in particular, $v_{\infty}$ is nonzero. Since $v_{\infty} \in H^{1}_{thr}$ implies \eqref{eq:th-res} and $v_\infty\in\tilLE_0$, $v_{\infty}$ is a threshold resonance, which is a contradiction.

\pfstep{(2) $\imp$ (3)}
Using \eqref{eq:no-thr-res-lap-pert}, we may find $v'$ such that $(\rho^{2} - H_{>R})  v' = (\rho^{2} - H) v$ and 
\begin{equation*}
	\nrm{v'}_{LE_{thr}^{1}} \aleq \nrm{(\rho^{2} - H) v}_{LE_{thr}^{\ast}}.
\end{equation*}
Rearranging terms, we see that
\begin{equation*}
	(\rho^{2} - H) (v - v') = (1 - \chi_{>R}) H_{\lot} v'.
\end{equation*}
As the right-hand side is supported in $\set{r \aleq R}$, it follows from \eqref{eq:no-th-res-H1} that
\begin{equation*}
	\nrm{v - v'}_{H^{1}_{thr}} \aleq \nrm{(1 - \chi_{>R}) H_{\lot} v'}_{H^{-1}_{thr}} \aleq_{R} \nrm{v'}_{LE^{1}_{thr}} \aleq \nrm{(\rho^{2} - H) v}_{LE_{thr}^{\ast}}. 
\end{equation*}
As $H^{1}_{thr} \hookrightarrow LE_{thr}^{1}$, the desired estimate \eqref{eq:no-th-res-LE} now follows.

\pfstep{(3) $\imp$ (1)}
Let $w \in \tilLE_{0}$ satisfy $(\rho^{2} - H) w = 0$ and \eqref{eq:th-res}.
We claim that $w$ also satisfies
\EQ{ \label{eq:th-res-0}
\| r^{-\frac{3}{2}} w \|_{L^2(r \simeq 2^j)} \to 0 \mas j \to \infty.
}
Then, using \eqref{eq:th-res}, we may estimate
\begin{equation*}
	\nrm{\brk{r}^{\frac{1}{2}} [H-\rho^{2}, \chi_{\leq R}] w}_{L^{2}} \to 0 \quad \hbox{ as } R \to \infty,
\end{equation*}
where $\chi_{\leq R} = 1 - \chi_{> R}$. On the other hand, by \eqref{eq:no-th-res-LE}, we would have
\begin{equation*}
	\nrm{\chi_{\leq R} w}_{LE_{thr}^{1}} \aleq \nrm{(H-\rho^{2}) \chi_{\leq R} w}_{LE^{\ast}_{thr}} \to 0.
\end{equation*}
Thus, it would follow that $w = 0$. 

To establish \eqref{eq:th-res-0}, we apply Lemma~\ref{l:thr-hardy} with $m = 1$ and $\gmm = \chi_{>R} \frac{\alp_{(j)}(r)}{r}$, where $\alp_{(j)} = \alp_{(j)}(r)$ is a slowly varying function satisfying
\begin{equation} \label{alp-j}
	\alp_{(j)}(r) \simeq \min \set{r 2^{-j}, r^{-1} 2^{j}}^{\dlt}, \quad
	r \abs{(\log \alp_{(j)})'} \aleq \dlt
\end{equation}
for some small parameter $\dlt > 0$. Then
\begin{align*}
	\int_{\bbH^d} \chi_{>R} \frac{\alp_{(j)}(r)}{r^{3}} \abs{w}^{2}\,\dh&
	\aleq \int_{\bbH^d} \chi_{>R} \frac{\alp_{(j)}(r)}{r} \abs{(\rd_{r} + \rho \coth r) w}^{2} \,\dh\\
	&\quad+ \int_{\bbH^d} \chi_{>R}'(r) \frac{\alp_{(j)}(r)}{r^{2}} \abs{w}^{2}\,\dh,
\end{align*}
provided that $\dlt$ is sufficiently small. Note that $\alp_{(j)}$ localizes the integrals to $A_{j}$. Since $w \in \tilLE_{0}$ and \eqref{eq:th-res} hold, the desired decay \eqref{eq:th-res-0} as $j \to \infty$ follows. \qedhere
\end{proof}

\subsection{Perturbation of symmetric, stationary case} \label{ss:LE2-pf}

The aim of this subsection is to prove Theorem~\ref{t:LE2}, i.e., to extend the validity of the global-in-time local smoothing estimate (without any lower order errors on a compact set) to small but possibly non-symmetric (i.e., complex-valued) and non-stationary (i.e., time-dependent) perturbations of a symmetric, stationary Hamiltonian $H_{\stat}$ considered in Theorem~\ref{t:LE-H}. We consider only the simple case when $H_{\stat}$ has neither a threshold resonance nor any eigenvalues in $(-\infty, \rho^2]$, and leave the more intricate general case to future investigation. 

\begin{proof}[Proof of Theorem~\ref{t:LE2}]
For sufficiently nice $u$ and $F = (- i \rd_{t} + H) u$, our goal is to derive a bound of the form
\begin{equation} \label{eq:LE-H-pert-goal}
	\nrm{u}_{\tilLE} \aleq \nrm{u_{0}}_{L^{2}} + \nrm{F}_{\tilLE^{\ast}} + \dlt(\kpp) \nrm{u}_{\tilLE} 
\end{equation}
where $\dlt(\kpp) = o(\kpp)$, and the implicit constant is allowed to depend on $H_{\stat}$ (in particular, the bounds in \eqref{eq:decay_assumptions} and the constant in Theorem~\ref{t:LE-H}). The desired conclusion would then easily follow.

We introduce a small parameter $0 < s_{0} < \frac{1}{8}$ and commute the equation with $\tilP_{\geq s_{0}}$. We claim that
\begin{equation} \label{eq:LE-s0}
	\nrm{\tilP_{\geq s_{0}} u}_{\tilLE} \aleq \nrm{\tilP_{\geq s_{0}} u_{0}}_{L^{2}} + \nrm{\tilP_{\geq s_{0}} F}_{\tilLE^{\ast}} + (s_{0}^{\frac{1}{2}} + s_{0}^{-1} \kpp) \nrm{u}_{\tilLE}.
\end{equation}
To prove this claim, we write
\begin{equation*}
	(-i \rd_{t} + H_{\stat}) \tilP_{\geq s_{0}} u = \tilP_{\geq s_{0}} F - [\tilP_{\geq s_{0}}, H_{\stat}] u - \tilP_{\geq s_{0}} H_{\pert} u
\end{equation*}
and then apply Theorem~\ref{t:LE-H}. We use \eqref{eq:comm-stat} and \eqref{eq:smooth-pert} to estimate the $\tilLE^{\ast}$ norm of the second and third terms on the right-hand side, respectively. 

To conclude the proof, we take $\kpp$ sufficiently small and apply Theorem~\ref{t:LE1} to the operator $H = H_{\stat} + H_{\pert}$, which implies
\begin{equation}\label{eq:complexLEtemp1}
	\nrm{u}_{\tilLE} \aleq \nrm{u_{0}}_{L^{2}} + \nrm{F}_{\tilLE^{\ast}} + \nrm{u}_{L^{2}(\bbR\times\set{r \leq R})}.
\end{equation}

To treat the last term on the right-hand side of \eqref{eq:complexLEtemp1}, we split $u = \int_{0}^{s_{0}} \tilP_{s} u \, \ds + \tilP_{\geq s_{0}} u$. We estimate the first term simply as follows:
\begin{align*}
	\nrm{\int_{0}^{s_{0}} \tilP_{s} u \, \ds}_{L^{2}(\bbR\times\set{r \leq R})}
	& \aleq \int_{0}^{s_{0}} \nrm{\tilP_{s} u}_{L^{2}(\bbR\times\set{r \leq R})} \ds \\
	& \aleq_{R} \int_{0}^{s_{0}} s^{\frac{1}{4}} s^{-\frac{1}{4}} \nrm{\tilP_{s} u}_{\tilLE_{s}} \ds \\
	& \aleq s_{0}^{\frac{1}{4}} \nrm{u}_{\tilLE}.
\end{align*}
For the latter term, we use Lemma~\ref{l:wL2LE} to bound
\begin{equation*}
	\nrm{\tilP_{\geq s_{0}} u}_{L^{2}(\bbR\times\set{r \leq R})} \aleq_{R} \nrm{\tilP_{\geq s_{0}} u}_{\tilLE}.
\end{equation*}
Below, we will omit the dependence of constants on $R$, following the convention set at the beginning of the proof. Combining these estimates with \eqref{eq:LE-s0}, we arrive at
\begin{equation*}
	\nrm{u}_{\tilLE} \aleq \nrm{u_{0}}_{L^{2}} + \nrm{\tilP_{\geq s_{0}} u_{0}}_{L^{2}} + \nrm{F}_{\tilLE^{\ast}} + \nrm{\tilP_{\geq s_{0}} F}_{\tilLE^{\ast}} + (s_{0}^{\frac{1}{4}} + s_{0}^{\frac{1}{2}} + s_{0}^{-1} \kpp) \nrm{u}_{\tilLE}.
\end{equation*}
Since $\tilP_{\geq s_{0}}$ is uniformly bounded on $L^{2}$ and $\tilLE^{\ast}$ (see \eqref{eq:LEs-heat}), the proof of \eqref{eq:LE-H-pert-goal} is complete by choosing $s_0$ small and assuming that $\kappa$ is such that $s_0^{-1}\kappa$ is sufficiently small. \qedhere
\end{proof}

\section{Proof of Corollaries~\ref{c:LE1} and \ref{c:Strichartz}}\label{s:cors}

\subsection{Local smoothing for the projections $\Pea_s$ and $\Pea_{\geq s}$}

We prove Corollary~\ref{c:LE1}.

\begin{proof}[Proof of Corollary~\ref{c:LE1}]
 The corollary follows immediately from Theorem~\ref{t:LE-H} and Theorem~\ref{t:LE2}, upon showing that there exists an absolute constant $C \geq 1$ such that
 \begin{equation} \label{equ:relation_LE_widetilde_LE}
  \| u \|_{\LE} \leq C \| u \|_{\tilLE}
 \end{equation}
 and 
 \begin{equation} \label{equ:relation_LE_ast_widetilde_LE_ast}
  \| F \|_{\tilLE^\ast} \leq C \| F \|_{\LE^\ast}.
 \end{equation}
First note that the estimate~\eqref{equ:relation_LE_ast_widetilde_LE_ast} follows by duality from the estimate~\eqref{equ:relation_LE_widetilde_LE} because 
 \begin{align*}
  \| F \|_{\tilLE^\ast} &\simeq \sup_{0 \neq \|u\|_{\tilLE} < \infty} \langle F, \|u\|_{\tilLE}^{-1}\,u \rangle_{t,x} \lesssim \sup_{0 \neq \|u\|_{\tilLE} < \infty} \langle F, \|u\|_{\LE}^{-1} \,u \rangle_{t,x} \\
  &\lesssim \sup_{0 \neq \|u\|_{\LE} < \infty} \langle F, \|u\|_{\LE}^{-1}\,u \rangle_{t,x} \simeq C \|F\|_{\LE^\ast}.
 \end{align*}
 We therefore only have to prove \eqref{equ:relation_LE_widetilde_LE}. Since
 \begin{align*}
  \|u\|_{\LE}^2 &= \int_{\frac{1}{8}}^4 \| P_{\geq s} u \|_{\tilLE_\low}^2 \, \ds + \int_0^{\frac{1}{2}} s^{-\frac{1}{2}} \| P_s u \|_{\tilLE_s}^2 \, \ds \\ 
  &= \int_{\frac{1}{8}}^4 \bigl\| e^{-s \rho^2} \tilde{P}_{\geq s} u \bigr\|_{\tilLE_\low}^2 \, \ds + \int_0^{\frac{1}{2}} s^{-\frac{1}{2}} \bigl\| e^{-s\rho^2} \tilP_s u + s \rho^2 e^{-s\rho^2} \tilP_{\geq s} u \bigr\|_{\tilLE_s}^2 \, \ds \\
  &\lesssim \int_{\frac{1}{8}}^4 \| \tilP_{\geq s} u \|_{\tilLE_\low}^2 \, \ds + \int_0^{\frac{1}{2}} s^{-\frac{1}{2}} \bigl\| \tilP_s u \bigr\|_{\tilLE_s}^2 + \int_0^{\frac{1}{2}} s^{\frac{3}{2}} \bigl\| \tilP_{\geq s} u \|_{\tilLE_s}^2 \, \ds,
 \end{align*}
 it remains to suitably estimate the last integral on the right-hand side. We further decompose this last integral and use the fact that for any $s\leq 1$
\begin{align}\label{eq:tilLEstilLElowrelation}
\begin{split}
\|v\|_{\tilLE_s}\lesssim s^{-\frac{1}{4}}\|v\|_{\tilLE_\low},
\end{split}
\end{align}
 to find that
 \begin{equation} \label{equ:last_integral_to_massage}
  \begin{aligned}
   \int_0^{\frac{1}{2}} s^{\frac{3}{2}} \bigl\| \tilP_{\geq s} u \|_{\tilLE_s}^2 \, \ds &\leq \int_0^{\frac{1}{8}} s^{\frac{3}{2}} \bigl\| \tilP_{\geq s} u \|_{\tilLE_s}^2 \, \ds + \int_{\frac{1}{8}}^{\frac{1}{2}}s^{\frac{3}{2}} \bigl\| \tilP_{\geq s} u \|_{\tilLE_s}^2 \, \ds. \\
   &\lesssim \int_0^{\frac{1}{8}} s^{\frac{3}{2}} \bigl\| \tilP_{\geq s} u \|_{\tilLE_s}^2 \, \ds + \int_{\frac{1}{8}}^{\frac{1}{2}} s^{\frac{3}{2}} s^{-\frac{1}{2}} \bigl\| \tilP_{\geq s} u \|_{\tilLE_\low}^2 \, \ds \\
   &\lesssim \int_0^{\frac{1}{8}} s^{\frac{3}{2}} \bigl\| \tilP_{\geq s} u \|_{\tilLE_s}^2 \, \ds + \int_{\frac{1}{8}}^{4} \bigl\| \tilP_{\geq s} u \|_{\tilLE_\low}^2 \, \ds.
  \end{aligned}
 \end{equation}
 The second integral on the right-hand side is now already of the right form and it remains to estimate the first integral on the right-hand side. To this end we write for any $0 < s \leq \frac{1}{8}$ and  any $\frac{1}{8} \leq s_0 \leq \frac{1}{2}$,
 \begin{equation} \label{equ:decompose_tildePs}
  \tilP_{\geq s} u = \int_s^{s_0} \tilP_{s'} u \, \frac{\ud s'}{s'} + \tilP_{\geq s_0} u
 \end{equation}
 and then average over~\eqref{equ:decompose_tildePs} to obtain that
 \begin{equation} \label{equ:averaged_tildeP}
  \tilP_{\geq s} u 
  = \frac{1}{\log(4)} \biggl( \int_{\frac{1}{8}}^{\frac{1}{2}} \int_s^{s_0} \tilP_{s'} u \, \frac{\ud s'}{s'} \, \frac{\ud s_0}{s_0} + \int_{\frac{1}{8}}^{\frac{1}{2}}  \tilP_{\geq s_0} u \, \frac{\ud s_0}{s_0} \biggr).
 \end{equation}
 Inserting \eqref{equ:averaged_tildeP} into the first integral on the right-hand side of~\eqref{equ:last_integral_to_massage} we find that
 \begin{align*}
  \int_0^{\frac{1}{8}} s^{\frac{3}{2}} \bigl\| \tilP_{\geq s} u \|_{\tilLE_s}^2 \, \ds &\lesssim \int_0^{\frac{1}{8}} s^{\frac{3}{2}} \biggl( \int_{\frac{1}{8}}^{\frac{1}{2}} \int_s^{s_0} \bigl\| \tilP_{s'} u \bigr\|_{\tilLE_s} \, \frac{\ud s'}{s'} \, \frac{\ud s_0}{s_0} \biggr)^2 \, \ds \\
  &\quad+ \int_0^{\frac{1}{8}} s^{\frac{3}{2}} \biggl( \int_{\frac{1}{8}}^{\frac{1}{2}} \bigl\| \tilP_{\geq s_0} u \bigr\|_{\tilLE_s} \, \frac{\ud s_0}{s_0} \biggr)^2 \, \ds =: I + II.
 \end{align*}
 To bound the term $I$ we first trivially raise the upper integration endpoint $s_0$ to $\frac{1}{2}$ in the inner most integral in order to integrate out the $\frac{ds_0}{s_0}$ integral
 \begin{align*}
  I \leq \int_0^{\frac{1}{8}} s^{\frac{3}{2}} \biggl( \int_{\frac{1}{8}}^{\frac{1}{2}} \int_s^{\frac{1}{2}} \bigl\| \tilP_{s'} u \bigr\|_{\tilLE_s} \, \frac{\ud s'}{s'} \, \frac{\ud s_0}{s_0} \biggr)^2 \, \ds \lesssim \int_0^{\frac{1}{8}} s^{\frac{3}{2}} \biggl( \int_s^{\frac{1}{2}} \bigl\| \tilP_{s'} u \bigr\|_{\tilLE_s} \, \frac{\ud s'}{s'} \biggr)^2 \, \ds.
 \end{align*}
 Then using the relation
\begin{align*}
\begin{split}
\|v\|_{\tilLE_s}\lesssim \Bigl( \frac{s'}{s} \Bigr)^{\frac{1}{4}}\|v\|_{\tilLE_{s'}},\qquad s\leq s',
\end{split}
\end{align*}
 we obtain 
 \begin{align*}
  I&\lesssim \int_0^{\frac{1}{8}} s^{\frac{3}{2}} \biggl( \int_s^{\frac{1}{2}} \Bigl( \frac{s'}{s} \Bigr)^{\frac{1}{4}} \bigl\| \tilP_{s'} u \bigr\|_{\tilLE_{s'}} \, \frac{\ud s'}{s'} \biggr)^2 \, \ds \\
  &\simeq \int_0^{\frac{1}{8}} s \biggl( \int_s^{\frac{1}{2}} \Bigl( (s')^{\frac{1}{2}} (s')^{-\frac{1}{4}} \bigl\| \tilP_{s'} u \bigr\|_{\tilLE_{s'}} \, \frac{\ud s'}{s'} \biggr)^2 \, \ds \\
  &\lesssim \int_0^{\frac{1}{2}} (s')^{-\frac{1}{2}} \bigl\| \tilP_{s'} u \bigr\|_{\tilLE_{s'}}^2 \, \frac{\ud s'}{s'},
 \end{align*}
 which is of the desired form. Here the last estimate just follows by Cauchy-Schwarz and by trivially estimating the integrals. 
 To bound the term $II$ we use \eqref{eq:tilLEstilLElowrelation} and Cauchy-Schwarz to obtain that
 \begin{align*}
  II 
  &\lesssim \int_0^{\frac{1}{8}} s^{\frac{3}{2}} \biggl( \int_{\frac{1}{8}}^{\frac{1}{2}} s^{-\frac{1}{4}} \bigl\| \tilP_{\geq s_0} u \bigr\|_{\tilLE_\low} \, \frac{\ud s_0}{s_0} \biggr)^2 \, \ds \\
  &\simeq \int_0^{\frac{1}{8}} s \biggl( \int_{\frac{1}{8}}^{\frac{1}{2}} \bigl\| \tilP_{\geq s_0} u \bigr\|_{\tilLE_\low} \, \frac{\ud s_0}{s_0} \biggr)^2 \, \ds \\
  &\lesssim \int_{\frac{1}{8}}^{\frac{1}{2}} \bigl\| \tilP_{\geq s_0} u \bigr\|_{\tilLE_\low}^2 \, \frac{\ud s_0}{s_0},
 \end{align*}
 which is again of the desired form. Putting all of the above estimates together concludes the proof of the corollary.
 \end{proof}

\subsection{Strichartz estimates}

In this subsection we establish Corollary~\ref{c:Strichartz}. To this end we follow closely the argument in~\cite[Theorem 1.22]{MMT}, see also \cite{Tat08ajm, RodS}. In the proof we will invoke the existing Strichartz estimates for the unperturbed Schr\"odinger equation on hyperbolic space from \cite{AP09, B07, IS}. 

\begin{proof}[Proof of Corollary~\ref{c:Strichartz}]
Since $P_c$ commutes with the equation, redefining $u$ to be $P_cu$ and $F$ to be $P_cF$ we may drop the projections $P_c$ in the proof and assume that there are no eigenvalues in $(-\infty,\rho^2]$. Let $S$ denote the free Schr\"odinger operator 
\begin{align*}
\begin{split}
 S=-i\partial_t-\Delta 
\end{split}
\end{align*}
and let us write our operator as
\begin{align*}
\begin{split}
S_\pert=-i\partial_t+H=S+H_{\lot}, 
\end{split}
\end{align*}
where as usual $H_\lot$ denotes the lower order terms. We use the notation $S^{-1}$ to denote the solution to the free inhomogeneous equation with vanishing initial data, that is,
\begin{align}\label{eq: S-1 def}
\begin{split}
&(S^{-1}f)(t)=i\int_0^te^{i(t-s)\Delta}f(s)ds,\\
 &S (S^{-1} f)=f,\qquad (S^{-1}f)(0)=0. 
\end{split}
\end{align}
Suppose
\begin{align*}
\begin{split}
S_\pert u = F. 
\end{split}
\end{align*}
We want to show 
\begin{align}\label{eq: Strich}
\begin{split}
 \|u\|_{L_t^{p_1}L_x^{q_1}}\lesssim \|u(0)\|_{L^2_x}+\|F\|_{L_t^{p'_2}L_x^{q'_2}} 
\end{split}
\end{align}
for any pair of Strichartz exponents $(p_1,q_1)$ and $(p_2,q_2)$, excluding the endpoint cases. First we prove this estimate assuming the bounds
%
%
\begin{align}
 \|S^{-1}g\|_{\tilLE} &\lesssim \|g\|_{L_t^{p'}L_x^{q'}}\label{eq:Strtemp1}\\
 \|w\|_{L_t^pL_x^q} &\lesssim \|w\|_{L_t^\infty L_x^2}+\|Sw\|_{\tilLE^\ast}.\label{eq:Strtemp2}
\end{align}
%
Assuming \eqref{eq:Strtemp1} and \eqref{eq:Strtemp2}, write $u=v+S^{-1}F$, that is, define $v:=u-S^{-1}F$. Then we have  
\begin{align*}
\begin{split}
 S_\pert v= S_\pert u-S (S^{-1}F) -H_\lot(S^{-1}F)=-H_\lot(S^{-1}F),
\end{split}
\end{align*} 
and by Propositions~\ref{p:Hlotbound_for_V} and~\ref{p:Hlotbound} we obtain
\begin{align}\label{eq: temp 1}
\begin{split}
 \|S_\pert v\|_{\tilLE^\ast}\lesssim \|S^{-1}F\|_{\tilLE}. 
\end{split}
\end{align}
On the other hand by definition $v(0)=u(0)$, so from  \eqref{eq:Strtemp1} we get
\begin{align*}
\begin{split}
 \|v(0)\|_{L^2_x}+\|S_\pert v\|_{\tilLE^\ast} \lesssim \|u(0)\|_{L^2_x}+\|F\|_{L_t^{p_2'}L_x^{q_2'}}.
\end{split}
\end{align*}
Combining with equation \eqref{equ:approx_mass_conservation}, Theorem~\ref{t:LE-H} or~\ref{t:LE2}, and another application of \eqref{eq: temp 1} and \eqref{eq:Strtemp1}, this shows
\begin{align*}
\begin{split}
 \|v\|_{L_t^\infty L_x^2}+\|S_\pert v\|_{\tilLE^\ast}&\lesssim  \|u(0)\|_{L^2_x}+\|v\|_{\tilLE} +\|S_\pert v\|_{\tilLE^\ast}+\|F\|_{L_t^{p_2'}L_x^{q_2'}}\\
 &\lesssim \|u(0)\|_{L^2_x} +\|F\|_{L_t^{p_2'}L_x^{q_2'}}.
\end{split}
\end{align*}
It then follows from \eqref{eq:Strtemp2}, writing $S=S_\pert-H_{\lot}$, and further applications of our local smoothing estimates that 
\begin{align*}
\begin{split}
 \|v\|_{L_t^{p_1}L_x^{q_1}}&\lesssim \|u(0)\|_{L^2_x} +\|H_\lot v\|_{\tilLE^\ast} +\|F\|_{L_t^{p_2'}L_x^{q_2'}}\\
 &\lesssim \|u(0)\|_{L^2_x} +\|v\|_{\tilLE}+\|F\|_{L_t^{p_2'}L_x^{q_2'}}\\
 &\lesssim  \|u(0)\|_{L^2_x} + \|v\|_{L_t^\infty L_x^2}+\|S_\pert v\|_{\tilLE^\ast}+\|F\|_{L_t^{p_2'}L_x^{q_2'}}\\
 &\lesssim \|u(0)\|_{L^2_x} +\|F\|_{L_t^{p_2'}L_x^{q_2'}}.
\end{split}
\end{align*}
The desired estimate \eqref{eq: Strich} now follows from the triangle inequality applied to $u=v+S^{-1}f$ and the Strichartz estimates for $S$.

Let us turn to the proofs of \eqref{eq:Strtemp1} and \eqref{eq:Strtemp2}.
Estimate \eqref{eq:Strtemp2} is actually a consequence of \eqref{eq:Strtemp1} and duality. To see this let $h$ be an arbitrary function in~$L_t^{p'}L_x^{q'}$. Then recalling that $S=-i\partial_t- \Delta$ and integrating on $I\times \bbH^d$ with $I=[t_1,t_2]$ we have 
\begin{align*}
\begin{split}
| \angles{w}{h}_{t,x} |= |\angles{w}{S(S^{-1}h)}_{t,x}|\leq |\angles{Sw}{S^{-1}h}_{t,x}|+|\angles{w}{S^{-1}h}_x\vert_{t_1}^{t_2}|.
\end{split}
\end{align*}
But then from \eqref{eq:Strtemp1} and $\|S^{-1}h\|_{L_t^\infty L_x^2}\lesssim \|h\|_{L_t^{p'}L_x^{q'}}$ we get
\begin{align*}
\begin{split}
 |\angles{w}{h}_{t,x}|\lesssim \|h\|_{L_t^{p'}L_x^{q'}}(\|w\|_{L_t^\infty L_x^2}+\|Sw\|_{\tilLE^\ast}), 
\end{split}
\end{align*}
which by duality proves \eqref{eq:Strtemp2}.

Finally we prove \eqref{eq:Strtemp1}. For any fixed $T>0$, and any $\alpha:[0,4]\to\calA$, let $\calX_\alpha^0$ be as in Definition~\ref{def:X0spaces}. From the local smoothing estimate for $S$ we have 
\begin{align*}
\begin{split}
 \left\| i\int_0^T e^{i(t-s)\Delta}g(s)ds\right\|_{L_t^2 \calX_\alpha^0}&= \left\| ie^{it\Delta}\int_0^T e^{-is\Delta}g(s)ds\right\|_{L_t^2 \calX_\alpha^0}\\
 &\lesssim \left\|\int_0^Te^{-is\Delta}g(s)ds\right\|_{L_x^2} \lesssim \|g\|_{L_t^{p'}L_x^{q'}},
\end{split}
\end{align*}
where in the last step we have applied the usual Strichartz estimates for $S$. But then from Christ-Kiselev and the Duhamel representation of $S^{-1}$ in \eqref{eq: S-1 def} it follows that as long as $p'<2$
\begin{align*}
\begin{split}
 \| S^{-1}g\|_{L_t^2X_\alpha^0}\lesssim  \|g\|_{L_t^{p'}L_x^{q'}}.
\end{split}
\end{align*}
Estimate \eqref{eq:Strtemp1} follows by taking the supremum over all $\alpha \colon (0,4] \to \calA$ and Lemma~\ref{l:albeLE}.
\end{proof}

\subsection{Combined Functional Framework}  In this final section we prove Corollary~\ref{c:combined}, which establishes a combined functional framework for the Strichartz and local smoothing estimates. 

\begin{proof}[Proof of Corollary~\ref{c:combined}]
We prove the following estimate for any pair of Strichartz exponents $(p_1,q_1)$ and $(p_2,q_2)$:
\begin{align}\label{eq: combined}
\begin{split}
  \|u\|_{L_t^{p_1}L_x^{q_1}\cap \LE}\lesssim \|u(0)\|_{L_x^2}+\|Hu\|_{L_t^{p_2'}L_x^{q_2'}+\LE^\ast}.
\end{split}
\end{align}
where $H$ is as in the statement of Corollary~\ref{c:combined}, and we have omitted $P_c$ (without loss of generality) as in the previous section. 
Note that for this it suffices to prove
\begin{align}\label{eq: Y Z}
\begin{split}
 \|u\|_{Y}\lesssim \|u(0)\|_{L_x^2}+\|u\|_Z
 \end{split}
\end{align}
for any choice of $Y= L_t^{p_1}L_x^{q_1},~\LE$ and $Z=L_t^{p_2'}L_x^{q_2'},~\LE^\ast$.
Estimate \eqref{eq: Y Z} for the choices $(Y,Z)=(\LE,\LE^\ast)$ and $(Y,Z)=(L_t^{p_1}L_x^{q_1},L_t^{p_2'}L_x^{q_2'})$ were already proved in previous sections.  To prove the estimate for other choices of $Y$ and $Z$  we first introduce some notation. Let $e^{-itA}$ denote the solution operator for $H$, that is,
\begin{align*}
\begin{split}
  H(e^{-itA}u_0)=(\frac{1}{i}\partial_t+A)(e^{-itA}u_0)=0,\qquad (e^{-itA}f)\vert_{t=0}=u_0.
\end{split}
\end{align*}
It follows that the solution to the inhomogeneous problem $Hu=f$, $u(0)=u_0$ is given by
\begin{align*}
\begin{split}
 u(t)= u_h+u_p:=e^{-itA}u_0+i\int_0^te^{-i(t-s)A}f(s)ds.
\end{split}
\end{align*} 
Since $Hu_h=0$ estimate \eqref{eq: Y Z} with $u$ replaced by $u_h$ and for any choice of $Y$ and $Z$ follows from the local smoothing and Strichartz estimates of the previous sections, so it suffices to consider $u_p$. The estimate for $(Y,Z)=(\LE,L_t^{p_2'}L_x^{q_2'})$ is similar to the proof of \eqref{eq:Strtemp1} in the previous section. Indeed for any time $T>0$ and $\alpha:[s_0,\infty)\to \calA$, by the local smoothing and Strichartz estimates for $H$, 
\begin{align*}
\begin{split}
   \left\| i\int_0^T e^{-i(t-s)A}f(s)ds\right\|_{L_t^2 X_\alpha^0}&= \left\| ie^{-itA}\int_0^T e^{isA}f(s)ds\right\|_{L_t^2 X_\alpha^0}\lesssim \left\|\int_0^Te^{isA}f(s)ds\right\|_{L_x^2}\\
   &\lesssim \|f\|_{L_t^{p'}L_x^{q'}}.
\end{split}
\end{align*}
Applying Christ-Kiselev we see that the same estimate holds with the left hand side replaced by $\|u_p\|_{L_t^2X_\alpha^0}$. Taking the supremum over $\alpha:[s_0,\infty)\to\calA$ proves \eqref{eq: Y Z} for the choice $(Y,Z)=(\LE,L_t^{p_2'}L_x^{q_2'})$ . Finally, we consider the case $(Y,Z)=(L_t^{p_1}L_x^{q_1},\LE^\ast)$. For $T$ and $\alpha$ as above and $I:=[0,T]$ we recall the local smoothing estimate
\begin{align*}
\begin{split}
 \|e^{-itA}u_0\|_{L_t^2X_\alpha^0(I\times \bbH^d)}\lesssim \|u_0\|_{L_x^2}. 
\end{split}
\end{align*}
By duality, this implies
\begin{align*}
\begin{split}
\left\| \int_0^Te^{isA}g(s)ds\right\|_{L_x^2}\lesssim \|g\|_{L_t^2(X_\alpha^0)^\ast(I\times\bbH^d)}.
\end{split}
\end{align*}
Combining with the Strichartz estimates for $H$ we then get
\begin{align*}
\begin{split}
 \left\|\int_0^Te^{-i(t-s)A}f(s)ds\right\|_{L_t^{p_1}L_x^{q_1}(I\times\bbH^d)}=&  \left\|e^{-itA}\int_0^Te^{isA}f(s)ds\right\|_{L_t^{p_1}L_x^{q_1}(I\times\bbH^d)}\\
 \lesssim &\left\|\int_0^Te^{isA}f(s)ds|\right\|_{L_x^2}\lesssim \|f\|_{L_t^2(X_\alpha^0)^\ast(I\times\bbH^d)}.
\end{split}
\end{align*}
Applying the standard Christ-Kiselev Lemma~\cite{CK01} we conclude that for $p_1>2$ 
\begin{align*}
\begin{split}
 \|u_p\|_{L_t^{p_1}L_x^{q_1}(I\times\bbH^d)}\lesssim \|f\|_{L_t^2(X_\alpha^0)^\ast(I\times\bbH^d)}.
\end{split}
\end{align*}
Taking the infimum over $\alpha:[s_0,\infty)\to\calA$ completes the proof of \eqref{eq: Y Z} for $(Y,Z)=(L_t^{p_1}L_x^{q_1},LE^\ast)$, and hence of \eqref{eq: combined}.
\end{proof}

\bibliographystyle{plain}
\bibliography{researchbib}

\bigskip 

\centerline{\scshape Andrew Lawrie}
\smallskip
{\footnotesize
 \centerline{Department of Mathematics, Massachusetts Institute of Technology}
\centerline{77 Massachusetts Ave, 2-267, Cambridge, MA 02139, U.S.A.}
\centerline{\email{ alawrie@mit.edu}}
} 

\medskip

\centerline{\scshape Jonas L\"uhrmann}
\smallskip
{\footnotesize
 \centerline{Department of Mathematics, Texas A\&M University}
 \centerline{Blocker 218B, College Station, TX 77843-3368, U.S.A.}
 \centerline{\email{luhrmann@math.tamu.edu}}
}

\medskip

\centerline{\scshape Sung-Jin Oh}
\smallskip
{\footnotesize
 \centerline{Department of Mathematics, UC Berkeley}
\centerline{Evans Hall 970, Berkeley, CA 94720-3840, U.S.A.}
\centerline{\email{sjoh@math.berkley.edu}}
} 

\medskip

\centerline{\scshape Sohrab Shahshahani}
\smallskip
{\footnotesize
\centerline{Department of Mathematics, University of Massachusetts, Amherst}
\centerline{710 N. Pleasant Street, Amherst, MA 01003-9305, U.S.A.}
\centerline{\email{sohrab@math.umass.edu}}
}

\end{document}